\documentclass[10pt,leqno]{amsart}
\usepackage{graphicx}
\usepackage{indentfirst,csquotes}

\usepackage{amssymb,amsthm,amsmath}
\usepackage{braket}
\usepackage{xfrac,mathtools}
\usepackage{paralist,hyperref,fancyhdr,etoolbox}
\usepackage[dvipsnames]{xcolor}
\usepackage[]{paralist}
\usepackage[]{enumerate,enumitem}
\usepackage{tikz}
\usepackage{tikz-cd}
\usetikzlibrary{hobby,patterns,positioning}
\usetikzlibrary{decorations.pathreplacing}
\usepackage{theoremref}
\newtheorem{theorem}{Theorem}[section]
\newtheorem{lemma}[theorem]{Lemma}
\newtheorem{proposition}[theorem]{Proposition}
\newtheorem{corollary}[theorem]{Corollary}

\newtheorem*{theorem*}{Theorem}
\newtheorem*{corollary*}{Corollary}
\theoremstyle{definition}
\newtheorem{definition}[theorem]{Definition}
\theoremstyle{remark}
\newtheorem{example}[theorem]{Example}
\newtheorem{remark}[theorem]{Remark}
\DeclarePairedDelimiter{\abs}{\lvert}{\rvert}

\newcommand{\Z}{\mathbb{Z}}

\newcommand{\Modd}{\operatorname{Mod}}

\begin{document}
\title{A presentation of the even spin mapping class group} 
\author[Filippo Bianchi]{Filippo Bianchi}
\address{Dipartimento di Matematica, Università di Pisa, Pisa, Italy}
\email{fil.bianchi5@gmail.com}
\begin{abstract}
We define a cell complex with an action of the even spin mapping class group, and use it to obtain a finite presentation. We also obtain a finite presentation with Dehn twist generators. 
\end{abstract}
\maketitle


\section{Introduction}

Every closed orientable surface $\Sigma_g$ of genus $g$ admits a spin structure. The group of self-diffeomorphisms of $\Sigma_g$ acts on the set of spin structures by pull-back, and this induces an action of the mapping class group $\Modd(\Sigma_g)$. The stabilizer of some spin structure $\xi$ under this action is the \emph{spin mapping class group} $\Modd(\Sigma_g)[\xi]$. The conjugacy class of $\Modd(\Sigma_g)[\xi]$ only depends on the $\sfrac{\Z}{2\Z}$-valued Arf invariant associated to $\xi$. In this paper, we will be primarily concerned with the even spin mapping class group.

Spin mapping class groups first appeared in the study of moduli spaces of Riemann surfaces with spin structures. Harer \cite{har1,har2} and Sierra \cite{sierra} computed their low-dimensional homology. More generally, Sipe \cite{sip1,sip2} considered the stabilizers of $r$-spin structures, i.e. $r$-th rooths of the canonical bundle. The homology of the corresponding stabilizers was investigated by Randal-Williams \cite{rw1,rw2}.

“Classical" spin mapping class groups have found applications in $4$-manifold topology. Indeed, by Stipsicz \cite{stip}, the monodromy of spin Lefschetz fibrations is a product of Dehn twists that stabilize a fixed spin structure on the regular fiber; see for example \cite[Section 2]{bh} for more details. On the other hand, (higher) spin mapping class group naturally appear in certain monodromy problems in algebraic geometry. In this context, Salter \cite{salt} and Calderon--Salter \cite{cs1,cs2} recently proved that these groups are generated by Dehn twists, and provided explicit generating sets. Their results were improved by Hamenstädt \cite{ham} for classical spin mapping class groups, where finite generating sets had already been found by Hirose \cite{hir1,hir2}.

The main result of this paper is the first finite presentation of the even spin mapping class group.

\begin{theorem*}[see \thref{thm:presDC}]
If $g \ge 3$, the even spin mapping class group $\Modd(\Sigma_g)[\xi]$ admits a finite presentation with Dehn twist generators, and the following relations:
\begin{enumerate}
\item commutators and braid relations between the generators;
\item a hyperelliptic relation of genus $3$;
\item various relations that are products of lanterns with total exponent $0$;
\item various relations that are products of a $3$-chain and some lanterns with total exponent $6$.
\end{enumerate}
\end{theorem*}

The spin mapping class group is not generated by Dehn twists for $g=1,2$ (see \cite{ham} and Remark \ref{rem:gdt}). In the even case, Hamenstädt found a generating set of Dehn twists for $g \ge 4$. We establish generation by Dehn twists also for $g=3$, although by Hamenstädt's results the generating set cannot be admissible in this case, i.e. the intersection graph of the corresponding curves has cycles. Our generating set coincides with Hamenstädt's for $g=4$, but is different in higher genus, and has a bigger cardinality.

As a corollary, we compute the abelianization of $\Modd(\Sigma_g)[\xi]$, recovering the results of Harer, Randal-Williams and Sierra. 

\begin{corollary*}[see Corollaries \ref{cor:abh} and \ref{cor:ab}]
The abelianization of the even spin mapping class group is
\[
H_1\big(\Modd(\Sigma_g)[\xi]\big)\cong \begin{cases}
\Z\oplus\sfrac{\Z}{4\Z} & \text{if $g=1$},\\
\Z\oplus\sfrac{\Z}{2\Z} & \text{if $g=2$},\\
\sfrac{\Z}{4\Z} & \text{if $g\ge3$}.
\end{cases}
\]
\end{corollary*}

This agrees with a conjecture of Ivanov \cite{iva}, which predicts that all finite-index subgroups of $\Modd(\Sigma_g)$ have finite abelianization if $g\ge3$. Notice that by a result of Putman \cite{put}, $\Modd(\Sigma_g)[\xi]$ cannot be a counterexample to Ivanov's conjecture, as it contains the Johnson kernel. For $g=2$, our calculation agrees with a result of Taherkhani \cite{tah}; in particular, it follows that $\Modd(\Sigma_2)[\xi]$ is conjugate to the group $H_7$ of \cite[Table 1]{tah}.

We obtain our presentation of $\Modd(\Sigma_g)[\xi]$ via the strategy of Hatcher-Thurston \cite{ht}, as seen through Wajnryb's combinatorial perspective \cite{waj:elem}. Namely, in Section \ref{scc} we construct a 2-dimensional cell complex $X_g$ with an action of $\Modd(\Sigma_g)[\xi]$, and prove that it is connected and simply connected. Then, a presentation of  $\Modd(\Sigma_g)[\xi]$ is obtained from a presentation of the stabilizer of a vertex, adding generators prescribed by the 1-skeleton of $X_g$ and relations prescribed by the $2$-skeleton. This program is carried out in Section \ref{fp}. Finally, in Section \ref{dt} we apply Tietze moves to obtain a presentation with Dehn twist generators. 

Our complex $X_g$ is inspired by Hatcher-Thurston's cut-system complex, but it has three key novelties. First, the vertices are cut-systems of curves with prescribed spin value. Second, there are two types of edges, with different intersection patterns. Finally, in addiction to triangles, squares and pentagons, there is a fourth kind of 2-cell, which we call \emph{hyperelliptic face}. In a forthcoming paper \cite{fb}, we interpret the presence of this extra 2-cell from a 4-dimensional perspective, using the presentation of the spin mapping class group to give a new proof of a classical theorem of Rokhlin \cite{rokh} on the signature of spin 4-manifolds.

It is easy to see that an even spin structure on $\Sigma_g$ extends to some handlebody $H_g$ bounded by $\Sigma_g$. As a byproduct of our construction, we obtain a finite presentation of the \emph{spin handlebody mapping class group} $\Modd(H_g)[\xi]:=\Modd(H_g) \cap \Modd(\Sigma_g)[\xi]$ (see \thref{thm:mhgq}).

\vspace{10pt}
\textbf{Acknowledgements.} The author wishes to thank Riccardo Giannini for his help during the first stages of this project.

\section{Spin mapping class groups}

In this section, we recall some basic facts about higher spin structures on surfaces and their stabilizers. The focus is on classical spin structures, as they will be our sole concern. For a more general treatment, we refer to the papers of Salter \cite{salt} and Calderon--Salter \cite{cs1}.

\subsection{Spin structures}

Fix a surface $\Sigma_g^b$ of genus $g$ with $b$ boundary components. We denote by $\mathcal{C}$ the set of isotopy classes of oriented simple closed curves on $\Sigma_g^b$. The following definition originates in the work of Humphries-Johnson \cite{hujo}.

\begin{definition}
An $r$-spin structure on $\Sigma_g^b$ is a map $\phi \colon\mathcal{C} \rightarrow \sfrac{\Z}{r\Z}$ such that:
\begin{enumerate}
\item $\phi(t_c(d))=\phi(d)+(d\cdot c)\,\phi(c)$ for every $c,d \in \mathcal{C}$, where $t_c$ denotes the Dehn twist along $c$ and $d\cdot c$ is the algebraic intersection number of $c$ and $d$ (\emph{twist linearity});
\item if the union of $c_1,\dots,c_m\in \mathcal{C}$ is the oriented boundary of a subsurface $S \subset \Sigma_g^b$, then $\sum \phi(c_i)=\chi(S)$ (\emph{homological coherence}).
\end{enumerate}
\end{definition}

\begin{remark}
\thlabel{rem:hj}
For closed surfaces, we can give an alternate definition as follows (see \cite{hujo} and \cite{salt}). Denote by $\pi \colon UT\Sigma_g\rightarrow \Sigma_g$ the unit tangent bundle of $\Sigma_g$. The inclusion of the fiber $i \colon S^1 \rightarrow UT\Sigma_g$ induces a short exact sequence
\[
0 \longrightarrow \sfrac{\Z}{r\Z} \overset{i_*}{\longrightarrow} H_1\left(UT\Sigma_g;\sfrac{\Z}{r\Z}\right) \overset{\pi_*}{\longrightarrow} H_1\left(\Sigma_g;\sfrac{\Z}{r\Z}\right) \longrightarrow 0.
\]
An $r$-spin structure is a class $\xi \in H^1(UT\Sigma_g;\sfrac{\Z}{2\Z})$ that evaluates to $1$ on the image of a generator of $\sfrac{\Z}{r\Z}$. Since 
\[
H^1(UT\Sigma_g;\sfrac{\Z}{r\Z})\cong \operatorname{Hom}(\Z^{2g}\oplus\sfrac{\Z}{(2g-2)\Z},\sfrac{\Z}{r\Z}),
\]
an $r$-spin structure exists if and only if $r$ divides $2g-2$. For $r=2$, this recovers Milnor's definition of spin structure \cite{mil:spin}.
\end{remark}

The case $r=2$ is special in many respects. 

\begin{theorem}[Johnson \cite{john0}]
\thlabel{thm:xiq}
Let $\phi$ be a $2$-spin structure on $\Sigma_g^b$. Then:
\begin{enumerate}
\item $\phi$ factors through the natural map $\mathcal{C} \rightarrow H_1(\Sigma_g^b;\sfrac{\Z}{2\Z})$, and we denote again by $\phi$ the induced map $H_1(\Sigma_g^b;\sfrac{\Z}{2\Z})\rightarrow \sfrac{\Z}{2\Z}$;
\item $q_{\phi}:=\phi+1$ is a quadratic enhancement of the intersection form, i.e. $q_{\phi}(a+b)=q_{\phi}(a)+q_{\phi}(b)+a\cdot b$ for all $a,b \in H_1(\Sigma_g^b;\sfrac{\Z}{2\Z})$;
\item the assignment $\phi \mapsto q_{\phi}$ defines a bijection between the set of $2$-spin structures on $\Sigma_g$ and the set of quadratic enhancements on $H_1(\Sigma_g^b;\sfrac{\Z}{2\Z})$.
\end{enumerate}
\end{theorem}

\begin{definition}
Let $\phi$ be an $r$-spin structure on $\Sigma_g^b$. If $r$ is even, the natural map $\sfrac{\Z}{r\Z}\rightarrow \sfrac{\Z}{2\Z}$ defines an associated $2$-spin structure $\overline{\phi}$. The \emph{Arf invariant} of $\phi$ is the Arf invariant of the corresponding quadratic enhancement $q_{\bar{\phi}}$. Explicitly, if $\{x_1,y_1,\dots,x_g,y_g\}$ is a symplectic basis for $H_1(\Sigma_g^b;\Z)$, we have
\[
\operatorname{Arf}(\phi):=\sum_{i=1}^g\big(\phi(x_i)+1\big)\big(\phi(y_i)+1\big) \quad (\operatorname{mod} \,2).
\]
We say that $\phi$ is \emph{even} or \emph{odd} according to the parity of $\operatorname{Arf}(\phi)$.
\end{definition}

The following theorem records some useful criteria for comparing different $r$-spin structures. Notice that $\operatorname{Mod}(\Sigma_g^b)$ acts naturally on the set of $r$-spin structures by $(f \cdot \phi)(c):=\phi(f^{-1}(c))$.

\begin{theorem}[\cite{hujo}, \cite{salt}]
\thlabel{thm:phipsi}
Let $\phi,\psi$ be two $r$-spin structures on $\Sigma_g^b$. Then:
\begin{enumerate}
\item $\phi=\psi$ if and only if they agree on a basis of $H_1(\Sigma_g^b;\Z)$;
\item if $b=0$, $\phi$ and $\psi$ are in the same $\operatorname{Mod}(\Sigma_g)$-orbit if and only if $r$ is odd or $\operatorname{Arf}(\phi)=\operatorname{Arf}(\psi)$.
\end{enumerate}
\end{theorem}

\begin{proof}[Proof (for $r=2$)]
It is well known that quadratic enhancements on a $\sfrac{\Z}{2\Z}$-vector space equipped with a nondegenerate symplectic pairing are completely determined by their value on a basis, and are completely classified up to automorphisms by their Arf invariant. \qedhere 
\end{proof}

\subsection{Operations on curves and surfaces}

We will often need to construct curves with certain properties and perform cut and paste operations on surfaces. In the spin context, this requires some extra care.  

We first introduce two operations on curves, following \cite[Subsection 3.2]{salt}. The \emph{smoothing} of a family of oriented curves is the multicurve obtained by resolving all intersections according to the orientations. If $\alpha$ and $\beta$ are curves with $\alpha\cdot \beta=1$, then the smoothing of $k$ copies of $\alpha$ and $\ell$ copies of $\beta$ has $\operatorname{gcd}(k,\ell)$ components. 

The \emph{arc sum} of two disjoint curves $\gamma$ and $\delta$ along an arc $c$ connecting them is the simple closed curve $\gamma+_c\delta$ that bounds a tubular neighborhood of the union $\gamma \cup c \cup \delta$ along with $\gamma$ and $\delta$. Clearly, its homology class satisfies $[\gamma+_c\delta]=[\gamma]+[\delta]$.

\begin{lemma}[{\cite[Lemmas 3.11 and 3.13]{salt}}]
\thlabel{lem:operations}
Consider two curves $\alpha,\beta$ on a spin surface $(\Sigma_g,\phi)$.
\begin{enumerate}
\item If $\gamma$ is the smoothing of $k$ copies of $\alpha$ and $\ell$ copies of $\beta$, then $\phi(\gamma)=k\phi(\alpha)+\ell\phi(\beta)$.
\item If $\alpha$ and $\beta$ are disjoint and $c$ is an arc connecting them, then $\phi(\alpha+_c\beta)=\phi(\alpha)+\phi(\beta)+1$.
\end{enumerate}
\end{lemma}

The next proposition describes the effect of cutting a $2$-spin surface on its Arf invariant. Here and elsewhere, we will assume that the unique spin structure on $S^2$ has Arf invariant zero. 

\begin{proposition}[Additivity of the Arf invariant]
Let $(\Sigma_g,\phi)$ be a $2$-spin surface, and consider a set of curves $\{\alpha_1,\dots,\alpha_n\}$ whose union separates $\Sigma_g$ into some subsurfaces $S_1,\dots,S_k$. Call $\phi_i$ the pullback spin structure on $S_i$. If $\phi(\alpha_j)=1$ for all $j$, then $\operatorname{Arf}(\phi)=\sum_i \operatorname{Arf}(\phi_i)$.
\end{proposition}

\begin{proof}
Fix a geometric symplectic basis $B_i$ for each $S_i$, then glue along $\alpha_j$ for $j=1,\dots,g$, and call $J$ the set of indices such that gluing along $\alpha_j$ for $j \in J$ produces new genus. Now, complete $\bigcup_i B_i \cup \{\alpha_j\}_{j \in J}$ to a geometric symplectic basis $B$, and compute the Arf invariant with respect to $B$. \qedhere
\end{proof}

\begin{corollary}
\thlabel{cor:cut}
Let $(\Sigma_g,\phi)$ be a $2$-spin surface. If $\alpha \subset \Sigma_g$ is a curve with $\phi(\alpha)=1$, then the pullback spin structure on $\Sigma \setminus \alpha$ has the same parity as $\phi$. 
\end{corollary}

\begin{remark}
Notice that the Arf invariant is not additive if we glue along curves with spin value $0$. For example, cut a $2$-spin torus along a curve $\gamma$ with $\phi(\gamma)=0$, obtaining an annulus. Then, the Arf invariant of the torus is decided by any curve that intersects $\gamma$ once, but we cannot read its spin value on the annulus.
\end{remark}

In the following, we are often going to consider $2$-spin surfaces that arise from cutting procedures. \thref{cor:cut} motivates the following standing assumption.

\begin{remark}[Surfaces with boundary]
\thlabel{rem:sgk}
Our spin structures on $\Sigma_g^b$ will always satisfy $\phi(\delta)=1$ for every boundary component $\delta$. In other words, we will only consider spin structures that extend to the surface $\Sigma_g$ obtained by capping all boundary components with disks.

Note that this choice is not standard: see for example \cite[Theorem 5.1]{bhm}. 
\end{remark}

\subsection{Stabilizer subgroups} Recall that $\Modd(\Sigma_g^b)$ acts naturally on the set of $r$-spin structures. 

\begin{definition}
The \emph{$r$-spin mapping class group} $\Modd(\Sigma_g^b)[\phi]$ is the stabilizer of an $r$-spin structure $\phi$ under the action of $\Modd(\Sigma_g^b)$.
\end{definition}

Clearly, the $r$-spin mapping class group is a finite index subgroup of $\Modd(\Sigma_g^b)$. If $b=0$, as a consequence of \thref{thm:phipsi}(2), the $r$-spin mapping class group is unique up to conjugation if $r$ is odd. If instead $r$ is even, there are exactly two conjugation classes, that are defined to be \emph{even} or \emph{odd} according to the parity of the induced $2$-spin structure.

We now introduce some important classes of elements of the $r$-spin mapping class group. Denote by $t_{\delta}$ the Dehn twist along $\delta$. Consider curves $\alpha,\beta,\gamma$ that bound a pair of pants on $\Sigma_g^b$. By homological coherence, we have $\phi(\alpha)+\phi(\beta)+\phi(\gamma)=-1$. Assume that $b:=\phi(\beta)=-\phi(\alpha)$; then $\phi(\gamma)=-1$. We say that $t_{\alpha}t_{\beta}^{-1}t_{\gamma}^b$ is a \emph{fundamental multitwist}.

\begin{lemma}[{\cite[Lemmas 3.15 and 3.18]{salt}}]
\thlabel{lem:elements}
Let $(\Sigma_g^b,\phi)$ be an $r$-spin surface.
\begin{enumerate}
\item Separating twists always preserve $\phi$. 
\item If $\delta$ is a nonseparating curve, then $t_{\delta}^k$ preserves $\phi$ if and only if $k\phi(\delta)\equiv0 \,(\operatorname{mod}\,r)$. 
\item Fundamental multitwists preserve $\phi$. 
\end{enumerate}
\end{lemma}

As a consequence of \thref{lem:elements}, we see that the only nonseparating Dehn twists contained in $\Modd(\Sigma_g^b)[\phi]$ are those along curves with spin value $0$. We say that such curves and the corresponding twists are \emph{admissible}. 

\begin{theorem}[Salter \cite{salt}, Calderon-Salter \cite{cs1}, Hamenstädt \cite{ham}]
    The $r$-spin mapping class group is generated by admissible twists if the genus is sufficiently high.
\end{theorem}

\begin{remark}
\thlabel{rem:gdt}
Consider now the case $r=2$. \thref{lem:elements} gives us two classes of elements of $\Modd(\Sigma_g^b)[\phi]$ that are not products admissible twists: namely, squared Dehn twists along curves with spin value $1$ and fundamental multitwists with $b=1$. We now explain a way of factoring these elements as products of admissible twists. This, along with the results of Hirose \cite{hir1,hir2}, can be used to give a short proof of the above theorem.

Consider the lantern $t_{\beta}t_{z_1}t_{y_1}=t_{\alpha}t_{\gamma} t_{d_1}t_{d_2}$ of Figure \ref{fig:lants}. Rearranging, we obtain
\begin{equation}
\label{eq:fml}
t_{\beta}t_{\alpha}^{-1}t_{\gamma}^{-1}=t_{d_1}t_{d_2}t_{y_1}^{-1}t_{z_1}^{-1},
\end{equation}
and this is a factorization of the fundamental multitwist $t_{\beta}t_{\alpha}^{-1}t_{\gamma}^{-1}$.

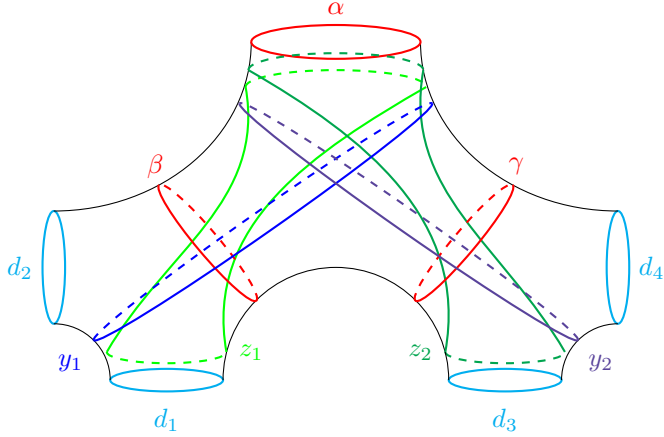
\begin{figure}
\centering
\begin{tikzpicture}[scale=.75]
\draw [cyan, thick] (0,0) arc (180:-180:1 and .2);
\draw [cyan, thick] (6,0) arc (180:-180:1 and .2);
\draw [green, thick, dashed] (-.06,.5) arc (180:360:1.06 and .2);
\draw [Green, thick, dashed] (5.94,.5) arc (180:360:1.06 and .2);
\draw [cyan, thick] (-1,1) arc (-90:270:.2 and 1);
\draw [cyan, thick] (9,1) arc (-90:270:.2 and 1);
\draw (0,0) arc (0:90:1 and 1);
\draw (8,0) arc (180:90:1 and 1);
\draw (2,0) arc (180:0:2 and 2);
\draw (-1,3) arc (-90:0:3.5 and 3);
\draw (9,3) arc (270:180:3.5 and 3);
\draw [red, thick] (2.5,6) arc (180:-180:1.5 and .3);
\draw [green, thick, dashed] (2.4,5.2) arc (180:0:1.6 and .3);
\draw [Green, thick, dashed] (2.45,5.5) arc (180:0:1.55 and .3);
\draw [green, thick] (-.06,.5) to [out=60, in=-80] (2.4,5.2);
\draw [green, thick] (2.06,.5) to [out=100, in=-150] (5.6,5.2);
\draw [Green, thick] (8.06,.5) to [out=120, in=-100] (5.55,5.5);
\draw [Green, thick] (5.94,.5) to [out=80, in=-30] (2.45,5.5);
\draw [red, thick] (2.59,1.41) {[rotate=40] arc[x radius=.2, y radius=1.34, start angle=270, end angle=90]};
\draw [red, thick, dashed] (2.59,1.41) {[rotate=40] arc[x radius=.2, y radius=1.34, start angle=-90, end angle=90]};
\draw [red, thick, dashed] (5.41,1.41) {[rotate=-40] arc[x radius=.2, y radius=1.34, start angle=270, end angle=90]};
\draw [red, thick] (5.41,1.41) {[rotate=-40] arc[x radius=.2, y radius=1.34, start angle=-90, end angle=90]};
\draw [blue, thick] (-.29,.71) {[rotate=-55] arc[x radius=.2, y radius=3.66, start angle=-90, end angle=90]};
\draw [blue, thick, dashed] (-.29,.71) {[rotate=-55] arc[x radius=.2, y radius=3.66, start angle=270, end angle=90]};
\draw [Violet, thick, dashed] (8.29,.71) {[rotate=55] arc[x radius=.2, y radius=3.66, start angle=-90, end angle=90]};
\draw [Violet, thick] (8.29,.71) {[rotate=55] arc[x radius=.2, y radius=3.66, start angle=270, end angle=90]};
\node [cyan] at (1,-.7) {$d_1$};
\node [cyan] at (-1.6,2) {$d_2$};
\node [cyan] at (7,-.7) {$d_3$};
\node [cyan] at (9.6,2) {$d_4$};
\node [red] at (4,6.6) {$\alpha$};
\node [red] at (.8,3.8) {$\beta$};
\node [red] at (7.2,3.8) {$\gamma$};
\node [blue] at (-.7,.3) {$y_1$};
\node [Violet] at (8.7,.3) {$y_2$};
\node [green] at (2.5,.5) {$z_1$};
\node [Green] at (5.5,.5) {$z_2$};
\end{tikzpicture}
\caption{Factoring squared twists and fundamental multitwists as products of admissible twists. All curves but the red ones are admissible.}
\label{fig:lants}
\end{figure}

Similarly, we can factor $t_{\gamma}t_{\beta}^{-1}t_{\alpha}^{-1}$ using the lantern $t_{\gamma}t_{y_2}t_{z_2}=t_{\alpha}t_{\beta}t_{d_3}t_{d_4}$ of Figure \ref{fig:lants}. Thus, we obtain
\begin{equation}
\label{eq:all}
t_{\alpha}^2=(t_{\alpha}^{-1}t_{\beta}t_{\gamma}^{-1})^{-1}(t_{\gamma}t_{\beta}^{-1}t_{\alpha}^{-1})^{-1}=t_{z_1}t_{y_1}t_{d_1}^{-1}t_{d_2}^{-1}t_{y_2}t_{z_2}t_{d_3}^{-1}t_{d_4}^{-1}.
\end{equation} 

Notice that the configuration of Figure \ref{fig:lants} embeds in a closed surface of genus at least $3$, if we require that the image of $\alpha$ be nontrivial.
\end{remark}

\subsection{Spin change of coordinates} The change of coordinates principle \cite[Section 1.3]{primer} can be roughly stated as follows: the mapping class group acts transitively on sets of curves with the same intersection pattern. We will use repeatedly a spin version of this tool, where the curves are also required to have the same spin values. We now illustrate this principle in a series of examples which are relevant for us, working with $2$-spin structures on closed surfaces for simplicity. See \cite[Section 4]{salt} and \cite[Subsection 5.2]{cs1} for a more general treatment. 

\subsubsection*{Geometric symplectic bases}

Let $\mathcal{B}:=\{\alpha_i,\beta_i\}$ and $\mathcal{B}':=\{\alpha_i',\beta_i'\}$ be two geometric symplectic bases for $\Sigma_g$, and assume that $\phi(\alpha_i)=\phi(\alpha_i')$ and $\phi(\beta_i)=\phi(\beta_i')$ for all $i$. By the usual change of coordinates principle, there exists a mapping class $f$ such that $f(\alpha_i)=\alpha_i'$ and $f(\beta_i)=\beta_i'$ for every $i$. By \thref{thm:phipsi}, $f$ fixes $\phi$.

\subsubsection*{Cut-systems}

Recall that a cut-system $\Braket{\alpha_1,\dots,\alpha_g}$ on $\Sigma_g$ is an unordered $g$-tuple of disjoint simple closed curves whose homology classes are linearly independent. Let $\Braket{\alpha_1',\dots,\alpha_g'}$ be another cut-system with $\phi(\alpha_i)=\phi(\alpha_{i}')$ for each $i$.

Complete the cut-systems to geometric symplectic bases $\mathcal{B}:=\{\alpha_i,\beta_i\}$ and $\mathcal{B}':=\{\alpha_i',\beta_i'\}$. Call $e_i$ the spin value of $\alpha_i$ and $\alpha_i'$. If $e_i=1$, then up to replacing $\beta_i'$ with $t_{\alpha_i'}(\beta_i')$ we may assume that $\phi(\beta_i)=\phi(\beta_i')$. 

Call $I$ the set of indices $i$ such that $e_i=0$. Since $\operatorname{Arf}(\phi)$ does not depend on the choice of a basis, the subsets 
\[
J:=\big\{j \in I\big|\phi(\beta_j)=0,\ \phi(\beta_j')=1\big\}, \quad 
J':=\big\{j \in I\big|\phi(\beta_j)=1,\ \phi(\beta_j')=0\big\}
\]
both have an even number of elements. We modify $\mathcal{B}'$ as follows: given $j_1,j_2 \in J'$, let $\gamma$ be the arc sum of $\alpha_{j_1}'$ and $\alpha_{j_2}'$ along an arc disjoint from all the other curves of $\mathcal{B}'$. Then $\phi(\gamma)=1$ by \thref{lem:operations}(2), and we can substitute $\beta_{j_1}'$ by $t_{\gamma}(\beta_{j_1}')$ and $\beta_{j_2}'$ by $t_{\gamma}(\beta_{j_2}')$. We perform this operation until $J'$ is empty, and we do the same for $J$. Now $\phi$ agrees on the two bases, and by the above we find a mapping class $f \in \Modd(\Sigma_g)[\phi]$ such that $f(\alpha_i)=\alpha_i'$ for every $i$.

In particular, $\Modd(\Sigma_g)[\phi]$ acts transitively on curves with the same spin value and on partial cut-systems with fixed spin values.

\subsubsection*{Chains} Recall that an $n$-chain $(\gamma_1,\dots,\gamma_n)$ is a set of curves such that $\abs{\gamma_i\cap \gamma_{i+1}}=1$ for every $i$ and $\gamma_i\cap\gamma_j=\emptyset$ if $\abs{i-j} \ne 1$. It is easy to see that a tubular neighborhood of $\gamma_1\cup\dots\cup\gamma_n$ has two boundary components if $n$ is odd, and a single boundary component if $n$ is even.

Let $(\gamma_1',\dots,\gamma_n')$ be another $n$-chain, and assume that $\phi(\gamma_i)=\phi(\gamma_i')$ for every $i$. Moreover, if $n$ is odd, assume that $\Sigma\setminus \bigcup_i\gamma_i$ is homeomorphic to $\Sigma\setminus \bigcup_i\gamma_i'$, and if they are disconnected, that the induced spin structures on corresponding components have the same Arf invariant. Then there exists an element $f$ of $\Modd(\Sigma_g)[\phi]$ such that $f(\gamma_i)=\gamma_i'$ for every $i$. 

To see this, construct two geometric symplectic bases $\mathcal{B}=\set{\alpha_i,\beta_i}$ and $\mathcal{B}'=\set{\alpha_i',\beta_i'}$ of $\Sigma$ as follows. Set $\beta_k:=\gamma_{2k}$ for all $k$. Orient each $\gamma_i$ so that $\gamma_i\cdot \gamma_{i+1}=1$ for all $i$. Define inductively $\alpha_k$ as follows:
\[
\alpha_1=\gamma_1, \qquad \alpha_{k+1}=\alpha_k+_{c_k}\gamma_{2k+1},
\]
where $c_k$ is the arc of $\gamma_{2k}$ that goes from $\gamma_{2k}\cap\alpha_k$ to $\gamma_{2k}\cap \gamma_{2k+1}$. Now, complete $\set{\alpha_k,\beta_k}$ to a geometric symplectic basis $\mathcal{B}$ on the whole of $\Sigma_g$, in such a way that $\mathcal{B}$ restricts to a geometric symplectic basis on every component of $\Sigma_g \setminus \bigcup_i \gamma_i$. Define similarly $\mathcal{B}'$. 

Now, by construction $\phi(\alpha_i)=\phi(\alpha_i')$ and $\phi(\beta_i)=\phi(\beta_i')$ if $2i \le n$, and by invariance of $\operatorname{Arf}(\phi)$ and the same reasoning as before we may assume that this holds for all $i$. Again, we conclude by the usual change of coordinates principle.

\section{The spin cut-system complex}
\label{scc}

In this section, we define the spin cut-system complex $X_g$ and prove that it is connected and simply connected for every $g \ge 1$. The complex $X_g$ is inspired by Hatcher and Thurston's cut-system complex \cite{ht}. Recall that the vertices of the cut-system complex are cut-systems, while edges and faces are determined by conditions on the intersections between curves in two or more cut-systems. Throughout this section, $\phi$ will be a fixed even $2$-spin structure on $\Sigma_g^b$. If $\phi(\gamma)=\epsilon$, we will say that $\gamma$ is an $\epsilon$-curve.

\subsection{Definition and first properties} Consider a surface $\Sigma_g^b$.

\begin{definition}
The \emph{spin cut-system complex} is the 2-dimensional cell complex $X_g$ defined as follows.
\begin{itemize}
\item[\textbf{--}] The vertices are isotopy classes of cut-systems of $1$-curves.
\item[\textbf{--}] An edge connects two vertices $\Braket{\alpha_1,\dots,\alpha_g}$ and $\Braket{\beta_1,\dots,\beta_g}$ if $\alpha_i=\beta_i$ for $i \ge 2$, and:
\begin{itemize}
    \item[$\bullet$] $\alpha_1$ and $\beta_1$ intersect once (type i), or
    \item[$\bullet$] $\alpha_1$ and $\beta_1$ intersect twice with the same sign (type ii).
\end{itemize}
We will often drop the common curves from the notation and write $\Braket{\alpha_1}-\Braket{\beta_1}$.
\item[\textbf{--}] The faces are of the following four kinds (see Figure \ref{fig:cells}):
\begin{itemize}
    \item[$\bullet$] \emph{triangles} $\Braket{\gamma_1}-\Braket{\gamma_1'}-\Braket{\gamma_1''}-\Braket{\gamma_1}$, where two edges are of type i and the third is of type ii;
    \item[$\bullet$] \emph{squares} $\Braket{\gamma_1,\gamma_2}-\Braket{\gamma_1,\gamma_2'}-\Braket{\gamma_1',\gamma_2'}-\Braket{\gamma_1',\gamma_2}-\Braket{\gamma_1,\gamma_2}$, where all edges are of type i;
    \item[$\bullet$] \emph{pentagons} $\Braket{\gamma_1,\gamma_2}-\Braket{\gamma_1,\gamma_2'}-\Braket{\gamma_1',\gamma_2'}-\Braket{\gamma_1',\gamma_2''}-\Braket{\gamma_2,\gamma_2''}-\Braket{\gamma_1,\gamma_2}$, where four edges are of type i and the fifth is of type ii;
    \item[$\bullet$] \emph{hyperelliptic faces}, which have $28$ edges of type i and will be described in detail later on (see \thref{def:hyp}).
\end{itemize}
\end{itemize} 
\end{definition}

\begin{figure}
\centering
\begin{tikzpicture}[scale=.3]
\draw [blue, thick] (3.5,2) arc(0:360:2 and 2);
\draw [green, thick] (0,-2.7) arc(-90:12:2 and 2);
\draw [green, thick] (0,-2.7) arc(270:32:2 and 2);
\draw [red, thick] (-1.5,0) arc (-90:32:2 and 2);
\draw [red, thick] (-1.5,0) arc (-90:-92:2 and 2);
\draw [red, thick] (-3.5,2) arc (180:250:2 and 2);
\draw [red, thick] (-1.5,4) arc (90:52:2 and 2);
\draw [red, thick] (-1.5,4) to [out=180, in=90] (-3,2.7) to [out=-90, in=180] (-1.5,1.4) to [out=0, in=-90] (0,2.2) arc (0:40:.5 and .8);
\draw [red, thick] (-3.5,2) arc (180:150:2 and 1.5);
\draw [red, thick] (-1.5,3.5) arc (90:123:2 and 1.5);
\draw [red, thick] (-1.5,3.5) arc (90:50:1.5 and 1.5);
\node [red] at (-4.1,2.8) {$\gamma_1$};
\node [blue] at (4.1,2.8) {$\gamma_1'$};
\node [green] at (0,-3.5) {$\gamma_1''$};
\node at (-4,-3.5) {(C1)};
\end{tikzpicture}
\begin{tikzpicture}[scale=.3]
\draw [cyan, thick] (3.5,1) arc(0:360:2 and 2);
\draw [orange, thick] (0,-3.7) arc(-90:12:2 and 2);
\draw [orange, thick] (0,-3.7) arc(270:32:2 and 2);
\draw [blue, thick] (-1,2) arc(0:360:2 and 2);
\draw [red, thick] (-4.5,-2.7) arc(-90:12:2 and 2);
\draw [red, thick] (-4.5,-2.7) arc(270:32:2 and 2);
\node [cyan] at (4,-.8) {$\gamma_2$};
\node [orange] at (0,-4.5) {$\gamma_2'$};
\node [blue] at (-5.5,3) {$\gamma_1$};
\node [red] at (-6.5,-3) {$\gamma_1'$};
\node at (-6,-5) {(C2)};
\end{tikzpicture}
\begin{tikzpicture}[scale=.3]
\draw [blue, thick] (3.5,2) arc(0:360:2 and 2);
\draw [green, thick] (0,-4.7) arc(-90:-15:2 and 2);
\draw [green, thick] (0,-4.7) arc(270:5:2 and 2);
\draw [red, thick] (-1.5,0) arc (-90:32:2 and 2);
\draw [red, thick] (-1.5,0) arc (-90:-135:2 and 2);
\draw [red, thick] (-3.5,2) arc (180:205:2 and 2);
\draw [red, thick] (-1.5,4) arc (90:52:2 and 2);
\draw [red, thick] (-1.5,4) to [out=180, in=90] (-3,2.7) to [out=-90, in=180] (-1.5,1.4) to [out=0, in=-90] (0,2.2) arc (0:40:.5 and .8);
\draw [red, thick] (-3.5,2) arc (180:150:2 and 1.5);
\draw [red, thick] (-1.5,3.5) arc (90:123:2 and 1.5);
\draw [red, thick] (-1.5,3.5) arc (90:50:1.5 and 1.5);
\draw [orange, thick] (-2.5,-3) arc (270:-65:2 and 2);
\draw [cyan, thick] (2.5,-3) arc (270:80:2 and 2);
\draw [cyan, thick] (2.5,-3) arc (270:420:2 and 2);
\node [red] at (-4.1,2.8) {$\gamma_1$};
\node [blue] at (4.1,2.8) {$\gamma_1'$};
\node [green] at (0,-5.5) {$\gamma_2'$};
\node [cyan] at (4.6,-3) {$\gamma_2$};
\node [orange] at (-4.7,-3) {$\gamma_2''$};
\node at (-6,-5.5) {(C3)};
\end{tikzpicture}
\caption{Configurations of curves for the $2$-cells of the spin cut-system complex.}
\label{fig:cells}
\end{figure}
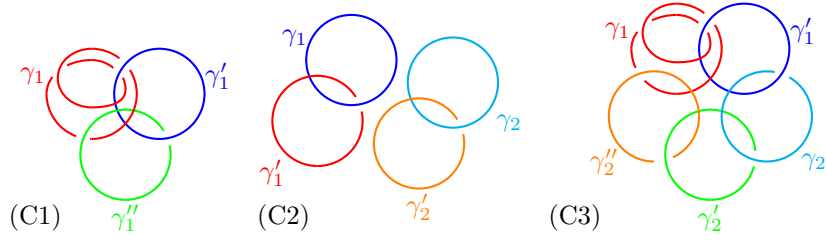

\begin{remark}
\thlabel{rem:112}
The following facts shed some light on the importance of edges of type ii, and will be used repeatedly.
\begin{enumerate}[label=\textit{(\roman*)}]
\item \emph{The 1-1-2 trick.} Every edge of type ii is contained in a triangle. Indeed, let $\Braket{\alpha}-\Braket{\beta}$ be an edge of type ii on a surface $\Sigma$, and call $P$ and $Q$ the two points of intersection of $\alpha$ and $\beta$. Construct two curves $\gamma_1$, $\gamma_2$ as follows: start from $P$, go along $\beta$ until $Q$, then turn right or left respectively, and run along $\alpha$ back to $P$ (Figure \ref{fig:112}). By \thref{thm:xiq}(2), we have 
\[
1=\phi(\alpha)=\phi(\gamma_1)+\phi(\gamma_2)+\gamma_1\cdot \gamma_2+1=\phi(\gamma_1)+\phi(\gamma_2),
\]
so one of the two is a $1$-curve, say $\gamma_1$, and $\Braket{\alpha}-\Braket{\beta}-\Braket{\gamma_1}-\Braket{\alpha}$ is a triangle. 

More generally, consider two nonseparating $1$-curves $\alpha$ and $\beta$ with $\abs{\alpha \cap \beta} \ge 2$, and assume that there is an arc of $\beta$ which connects the two boundary components of $\Sigma_g^b \setminus \alpha$ which correspond to $\alpha$. Call $P$ and $Q$ the endpoints of such arc. Equivalently, assume that there are two consecutive intersection points with the same sign $P$, $Q$ on $\beta$. Then the same trick can be used to obtain a nonseparating $1$-curve that intersects both $\alpha$ and $\beta$ in less than $\abs{\alpha \cap \beta}$ points.
\item \emph{No i-i-i triangles.} \label{noiii} A closed path of length $3$ must have two edges of type i and one edge of type ii. Indeed, assume for example that a closed path $\Braket{\gamma_1}-\Braket{\gamma_2}-\Braket{\gamma_3}-\Braket{\gamma_1}$ on $\Sigma_{g,b}$ only contains edges of type i. A tubular neighborhood $\nu(\gamma_1\cup \gamma_2\cup \gamma_3)$ has 3 boundary components and Euler characteristic $\chi=-3$. Notice that one of the boundary components, call it $\delta_1$, satisfies the relation $[\delta_1]= [\gamma_1]+[\gamma_2]+[\gamma_3]$ in $H_1(\Sigma;\Z/2)$, so $\phi(\delta_1)=0$. In particular, $\delta_1$ is nonseparating, so the complement of $\nu(\gamma_1\cup \gamma_2\cup \gamma_3)$ has at most two connected components. If it has two connected components, they are homeomorphic to $\Sigma_{g_1}^{b_1+1}$ and $\Sigma_{g_2}^{b_2+2}$, with $g_1+g_2=g-2$ and $b_1+b_2=b$. If it is connected, it is homeomorphic to $\Sigma_{g-3}^{b+3}$. In either case, two boundary components out of three are $0$-curves, so it is impossible to find $g-1$ disjoint linearly independent $1$-curves in the complement of $\gamma_1\cup\gamma_2\cup\gamma_3$. The existence of other kinds of triangles (and pentagons) can be ruled out in a similar way.
\item \emph{Other squares.} Squares with edges of type ii are null-homotopic in $X_g$. Indeed, such a square has two opposite edges of type ii, and by the 1-1-2 trick we get the following null-homotopy, where $\gamma_1$ is the curve obtained from $\alpha_1$ and $\beta_1$ via the 1-1-2 trick:
\[
\begin{tikzcd}
\Braket{\alpha_1,\alpha_2} \arrow[rrr, no head] \arrow[dd, no head, "ii"'] \arrow[dr, no head, "i"] & & & \Braket{\alpha_1,\beta_2} \arrow[dl, no head, "i"'] \arrow[dd, no head, "ii"] \\
& \Braket{\gamma_1,\alpha_2} \arrow [r, no head] \arrow[dl, no head, "i"] & \Braket{\gamma_1,\beta_2} \arrow[dr, no head, "i"'] & \\
\Braket{\beta_1,\alpha_2} \arrow[rrr, no head] & & & \Braket{\beta_1,\beta_2}.
\end{tikzcd}
\]
\end{enumerate}
\end{remark}

\begin{figure}
\centering
\begin{tikzpicture}[scale=.7]
\draw [thick, red] (0,-2) -- (0,2);
\draw [thick, blue] (-3,1) -- (3,1);
\draw [thick, blue] (-3,-1) -- (3,-1);
\draw [thick, orange] (-3,-1.2) -- (-.7,-1.2) to [out=0, in=90] (-.3,-2);
\draw [thick, orange] (-.3,2) to [out=-90, in=180] (.1,.8) -- (3,.8);
\draw [thick, green] (-3,-.8) -- (-.2,-.8) to [out=0, in=180] (.7,1.2) -- (3,1.2);
\draw [->, thick, blue] (0,1) -- (.5,1);
\draw [->, thick, blue] (0,-1) -- (.5,-1);
\node [red] at (.3,-1.5) {$\alpha$}; 
\node [blue] at (-2.5,.6) {$\beta$}; 
\node [orange] at (2.5,.4) {$\gamma_1$}; 
\node [green] at (2.5,1.5) {$\gamma_2$}; 
\end{tikzpicture}
\caption{The 1-1-2 trick: either $\gamma_1$ or $\gamma_2$ must be spin.}
\label{fig:112}
\end{figure}
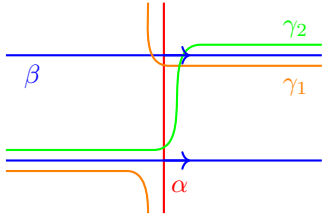

The main result of this section is the following.

\begin{theorem}
\thlabel{thm:x}
The spin cut-system complex $X_g$ is connected and simply connected for every $g \ge1$.
\end{theorem}

Following \cite{waj:elem}, we will prove \thref{thm:x} by induction on the genus and on a measure of complexity for edge paths, the \emph{radius}. Let $\mathbf{p}$ be a path in $X$ and let $v_0$ be a vertex of $\mathbf{p}$. Fix a curve $\alpha$ of $v_0$. The \emph{distance} of some vertex $v$ from $\alpha$ is defined as
\[
d_{\alpha}(v):=\min\set{\abs{\gamma \cap \alpha}:\gamma \in v}. 
\]
The \emph{radius} of $\mathbf{p}$ around $\alpha$ is the maximum distance of its vertices from $\alpha$. If all its vertices contain $\alpha$, $\mathbf{p}$ is called an $\alpha$-\emph{segment}. We will denote an $\alpha$-segment by a dashed line.

\subsection{Surfaces of genus 1} In this section, we are going to prove that $X_1$ is connected and simply connected.


\begin{proposition}
\thlabel{prop:conn1}
The complex $X_1$ is connected via paths that contain only edges of type i.
\end{proposition}

\begin{proof}
This follows by adapting the proof of \cite[Lemma 8]{waj:elem}. Let $\alpha,\beta$ be two nonseparating spin curves on $\Sigma$. We want to prove that there exists an edge-path from $\Braket{\alpha}$ to $\Braket{\beta}$. 

If $\alpha$ and $\beta$ are disjoint, they have a common geometric dual $\gamma$, and we can assume that it is a $1$-curve by Dehn twisting along $\alpha$ if necessary. Then $\Braket{\alpha}-\Braket{\gamma}-\Braket{\beta}$ is the required path.

In general, after cutting off any bigons as explained in \cite{waj:elem}, we may assume that the geometric intersection and the algebraic intersection between $\alpha$ and $\beta$ coincide by the bigon criterion \cite[Proposition 1.7]{primer}. Then, it suffices to apply the generalized 1-1-2 trick and conclude by induction on $\abs{\alpha \cap \beta}$. \qedhere
\end{proof}

Edges of type ii are necessary for simple connectivity, as the following Lemma shows.

\begin{lemma}[Square lemma]
Let $\mathbf{p}$ be the edge-path $\Braket{\delta_1}-\Braket{\delta_2}-\Braket{\delta_3}-\Braket{\delta_4}-\Braket{\delta_1}$. Assume that all the edges are of type i. If $\abs{\delta_2\cap \delta_4}=0$, then $\mathbf{p}$ is null-homotopic.
\end{lemma}

\begin{proof}
This is proven in the same way as \cite[Lemma 9]{waj:elem}, setting $\beta:=\tau_{\delta_2}^{\pm2}(\delta_3)$. Notice that such a curve cannot intersect $\delta_1$ once as there are no i-i-i triangles. \qedhere
\end{proof}

\begin{lemma}[{\cite[Lemma 10]{waj:elem}}]
\thlabel{lem:bigons}
Every closed path $\mathbf{p}$ in $X_1$ where all the edges are of type i is homotopic to another closed path $\mathbf{p}'$ where each curve is homologous to the corresponding curve in $\mathbf{p}$ but no two curves form a bigon. 
\end{lemma}

\begin{proposition}
\thlabel{prop:1conn1}
The complex $X_1$ is simply connected.
\end{proposition}

\begin{proof}
Let $\mathbf{p}=\Braket{\alpha_1}-\dots-\Braket{\alpha_k}-\Braket{\alpha_1}$ be a closed path. By the 1-1-2 trick, we can assume that it contains only edges of type i (hence $k \ne 3$). Then we proceed as in the proof of \cite[Proposition 7]{waj:elem}, using a squared twist to construct the curve $\beta$ instead of a single twist, just as in the proof of the square lemma. \qedhere
\end{proof}

\subsection{Connectivity}

From now on, $\Sigma_{g,b}$ will be a fixed surface of genus $g\ge2$, and we will call $\overline{\Sigma}_g$ the surface obtained by capping each boundary component with a disk. In this section, we are going to prove that the complex $X_g$ associated to $\Sigma_{g,b}$ is connected.

\begin{remark}
\thlabel{rem:oddss}
As the genus increases, our proof becomes a bit more involved than Wajnryb's. Indeed, Wajnryb's arguments often involve cutting along certain curves in one or more cut-systems, but when we cut a spin surface along a separating union of $1$-curves, we may get some subsurfaces where there are no nonseparating $1$-curves: one-holed tori with an odd pull-back spin structure, two-holed odd tori whose boundary components have spin value $1$, or annuli whose belt curve has spin value $0$. This will require some extra care in our arguments. See Figure \ref{fig:odd} for an example.
\end{remark}

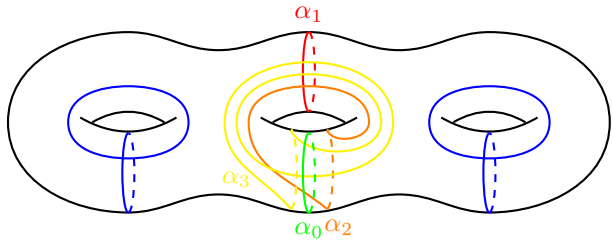
\begin{figure}
\centering
\begin{tikzpicture}[scale=.8]
\draw [thick] (0,0) to [out=90, in=-180] (2,1.5) to [out=0, in=180] (3.5,1.2) to [out=0, in=180] (5,1.5) to [out=0, in=180] (6.5,1.2) to [out=0, in=180] (8,1.5)  to [out=0, in=90] (10,0) to [out=-90, in=0] (8,-1.5) to [out=-180, in=0] (6.5,-1.2) to [out=-180, in=0] (5,-1.5) to [out=-180, in=0] (3.5,-1.2) to [out=-180, in=0] (2,-1.5) to [out=-180, in=-90] (0,0);
\draw [thick] (1.2,.08) to [out=-30, in=-150] (2.8,.08);
\draw [thick] (1.4,0) to [out=30, in=150] (2.6,0);
\draw [thick] (4.2,.08) to [out=-30, in=-150] (5.8,.08);
\draw [thick] (4.4,0) to [out=30, in=150] (5.6,0);
\draw [thick] (7.2,.08) to [out=-30, in=-150] (8.8,.08);
\draw [thick] (7.4,0) to [out=30, in=150] (8.6,0);
\draw [thick, red] (5,1.5) arc (90:270:.1 and .66);
\draw [thick, red, dashed] (5,1.5) arc (90:-90:.1 and .66);
\draw [thick, orange] (5.3,-1.44) to [out=140, in=-90] (4,0) to [out=90, in=180] (5,.6) to [out=0, in=90] (6,0) to [out=-90, in=-60] (5.3,-.14);
\draw [thick, orange, dashed] (5.3,-1.44) arc (-90:90:.1 and .66);
\draw [thick, yellow] (4.7,-1.44) to [out=130, in=-90] (3.6,0) to [out=90, in=180] (5,1) to [out=0, in=90] (6.4,0) to [out=-90, in=0] (5,-.9) to [out=180, in=-90] (3.8,0) to [out=90, in=180] (5,.8) to [out=0, in=90] (6.2,0) to [out=-90, in=-60] (4.7,-.14);
\draw [thick, yellow, dashed] (4.7,-1.44) arc (-90:90:.1 and .66);
\draw [thick, green] (5,-1.5) arc (270:90:.1 and .66);
\draw [thick, green, dashed] (5,-1.5) arc (-90:90:.1 and .66);
\draw [thick, blue] (2,-1.5) arc (270:90:.1 and .66);
\draw [thick, blue, dashed] (2,-1.5) arc (-90:90:.1 and .66);
\draw [thick, blue] (8,-1.5) arc (270:90:.1 and .66);
\draw [thick, blue, dashed] (8,-1.5) arc (-90:90:.1 and .66);
\draw [thick, blue] (1,0) to [out=90, in=180] (2,.6) to [out=0, in=90] (3,0) to [out=-90, in=0] (2,-.6) to [out=180, in=-90] (1,0);
\draw [thick, blue] (7,0) to [out=90, in=180] (8,.6) to [out=0, in=90] (9,0) to [out=-90, in=0] (8,-.6) to [out=180, in=-90] (7,0);
\node [red] at (5,1.8) {$\alpha_1$};
\node [orange] at (5.5,-1.75) {$\alpha_2$};
\node [yellow] at (3.8,-.95) {$\alpha_3$};
\node [green] at (5,-1.8) {$\alpha_0$};
\end{tikzpicture}
\caption{An impossible i-i-ii triangle: here the blue curves have all spin value $0$, hence the nonseparating $1$-curves $\alpha_1,\alpha_2,\alpha_3$ cannot be completed simultaneously to three spin cut-systems. Note, however, that after replacing $\alpha_1$ with the homologous curve $\alpha_0$ it is possible to form a well-defined triangle.}
\label{fig:odd}
\end{figure}

The proof that $X_g$ is connected will be by induction on the genus. The base case is \thref{prop:conn1}. For the inductive step, we will assume that the complex is connected when the genus is less than $g$, and use the following easy observation.

\begin{lemma}[{\cite[Lemma 12]{waj:elem}}]
\thlabel{lemma:conn}
If two vertices of $X_g$ have a curve $\alpha$ in common, they are connected via an $\alpha$-segment that contains only edges of type i.
\end{lemma}

We start by recalling the following construction of Wajnryb. Let $\gamma_1,\gamma_2\subset \Sigma_g^b$ be two nonseparating $1$-curves such that $\abs{\gamma_1\cap\gamma_2}=n\ge2$. We want to find a third nonseparating $1$-curve $\gamma$ such that $\abs{\gamma\cap\gamma_i}<n$ for $i=1,2$. As in the proof of \cite[Lemma 15]{waj:elem}, choose orientations on $\gamma_1$ and $\gamma_2$, and let $P_1$ be an intersection point. Construct a curve $\delta_1$ as follows: following the orientations of $\gamma_1$ and $\gamma_2$, go from $P_1$ to the next intersection point $P_2$ along $\gamma_1$, then follow $\gamma_2$ until getting back to $P_1$. Then construct $\delta_2$ as follows: go from $P_2$ along $\gamma_1$ until the first intersection point that was not met by $\delta_1$, or until $P_1$ if there is no such point, and then follow $\gamma_2$ all the way back to $P_2$. Repeat this construction until every arc of $\gamma_1$ and $\gamma_2$ is covered by an arc of some $\delta_i$, $i=1,\dots,k$. Then choose the opposite orientation of $\gamma_2$ and start again, obtaining curves $\epsilon_1,\dots,\epsilon_r$. Notice that the following relations hold in $H_1(\overline{\Sigma_g};\Z)$: $[\delta_1]-[\varepsilon_1]=[\gamma_2]$, $[\delta_1]+\dots+[\delta_k]=[\gamma_1]+[\gamma_2]$ and $[\varepsilon_1]+\dots+[\varepsilon_r]=[\gamma_1]-[\gamma_2]$. This implies that $\delta_1$ and some $\delta_i$, $i\ge2$, or $\epsilon_1$ and some $\epsilon_j$, $j\ge 2$, are nonseparating. 

Now we study the spin values of the above curves. Observe first that if $P_1$ and $P_2$ have the same sign, then $\delta_1$ and $\varepsilon_1$ intersect $\gamma_2$ and each other exactly once, and exactly one of them is a $1$-curve; this is an instance of the 1-1-2 trick. A finer observation is the following.

\begin{lemma}
Let $\gamma_1,\gamma_2\subset \Sigma_g^b$ be two oriented $1$-curves such that $\abs{\gamma_1\cap\gamma_2}=n\ge 1$, and construct $\delta_1,\dots,\delta_k$ as above. Let $\ell$ be the number of intersection points between some $\delta_i$ and $\delta_j$, for $i,j=1,\dots,k$. Then $k+\ell = n$.
\end{lemma}

\begin{proof}
We do induction on $n$. If $n=1$, we obtain a single curve $\delta_1$, which is the oriented resolution of $\gamma_1\cup\gamma_2$. Assume now that $n>1$. Remove an intersection point $p$ between $\gamma_1$ and $\gamma_2$ via some surgery on $\Sigma_g^b$ (for example, gluing in a tube). We will show that $k+\ell$ decreases by one. 

Notice that since $n \ge 2$, our surgery only affects two curves $\delta_i, \delta_j$, that either meet at $p$ or both turn at $p$ following the orientations of $\gamma_1$ and $\gamma_2$. If they intersect at $p$, after the surgery $\ell$ decreases by one, and $k$ stays the same. If both turn at $p$, they merge after the surgery, so $k$ decreases by one and $\ell$ stays the same. \qedhere
\end{proof}

\begin{corollary}
\thlabel{cor:15}
Let $\gamma_1,\gamma_2\subset \Sigma_g^b$ be two oriented $1$-curves such that $\abs{\gamma_1\cap\gamma_2}=n\ge1$, and construct $\delta_1,\dots,\delta_k$ as above. Then $\phi(\delta_1)+\dots+\phi(\delta_k) \equiv 0 \ (\operatorname{mod}2)$.
\end{corollary}

\begin{proposition}
\thlabel{prop:conn}
The complex $X_g$ is connected via paths that contain only edges of type i.
\end{proposition}

\begin{proof}
Let $\alpha_1,\alpha_2$ be two spin nonseparating curves on $\Sigma_g^b$, and let $v_1,v_2$ be two vertices of $X_g$ with $\alpha_i \in v_i$. We are going to prove that there is a path from $v_1$ to $v_2$ by induction on $n:=\abs{\alpha_1\cap \alpha_2}$.

If $\alpha_1=\alpha_2$, the statement is \thref{lemma:conn}. We will deal later with other cases where $n=0$.

If $n=1$, we can cut $\Sigma_g^b$ along $\alpha_1\cup\alpha_2$, obtaining a surface $\Sigma_{g-1}^{b+1}$ with an even pull-back spin structure, and find a spin cut system $u'$ on $\Sigma_{g-1}^{b+1}$. Setting $u_i:=u'\cup\set{\alpha_i}$ for $i=1,2$, we get an edge of type i $u_1-u_2$, and by \thref{lemma:conn} there exist paths from $v_1$ to $u_1$ and from $u_2$ to $v_2$ containing only edges of type i. 

If $n=0$ and $[\alpha_1],[\alpha_2]$ are linearly independent in $H_1(\overline{\Sigma}_g;\Z)$, there exists a spin cut-system $v$ containing both curves, and by \thref{lemma:conn} we can connect $v_1$ to $v$ and $v$ to $v_2$. 

If $n=0$ and $[\alpha_1]=[\alpha_2]$ in $H_1(\overline{\Sigma}_g;\Z)$, but the two curves are not isotopic, they have a common geometric dual $\beta$, and up to Dehn twisting along $\alpha_1$, we can assume that it is a $1$-curve. Now, as in the case $n=1$, there are edges of type i $\Braket{\alpha_1}-\Braket{\beta}$ and $\Braket{\beta}-\Braket{\alpha_2}$, and applying \thref{lemma:conn} repeatedly we get a path from $v_1$ to $v_2$ which interpolates between them. 

If $n\ge2$ and $\alpha_1,\alpha_2$ have two consecutive points of intersection with the same sign, we can apply the generalized 1-1-2 trick to find a $1$-curve $\alpha_3$ that intersects both $\alpha_1$ and $\alpha_2$ in less than $n$ points, and conclude by induction. 

Assume now that $n \ge 2$ and the signs of all intersection points between $\alpha_1$ and $\alpha_2$ are alternating. Fix an orientation on $\alpha_1$. Call $\gamma_1^r,\dots,\gamma_{k_r}^r$ the boundary components of a tubular neighborhood $\nu(\alpha_1\cup \alpha_2)$ that sit on the right with respect to $\alpha_1$, and $\gamma_1^{\ell},\dots,\gamma_{k_{\ell}}^{\ell}$ the remaining boundary components. We can orient these curves so that
\[
\left[\gamma_1^r\right]+\dots+\left[\gamma_{k_r}^r\right]=[\alpha_1]=\left[\gamma_1^{\ell}\right]+\dots+\left[\gamma_{k_{\ell}}^{\ell}\right],
\]
so at least one right and one left component are nonseparating. If some right or left component is nonseparating and has spin value $1$, or cobounds a subsurface of $\Sigma_{g,b} \setminus \nu(\alpha_1\cup\alpha_2)$ which contains a $1$-curve that does not separate $\Sigma_{g,b}$, then we conclude. If that is not the case, it is easy to see that each left and right component falls in one of the following cases (compare \thref{rem:oddss}):
\begin{itemize}
\item it bounds a disk;
\item it bounds a one-holed torus with an induced odd spin structure;
\item it is a $0$-curve and is one boundary component of an annulus. 
\end{itemize}
Notice that there are at least two annuli $A_1,A_2$ with one right and one left boundary component (in particular, $n$ is at least $4$). Indeed, if there is only one such annulus, up to renaming we can assume that its boundary components are $\gamma_1^r$ and $\gamma_1^{\ell}$; then $\left[\gamma_1^r\right]=[\alpha_1]=\left[\gamma_1^{\ell}\right]$, so they cannot be $0$-curves. Assume that $\partial A_i=\gamma_i^r\cup \gamma_i^{\ell}$ for $i=1,2$. 

We form a $1$-curve $\delta$ by arc-summing one component of $\partial A_1$ and one component of $\partial A_2$ along an arc that minimizes the intersections with $\alpha_1$ and $\alpha_2$. Such an arc can be constructed as follows. Isotop $\gamma_i^r$ and $\gamma_i^{\ell}$ so that they stay disjoint from $\alpha_2$, and exactly one point $p_i^r\in\gamma_i^r$ and $p_i^{\ell}\in\gamma_i^{\ell}$ lies on $\alpha_1$, for $i=1,2$. Consider the arcs of $\alpha_1$ between $p_1^r$ or $p_1^{\ell}$ and $p_2^r$ or $p_2^{\ell}$. If any of this arc intersects $\alpha_2$ in less than $n$ points, we are done. Otherwise, notice that the same arc of $\alpha_1$ cannot form two bigons with arcs of $\alpha_2$, hence some boundary component of $A_1$ or $A_2$ must come close to $\alpha_1$ also at some other point (see for example the dashed orange arc of $\gamma_1^r$ in Figure \ref{fig:arc}). Repeating the procedure using that point yields the desired arc. \qedhere
\end{proof}

\begin{figure}
\centering
\begin{tikzpicture}
\draw [->, thick, red] (-5,0) -- (8.5,0);
\draw [->, thick, blue] (-3.5,-2) -- (-3.5,2);
\draw [->, thick, blue] (-2,2) -- (-2,-2);
\draw [->, thick, blue] (-.5,-2) -- (-.5,2);
\draw [->, thick, blue] (1,2) -- (1,-2);
\draw [->, thick, blue] (2.5,-2) -- (2.5,2);
\draw [->, thick, blue] (4,2) -- (4,-2);
\draw [->, thick, blue] (5.5,-2) -- (5.5,2);
\draw [->, thick, blue] (7,2) -- (7,-2);
\draw [thick, orange] (-3.3,-2) -- (-3.3,-.6) to [out=90, in=180] (-3,-.2) -- (-2.5,-.2) to [out=0, in=90] (-2.2,-.6) -- (-2.2,-2);
\draw [thick, dashed, orange] (-.3,-2) -- (-.3,-.6) to [out=90, in=180] (0,-.2) -- (.5,-.2) to [out=0, in=90] (.8,-.6) -- (.8,-2);
\draw [thick, green] (-3.3,2) -- (-3.3,.6) to [out=-90, in=180] (-3,.2) -- (-2.5,.2) to [out=0, in=-90] (-2.2,.6) -- (-2.2,2);
\draw [thick, cyan] (2.7,-2) -- (2.7,-.6) to [out=90, in=180] (3,-.2) -- (3.5,-.2) to [out=0, in=90] (3.8,-.6) -- (3.8,-2);
\draw [thick, yellow] (2.7,2) -- (2.7,.6) to [out=-90, in=180] (3,.2) -- (3.5,.2) to [out=0, in=-90] (3.8,.6) -- (3.8,2);
\draw [dotted] (.5,-.2) -- (3,-.2);
\node at (1.75,-.5) {$c$};
\node [red] at (8,.35) {$\alpha_1$}; 
\node [blue] at (7.3,1.5) {$\alpha_2$}; 
\node [orange] at (-2.75,-.6) {$\gamma_1^r$}; 
\node [orange] at (.25,-.6) {$\gamma_1^r$}; 
\node [green] at (-2.75,.6) {$\gamma_1^{\ell}$}; 
\node [cyan] at (3.25,-.6) {$\gamma_2^r$}; 
\node [yellow] at (3.25,.6) {$\gamma_2^{\ell}$}; 
\end{tikzpicture}
\caption{Constructing the arc $c$ required in the proof of \thref{prop:conn}: choose segments of $\gamma_1^r,\gamma_1^{\ell},\gamma_2^r,\gamma_2^{\ell}$ which run parallel to $\alpha_1$, and connect them via arcs parallel to $\alpha_1$. If all these arcs intersect $\alpha_2$ in $n/2$ points, then we are in the situation depicted above, and it is possible to choose a different segment (such as the dashed orange one) of at least one curve. (!)}
\label{fig:arc}
\end{figure}
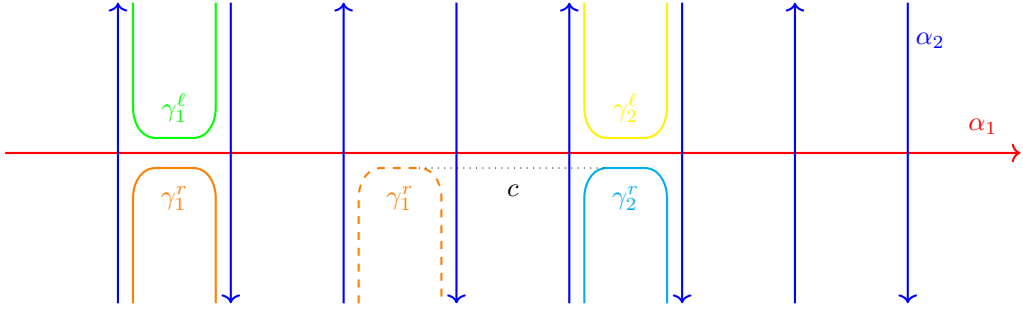

Call $X_g'$ the complex whose vertices are isotopy classes of spin simple closed curves on $\Sigma$ and whose edges connect two curves which intersect once. We have the following.

\begin{corollary}
\thlabel{cor:conn}
The complex $X_g'$ is connected.
\end{corollary}

We will also need the following refined version of \thref{prop:conn}, where we take into account intersections with other curves.

\begin{lemma}
\thlabel{lem:17}
Let $\delta_1,\delta_2$ be two distinct nonseparating $1$-curves that are either disjoint and homologous or intersect more than once, and assume that there exist an integer $m \ge 1$ and nonseparating $1$-curves $\gamma,\gamma'$ such that:
\begin{enumerate}[label=\textit{(\alph*)}]
\item if $m=1$, then $\gamma,\gamma'$ are disjoint and homologous, and $\abs{\gamma\cap\delta_i}=\abs{\gamma'\cap\delta_i}=1$ for $i=1,2$;
\item if $m\ge2$, then $\abs{\gamma\cap\gamma'}=m$, $\abs{\gamma\cap\delta_1}<m$, $\abs{\gamma\cap\delta_2}\le m$ and $\abs{\gamma'\cap\delta_i}\le1$ for $i=1,2$.
\end{enumerate}
There exists a nonseparating $1$-curve $\delta$ that intersects $\gamma$ and $\gamma'$ once if $m=1$, and less than $m$ if $m \ge 2$, and moreover:
\begin{enumerate}
\item if $\delta_1,\delta_2$ are disjoint and homologous, then $\abs{\delta\cap\delta_i}=1$ for $i=1,2$;
\item otherwise, $\abs{\delta\cap\delta_i}<\abs{\delta_1\cap\delta_2}$ for $i=1,2$.
\end{enumerate}
\end{lemma}

\begin{proof}
\textbf{Case 1a.} Choose a component $S$ of the complement of $\gamma \cup \gamma'$ of positive genus, and let $\beta \subset S$ be a curve that intersects both $\delta_1$ and $\delta_2$ once. If $\beta$ is a $0$-curve, we set $\delta:=\tau_{\delta_1}(\beta)$. Assume that $\beta$ is a $1$-curve. Construct two more curves as follows. Let $\eta_1$ be a boundary component of a tubular neighborhood of $\gamma \cup \delta_1 \cup \delta_2 \cup \beta$ in $S$ that is nontrivial in $\overline{\Sigma}_g$, and choose a curve $\eta_2$ in $S \setminus (\delta_1\cup\delta_2)$ that meets $\beta$ and $\eta_1$ once. If $\eta_1$ is a $0$-curve, call $\beta'$ the arc sum of $\beta$ and $\eta_1$ along an arc of $\eta_2$. If $\eta_1$ is a $1$-curve, up to replacing $\eta_2$ with $\tau_{\eta_1}(\eta_2)$ we may assume that $\eta_2$ is a $1$-curve, and we set $\beta':=\tau_{\eta_2}(\beta)$. In any case, $\beta'$ is a $0$-curve, and $\delta:=\tau_{\delta_1}(\beta)$ is the desired $1$-curve.

\textbf{Case 2a.} Choose an orientation on $\delta_2$, and call $p$ the first intersection point with $\delta_1$ that is found on $\delta_2$ after meeting $\gamma$. Construct an arc $c$ as follows: starting from the intersection point between $\gamma$ and $\delta_2$, move along $\delta_2$ towards $p$, then go along $\delta_1$ crossing $\gamma'$ and then going back to $\gamma$. Now, $c$ can be completed to two different curves $\xi_1$ and $\xi_2$ using arcs of $\gamma$, and one of the two is a $1$-curve that satisfies our requirements.

\textbf{Case 1b: $\abs{\gamma'\cap \delta_i}=0$.} Call $S_1$ and $S_2$ the two components of $\Sigma_{g,b} \setminus (\delta_1\cup\delta_2)$, and assume that $\gamma'$ lies in $S_2$. It is easy to see that there exists an arc $c_1\subset S_1 \setminus \gamma$ that connects $\delta_1$ and $\delta_2$. We claim that there exists an arc $c_2 \subset S_2$ that connects $\delta_1$ and $\delta_2$ and meets $\gamma$ at most once and $\gamma'$ less than $m$ times. 

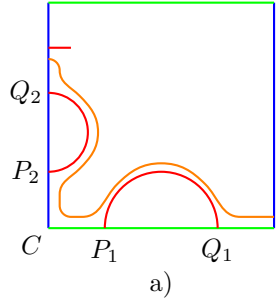
\begin{figure}
\centering
\begin{tikzpicture}[scale=.75]
\draw [blue, thick] (0,0) -- (0,4);
\draw [blue, thick] (4,0) -- (4,4);
\draw [green, thick] (0,0) -- (4,0);
\draw [green, thick] (0,4) -- (4,4);
\draw [red, thick] (1,0) arc (180:0:1 and 1);
\draw [red, thick] (0,1) arc (-90:90:.7 and .7);
\draw [red, thick] (0,3.2) -- (.4,3.2);
\draw [orange, thick] (0,3) to [out=0, in=90] (.2,2.8) to [out=-90, in=135] (.55,2.35) to [out=-45, in=90] (.88,1.7) to [out=-90, in=45] (.55,1.05) to [out=-135, in=90] (.2,.6) -- (.2,.4) to [out=-90, in=180] (.4,.2) -- (.6,.2) to [out=0, in=-135] (1.22,.78) to [out=45, in=180] (2,1.15) to [out=0, in=135] (2.78,.78) to [out=-45, in=180] (3.4,.2) -- (4,.2);
\node at (1,-.4) {$P_1$};
\node at (3,-.4) {$Q_1$};
\node at (-.4,1) {$P_2$};
\node at (-.4,2.4) {$Q_2$};
\node at (-.3,-.3) {$C$};
\node at (2,-1) {a)};
\end{tikzpicture}
\caption{Construction of the arc $c_2$ in the proof of \thref{lem:17}.}
\label{fig:c2}
\end{figure}

Assume first that some arc $d$ of $\gamma$ connects $\delta_1$ and $\delta_2$ in $S_2$. If $\abs{d \cap \gamma'}<m$, we set $c_2:=d$. If instead $\abs{d\cap\gamma'}=m$, define $c_2$ as follows: start from $\delta_1$ and go along $\delta$ until the first intersection point with $\gamma'$, then follow $\gamma'$ in either direction until the next intersection point with $d$, and finally go along $d$ until $\delta_2$ (see Figure \ref{fig:c2}a). 

If no such $d$ exists, but there is an arc $d'$ of $\gamma'$ that connects a $\gamma$-arc $a_1$ with endpoints on $\delta_1$ to a $\gamma$-arc $a_2$ with endpoints of $\delta_2$, we define $c_2$ as follows: go along $a_1$ until $d'$, then follow $d'$ until $a_2$ and go along $a_2$ until $\delta_2$ (see Figure \ref{fig:c2}b). 

Finally, if no such $d$ or $d'$ exist, we may assume that $\gamma'$ only intersects $\gamma$-arcs with endpoints on $\delta_2$. Construct $c_2$ by going along one such arc until the first intersection point with $\gamma'$, then following $\gamma'$ until entering a component of $S_2\setminus(\gamma\cup\gamma')$ that meets $\delta_1$, and going through such a component until $\delta_1$.

Now, join $c_1$ and $c_2$ via an arc of $\delta_1$ and an arc of $\delta_2$. There are four possible choices for such a couple of arcs: two of them produce a $0$-curve, and it is easy to see that at least one of the two curves meets $\gamma$ in less than $m$ points. 

\textbf{Case 1b: $\abs{\gamma'\cap \delta_i}=1$.} Again, let $S_1$ and $S_2$ be the two components of $\Sigma_{g,b}\setminus(\delta_1\cup\delta_2)$. Assume first that $\gamma\cap\delta_1=\emptyset$. Construct an arc $c_1 \subset S_1$ as follows. If $\gamma$ does not meet $\gamma'$ in $S_1$, simply set $c_1:=\gamma'\cap S_1$. Otherwise, go along $\gamma'$ starting from $\delta_1$, and turn left at the first intersection point with $\gamma$, following $\gamma$ until meeting $\delta_2$. Define similarly $c_2\subset S_2$, turning right at the first intersection point with $\gamma$. Close up $c_1\cup c_2$ with the arc of $\delta_2$ that intersects $\gamma$ in less points. Up to twisting the resulting curve along $\delta_1$, we are done.

Assume now that there is an arc $d$ of $\gamma$ that joins $\delta_1$ to $\delta_2$ on $S_1$. Construct an arc $d'\subset S_2$ as follows. If there is an arc of $\gamma$ that joins $\delta_1$ to $\delta_2$ on $S_2$, take it as $\delta'$. If that is not the case, and there are no intersection points between $\gamma$ and $\gamma'$ on $S_2$, set $d':=\gamma'\cap S_2$. Otherwise, let $p,q\in S_2$ be two intersection points between $\gamma'$ and $\gamma$, $\delta_1$ or $\delta_2$ that are consecutive on $\gamma'$, and such that if $p,q \in \gamma'\cap\gamma$, then $p$ is joined to $\delta_1$ by a $\gamma$-arc on $S_2$ and $q$ is joined to $\delta_2$ by a $\gamma$-arc on $S_2$. Call $d'$ the union of these $\gamma$-arc and of the $\gamma'$-arc from $p$ to $q$ on $S_2$. There are four possible choices of arcs of $\delta_1$ and $\delta_2$ to close up $d \cup d'$, and at least one of them results in a curve that satisfies our requirements.

\textbf{Case 2b.} Set $n:=\abs{\delta_1\cap\delta_2}\ge2$. Assume first that there are two consecutive intersection points $P_1,P_2$ with the same sign, say on $\delta_1$. We apply the construction of the proof of \cite[Lemma 15]{waj:elem} that we recalled earlier. As we already observed, in this case we can assume that $\delta_1$ is a nonseparating $1$-curve. Assume that it intersects $\gamma$ or $\gamma'$ in at least $m$ points By \thref{cor:15}, some other $\delta_i$ is a $1$-curve, and if it is nonseparating, it satisfies our requirements. If every $\delta_i$ with spin value $1$ is separating, we define $\delta$ as the arc sum of some $\epsilon_j$ with spin value $0$ and some curve lying on a subsurface cut out by some $\delta_i$, minimizing the intersections with $\gamma$ and $\gamma'$. 

Assume now that on both curves the intersection points have alternating signs. In this case, we adapt the last part of the proof of \thref{prop:conn}: to construct $\delta$ we have to choose a component of $\partial A_1$, a component of $\partial A_2$ and an arc that joins them, and we can perform these choices to ensure that $\delta$ meets both $\gamma$ and $\gamma'$ in less than $m$ points. \qedhere 
\end{proof}

\subsection{Simple connectivity: paths of radius 0}

Now we turn our attention to simple connectivity. The proof will be by induction on the genus and on the radius, and will be split among this and the following subsections. The base of the induction is \thref{prop:1conn1}. We are going to assume that $X_{g'}$ is simply connected for every $g'<g$, and prove that all closed paths in $X_g$ are null-homotopic. In this section, we consider paths of radius $0$. 

The following observation is analogous to \thref{lemma:conn}, and will be used in the inductive step.

\begin{lemma}[{\cite[Lemma 11]{waj:elem}}]
\thlabel{lemma:1conn}
Every closed segment of $X_g$ is null-homotopic.
\end{lemma}

We will also need the following lemmas. The first concerns a sort of generalized square. The second is where pentagons make their appearance.

\begin{lemma}[Ladder lemma]
Let $\alpha_1,\alpha_2,\beta_1,\beta_2$ be nonseparating $1$-curves on $\Sigma_{g,b}$ such that the pairs $(\alpha_1,\alpha_2)$, $(\alpha_1,\beta_1)$, $(\beta_1,\beta_2)$ and $(\alpha_2,\beta_2)$ can be completed to spin cut-systems. Then there exists a null-homotopic path in $X_g$
\[
\begin{tikzcd}
\Braket{\alpha_1,\alpha_2} \arrow[r, dashed, "\alpha_2", no head] \arrow[d, dashed, "\alpha_1"', no head] & \Braket{\beta_2,\alpha_2} \arrow[d, dashed,"\beta_2", no head]\\
\Braket{\alpha_1,\beta_1} \arrow[r, dashed, "\beta_1"', no head] & \Braket{\beta_1,\beta_2}.
\end{tikzcd}
\]
\end{lemma}

\begin{proof}
Cut $\Sigma_{g,b}$ along $\beta_1\cup\alpha_2$. If the result is disconnected and $\alpha_1$ and $\beta_2$ lie on different components, then we can actually prove more: $\Braket{\alpha_1,\alpha_2,\beta_2}$ and $\Braket{\alpha_1,\beta_1,\beta_2}$ can be completed to spin cut-systems, and these can be connected by an $(\alpha_1,\beta_2)$-segment by \thref{lemma:1conn}.

Otherwise, by \thref{prop:conn} we can find nonseparating $1$-curves $\gamma_0:=\alpha_1,\gamma_1,\dots,\gamma_n:=\beta_2$ such that $\abs{\gamma_i\cap \gamma_{i+1}}=1$ and every $\gamma_i$ lies on the component of $\Sigma \setminus (\beta_1\cup \alpha_2)$ that contains $\alpha_1,\beta_2$. Complete $\Braket{\gamma_i,\alpha_2}$ to a vertex $v_i$ of $X_g$ and $\Braket{\gamma_i,\beta_1}$ to $w_i$, for all $i=0,1,\dots,n-1$, in such a way that the only curve of $v_i$ that intersects $\gamma_{i+1}$ is $\gamma_i$, and the same is true for $w_i$. This can be done by cutting along $\gamma_i,\gamma_{i+1}$ and $\alpha_2$ or $\beta_1$ and finding a spin cut-system on the resulting surface. Set $v_i':=v_{i}\setminus \set{\gamma_{i-1}}\cup \set{\gamma_i}$ and $w_i':=w_{i}\setminus \set{\gamma_{i-1}}\cup \set{\gamma_i}$. Clearly, there are edges $v_i-v_{i+1}'$ and $w_i-w_{i+1}'$ for every $i$. Construct a $\gamma_i$-segment $\mathbf{p}_i$ from $v_i$ to $w_i$ such that for each vertex of $\mathbf{p}_i$, the only curve which intersects $\gamma_{i+1}$ is $\gamma_i$. Replacing each occurrence of $\gamma_i$ in $\mathbf{p}_i$ with $\gamma_{i+1}$ gives a path $\mathbf{p}_i'$ from $v_i'$ to $w_i'$, and each vertex of $\mathbf{p}_i$ is connected to the corresponding vertex of $\mathbf{p_i}'$ by an edge $\Braket{\gamma_i}-\Braket{\gamma_{i+1}}$. Finally, construct $\gamma_i$-segments from $v_i'$ to $v_{i+1}$ and from $w_i$ to $w_{i+1}$. We get the following:
\[
\begin{tikzcd}[column sep=tiny, row sep=tiny]
v_0 \arrow[r, no head] \arrow[dddd, dashed, no head, "\mathbf{p}_0"'] & v_0' \arrow[rrrr, no head, dashed, "\gamma_1"] \arrow[dddd, dashed, no head, "\mathbf{p}_0'"] & & & & v_1 \arrow[r, no head] \arrow[dddd, dashed, no head, "\mathbf{p}_1"']  & v_1' \arrow[rrrr, no head, dashed, "\gamma_2"] \arrow[dddd, dashed, no head, "\mathbf{p}_1'"] & & & & \dots \arrow[rrrr, no head, dashed, "\gamma_{n-1}"] & & & & v_{n-1} \arrow[r, no head] \arrow[dddd, dashed, no head, "\mathbf{p}_{n-1}"'] & v_{n-1}' \arrow[dddd, dashed, no head, "\mathbf{p}_{n-1}'"]\\
{}\arrow[r, no head] & {} & & & & {}\arrow[r, no head] & {} & & & & & & & & {} \arrow[r, no head] & {} \\
{}\arrow[r, no head] & {} & & & & {}\arrow[r, no head] & {} & & & & & & & & {} \arrow[r, no head] & {} \\
{}\arrow[r, no head] & {} & & & & {}\arrow[r, no head] & {} & & & & & & & & {} \arrow[r, no head] & {} \\
w_0 \arrow[r, no head] & w_0' \arrow[rrrr, no head, dashed, "\gamma_1"] & & & & w_1 \arrow[r, no head]  & w_1' \arrow[rrrr, no head, dashed, "\gamma_2"] & & & & \dots \arrow[rrrr, no head, dashed, "\gamma_{n-1}"] & & & & w_{n-1} \arrow[r, no head] & w_{n-1}'. \\
\end{tikzcd}
\]
This is a sequence of ladders of squares and closed $\gamma_i$-segments, which are null-homotopic by \thref{lemma:1conn}, so we are done.\qedhere
\end{proof}

\begin{lemma}[Hexagon lemma]
Let $\alpha_1,\alpha_2,\alpha_3$ be three disjoint nonseparating $1$-curves on $\Sigma$ that are pairwise not homologous but whose union separates $\Sigma_{g,b}$. Then there exists a null-homotopic path
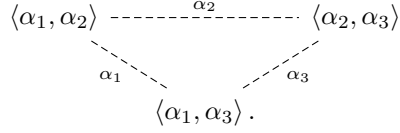
\[
\begin{tikzcd}[column sep=tiny]
\Braket{\alpha_1,\alpha_2} \arrow[rrrr, dashed, no head, "\alpha_2"] \arrow[drr, dashed, no head, "\alpha_1"'] & & & & \Braket{\alpha_2,\alpha_3} \arrow[dll, dashed, no head, "\alpha_3"] \\
& & \Braket{\alpha_1,\alpha_3}. & &
\end{tikzcd}
\]
\end{lemma}

\begin{proof}
By the assumptions, $\Sigma_{g,b} \setminus (\alpha_1\cup\alpha_2\cup\alpha_3)$ has exactly two components, which we will call $S_1$ and $S_2$. Assume first that the restriction of the spin structure to both $S_1$ and $S_2$ is even. In this case, the proof is the same as that of \cite[Lemma 13]{waj:elem}, setting $\delta:=\tau_{\alpha_2}^2(\beta_3)$. 

If instead the induced spin structures on $S_1$ and $S_2$ are odd, we can replace $\alpha_1$ with a disjoint curve $\alpha_1'$ in the same homology class such that $\alpha_1',\alpha_2,\alpha_3$ still satisfy our hypotheses and $\alpha_1' \cup \alpha_2\cup \alpha_3$ cuts the surface into two even subsurfaces. Indeed, choose $0$-curves $\gamma_1,\gamma_2$ on $S_1$ that intersect once, and call $\gamma_3$ the arc sum of $\gamma_1$ and $\alpha_1$ along an arc that is disjoint from $\gamma_2$. Then, $\gamma_1,\gamma_2,\gamma_3$ is a $3$-chain of $0$-curves, and a tubular neighborhood of the union $\gamma_1\cup\gamma_2\cup\gamma_3$ has boundary given by $\alpha_1$ and the desired curve $\alpha_1'$.

Now, apply the first part of the proof to the triple $\alpha_1',\alpha_2,\alpha_3$, and construct the required path applying the ladder lemma to the edges of the hexagon involving $\alpha_1'$ and finding additional segments via \thref{lemma:conn}. \qedhere
\end{proof}

\begin{proposition}
\thlabel{prop:r0}
All paths of radius zero in $X_g$ are null-homotopic. 
\end{proposition}

\begin{proof}
Let $\mathbf{p}$ be a path of radius zero with respect to some curve $\alpha$ contained in a vertex $v_0$ of $\mathbf{p}$. By the 1-1-2 trick, we can assume that all the edges of $p$ are of type i. The proof is then the same as that of \cite[Proposition 14]{waj:elem}, using the ladder lemma to construct the squares of \cite[Figures 6 and 8]{waj:elem}. \qedhere
\end{proof}

\subsection{Simple connectivity: paths of radius 1}

Our inductive step works only when the radius is at least 2. We now deal separately with paths of radius 1. Here we will need to use \emph{hyperelliptic face}. The reasons why a new $2$-cell is needed can be traced back to the following observation. 

\begin{remark}
\thlabel{rem:hyp}
Let $v_0,v_1$ be two spin cut-systems, and assume that there exist two disjoint curves $\alpha_0 \in v_0$, $\alpha_1 \in v_1$, i.e. $d_{\alpha_0}(v_1)=0$. Then, unlike in the standard cut-system complex, there is not always a path in $X_g$ from $v_0$ to $v_1$ with radius $0$ around $\alpha_0$. For example, choose $\alpha_0,\alpha_1$ as in Figure \ref{fig:odd} (in particular, take $g=3$). Recall that there are no $1$-curves in $\Sigma_3\setminus(\alpha_0\cup\alpha_1)$ by \thref{rem:oddss}. As a consequence, if $\mathbf{p}$ is a path from $v_0$ to $v_1$, all the curves of the last vertex before the final $\alpha_1$-segment must intersect $\alpha_0$. 

This problem does not arise in genus $2$. Indeed, in this case two disjoint, nonseparating curves are either independent in homology or bound an annulus with holes. As a consequence, the arguments of this section and the next one become much simpler in that case, and can be followed to prove that $X_2$ is simply connected. This fact will be used in the inductive steps. 
\end{remark}

We construct the hyperelliptic face by reverse engineering the genus 3 hyperelliptic relation. Before giving the detailed construction, we state the key fact that the hyperelliptic face allows us to prove.

\begin{proposition}
\thlabel{prop:hyp}
Let $\alpha,\alpha'$ and $\beta,\beta'$ be two couples of nonseparating $1$-curves on $\Sigma_g^b$ with the following properties:
\begin{enumerate}[label=\textit{(\roman*)}, ref=(\roman*)]
\item $\alpha,\alpha'$ (resp. $\beta,\beta'$) are disjoint and homologous, and separate $\Sigma_g^1$ into two odd subsurfaces; \label{hy1}
\item $\abs{\alpha \cap \beta}=\abs{\beta\cap \alpha'}=\abs{\alpha'\cap\beta'}=\abs{\beta'\cap \alpha}=1$. \label{hy3}
\end{enumerate}
Then there exists a null-homotopic path in $X_g$ of the form
\[
\begin{tikzcd}[column sep=tiny, row sep=small]
\Braket{\alpha} \arrow[r, no head] \arrow[dd, dashed, no head, "\alpha"'] & \Braket{\beta} \arrow[rr, dashed, no head, "\beta"] & & \Braket{\beta} \arrow[r, no head] & \Braket{\alpha'} \arrow[dd, dashed, no head, "\alpha'"] \\
& & & & \\
\Braket{\alpha} \arrow[r, no head]  & \Braket{\beta'} \arrow[rr, dashed, no head, "\beta'"'] & & \Braket{\beta'} \arrow[r, no head]  & \Braket{\alpha'}.
\end{tikzcd}
\]
\end{proposition}

\subsubsection{Construction of the hyperelliptic face}

Consider curves $\alpha,\alpha',\beta,\beta'$ on $\Sigma_g^b$ as in the statement of \thref{prop:hyp}. The union $\alpha \cup \alpha'\cup\beta\cup\beta'$ splits the surface into four components $S_1$, $S_2$, $S_3$ and $S_4$, each with a single new boundary component. We may assume that $\alpha$ and $\alpha'$ split the surface into $S_1 \cup_{\partial} S_2$ and $S_3 \cup_{\partial} S_4$, while $\beta$ and $\beta'$ into $S_1 \cup_{\partial} S_3$ and $S_2 \cup_{\partial} S_4$. 

Observe that exactly two surfaces $S_i$ inherit an odd spin structure. Indeed, it is easy to see that $g(S_i \cup_{\partial} S_j)=g(S_i)+g(S_j)$ in all the above cases, hence the union of symplectic bases for $S_i$ and $S_j$ gives a symplectic basis for $S_i\cup_{\partial} S_j$, and the Arf invariant is additive (even if we are not gluing along a whole boundary component).

We can then assume that $S_1$ and $S_4$ inherit an odd spin structure. Choose spin cut-systems on $S_2$ and $S_3$ and (partial) spin cut-systems on $S_1$ and $S_4$ with $g(S_1)-1$ and $g(S_4)-1$ curves respectively, and cut the surface along all these $1$-curves. We get a surface $S$ of genus $3$ with an induced even spin structure.

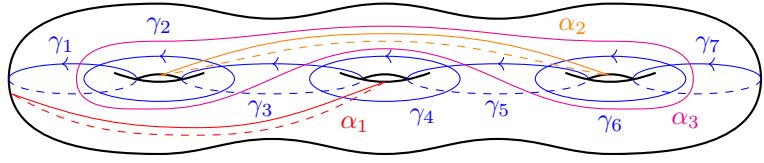
\begin{figure}
\centering
\begin{tikzpicture}
\draw [thick] (0,0) to [out=90, in=-180] (2,1) to [out=0, in=180] (3.5,.8) to [out=0, in=180] (5,1) to [out=0, in=180] (6.5,.8) to [out=0, in=180] (8,1) to [out=0, in=90] (10,0) to [out=-90, in=0] (8,-1) to [out=-180, in=0] (6.5,-.8) to [out=-180, in=0] (5,-1) to [out=-180, in=0] (3.5,-.8) to [out=-180, in=0] (2,-1) to [out=-180, in=-90] (0,0);
\draw [thick] (1.4,.08) to [out=-20, in=-160] (2.6,.08);
\draw [thick] (1.7,0) to [out=20, in=160] (2.3,0);
\draw [blue] (1,0) arc (180:90:1 and .3);
\draw [blue,->] (1,0) arc (-180:90:1 and .3);
\node [blue] at (2,.7) {$\gamma_2$};
\draw [blue] (0,0) arc (180:100:.85 and .2);
\draw [blue,->] (1.7,0) arc (0:100:.85 and .2);
\draw [blue, dashed] (0,0) arc (180:360:.85 and .2);
\node [blue] at (.7,.45) {$\gamma_1$};
\draw [blue] (2.3,0) arc (180:90:1.2 and .2);
\draw [blue,->] (4.7,0) arc (0:90:1.2 and .2);
\draw [blue, dashed] (2.3,0) arc (-180:0:1.2 and .2);
\node [blue] at (3.35,-.4) {$\gamma_3$};
\draw [thick] (4.4,.08) to [out=-20, in=-160] (5.6,.08);
\draw [thick] (4.7,0) to [out=20, in=160] (5.3,0);
\draw [blue] (4,0) arc (180:90:1 and .3);
\draw [blue,->] (4,0) arc (-180:90:1 and .3);
\node [blue] at (5.5,-.5) {$\gamma_4$};;
\node [red] at (4.6,-.6) {$\alpha_1$};
\draw [red] (5,-.04) to [out=-160, in=0] (2.5,-.65) to [out=180, in=-20] (.04,-.2);
\draw [red, dashed] (5,-.04) to [out=-150, in=0] (2.5,-.75) to [out=180, in=-30] (.04,-.2);
\draw [blue] (5.3,0) arc (180:90:1.2 and .2);
\draw [blue,->] (7.7,0) arc (0:90:1.2 and .2);
\draw [blue, dashed] (5.3,0) arc (-180:0:1.2 and .2);
\node [blue] at (6.5,-.4) {$\gamma_5$};
\draw [thick] (7.4,.08) to [out=-20, in=-160] (8.6,.08);
\draw [thick] (7.7,0) to [out=20, in=160] (8.3,0);
\draw [blue] (7,0) arc (180:90:1 and .3);
\draw [blue,->] (7,0) arc (-180:90:1 and .3);
\node [blue] at (8,-.6) {$\gamma_6$};
\draw [blue,->] (10,0) arc (0:80:.85 and .2);
\draw [blue] (8.3,0) arc (180:80:.85 and .2);
\draw [blue, dashed] (10,0) arc (360:180:.85 and .2);
\node [blue] at (9.3,.45) {$\gamma_7$};
\draw [orange] (2,.04) to [out=20, in=180] (5,.6) to [out=0, in=160] (8,.04);
\draw [orange, dashed] (2,.04) to [out=10, in=180] (5,.5) to [out=0, in=170] (8,.04);
\node [orange] at (7.5,.7) {$\alpha_2$};
\draw [magenta] (5,.7) to [out=180, in=0] (2,.5) to [out=180, in=90] (.9,0) to [out=-90, in=180] (2,-.4) to [out=0, in=180] (5,.4) to [out=0, in=180] (8,-.4) to [out=0, in=-90] (9.1,0) to [out=90, in=0] (8,.5) to [out=180, in=0] (5,.7);
\node [magenta] at (9,-.55) {$\alpha_3$};
\end{tikzpicture}
\caption{A $7$-chain of admissible curve on a surface of genus $3$ with an even spin structure, and the spin cut-system $v$ corresponding to the chosen orientations.}
\label{fig:hypv}
\end{figure}

Assume that $\gamma_1,\dots,\gamma_7$ is a $7$-chain of admissible curves on $S$ (see Figure \ref{fig:hypv}). Recall that if $\overline{S}$ is the surface obtained by capping all boundary components of $S$ with disks, we have the relation
\[
(t_{\gamma_1}\dots t_{\gamma_6}t_{\gamma_7}^2t_{\gamma_6}\dots t_{\gamma_1})^2=1
\]
in $\Modd(\overline{S})$. More generally, let $\delta$ be the boundary of $\gamma_1 \cup \dots \cup \gamma_6$ in $S$, and let $\delta_1,\delta_2$ be the two boundary components of $\gamma_1\cup\dots \cup \gamma_7$. Then the relation
\begin{equation}
\label{hyp3}
(t_{\gamma_1}\dots t_{\gamma_6}t_{\gamma_7}^2t_{\gamma_6}\dots t_{\gamma_1})^2=t_{\delta_1}^2t_{\delta_2}^2t_{\delta}^{-1}
\end{equation}
holds in $\Modd(S)$, as the result of combining two positive $7$-chain relations and a negative $6$-chain relation. 

Now we construct a spin cut-system on $S$, which will be the first vertex of the hyperelliptic face. Orient $\gamma_1,\dots,\gamma_7$ so that $\gamma_i\cdot \gamma_{i+1}=1$. Consider the following curves on $S$ (see Figure \ref{fig:hypv}): 
\begin{itemize}
\item $\alpha_1:=\gamma_1+_{a_1}\gamma_3$, where $a_1$ is the arc of $\gamma_2$ going from $\gamma_1\cap\gamma_2$ to $\gamma_2\cap\gamma_3$;
\item $\alpha_2:=\gamma_3+_{a_2}\gamma_5$, where $a_2$ is the arc of $\gamma_4$ from $\gamma_4\cap\gamma_5$ to $\gamma_3\cap\gamma_4$;
\item $\alpha_3:=\gamma_2+_{a_3}\gamma_6$, where $a_3$ is the arc that goes from $\gamma_5\cap\gamma_6$ to $\gamma_2\cap \gamma_3$ along $\gamma_5$, $\gamma_4$ and $\gamma_3$. 
\end{itemize}
It is easy to see that $\alpha_1$, $\alpha_2$, $\alpha_3$ form a spin cut-system $v$ on $S$. Note that they all lie on a tubular neighborhood of $\gamma_1\cup\dots\cup\gamma_6$.

\begin{definition}
\thlabel{def:hyp}
Let $c_i$ be the $i$-th Dehn twist involved in the left hand side of the hyperelliptic relation (\ref{hyp3}). Set $h_1:=c_1$ and $h_{i}=(c_1\dots c_{i-1})*c_{i}$ for $i=2,\dots,28$. Complete $v$ to a spin cut-system on the whole of $\Sigma_g^b$. A \emph{hyperelliptic face} is a $28$-gon of the form
\begin{equation}
\label{eq:hypf}
v-h_1(v)-(h_2h_1)(v)-\dots-(h_{28}\dots h_1)(v)=v,
\end{equation}
where all the edges are of type (i), and as usual we are only writing the curves that change.
\end{definition}

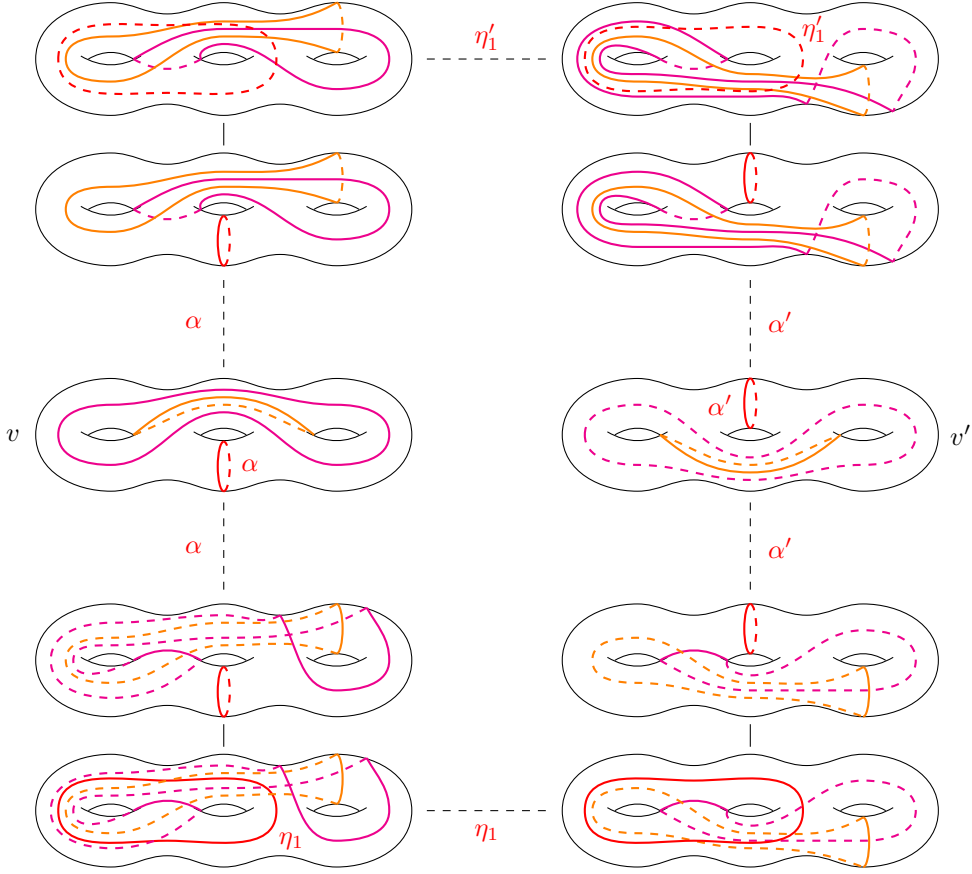
\begin{figure}
\centering
\begin{tikzpicture}
\draw (0,0) to [out=90, in=-180] (1,.75) to [out=0, in=180] (1.75,.6) to [out=0, in=180] (2.5,.75) to [out=0, in=180] (3.25,.6) to [out=0, in=180] (4,.75)  to [out=0, in=90] (5,0) to [out=-90, in=0] (4,-.75) to [out=-180, in=0] (3.25,-.6) to [out=-180, in=0] (2.5,-.75) to [out=-180, in=0] (1.75,-.6) to [out=-180, in=0] (1,-.75) to [out=-180, in=-90] (0,0);
\draw (.6,.04) to [out=-30, in=-150] (1.4,.04);
\draw (.7,0) to [out=30, in=150] (1.3,0);
\draw (2.1,.04) to [out=-30, in=-150] (2.9,.04);
\draw (2.2,0) to [out=30, in=150] (2.8,0);
\draw (3.6,.04) to [out=-30, in=-150] (4.4,.04);
\draw (3.7,0) to [out=30, in=150] (4.3,0);
\draw [thick, red] (2.5,-.75) arc (270:90:.08 and .33);
\draw [thick, red, dashed] (2.5,-.75) arc (-90:90:.08 and .33);
\draw [thick, magenta] (.3,0) to [out=90, in=-180] (1,.4) to [out=0, in=180] (2.5,.6) to [out=0, in=180] (4,.4)  to [out=0, in=90] (4.7,0) to [out=-90, in=0] (4,-.4) to [out=-180, in=0] (2.5,.3) to [out=-180, in=0] (1,-.4) to [out=-180, in=-90] (.3,0);
\draw [thick, orange] (1.3,0) to [out=40, in=180] (2.5,.5) to [out=0, in=140] (3.7,0);
\draw [thick, orange, dashed] (1.3,0) to [out=20, in=180] (2.5,.4) to [out=0, in=160] (3.7,0);
\node [red] at (2.85,-.4) {$\alpha$};
\draw (7,0) to [out=90, in=-180] (8.0,.75) to [out=0, in=180] (8.75,.6) to [out=0, in=180] (9.5,.75) to [out=0, in=180] (10.25,.6) to [out=0, in=180] (11,.75)  to [out=0, in=90] (12,0) to [out=-90, in=0] (11,-.75) to [out=-180, in=0] (10.25,-.6) to [out=-180, in=0] (9.5,-.75) to [out=-180, in=0] (8.75,-.6) to [out=-180, in=0] (8,-.75) to [out=-180, in=-90] (7,0);
\draw (7.6,.04) to [out=-30, in=-150] (8.4,.04);
\draw (7.7,0) to [out=30, in=150] (8.3,0);
\draw (9.1,.04) to [out=-30, in=-150] (9.9,.04);
\draw (9.2,0) to [out=30, in=150] (9.8,0);
\draw (10.6,.04) to [out=-30, in=-150] (11.4,.04);
\draw (10.7,0) to [out=30, in=150] (11.3,0);
\draw [thick, red] (9.5,.75) arc (90:270:.08 and .33);
\draw [thick, red, dashed] (9.5,.75) arc (90:-90:.08 and .33);
\draw [thick, magenta, dashed] (7.3,0) to [out=90, in=-180] (8,.4) to [out=0, in=180] (9.5,-.3) to [out=0, in=180] (11,.4)  to [out=0, in=90] (11.7,0) to [out=-90, in=0] (11,-.4) to [out=-180, in=0] (9.5,-.6) to [out=-180, in=0] (8,-.4) to [out=-180, in=-90] (7.3,0);
\draw [thick, orange] (8.3,0) to [out=-40, in=180] (9.5,-.5) to [out=0, in=-140] (10.7,0);
\draw [thick, orange, dashed] (8.3,0) to [out=-20, in=180] (9.5,-.4) to [out=0, in=-160] (10.7,0);
\node [red] at (9.1,.4) {$\alpha'$};
\draw (0,3) to [out=90, in=-180] (1,3.75) to [out=0, in=180] (1.75,3.6) to [out=0, in=180] (2.5,3.75) to [out=0, in=180] (3.25,3.6) to [out=0, in=180] (4,3.75)  to [out=0, in=90] (5,3) to [out=-90, in=0] (4,2.25) to [out=-180, in=0] (3.25,2.4) to [out=-180, in=0] (2.5,2.25) to [out=-180, in=0] (1.75,2.4) to [out=-180, in=0] (1,2.25) to [out=-180, in=-90] (0,3);
\draw (0.6,3.04) to [out=-30, in=-150] (1.4,3.04);
\draw (0.7,3) to [out=30, in=150] (1.3,3);
\draw (2.1,3.04) to [out=-30, in=-150] (2.9,3.04);
\draw (2.2,3) to [out=30, in=150] (2.8,3);
\draw (3.6,3.04) to [out=-30, in=-150] (4.4,3.04);
\draw (3.7,3) to [out=30, in=150] (4.3,3);
\draw [thick, red] (2.5,2.25) arc (270:90:.08 and .33);
\draw [thick, red, dashed] (2.5,2.25) arc (-90:90:.08 and .33);
\draw [thick, magenta, dashed] (1.3,3) to [out=-30, in=-150] (2.2,3);
\draw [thick, magenta] (1.3,3) to [out=30, in=180] (2.5,3.4) to [out=0, in=180] (4,3.4)  to [out=0, in=90] (4.7,3) to [out=-90, in=0] (4,2.6) to [out=-180, in=0] (2.5,3.2) to [out=180, in=120] (2.2,3);
\draw [thick, orange, dashed] (4,3.75) arc (90:-90:.08 and .33);
\draw [thick, orange] (4,3.75) to [out=-150, in=0] (2.5,3.5) to [out=180, in=0] (1,3.3) to [out=180, in=90] (0.4,3) to [out=-90, in=180] (1,2.7) to [out=0, in=180] (2.5,3.3) to [out=0, in=160] (4,3.09);
\draw (7,3) to [out=90, in=-180] (8.0,3.75) to [out=0, in=180] (8.75,3.6) to [out=0, in=180] (9.5,3.75) to [out=0, in=180] (10.25,3.6) to [out=0, in=180] (11,3.75)  to [out=0, in=90] (12,3) to [out=-90, in=0] (11,2.25) to [out=-180, in=0] (10.25,2.4) to [out=-180, in=0] (9.5,2.25) to [out=-180, in=0] (8.75,2.4) to [out=-180, in=0] (8,2.25) to [out=-180, in=-90] (7,3);
\draw (7.6,3.04) to [out=-30, in=-150] (8.4,3.04);
\draw (7.7,3) to [out=30, in=150] (8.3,3);
\draw (9.1,3.04) to [out=-30, in=-150] (9.9,3.04);
\draw (9.2,3) to [out=30, in=150] (9.8,3);
\draw (10.6,3.04) to [out=-30, in=-150] (11.4,3.04);
\draw (10.7,3) to [out=30, in=150] (11.3,3);
\draw [thick, red] (9.5,3.75) arc (90:270:.08 and .33);
\draw [thick, red, dashed] (9.5,3.75) arc (90:-90:.08 and .33);
\draw [thick, magenta, dashed] (8.3,3) to [out=-30, in=-150] (9.2,3);
\draw [thick, magenta] (9.2,3) to [out=150, in=0] (8,3.5) to [out=180, in=90] (7.2,3) to [out=-90, in=180] (8,2.5) to [out=0, in=180] (9.5,2.5) to [out=0, in=150] (10.25,2.4);
\draw [thick, magenta, dashed] (10.25,2.4) to [out=60, in=180] (11,3.4) to [out=0, in=90] (11.7,3) to [out=-90, in=70] (11.4,2.3);
\draw [thick, magenta] (8.3,3) to [out=150, in=90] (7.5,3) to [out=-90, in=180] (8,2.8) to [out=0, in=180] (9.5,2.7) to [out=0, in=150] (11.4,2.3);
\draw [thick, orange, dashed] (11,2.25) arc (-90:90:.08 and .33);
\draw [thick, orange] (11,2.91) to [out=-150, in=0] (9.5,2.8) to [out=180, in=0] (8,3.3) to [out=180, in=90] (7.4,3) to [out=-90, in=180] (8,2.7) to [out=0, in=180] (9.5,2.6) to [out=0, in=160] (11,2.25);
\draw (0,5) to [out=90, in=-180] (1,5.75) to [out=0, in=180] (1.75,5.6) to [out=0, in=180] (2.5,5.75) to [out=0, in=180] (3.25,5.6) to [out=0, in=180] (4,5.75)  to [out=0, in=90] (5,5) to [out=-90, in=0] (4,4.25) to [out=-180, in=0] (3.25,4.4) to [out=-180, in=0] (2.5,4.25) to [out=-180, in=0] (1.75,4.4) to [out=-180, in=0] (1,4.25) to [out=-180, in=-90] (0,5);
\draw (0.6,5.04) to [out=-30, in=-150] (1.4,5.04);
\draw (0.7,5) to [out=30, in=150] (1.3,5);
\draw (2.1,5.04) to [out=-30, in=-150] (2.9,5.04);
\draw (2.2,5) to [out=30, in=150] (2.8,5);
\draw (3.6,5.04) to [out=-30, in=-150] (4.4,5.04);
\draw (3.7,5) to [out=30, in=150] (4.3,5);
\draw [thick, magenta, dashed] (1.3,5) to [out=-30, in=-150] (2.2,5);
\draw [thick, magenta] (1.3,5) to [out=30, in=180] (2.5,5.4) to [out=0, in=180] (4,5.4)  to [out=0, in=90] (4.7,5) to [out=-90, in=0] (4,4.6) to [out=-180, in=0] (2.5,5.2) to [out=180, in=120] (2.2,5);
\draw [thick, orange, dashed] (4,5.75) arc (90:-90:.08 and .33);
\draw [thick, orange] (4,5.75) to [out=-150, in=0] (2.5,5.5) to [out=180, in=0] (1,5.3) to [out=180, in=90] (0.4,5) to [out=-90, in=180] (1,4.7) to [out=0, in=180] (2.5,5.3) to [out=0, in=160] (4,5.09);
\draw [red, thick, dashed] (.3,5) to [out=90, in=180] (1.75,5.45) to [out=0, in=90] (3.2,5) to [out=-90, in=0] (1.75,4.55) to [out=180, in=-90] (.3,5);
\draw (7,5) to [out=90, in=-180] (8.0,5.75) to [out=0, in=180] (8.75,5.6) to [out=0, in=180] (9.5,5.75) to [out=0, in=180] (10.25,5.6) to [out=0, in=180] (11,5.75)  to [out=0, in=90] (12,5) to [out=-90, in=0] (11,4.25) to [out=-180, in=0] (10.25,4.4) to [out=-180, in=0] (9.5,4.25) to [out=-180, in=0] (8.75,4.4) to [out=-180, in=0] (8,4.25) to [out=-180, in=-90] (7,5);
\draw (7.6,5.04) to [out=-30, in=-150] (8.4,5.04);
\draw (7.7,5) to [out=30, in=150] (8.3,5);
\draw (9.1,5.04) to [out=-30, in=-150] (9.9,5.04);
\draw (9.2,5) to [out=30, in=150] (9.8,5);
\draw (10.6,5.04) to [out=-30, in=-150] (11.4,5.04);
\draw (10.7,5) to [out=30, in=150] (11.3,5);
\draw [thick, magenta, dashed] (8.3,5) to [out=-30, in=-150] (9.2,5);
\draw [thick, magenta, dashed] (10.25,4.4) to [out=60, in=180] (11,5.4) to [out=0, in=90] (11.7,5) to [out=-90, in=70] (11.4,4.3);
\draw [thick, magenta] (8.3,5) to [out=150, in=90] (7.5,5) to [out=-90, in=180] (8,4.8) to [out=0, in=180] (9.5,4.7) to [out=0, in=150] (11.4,4.3);
\draw [thick, magenta] (9.2,5) to [out=150, in=0] (8,5.5) to [out=180, in=90] (7.2,5) to [out=-90, in=180] (8,4.5) to [out=0, in=180] (9.5,4.5) to [out=0, in=150] (10.25,4.4);
\draw [thick, orange, dashed] (11,4.25) arc (-90:90:.08 and .33);
\draw [thick, orange] (11,4.91) to [out=-150, in=0] (9.5,4.8) to [out=180, in=0] (8,5.3) to [out=180, in=90] (7.4,5) to [out=-90, in=180] (8,4.7) to [out=0, in=180] (9.5,4.6) to [out=0, in=160] (11,4.25);
\draw [red, thick, dashed] (7.3,5) to [out=90, in=180] (8.75,5.4) to [out=0, in=90] (10.2,5) to [out=-90, in=0] (8.75,4.6) to [out=180, in=-90] (7.3,5);
\node [red] at (10.35,5.4) {$\eta_1'$};
\draw (0,-3) to [out=-90, in=-180] (1,-3.75) to [out=0, in=180] (1.75,-3.6) to [out=0, in=180] (2.5,-3.75) to [out=0, in=180] (3.25,-3.6) to [out=0, in=180] (4,-3.75)  to [out=0, in=-90] (5,-3) to [out=90, in=0] (4,-2.25) to [out=-180, in=0] (3.25,-2.4) to [out=-180, in=0] (2.5,-2.25) to [out=-180, in=0] (1.75,-2.4) to [out=-180, in=0] (1,-2.25) to [out=-180, in=90] (0,-3);
\draw (0.6,-2.96) to [out=-30, in=-150] (1.4,-2.96);
\draw (0.7,-3) to [out=30, in=150] (1.3,-3);
\draw (2.1,-2.96) to [out=-30, in=-150] (2.9,-2.96);
\draw (2.2,-3) to [out=30, in=150] (2.8,-3);
\draw (3.6,-2.96) to [out=-30, in=-150] (4.4,-2.96);
\draw (3.7,-3) to [out=30, in=150] (4.3,-3);
\draw [thick, red] (2.5,-3.75) arc (270:90:.08 and .33);
\draw [thick, red, dashed] (2.5,-3.75) arc (-90:90:.08 and .33);
\draw [thick, magenta] (1.3,-3) to [out=30, in=150] (2.2,-3);
\draw [thick, magenta] (3.25,-2.4) to [out=-60, in=180] (4,-3.4) to [out=0, in=-90] (4.7,-3) to [out=90, in=-60] (4.4,-2.3);
\draw [thick, magenta, dashed] (3.25,-2.4) to [out=-150, in=0] (2.5,-2.4) to [out=180, in=0] (1,-2.5) to [out=180, in=90] (.2,-3) to [out=-90, in=180] (1,-3.5) to [out=0, in=-140] (2.2,-3);
\draw [thick, magenta, dashed] (4.4,-2.3) to [out=-150, in=0] (2.5,-2.7) to [out=180, in=0] (1,-2.8) to [out=180, in=90] (.5,-3) to [out=-90, in=-150] (1.3,-3);
\draw [thick, orange] (4,-2.25) arc (90:-90:.08 and .33);
\draw [thick, orange, dashed] (4,-2.25) to [out=-150, in=0] (2.5,-2.5) to [out=180, in=0] (1,-2.7) to [out=180, in=90] (0.4,-3) to [out=-90, in=180] (1,-3.3) to [out=0, in=180] (2.5,-2.8) to [out=0, in=160] (4,-2.91);
\draw (7,-3) to [out=-90, in=-180] (8.0,-3.75) to [out=0, in=180] (8.75,-3.6) to [out=0, in=180] (9.5,-3.75) to [out=0, in=180] (10.25,-3.6) to [out=0, in=180] (11,-3.75)  to [out=0, in=-90] (12,-3) to [out=90, in=0] (11,-2.25) to [out=-180, in=0] (10.25,-2.4) to [out=-180, in=0] (9.5,-2.25) to [out=-180, in=0] (8.75,-2.4) to [out=-180, in=0] (8,-2.25) to [out=-180, in=90] (7,-3);
\draw (7.6,-2.96) to [out=-30, in=-150] (8.4,-2.96);
\draw (7.7,-3) to [out=30, in=150] (8.3,-3);
\draw (9.1,-2.96) to [out=-30, in=-150] (9.9,-2.96);
\draw (9.2,-3) to [out=30, in=150] (9.8,-3);
\draw (10.6,-2.96) to [out=-30, in=-150] (11.4,-2.96);
\draw (10.7,-3) to [out=30, in=150] (11.3,-3);
\draw [thick, red] (9.5,-2.25) arc (90:270:.08 and .33);
\draw [thick, red, dashed] (9.5,-2.25) arc (90:-90:.08 and .33);
\draw [thick, magenta] (8.3,-3) to [out=30, in=150] (9.2,-3);
\draw [thick, magenta, dashed] (8.3,-3) to [out=-30, in=180] (9.5,-3.4) to [out=0, in=180] (11,-3.4)  to [out=0, in=-90] (11.7,-3) to [out=90, in=0] (11,-2.6) to [out=-180, in=0] (9.5,-3.2) to [out=180, in=-120] (9.2,-3);
\draw [thick, orange] (11,-3.75) arc (-90:90:.08 and .33);
\draw [thick, orange, dashed] (11,-3.09) to [out=-150, in=0] (9.5,-3.3) to [out=180, in=0] (8,-2.7) to [out=180, in=90] (7.4,-3) to [out=-90, in=180] (8,-3.3) to [out=0, in=180] (9.5,-3.5) to [out=0, in=160] (11,-3.75);
\draw (0,-5) to [out=-90, in=-180] (1,-5.75) to [out=0, in=180] (1.75,-5.6) to [out=0, in=180] (2.5,-5.75) to [out=0, in=180] (3.25,-5.6) to [out=0, in=180] (4,-5.75)  to [out=0, in=-90] (5,-5) to [out=90, in=0] (4,-4.25) to [out=-180, in=0] (3.25,-4.4) to [out=-180, in=0] (2.5,-4.25) to [out=-180, in=0] (1.75,-4.4) to [out=-180, in=0] (1,-4.25) to [out=-180, in=90] (0,-5);
\draw (0.6,-4.96) to [out=-30, in=-150] (1.4,-4.96);
\draw (0.7,-5) to [out=30, in=150] (1.3,-5);
\draw (2.1,-4.96) to [out=-30, in=-150] (2.9,-4.96);
\draw (2.2,-5) to [out=30, in=150] (2.8,-5);
\draw (3.6,-4.96) to [out=-30, in=-150] (4.4,-4.96);
\draw (3.7,-5) to [out=30, in=150] (4.3,-5);
\draw [red, thick] (.3,-5) to [out=90, in=180] (1.75,-4.6) to [out=0, in=90] (3.2,-5) to [out=-90, in=0] (1.75,-5.4) to [out=180, in=-90] (.3,-5);
\draw [thick, magenta] (1.3,-5) to [out=30, in=150] (2.2,-5);
\draw [thick, magenta] (3.25,-4.4) to [out=-60, in=180] (4,-5.4)  to [out=0, in=-90] (4.7,-5) to [out=90, in=-60] (4.4,-4.3);
\draw [thick, magenta, dashed] (3.25,-4.4) to [out=-150, in=0] (2.5,-4.4) to [out=180, in=0] (1,-4.5) to [out=180, in=90] (.2,-5) to [out=-90, in=180] (1,-5.5) to [out=0, in=-140] (2.2,-5);
\draw [thick, magenta, dashed] (4.4,-4.3) to [out=-150, in=0] (2.5,-4.7) to [out=180, in=0] (1,-4.8) to [out=180, in=90] (.5,-5) to [out=-90, in=-150] (1.3,-5);
\draw [thick, orange] (4,-4.25) arc (90:-90:.08 and .33);
\draw [thick, orange, dashed] (4,-4.25) to [out=-150, in=0] (2.5,-4.5) to [out=180, in=0] (1,-4.7) to [out=180, in=90] (0.4,-5) to [out=-90, in=180] (1,-5.3) to [out=0, in=180] (2.5,-4.8) to [out=0, in=160] (4,-4.91);
\node [red] at (3.4,-5.4) {$\eta_1$};
\draw (7,-5) to [out=-90, in=-180] (8.0,-5.75) to [out=0, in=180] (8.75,-5.6) to [out=0, in=180] (9.5,-5.75) to [out=0, in=180] (10.25,-5.6) to [out=0, in=180] (11,-5.75)  to [out=0, in=-90] (12,-5) to [out=90, in=0] (11,-4.25) to [out=-180, in=0] (10.25,-4.4) to [out=-180, in=0] (9.5,-4.25) to [out=-180, in=0] (8.75,-4.4) to [out=-180, in=0] (8,-4.25) to [out=-180, in=90] (7,-5);
\draw (7.6,-4.96) to [out=-30, in=-150] (8.4,-4.96);
\draw (7.7,-5) to [out=30, in=150] (8.3,-5);
\draw (9.1,-4.96) to [out=-30, in=-150] (9.9,-4.96);
\draw (9.2,-5) to [out=30, in=150] (9.8,-5);
\draw (10.6,-4.96) to [out=-30, in=-150] (11.4,-4.96);
\draw (10.7,-5) to [out=30, in=150] (11.3,-5);
\draw [thick, magenta] (8.3,-5) to [out=30, in=150] (9.2,-5);
\draw [thick, magenta, dashed] (8.3,-5) to [out=-30, in=180] (9.5,-5.4) to [out=0, in=180] (11,-5.4)  to [out=0, in=-90] (11.7,-5) to [out=90, in=0] (11,-4.6) to [out=-180, in=0] (9.5,-5.2) to [out=180, in=-120] (9.2,-5);
\draw [thick, orange] (11,-5.75) arc (-90:90:.08 and .33);
\draw [thick, orange, dashed] (11,-5.09) to [out=-150, in=0] (9.5,-5.3) to [out=180, in=0] (8,-4.7) to [out=180, in=90] (7.4,-5) to [out=-90, in=180] (8,-5.3) to [out=0, in=180] (9.5,-5.5) to [out=0, in=160] (11,-5.75);
\draw [red, thick] (7.3,-5) to [out=90, in=180] (8.75,-4.6) to [out=0, in=90] (10.2,-5) to [out=-90, in=0] (8.75,-5.4) to [out=180, in=-90] (7.3,-5);
\draw [dashed] (2.5,.9) -- (2.5,2.1);
\draw (2.5,3.85) -- (2.5,4.15);
\draw [dashed] (2.5,-.9) -- (2.5,-2.1);
\draw (2.5,-3.85) -- (2.5,-4.15);
\draw [dashed] (9.5,.9) -- (9.5,2.1);
\draw (9.5,3.85) -- (9.5,4.15);
\draw [dashed] (9.5,-.9) -- (9.5,-2.1);
\draw (9.5,-3.85) -- (9.5,-4.15);
\draw [dashed] (5.2,5) -- (6.8,5);
\draw [dashed] (5.2,-5) -- (6.8,-5);
\node at (-.3,0) {$v$};
\node at (12.3,0) {$v'$};
\node [red] at (2.1,1.5) {$\alpha$};
\node [red] at (9.9,1.5) {$\alpha'$};
\node [red] at (6,5.3) {$\eta_1'$};
\node [red] at (2.1,-1.5) {$\alpha$};
\node [red] at (9.9,-1.5) {$\alpha'$};
\node [red] at (6,-5.3) {$\eta_1$};
\end{tikzpicture}
\caption{Some of the vertices of the hyperelliptic face of Figure \ref{fig:hypv}.}
\label{fig:hypface}
\end{figure}

Some verifications are needed. First of all, notice that
\begin{equation}
\label{eq:hk}
h_{k}\dots h_1=(c_1\dots c_{k-1})c_{k}(c_{k-1}^{-1}\dots c_1^{-1}) (c_1\dots c_{k-2})c_{k-1}(c_{k-2}^{-1}\dots c_1^{-1})\dots c_1=c_1\dots c_{k},
\end{equation}
hence indeed $h_{28}(v)=v$ by (\ref{hyp3}), as $\alpha_1,\alpha_2,\alpha_3$ are contained in a neighborhood of $\gamma_1\cup \dots \cup \gamma_6$. Moreover, each curve $\gamma_i$ intersects once a curve of $v=\Braket{\alpha_1,\alpha_2,\alpha_3}$ and is disjoint from the other two. Hence, the same is true for $(h_k\dots h_1)(\gamma_i)$ and $(h_k\dots h_1)(v)$. Now, 
\[
h_{k+1}=(c_1\dots c_k)*c_{k+1}=(h_k\dots h_1)*c_{k+1}
\]
is the Dehn twist along some curve $(h_k\dots h_1)(\gamma_i)$, so there is an edge of type (i) $(h_k\dots h_1)(v)-(h_{k+1}h_k\dots h_1)(v)$.

Before proving \thref{prop:hyp}, we need the following lemmas.

\begin{lemma}
\thlabel{lem:hyp}
The hyperelliptic face is made up of four segments, whose fixed curves satisfy the properties \ref{hy1} and \ref{hy3} of \thref{prop:hyp}.
\end{lemma}

\begin{proof}
Let $\mathbf{p}$ be a hyperelliptic face. We can assume that $g=3$. Let $\gamma_1,\dots,\gamma_7$ be the $7$-chain that defines $\mathbf{p}$. Choose orientations as before, and let $v=\Braket{\alpha_1,\alpha_2,\alpha_3}$ be the corresponding vertex. Then, $\mathbf{p}$ is of the form (\ref{eq:hypf}). 

By construction, $\alpha_1$ only intersects $\gamma_4$, hence by the above reasoning the curve $h_k\dots h_1(\alpha_1)$ is involved in an edge only at the four occurrences of $\gamma_4$ in the hyperelliptic relation, and by (\ref{eq:hk}) the corresponding curves are
\begin{gather*}
\eta_1:=(t_{\gamma_1}t_{\gamma_2}t_{\gamma_3}t_{\gamma_4})(\alpha_1), \qquad \alpha_1':=(t_{\gamma_1}\dots t_{\gamma_6}t_{\gamma_7}^2t_{\gamma_6}t_{\gamma_5}t_{\gamma_4})(\alpha_1), \\
\eta_1':=(t_{\gamma_1}\dots t_{\gamma_6}t_{\gamma_7}^2t_{\gamma_6}\dots t_{\gamma_1})(t_{\gamma_1}t_{\gamma_2}t_{\gamma_2}t_{\gamma_4})(\alpha_1)=(t_{\gamma_1}^{-1}t_{\gamma_2}^{-1}t_{\gamma_3}^{-1}t_{\gamma_4}^{-1})(\alpha_1), \\
(t_{\gamma_1}\dots t_{\gamma_6}t_{\gamma_7}^2t_{\gamma_6}\dots t_{\gamma_1})(t_{\gamma_1}\dots t_{\gamma_6}t_{\gamma_7}^2t_{\gamma_6}t_{\gamma_5}t_{\gamma_4})(\alpha_1)=\alpha_1
\end{gather*}
(see Figure \ref{fig:hypface}). All these are nonseparating $1$-curves, and satisfy \ref{hy3} by construction. 

It can be shown that $\eta_1$ and $\eta_1'$ are isotopic to the arc sums $\gamma_2+_c\gamma_4$ and $\gamma_2+_{c'}\gamma_4$ respectively, where $c$ is the arc of $\gamma_3$ from $\gamma_3\cap\gamma_4$ to $\gamma_2 \cap \gamma_3$ and $c'=\gamma_3\setminus c$ is the complement. Hence, that $\eta_1$ and $\eta_1'$ are homologous and cobound a two-holed torus which is a tubular neighborhood of the $0$-curves $\gamma_2,\gamma_3,\gamma_4$. Similarly, $\alpha_1'=\gamma_1+_{a_1'}\gamma_3$, where $a_1'$ is the complement of $a_1$ in $\gamma_2$, i.e. the arc of $\gamma_2$ that goes from $\gamma_2\cap \gamma_3$ to $\gamma_1 \cap \gamma_3$, and $\alpha_1,\alpha_1'$ cobound a tubular neighborhood of $\gamma_1\cup\gamma_2\cup\gamma_3$. Hence, \ref{hy1} is also verified. \qedhere
\end{proof}

\begin{remark}
\thlabel{rem:h1}
With similar arguments, it can be shown that choosing the opposite orientation on the $\gamma_i$ we get a spin cut-system $v'$ which is already included in (\ref{eq:hypf}), namely, $(h_{14}\dots h_1)(v)$, and the exact same path as (\ref{eq:hypf}) but starting at $v'$ (see Figure \ref{fig:hypface}). 
\end{remark}

\begin{lemma}
\thlabel{lem:r1s1}
Let $\mathbf{p}$ be a path in $X_g$ of radius $1$ with respect to some curve $\alpha$. If $\mathbf{p}$ contains just one segment with $d_{\alpha}=1$, then it is null-homotopic.
\end{lemma}

\begin{proof}
It suffices to adapt the proof of \cite[Proposition 19]{waj:elem}. Using its notations, we can construct $1$-curves $\delta_1,\delta_2$ as follows. If $\abs{\gamma_i \cap \beta}=0$ we set $\delta_i:=\gamma_i$. Otherwise, if $\gamma_i$ is homologous to $\alpha$, then $\alpha \cup \beta \cup \gamma_i$ splits $\Sigma_g^b$ into two subsurfaces, one of which must contain some $1$-curve $\delta_i$ (compare \thref{rem:hyp}). Finally, if $\gamma_i$ and $\alpha$ are not homologous, the boundary components of a regular neighborhood of $\alpha \cup \beta \cup \gamma_i$ are nonseparating $1$-curves, and we call $\delta_i$ one of them. Now, by \thref{prop:conn} we can construct a path $\Braket{\beta,\delta_1}--\Braket{\beta,\delta_2}$.\qedhere
\end{proof}

\begin{proof}[Proof of \thref{prop:hyp}]
As already observed, we may assume that the genus is $3$ and that the union $\alpha\cup\alpha'\cup\beta\cup\beta'$ splits the surface into two disks $S_2$ and $S_3$ and two one-holed odd tori $S_1$ and $S_4$ (with extra boundary components coming from those of $\Sigma_g^b$). We are going to construct a $7$-chain of $0$-curves $\gamma_1,\dots,\gamma_7$ such that $\alpha=\alpha_1$ and $\beta$ is equal to $\eta_1$ or $\eta_1'$, in the notation of \thref{lem:hyp}. 

Choose geometric symplectic bases $\gamma_2,\gamma_3$ for $S_1$ and $\gamma_7,\gamma_6$ for $S_4$. By construction, $\gamma_2,\gamma_3,\gamma_6$  and $\gamma_7$ are all admissible. Define new curves as follows:
\begin{itemize}
\item $\gamma_4:=\gamma_2+_{b_1}\beta$, where $b_1$ is some arc in $S_1\setminus (\gamma_2\cup\gamma_3)$;
\item $\gamma_1:=\gamma_3+_{b_2}\alpha$, where $b_2$ is some arc contained in the pair of pants bounded by $\beta \cup \gamma_2\cup \gamma_4$ and disjoint from $b_1$;
\item $\gamma_5:=\gamma_7+_{b_3}\alpha$, where $b_3$ is some arc in $S_4\setminus (\gamma_6\cup\gamma_7)$.
\end{itemize}
By construction, $\alpha$ is isotopic to the arc sum of $\gamma_1$ and $\gamma_3$ along some arc of $\gamma_2$, and we can choose the orientations of the $\gamma_i$ so that $\alpha=\alpha_1$. Moreover, $\beta$ is isotopic to the arc sum of $\gamma_2$ and $\gamma_4$ along some arc of $\gamma_3$, and up to renaming we can assume that it coincides with $\eta_1$. Note also that $\eta_1'$ intersects $\alpha$ and $\alpha'$ once, and $\beta',\eta_1'$ cobound an annulus (possibly with holes). Similarly, $\alpha_1'$ intersects $\beta$ and $\beta'$ once, and $\alpha',\alpha_1'$ cobound an annulus. 

As a consequence, given a vertex $w$ containing $\eta_1'$, we can connect it to some vertex containing $\beta'$ through a path with $d_{\eta_1'}=0$. Indeed, let $\xi$ be a curve that goes once through the annulus bounded by $\eta_1'\cup\beta'$ and is disjoint from the other curves of $w$. Up to Dehn twisting along $\eta_1'$, we may assume that $\xi$ is a $1$-curve, and we have an edge-path $\Braket{\eta_1'}-\Braket{\xi}-\Braket{\beta'}$. Similarly, we can connect a vertex containing $\alpha_1'$ to some vertex containing $\alpha'$ through a path with $d_{\alpha_1'}=0$.

We construct the required null-homotopic path as follows:
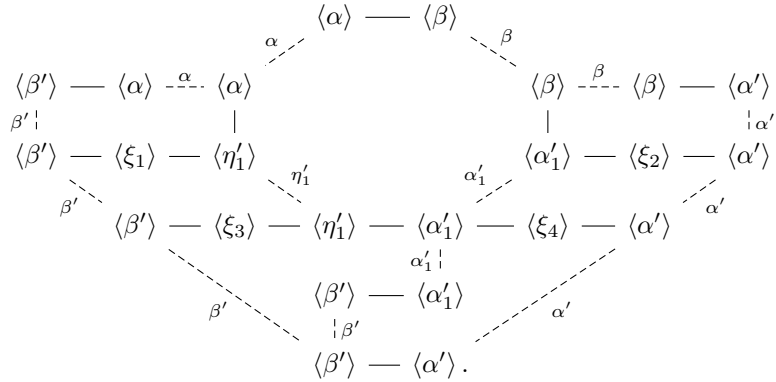
\[
\begin{tikzcd}[column sep=small, row sep=small]
& & & \Braket{\alpha} \arrow[r, no head] \arrow[dl, dashed, no head, "\alpha"'] & \Braket{\beta} \arrow[dr, dashed, no head, "\beta"] & & & \\
\Braket{\beta'} \arrow[r, no head] \arrow[d, dashed, no head, "\beta'"'] & \Braket{\alpha} \arrow [r, dashed, no head, "\alpha"] & \Braket{\alpha} \arrow[d, no head] & & & \Braket{\beta} \arrow[r, dashed, no head, "\beta"] \arrow[d, no head] & \Braket{\beta} \arrow[r, no head] & \Braket{\alpha'} \arrow[d, dashed, no head, "\alpha'"] \\
\Braket{\beta'} \arrow[r, no head] \arrow[dr, dashed, no head, "\beta'"'] & \Braket{\xi_1} \arrow [r, no head] & \Braket{\eta_1'} \arrow[dr, dashed, no head, "\eta_1'"] & & & \Braket{\alpha_1'} \arrow[r, no head] \arrow[dl, dashed, no head, "\alpha_1'"'] & \Braket{\xi_2} \arrow[r, no head] & \Braket{\alpha'} \arrow[dl, dashed, no head, "\alpha'"] \\
& \Braket{\beta'} \arrow[r, no head] \arrow[ddrr, dashed, no head, "\beta'"'] & \Braket{\xi_3} \arrow [r, no head] & \Braket{\eta_1'} \arrow[r, no head] & \Braket{\alpha_1'} \arrow[r, no head] \arrow [d, dashed, no head, "\alpha_1'"'] & \Braket{\xi_4} \arrow[r, no head] & \Braket{\alpha'} \arrow[ddll, dashed, no head, "\alpha'"] & \\
& & & \Braket{\beta'} \arrow[r, no head] \arrow[d, dashed, no head, "\beta'"] & \Braket{\alpha_1'} & & & \\
& & & \Braket{\beta'} \arrow[r, no head] & \Braket{\alpha'}. & & &
\end{tikzcd}
\]
Here, $\xi_1,\xi_2,\xi_3$ and $\xi_4$ are $1$-curves constructed as above. The central “octagon" is a hyperelliptic face, while the other paths have either radius $0$ or radius $1$ with a single segment of distance $1$ from the base curve. By \thref{prop:r0} and \thref{lem:r1s1} we conclude. \qedhere
\end{proof}

\begin{remark}
\thlabel{rem:h2}
A $7$-chain on a surface of genus $3$ is necessarily separating. However, notice that the $7$-chain that we constructed in the proof of \thref{prop:hyp} also separates $\Sigma_g^b$, and one of the two components is a sphere with $4$ holes, as it is a tubular neighborhood of $\gamma_3 \cup b_2 \cup \alpha \cup b_3 \cup \gamma_5$.
\end{remark}

\subsubsection{General paths of radius $1$}

Now we are ready to prove the main result of this subsection.

\begin{proposition}
All paths of radius 1 in $X_g$ are null-homotopic.
\end{proposition}

\begin{proof}
Let $\mathbf{p}$ be a path of radius $1$ around some curve $\alpha$, and let $v_0$ be a vertex of $\mathbf{p}$ containing $\alpha$. Then $\mathbf{p}$ can be split into a finite number of $\eta_i$-segments, in such a way that each $\eta_i$ is either disjoint from $\alpha$ or it intersects $\alpha$ exactly once. If $\abs{\alpha \cap \eta_i}=1$, choose a vertex $w_i$ of the $\eta_i$-segment. Then we can construct a shortcut 
\[
\begin{tikzcd}
v_0 \arrow[rr, no head, dashed, "\alpha"] & & \Braket{\alpha} \arrow[r, no head] & \Braket{\eta_i} \arrow[rr, no head, dashed, "\eta_i"] & & w_i.
\end{tikzcd}
\]
These shortcuts split $\mathbf{p}$ into a finite number of closed paths of radius $1$ around $\alpha$, each containing up to two segments with $d_{\alpha}=1$. Hence, we can assume that $\mathbf{p}$ contains up to two segments with $d_{\alpha}=1$. We dealt with the case of a single segment in \thref{lem:r1s1}. Assume then that $\mathbf{p}$ has two segments with $d_{\alpha}=1$, and call $\beta$ and $\gamma$ the fixed curves of the two segments. 

\textbf{Case 1: the $\beta$-segment and the $\gamma$-segment share a vertex.} In this case, $\beta$ and $\gamma$ are disjoint and not homologous, so the two boundary components of $\beta \cup \alpha \cup \gamma$ are nonseparating $1$-curves. If we define $\delta_2$ to be one of the boundary components, we can proceed as in the proof of \thref{lem:r1s1} and split $\mathbf{p}$ into three paths of radius $0$ and a path with a single segment of distance $1$, which are null-homotopic by \thref{prop:r0} and \thref{lem:r1s1}.

\textbf{Case 2: there is an edge $\Braket{\beta}-\Braket{\gamma}$.} Note that $\set{\alpha,\beta,\gamma}$ is a triple of $1$-curves which intersect pairwise once. We briefly described such triples in \thref{rem:112}\ref{noiii}, where we proved that their complement contains at most $g-2$ disjoint linearly independent $1$-curves. On the other hand, observe that it contains at least $g-3$ disjoint linearly independent $1$-curves. Indeed, suppose that it is the union of two subsurfaces $\Sigma_{g_1}^{b_1+1}$ and $\Sigma_{g_2}^{b_2+2}$ with $g_1+g_2=g-2$ and $b_1+b_2=b$. Then $\Sigma_{g_1}^{b_1+1}$ contains either $g_1$ or $g_1-1$ disjoint linearly independent $1$-curves, depending on whether its Arf invariant is $0$ or $1$. On the other hand, $\Sigma_{g_2}^{b_2+2}$ always contains $g_2$ disjoint linearly independent $1$-curves. Indeed, if its Arf invariant is $0$, just take a spin cut-system. If instead its Arf invariant is $1$, there exists a cut-system with $g_2-1$ $1$-curves and one $0$-curve $\eta$. Taking the arc sum of $\eta$ and one boundary component, which is a $0$-curve as we already observed, via an arc that is disjoint from the other curves, we get the last $1$-curve. If the complement of our triple is a connected subsurface $\Sigma_{g-3}^{b+3}$, we can repeat the argument for $\Sigma_{g_2}^{b_2+2}$ and take also the boundary component with spin value $1$. We say that a triple $\set{\alpha,\beta,\gamma}$ is \emph{good} if its complement contains $g-2$ disjoint linearly independent $1$-curves, and \emph{bad} otherwise (i.e. if its complement is disconnected and $\Sigma_{g_1}^{b_1+1}$ inherits an odd spin structure). By induction, it suffices to deal with bad triples when $g=3$ and with good triples when $g=2$.

\textbf{Case 2A: $g=2$ and $\set{\alpha,\beta,\gamma}$ is a good triple.} The complement of $\alpha\cup\beta\cup\gamma$ is the union of a disk and a cylinder. Call $\xi_1$ and $\xi_2$ the boundary components of the cylinder; by \thref{rem:oddss}, they are $0$-curves. Construct a curve $\delta$ which runs from $\xi_1$ to $\xi_2$ crossing only $\gamma$ once, then goes back along the cylinder. Now, $\delta$ must be a $1$-curve as $\set{\alpha,\beta,\xi_1,\delta}$ is a geometric symplectic basis for $H_1(\Sigma_2^b;\Z)$. We can then construct the following shortcut:
\[
\begin{tikzcd}[column sep=small]
\dots \arrow[rr, dashed, no head, "\beta"] & & v_2 \arrow[rrr, no head] \arrow[d, dashed, no head, "\beta"'] & & & u_2 \arrow[rr, dashed, no head, "\gamma"] \arrow[d, dashed, no head, "\gamma"] & & \dots\\ 
& & \Braket{\beta,\delta} \arrow[rr, dashed, no head, "\delta"'] & & \Braket{\delta} \arrow[r, no head] & \Braket{\gamma} & & 
\end{tikzcd}
\]
Since $\delta$ and $\alpha$ are disjoint and linearly independent, we can connect some vertex containing $\delta$ to $v_0$ via a path of radius 0 around $\alpha$. Thus, we split $\mathbf{p}$ into three paths with a single segment of distance 1 each, and we conclude by \thref{lem:r1s1}.

\textbf{Case 2B: $g=3$ and $\set{\alpha,\beta,\gamma}$ is a bad triple.} The arc sum of $\alpha$ with the separating boundary component of $\alpha\cup\beta\cup\gamma$ is a $1$-curve $\alpha'$ that is disjoint from $\alpha$ and homologous to it, and such that $\set{\alpha',\beta,\gamma}$ is a good triple (see Figure \ref{fig:bad}). So we can construct a shortcut
\[
\begin{tikzcd}[column sep=small]
\dots \arrow[rr, dashed, no head, "\beta"] & & v_2 \arrow[rrrr, no head] \arrow[d, dashed, no head, "\beta"'] & & & & u_2 \arrow[rr, dashed, no head, "\gamma"] \arrow[d, dashed, no head, "\gamma"] & & \dots\\ 
& & \Braket{\beta} \arrow[r, no head] & \Braket{\alpha'} \arrow[rr, dashed, no head, "\alpha'"'] & & \Braket{\alpha'} \arrow[r, no head] & \Braket{\gamma} & & 
\end{tikzcd}
\]
that splits $\mathbf{p}$ into a path with a good triple and a path with two non-adjacent segments of distance $1$. 

\textbf{Case 3: the $\beta$-segment and the $\gamma$ segment are joined by a subpath that has radius $0$ around $\alpha$.} Let $w$ be a vertex of this subpath, and call $\alpha'$ a curve of $w$ that is disjoint from $\alpha$. If we can join $v_0$ to $w$ by a path of radius $0$ around $\alpha$, we are done by \thref{lem:r1s1}. If such a path does not exist, then $\alpha$ and $\alpha'$ are homologous, and we are in the situation described by \thref{rem:hyp}. In particular, the genus is at least $3$, and we can assume that $g=3$ by induction. Moreover, the two components of $\Sigma \setminus (\alpha \cup \alpha')$ are odd tori. Call $v_1,v_2$ and $v_3, v_4$ the endpoints of the $\beta$-segment and of the $\gamma$-segment respectively. We are going to reduce to a situation where we can apply \thref{prop:hyp}.

First of all, we claim that $\abs{\alpha'\cap\beta}=\abs{\alpha'\cap \gamma}=1$. As a consequence, applying \thref{lem:r1s1} to suitable shortcuts, we can assume that $\mathbf{p}$ is made up of four segments. To prove the claim for $\beta$, call $u_2$ the first vertex after $v_2$. Since $d_{\alpha}(u_2)=0$, there is a nonseparating $1$-curve $\eta \in u_2$ that is disjoint from $\alpha$. Hence, $\alpha$ and $\eta$ must be homologous, and they must cut $\Sigma_3$ into two odd tori, otherwise we would be able to connect $v_0$ to $w$ via a path of radius $0$ around $\alpha$, contradicting \thref{rem:hyp}. In particular, $\beta$ cannot be disjoint from $\eta$, so $\abs{\beta \cap \eta}=1$. Now, there is a path from $w$ to a vertex containing $\eta$ that has radius $0$ around $\alpha$, and since there are no $1$-curves in $\Sigma_{3,b}\setminus(\alpha\cup\alpha')$ and $\alpha',\eta$ are homologous, they must coincide. The same reasoning works for $\abs{\alpha'\cap\gamma}$. 

Set $m:=\abs{\beta\cap \gamma}$. Applying \thref{lem:17}(a), we find nonseparating $1$-curves $\beta_1:=\beta$, $\beta_2,\dots,\beta_k:=\gamma$ such that $\abs{\beta_i\cap\alpha}=\abs{\beta_i\cap \alpha'}=1$ and $\abs{\beta_i \cap \beta_{i+1}} \le 1$. Hence, we can assume that $\abs{\beta \cap \gamma} \le 1$. If $\beta$ and $\gamma$ coincide, or if they are disjoint and not homologous, or if $\set{\alpha,\beta,\gamma}$ and $\set{\alpha',\beta,\gamma}$ are both good triples, then we can connect $v_2$ to $v_3$ via a path that contains only two segments, with fixed curves $\beta$ and $\gamma$ respectively, and split $\mathbf{p}$ into two paths that are null-homotopic by Cases 1 and 2A. 

\textbf{Case 3A: $\beta$ and $\gamma$ are disjoint and homologous.} Notice that the union $\beta\cup\gamma$ splits the surface into two odd tori. Hence, by \thref{prop:hyp} there is a null-homotopic path with exactly 4 segments, with fixed curves $\alpha, \beta,\alpha'$ and $\gamma$. Constructing shortcuts as follows, we reduce to \thref{lem:r1s1}:
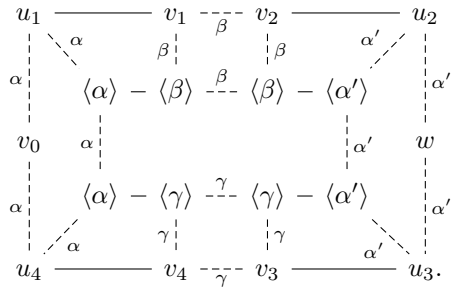
\[
\begin{tikzcd}[column sep=tiny, row sep=tiny]
u_1 \arrow[dddd, no head, dashed, "\alpha"'] \arrow[rr, no head] \arrow[dddr, dashed, no head, "\alpha"] & & v_1 \arrow[ddd, dashed, no head, "\beta"'] \arrow[rr, no head, dashed, "\beta"'] & & v_2 \arrow[rr, no head] \arrow[ddd, dashed, no head, "\beta"] & & u_2 \arrow[dddd, dashed, no head, "\alpha'"] \arrow[dddl, dashed, no head, "\alpha'"'] \\
& & & & & & \\
& & & & & & \\
& \Braket{\alpha} \arrow[dd, no head, dashed, "\alpha"'] \arrow[r, no head] & \Braket{\beta}  \arrow[rr, no head, dashed, "\beta"] & & \Braket{\beta} \arrow[r, no head] & \Braket{\alpha'} \arrow[dd, no head, dashed, "\alpha'"] & \\
v_0 \arrow[dddd, no head, dashed, "\alpha"'] & & & & & & w \arrow[dddd, dashed, no head, "\alpha'"] \\
& \Braket{\alpha} \arrow[r, no head] \arrow[dddl, dashed, no head, "\alpha"] & \Braket{\gamma} \arrow[rr, no head, dashed, "\gamma"] \arrow[ddd, dashed, no head, "\gamma"'] & & \Braket{\gamma} \arrow[ddd, dashed, no head, "\gamma"] \arrow[r, no head] & \Braket{\alpha'} \arrow[dddr, dashed, no head, "\alpha'"'] & \\
& & & & & & \\
& & & & & & \\
u_4 \arrow[rr, no head] & & v_4 \arrow[rr, no head, dashed, "\gamma"'] & & v_3 \arrow[rr, no head] & & u_3.
\end{tikzcd}
\]

\textbf{Case 3B: $\abs{\beta \cap \gamma}=1$.} Notice that $\set{\alpha,\beta,\gamma}$ and $\set{\alpha',\beta,\gamma}$ cannot be both bad triples. Assume that $\set{\alpha,\beta,\gamma}$ is a bad triple. Let $\beta'$ be the arc sum of $\beta$ and the separating boundary component of $\alpha\cup\beta\cup\gamma$ as in Figure \ref{fig:bad}. Then $\beta$ and $\beta'$ are disjoint and homologous, and they separate $\Sigma_3$ into two odd tori. Note that $\set{\alpha,\beta',\gamma}$ and $\set{\alpha',\beta',\gamma}$ are both good triples. Hence, we can apply \thref{prop:hyp} and find shortcuts as follows, reducing to Case 2A and \thref{lem:r1s1}:
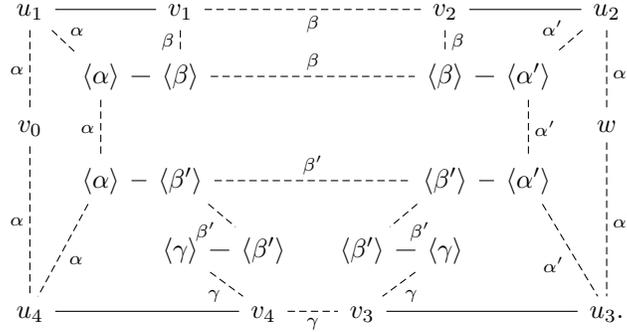
\[
\begin{tikzcd}[column sep=tiny, row sep=tiny]
u_1 \arrow[ddd, no head, dashed, "\alpha"'] \arrow[rr, no head] \arrow[ddr, dashed, no head, "\alpha"] & & v_1 \arrow[dd, dashed, no head, "\beta"'] \arrow[rrrr, no head, dashed, "\beta"'] & & & & v_2 \arrow[rr, no head] \arrow[dd, dashed, no head, "\beta"] & & u_2 \arrow[ddd, dashed, no head, "\alpha'"] \arrow[ddl, dashed, no head, "\alpha'"'] \\
& & & & & & & & \\
& \Braket{\alpha} \arrow[dd, no head, dashed, "\alpha"'] \arrow[r, no head] & \Braket{\beta} \arrow[rrrr, no head, dashed, "\beta"] & & & & \Braket{\beta} \arrow[r, no head] & \Braket{\alpha'} \arrow[dd, no head, dashed, "\alpha'"] & \\
v_0 \arrow[ddddd, no head, dashed, "\alpha"'] & & & & & & & & w \arrow[ddddd, dashed, no head, "\alpha'"] \\
& \Braket{\alpha} \arrow[r, no head] \arrow[ddddl, dashed, no head, "\alpha"] & \Braket{\beta'} \arrow[rrrr, no head, dashed, "\beta'"] \arrow[ddr, dashed, no head, "\beta'"'] & & & & \Braket{\beta'} \arrow[ddl, dashed, no head, "\beta'"] \arrow[r, no head] & \Braket{\alpha'} \arrow[ddddr, dashed, no head, "\alpha'"'] & \\
& & & & & & & & \\
& & \Braket{\gamma} \arrow[ddr, dashed, no head, "\gamma"'] \arrow[r, no head] & \Braket{\beta'} & & \Braket{\beta'} \arrow[r, no head] & \Braket{\gamma} \arrow[ddl, dashed, no head, "\gamma"] & & \\
& & & & & & & & \\
u_4 \arrow[rrr, no head] & & & v_4 \arrow[rr, no head, dashed, "\gamma"'] & & v_3 \arrow[rrr, no head] & & & u_3.
\end{tikzcd}\qedhere
\]
\end{proof}

\begin{figure}
\centering
\begin{tikzpicture}[scale=.7]
\draw [red, thick] (-6,0) arc (180:-5:3 and 3);
\draw [red, thick] (-6,0) arc (180:330:3 and 3);
\draw [blue, thick] (-7,1.5) arc (180:-117:3 and 3);
\draw [blue, thick] (-7,1.5) arc (180:218:3 and 3);
\draw [green, thick] (1,1.5) arc (0:-238:3 and 3);
\draw [green, thick] (1,1.5) arc (0:97:3 and 3);
\draw [thick] (-3.8,-.5) to [out=30, in=-90] (-3.6,0) to [out=90, in=-90] (-4,1.2) to [out=90, in=180] (-3,2.5) to [out=0, in=90] (-2,1.2) to [out=-90, in=90] (-2.4,0)  to [out=-90, in=150] (-2.2,-.5);
\draw [thick] (-3.15,1.2) arc (180:90:.3 and .6);
\draw [thick] (-3.15,1.2) arc (180:270:.3 and .6);
\draw [thick] (-2.85,1.2) arc (0:60:.3 and .6);
\draw [thick] (-2.85,1.2) arc (0:-60:.3 and .6);
\draw [orange, thick] (-.2,0) arc (0:-8:2.8 and 2.8);
\draw [orange, thick] (-3,-2.8) arc (-90:-33:2.8 and 2.8);
\draw [orange, thick] (-3,-2.8) to [out=180, in=-90] (-5.8,0) to [out=90, in=180] (-3,-1) to [out=0, in=90] (-.2,0);
\draw [cyan, thick] (-6.8,1.5) arc (180:215:2.8 and 2.8);
\draw [cyan, thick] (-6.8,1.5) arc (180:60:2.8 and 2.8) to [out=-30, in=60] (-4.6,2.1) to [out=-120, in=-30] (-5.4,-.925);
\node at (-3,2.05) {\tiny $\Sigma_1^{b_1+1}$};
\node [green] at (.4,2.4) {$\gamma$};
\node [cyan] at (-6.3,2.2) {$\beta'$};
\node [orange] at (-3,-2.3) {$\alpha'$};
\node [red] at (-.5,-2.2) {$\alpha$};
\node [blue] at (-7,-.1) {$\beta$};
\end{tikzpicture}
\caption{Bad and good triples: if $\Sigma_1^{b_1+1}$ inherits an odd spin structure, the triple $\set{\alpha,\beta,\gamma}$ is bad. Notice that the triples $\set{\alpha',\beta,\gamma}$, $\set{\alpha',\beta',\gamma}$ and $\set{\alpha,\beta',\gamma}$ are good. Moreover, $\alpha, \alpha',\beta,\beta'$ satisfy the hypotheses of \thref{prop:hyp}.}
\label{fig:bad}
\end{figure}
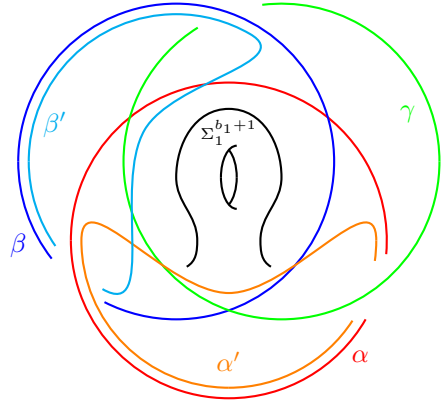

\subsection{Simple connectivity: the general case}

For paths of radius at least two, we are finally able to do a proper induction on the radius. The key lemma is the following.

\begin{lemma}
\thlabel{lem:abcd}
Let $\alpha,\beta,\gamma$ be nonseparating spin curves on $\Sigma_g$, and assume that $\abs{\alpha \cap \beta}=m \ge 2$, $\abs{\alpha \cap \gamma} \le m$ and $\abs{\beta \cap \gamma}=1$. Then there exists a nonseparating spin curve $\delta$ such that $\abs{\alpha \cap \delta}<m$, $\abs{\beta\cap\delta}=0,1$ and $\abs{\gamma \cap \delta}=0,1$.
\end{lemma}

\begin{proof}[Proof {\cite[Lemma 18]{waj:elem}}]
Cutting $\Sigma_g$ along $\beta \cup \gamma$, we can think of it as a square with some handles attached on it. Opposite edges of the square correspond to the same curve, $\beta$ or $\gamma$. Observe first that if $\abs{\alpha \cap \gamma}=1$ then we can set $\delta:=\gamma$. Assume then that $\abs{\alpha \cap \gamma} \ge 2$.

\begin{figure}
\centering
\begin{tikzpicture}[scale=.75]
\draw [blue, thick] (0,0) -- (0,4);
\draw [blue, thick] (4,0) -- (4,4);
\draw [green, thick] (0,0) -- (4,0);
\draw [green, thick] (0,4) -- (4,4);
\draw [red, thick] (3,0) to [out=90, in=-90] (1,4);
\draw [red, thick] (1,0) -- (1,.2);
\draw [red, thick] (3,4) -- (3,3.8);
\draw [orange, thick] (3.3,0) to [out=90, in=-45] (2.2,2.2) to [out=135, in=0] (.4,3.7) -- (0,3.7);
\draw [orange, thick] (4,3.7) -- (3.6,3.7) arc(-90:-180:.3 and .3);
\draw [magenta, thick] (2.7,0) to [out=90, in=180] (1.8,3.7) -- (2.4,3.7) arc(-90:0:.3 and .3);
\node [blue] at (-.3,2) {$\beta$};
\node [blue] at (4.3,2) {$\beta$};
\node [green] at (2,-.3) {$\gamma$};
\node [green] at (2,4.3) {$\gamma$};
\node [red] at (3,-.3) {$\alpha$};
\node [magenta] at (2,1) {$\eta_1$};
\node [orange] at (3,2) {$\eta_2$};
\end{tikzpicture}
\caption{Construction of the curve $\delta$ of \thref{lem:abcd} when an $\alpha$-segment has its endpoints on opposite $\gamma$-edges. Note that $[\eta_1]+[\eta_2]\equiv\gamma \ (\operatorname{mod}2)$.}
\label{fig:square1}
\end{figure}
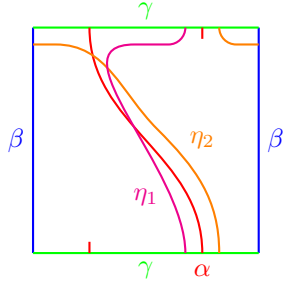

If an arc of $\alpha$ has its endpoints on two opposite edges of the square, say on the $\gamma$-edges, then we can perform a sort of 1-1-2 trick. Let $\eta_1,\eta_2$ be curves as in Figure \ref{fig:square1}. They are both nonseparating as they intersect $\gamma$ once, and both intersect $\alpha$ in at most $\abs{\alpha \cap \gamma}-1 \le m-1$ points. Moreover, we have $\phi(\eta_1)+\phi(\eta_2)=\phi(\gamma)$, so exactly one of the two is a $1$-curve, and we can take it as $\delta$. The same reasoning applies if an arc of $\alpha$ has its endpoints on the two opposite $\beta$-edges. Furthermore, we can start from any arc on the square that connects two opposite edges and does not intersect $\alpha$, provided that its endpoints are separated by some intersection points with $\alpha$ (otherwise, we cannot ensure that the resulting spin curve intersects $\alpha$ in less than $m$ points). We will refer to such arcs as \emph{nice arcs}. 

Given an arc $c$ of $\alpha$ with endpoints on the same edge or on adjacent edges, denote by $\widetilde{c}$ the curve obtained as the union of $c$ and the portion of the boundary of the square that connects the endpoints of $c$ and contains at most one corner point. We will define the spin value of $c$ as the spin value of $\widetilde{c}$. 

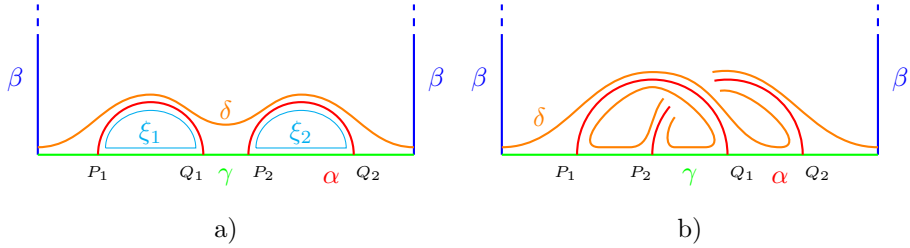
\begin{figure}
\centering
\begin{tikzpicture}
\draw [blue, thick] (0,0) -- (0,1.5);
\draw [blue, thick] (5,0) -- (5,1.5);
\draw [blue, thick, dashed] (0,1.5) -- (0,2);
\draw [blue, thick, dashed] (5,1.5) -- (5,2);
\draw [green, thick] (0,0) -- (5,0);
\draw [red, thick] (.8,0) arc (180:0:.7 and .7);
\draw [red, thick] (2.8,0) arc (180:0:.7 and .7);
\draw [cyan] (.9,.09) arc (180:0:.6 and .5) -- (.9,.09);
\draw [cyan] (2.9,.09) arc (180:0:.6 and .5) -- (2.9,.09);
\draw [orange, thick] (0,.1) to [out=0, in=180] (1.5,.8) to [out=0, in=180] (2.5,.4) to [out=0, in=180] (3.5,.8) to [out=0, in=180] (5,.1);
\node [blue] at (-.3,1) {$\beta$};
\node [blue] at (5.3,1) {$\beta$};
\node [green] at (2.5,-.3) {$\gamma$};
\node [red] at (3.9,-.3) {$\alpha$};
\node [cyan] at (1.5,.3) {$\xi_1$};
\node [cyan] at (3.5,.3) {$\xi_2$};
\node [orange] at (2.5,.6) {$\delta$};
\node at (.8,-.25) {\tiny $P_1$};
\node at (2.05,-.25) {\tiny $Q_1$};
\node at (3,-.25) {\tiny $P_2$};
\node at (4.4,-.25) {\tiny $Q_2$};
\node at (2.5,-1) {a)};
\end{tikzpicture}
\begin{tikzpicture}
\draw [blue, thick] (0,0) -- (0,1.5);
\draw [blue, thick] (5,0) -- (5,1.5);
\draw [blue, thick, dashed] (0,1.5) -- (0,2);
\draw [blue, thick, dashed] (5,1.5) -- (5,2);
\draw [green, thick] (0,0) -- (5,0);
\draw [red, thick] (1,0) arc (180:0:1 and 1);
\draw [red, thick] (2,0) arc (180:140:1 and 1);
\draw [red, thick] (4,0) arc (0:100:1 and 1);
\draw [orange, thick] (0,.1) to [out=0, in=180] (2,1.1) to [out=0, in=135] (2.78,.78) to [out=-45, in=180] (3.7,.1) to [out=0, in=10] (2.9,.85);
\draw [orange, thick] (2.3,.5) to [out=-130, in=180] (2.3,.1) -- (2.7,.1) to [out=0, in=0] (2,.9) to [out=180, in=180] (1.3,.1) -- (1.7,.1) to [out=0, in=-130] (2.15,.75);
\draw [orange, thick] (2.8,1.1) to [out=10, in=180] (5,.1);
\node [blue] at (-.3,1) {$\beta$};
\node [blue] at (5.3,1) {$\beta$};
\node [green] at (2.5,-.3) {$\gamma$};
\node [red] at (3.7,-.3) {$\alpha$};
\node [orange] at (.5,.5) {$\delta$};
\node at (4.2,-.25) {\tiny $Q_2$};
\node at (.85,-.25) {\tiny $P_1$};
\node at (1.85,-.25) {\tiny $P_2$};
\node at (3.2,-.25) {\tiny $Q_1$};
\node at (2.5,-1) {b)};
\end{tikzpicture}
\caption{Construction of the curve $\delta$ of \thref{lem:abcd} when two  $\alpha$-arcs have endpoints on the same $\gamma$-edge and are not nested. Here $\phi(\xi_1)=\phi(\xi_2)=0$.}
\label{fig:aa1}
\end{figure}

Assume now that there are two arcs $a_1,a_2$ of $\alpha$ with respective endpoints $P_1, Q_1$ and $P_2,Q_2$ all lying on the same edge $\ell$ of the square. Choose an orientation for the edge and enumerate its intersection with $\alpha$. If the endpoints appear in the order $P_1,Q_1,P_2,Q_2$ or $P_1,P_2,Q_1,Q_2$ (up to renaming), then we construct a $\delta$ as follows. Let $\overline{P_iQ_i}$ be the segment of $\ell$ with endpoints $P_i$, $Q_i$. Consider the curves $\xi_1:=\widetilde{a_1}$ and $\xi_2:=\widetilde{a_2}$. If one of them, say $\xi_1$, is a $1$-curve (not necessarily nonseparating), then we set $\delta:=(\ell\setminus \overline{P_1Q_1})\cup a_1$. If $\xi_1$ and $\xi_2$ are disjoint $0$-curves, we set $\delta:=(\ell\setminus (\overline{P_1Q_1}\cup\overline{P_2Q_2}))\cup a_1 \cup a_2$; see Figure \ref{fig:aa1}a). Finally, if $\abs{\xi_1 \cap \xi_2}=1$, we can take the boundary of a tubular neighborhood of $\ell \cup \xi_1\cup\xi_2$ as $\delta$; see Figure \ref{fig:aa1}b).

From now on, we will assume that the above cases do not occur, i.e. that there are no nice arcs, and if two $\alpha$-arcs have their endpoints on the same edge then they are nested. Moreover, if an $\alpha$-arc $c$ has its endpoints on the same edge, we will assume that $\widetilde{c}$ is a $0$-curve. 

Consider a corner $C$ between two edges $\ell_1$ and $\ell_2$ of the square. Let $P_1$ and $P_2$ be the first intersection points with $\alpha$ that are found on $\ell_1$ and $\ell_2$ respectively, starting at $C$. Let $c_i$ be the $\alpha$-arc starting at $P_i$, and let $Q_i$ be its other endpoints. There are various possibilities.

\begin{figure}
\centering
\begin{tikzpicture}[scale=.75]
\draw [blue, thick] (0,0) -- (0,4);
\draw [blue, thick] (4,0) -- (4,4);
\draw [green, thick] (0,0) -- (4,0);
\draw [green, thick] (0,4) -- (4,4);
\draw [red, thick] (1,0) arc (180:0:1 and 1);
\draw [red, thick] (0,1) arc (-90:90:.7 and .7);
\draw [red, thick] (0,3.2) -- (.4,3.2);
\draw [orange, thick] (0,3) to [out=0, in=90] (.2,2.8) to [out=-90, in=135] (.55,2.35) to [out=-45, in=90] (.88,1.7) to [out=-90, in=45] (.55,1.05) to [out=-135, in=90] (.2,.6) -- (.2,.4) to [out=-90, in=180] (.4,.2) -- (.6,.2) to [out=0, in=-135] (1.22,.78) to [out=45, in=180] (2,1.15) to [out=0, in=135] (2.78,.78) to [out=-45, in=180] (3.4,.2) -- (4,.2);
\node at (1,-.4) {$P_1$};
\node at (3,-.4) {$Q_1$};
\node at (-.4,1) {$P_2$};
\node at (-.4,2.4) {$Q_2$};
\node at (-.3,-.3) {$C$};
\node at (2,-1) {a)};
\end{tikzpicture}
\begin{tikzpicture}[scale=.75]
\draw [blue, thick] (0,0) -- (0,4);
\draw [blue, thick] (4,0) -- (4,4);
\draw [green, thick] (0,0) -- (4,0);
\draw [green, thick] (0,4) -- (4,4);
\draw [red, thick] (1,0) arc (180:0:1 and 1);
\draw [red, thick] (0,1) arc (-90:90:1 and 1);
\draw [red, thick] (4,1) arc (270:90:.7 and .7);
\draw [red, thick] (4,3.2) -- (3.6,3.2);
\draw [orange, thick] (4,3) to [out=180, in=90] (3.8,2.8) to [out=-90, in=45] (3.45,2.35) to [out=-135, in=90] (3.12,1.7) to [out=-90, in=135] (3.45,1.05) to [out=-45, in=90] (3.8,.6) -- (3.8,.4) to [out=-90, in=0] (3.6,.2) -- (3.4,.2) to [out=180, in=-45] (2.78,.78) to [out=135, in=0] (2,1.15) to [out=180, in=45] (1.22,.78) to [out=-135, in=0] (.6,.2) -- (0,.2);
\node at (1,-.4) {$P_1$};
\node at (3,-.4) {$Q_1$};
\node at (-.4,1) {$P_2$};
\node at (4.4,1) {$P_2$};
\node at (-.4,3) {$Q_2$};
\node at (-.3,-.3) {$C$};
\node at (2,-1) {b)};
\end{tikzpicture}
\begin{tikzpicture}[scale=.75]
\draw [blue, thick] (0,0) -- (0,4);
\draw [blue, thick] (4,0) -- (4,4);
\draw [green, thick] (0,0) -- (4,0);
\draw [green, thick] (0,4) -- (4,4);
\draw [red, thick] (1,0) arc (180:0:1 and 1);
\draw [red, thick] (0,1) arc (-90:90:1 and 1);
\draw [red, thick] (4,1) arc (270:180:2 and 3);
\draw [red, thick] (1,4) -- (1,3.6);
\draw [orange, thick] (3.8,0) to [out=90, in=-60] (2.5,1.7) to [out=120, in=-90] (1.8,4);
\node at (1,-.4) {$P_1$};
\node at (3,-.4) {$Q_1$};
\node at (-.4,1) {$P_2$};
\node at (-.4,3) {$Q_2$};
\node at (1,4.4) {$P_1$};
\node at (4.4,1) {$P_2$};
\node at (-.3,-.3) {$C$};
\node at (2,-1) {c)};
\end{tikzpicture}
\begin{tikzpicture}[scale=.75]
\draw [blue, thick] (0,0) -- (0,4);
\draw [blue, thick] (4,0) -- (4,4);
\draw [green, thick] (0,0) -- (4,0);
\draw [green, thick] (0,4) -- (4,4);
\draw [red, thick] (1,0) arc (180:0:.7 and .7);
\draw [red, thick] (0,1) arc (-90:90:.5 and .5);
\draw [red, thick] (3,0) arc (180:90:1 and 1.8);
\draw [orange, thick] (4,2) to [out=180, in=60] (3.1,1.2) to [out=-120, in=0] (2.7,.4) to [out=180, in=-45] (2.35,.55) to [out=135, in=0] (1.7,.85) to [out=180, in=45] (1.05,.55) to [out=-135, in=0] (.6,.2) -- (0,.2);
\draw [purple, thick] (3.8,4) -- (3.8,2.5) to [out=-90, in=60] (3.2,1.6) to [out=-120, in=90] (2.7,0);
\node at (1,-.4) {$P_1$};
\node at (2.4,-.4) {$Q_1$};
\node at (3.1,-.4) {$R_1$};
\node at (-.4,1) {$P_2$};
\node at (-.4,2) {$Q_2$};
\node at (4.4,1.8) {$S_1$};
\node at (-.3,-.3) {$C$};
\node at (2,-1) {d)};
\end{tikzpicture}
\begin{tikzpicture}[scale=.75]
\draw [blue, thick] (0,0) -- (0,4);
\draw [blue, thick] (4,0) -- (4,4);
\draw [green, thick] (0,0) -- (4,0);
\draw [green, thick] (0,4) -- (4,4);
\draw [red, thick] (1,0) arc (180:0:.7 and .7);
\draw [red, thick] (0,1) arc (-90:90:.5 and .5);
\draw [red, thick] (3,0) arc (0:75:3 and 1.5);
\draw [red, thick] (0,1.5) arc (90:85:3 and 1.5);
\draw [cyan, thick] (.3,1.4) to [out=180, in=90] (.2,1) -- (.2,.4) to [out=-90, in=180] (.4,.2) -- (.6,.2) to [out=0, in=-135] (1.05,.55) to [out=45, in=180] (1.7,.8) to [out=0, in=180] (2.7,.4) to [out=0, in=-30] (1.7,1.1) to [out=150, in=-20] (.8,1.3);
\draw [orange, thick] (4,.2) -- (3.4,.2) to [out=180, in=-45] (2.7,.8) to [out=135, in=-20] (.8,1.6);
\draw [orange, thick] (.3,1.6) to [out=180, in=90] (.1,1) -- (.1,.4) to [out=-90, in=0] (0,.2);
\draw [purple, thick] (.2,4) -- (.2,2.1) to [out=-90, in=180] (.3,1.7);
\draw [purple, thick] (.8,1.7) to [out=-20, in=135] (2.7,.9) to [out=-45, in=0] (2.9,.1) -- (.4,.1) to [out=180, in=90] (.2,0);
\node at (1,-.4) {$P_1$};
\node at (2.4,-.4) {$Q_1$};
\node at (3.1,-.4) {$R_1$};
\node at (-.4,.9) {$P_2$};
\node at (-.4,2.1) {$Q_2$};
\node at (-.4,1.5) {$S_1$};
\node at (-.3,-.3) {$C$};
\node [cyan] at (.5,.6) {$\xi$};
\node [purple] at (.6,3.6) {$\eta_1$};
\node [orange] at (3.6,.5) {$\eta_2$};
\node at (2,-1) {e)};
\end{tikzpicture}
\begin{tikzpicture}[scale=.75]
\draw [blue, thick] (0,0) -- (0,4);
\draw [blue, thick] (4,0) -- (4,4);
\draw [green, thick] (0,0) -- (4,0);
\draw [green, thick] (0,4) -- (4,4);
\draw [red, thick] (1,0) arc (180:0:.7 and .7);
\draw [red, thick] (0,1) arc (-90:90:.5 and .5);
\draw [red, thick] (3.1,0) arc (0:90:3.1 and 2.6);
\draw [cyan, thick] (.6,1.5) to [out=-90, in=90] (.2,.6) -- (.2,.4) to [out=-90, in=180] (.4,.2) -- (.6,.2) to [out=0, in=-135] (1.05,.55) to [out=45, in=180] (1.7,.8) to [out=0, in=180] (2.7,.4) to [out=0, in=-30] (1.7,1.9) to [out=150, in=90] (.3,2.2) to [out=-90, in=90] (.6,1.5);
\draw [orange, thick] (.1,4) -- (.1,3.6) to [out=-90, in=140] (2.35,1.85) to [out=-40, in=180] (4,.1);
\draw [orange, thick] (.1,0) to [out=90, in=0] (0,.1);
\draw [thick] (.5,1.5) -- (.6,1.5);
\draw [thick] (1.7,.7) -- (1.7,.8);
\node at (1,-.4) {$P_1$};
\node at (2.4,-.4) {$Q_1$};
\node at (3.2,-.4) {$R_1$};
\node at (-.4,.9) {$P_2$};
\node at (-.4,1.9) {$Q_2$};
\node at (-.4,2.6) {$S_1$};
\node at (-.3,-.3) {$C$};
\node [cyan] at (.5,.6) {$\xi$};
\node [orange] at (.4,3.6) {$\eta$};
\node at (2,-1) {f)};
\end{tikzpicture}
\caption{Construction of nice arcs and spin curves in Case A of \thref{lem:abcd}.}
\label{fig:6.24A0}
\end{figure}
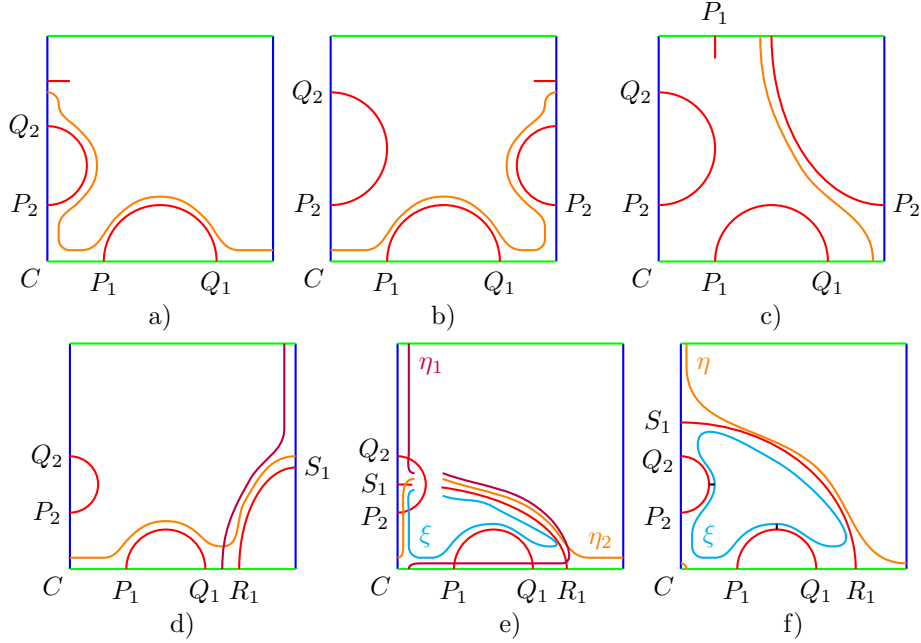

\paragraph{Case A: $Q_1\in \ell_1$ and $Q_2\in\ell_2$.} If $Q_1$ is the last intersection point with $\alpha$ on $\ell_1$, then also $Q_2$ is the last intersection point on $\ell_2$, otherwise we would find a nice arc; see Figure \ref{fig:6.24A0}a). Moreover, by the same reason (see Figure \ref{fig:6.24A0}b) and c)), the points corresponding to $P_1$ and $P_2$ on the edges opposite to $\ell_1$ and $\ell_2$ are joined by an $\alpha$-arc $d$. We may assume that $\widetilde{d}$ is a $1$-curve, otherwise the orange arc of Figure \ref{fig:6.24A0}c) closes up in the obvious way to a curve $\delta$ as in the statement. This is a “bad configuration" (see Figure \ref{fig:bad24}a)) and we will deal with it later on.

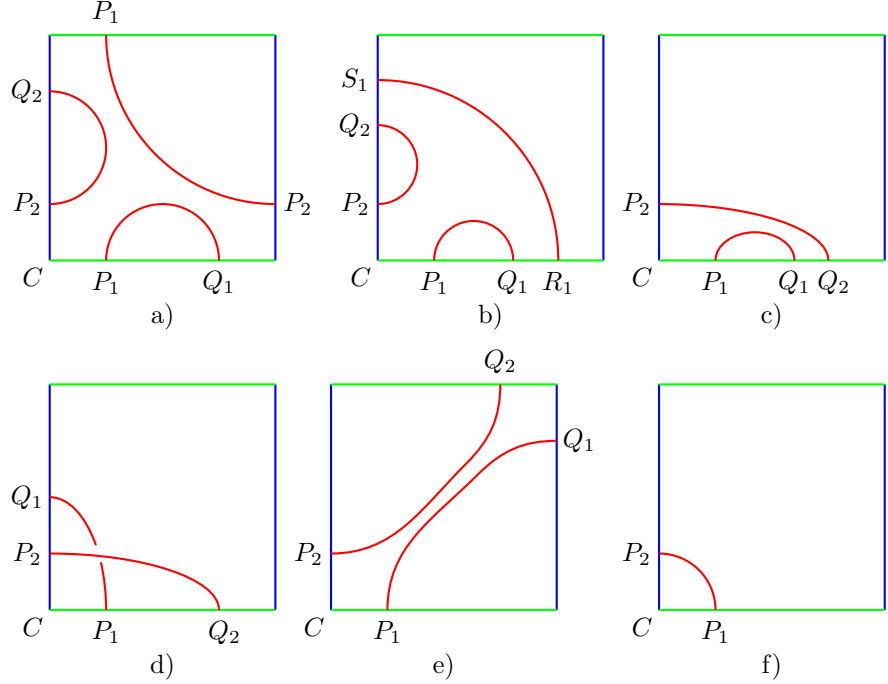
\begin{figure}
\centering
\begin{tikzpicture}[scale=.75]
\draw [blue, thick] (0,0) -- (0,4);
\draw [blue, thick] (4,0) -- (4,4);
\draw [green, thick] (0,0) -- (4,0);
\draw [green, thick] (0,4) -- (4,4);
\draw [red, thick] (1,0) arc (180:0:1 and 1);
\draw [red, thick] (0,1) arc (-90:90:1 and 1);
\draw [red, thick] (4,1) arc (270:180:3 and 3);
\node at (1,-.4) {$P_1$};
\node at (3,-.4) {$Q_1$};
\node at (-.4,1) {$P_2$};
\node at (-.4,3) {$Q_2$};
\node at (1,4.4) {$P_1$};
\node at (4.4,1) {$P_2$};
\node at (-.3,-.3) {$C$};
\node at (2,-1) {a)};
\end{tikzpicture}
\begin{tikzpicture}[scale=.75]
\draw [blue, thick] (0,0) -- (0,4);
\draw [blue, thick] (4,0) -- (4,4);
\draw [green, thick] (0,0) -- (4,0);
\draw [green, thick] (0,4) -- (4,4);
\draw [red, thick] (1,0) arc (180:0:.7 and .7);
\draw [red, thick] (0,1) arc (-90:90:.7 and .7);
\draw [red, thick] (3.2,0) arc (0:90:3.2 and 3.2);
\node at (1,-.4) {$P_1$};
\node at (2.4,-.4) {$Q_1$};
\node at (3.2,-.4) {$R_1$};
\node at (-.4,1) {$P_2$};
\node at (-.4,2.4) {$Q_2$};
\node at (-.4,3.2) {$S_1$};
\node at (-.3,-.3) {$C$};
\node at (2,-1) {b)};
\end{tikzpicture}
\begin{tikzpicture}[scale=.75]
\draw [blue, thick] (0,0) -- (0,4);
\draw [blue, thick] (4,0) -- (4,4);
\draw [green, thick] (0,0) -- (4,0);
\draw [green, thick] (0,4) -- (4,4);
\draw [red, thick] (1,0) arc (180:0:.7 and .5);
\draw [red, thick] (3,0) arc (0:90:3 and 1);
\node at (1,-.4) {$P_1$};
\node at (2.4,-.4) {$Q_1$};
\node at (3.1,-.4) {$Q_2$};
\node at (-.4,1) {$P_2$};
\node at (-.3,-.3) {$C$};
\node at (2,-1) {c)};
\end{tikzpicture}
\begin{tikzpicture}[scale=.75]
\draw [blue, thick] (0,0) -- (0,4);
\draw [blue, thick] (4,0) -- (4,4);
\draw [green, thick] (0,0) -- (4,0);
\draw [green, thick] (0,4) -- (4,4);
\draw [red, thick] (1,0) arc (0:25:1 and 2);
\draw [red, thick] (0,2) arc (90:35:1 and 2);
\draw [red, thick] (3,0) arc (0:90:3 and 1);
\node at (1,-.4) {$P_1$};
\node at (-.4,2) {$Q_1$};
\node at (3.1,-.4) {$Q_2$};
\node at (-.4,1) {$P_2$};
\node at (-.3,-.3) {$C$};
\node at (2,-1) {d)};
\end{tikzpicture}
\begin{tikzpicture}[scale=.75]
\draw [blue, thick] (0,0) -- (0,4);
\draw [blue, thick] (4,0) -- (4,4);
\draw [green, thick] (0,0) -- (4,0);
\draw [green, thick] (0,4) -- (4,4);
\draw [red, thick] (1,0) to [out=90, in=-135] (2.6,2.4) to [out=45, in=180] (4,3);
\draw [red, thick] (0,1) to [out=0, in=-135] (2.4,2.6) to [out=45, in=-90] (3,4);
\node at (1,-.4) {$P_1$};
\node at (4.4,3) {$Q_1$};
\node at (3,4.4) {$Q_2$};
\node at (-.4,1) {$P_2$};
\node at (-.3,-.3) {$C$};
\node at (2,-1) {e)};
\end{tikzpicture}
\begin{tikzpicture}[scale=.75]
\draw [blue, thick] (0,0) -- (0,4);
\draw [blue, thick] (4,0) -- (4,4);
\draw [green, thick] (0,0) -- (4,0);
\draw [green, thick] (0,4) -- (4,4);
\draw [red, thick] (1,0) arc (0:90:1 and 1);
\node at (1,-.4) {$P_1$};
\node at (-.4,1) {$P_2$};
\node at (-.3,-.3) {$C$};
\node at (2,-1) {f)};
\end{tikzpicture}
\caption{Bad configurations in the proof of \thref{lem:abcd}.}
\label{fig:bad24}
\end{figure}

Suppose then that $Q_1$ is not the last intersection point, and call $R_1$ the next one, going further from $C$. The $\alpha$-arc $d_1$ starting at $R_1$ must have its other endpoint $S_1$ on $\ell_2$. Indeed, $S_1$ cannot lie on $\ell_1$ by assumption, and if $S_1 \notin \ell_2$ then either the orange arc or the purple arc in Figure \ref{fig:6.24A0}d) is nice. 

Assume that $S_1$ is between $P_2$ and $Q_2$. Observe that $\widetilde{d_1}$ is nonseparating as it intersects $\widetilde{c_2}$ once. If it is a $0$-curve, then the curve $\xi$ in Figure \ref{fig:6.24A0}e) is a nonseparating $1$-curve. If instead $\widetilde{d_1}$ is a $1$-curve, then also curves $\eta_1$ and $\eta_2$ of Figure \ref{fig:6.24A0}e) are nonseparating $1$-curves. Let $m_1$ be the number of intersection points of $\ell_2$ with $\alpha$ that are further than $S_1$ from $C$, and let $m_2$ be the number of those that are closer. Define similarly $n_1$ and $n_2$ for $\ell_1$. Then $m_1+m_2+1 \le m$ and $n_1+n_2+1 \le m$. Observe that 
\[
\abs{\eta_1\cap \alpha}=m_1+1+n_2, \quad  \abs{\eta_2\cap \alpha}=m_2+1+n_1, \quad  \big|\widetilde{d_1}\cap\alpha\big|=m_2+n_2.
\]
If all these three quantities were at least equal to $m$, we would get $2n_2 \ge m$ and $2m_2 \ge m$, hence $m_2+n_2 \ge m$ and $m_1+n_1+2 \le m$. Now, $t_{\widetilde{c_2}}(\eta_2)$ is a nonseparating $1$-curve, and it intersects $\alpha$ in at most $m_1+n_1<m$ points. 

Suppose on the other hand that $S_1$ is further than $Q_2$ from $C$. If $\widetilde{d_1}$ is a $0$-curve, then the curve $\xi$ of Figure \ref{fig:6.24A0}f) must be a $0$-curve. The arc-sums of $\xi$ with $\widetilde{c_1}$ and $\widetilde{c_2}$ along the black arcs of Figure \ref{fig:6.24A0}f) are $1$-curves, and cannot be both separating. It is clear that $\xi+\widetilde{c_2}$ intersects $\alpha$ in less than $m$ points. On the other hand, this need not be true for $\xi+\widetilde{c_1}$, but if $\abs{(\xi+\widetilde{c_1})\cap\alpha} \ge m$ then clearly the curve $\eta$ of Figure \ref{fig:6.24A0} is a nonseparating $1$-curve that satisfies $\abs{\eta \cap \alpha}<m$.

Assume that $\widetilde{d}$ is a $1$-curve. Defining $m_1,m_2,n_1$ and $n_2$ as before, we can take one of $t_{\gamma}^{\pm}(\eta)$, $t_{\beta}^{\pm}(\eta)$ and $\widetilde{d}$ as $\delta$ unless $m_2=n_2 \ge \lfloor m/2\rfloor+1$ and $m_1=n_1$. Moreover, in this case the curve $\xi$ of Figure \ref{fig:6.24A0}f) is a $1$-curve and always intersects $\alpha$ in less than $m$ points, so we can take it as $\delta$ unless it is separating. It is clear that if $\xi$ is separating then there are no points of $\alpha \cap \ell_2$ between $Q_2$ and $S_1$. 
This is the bad configuration in Figure \ref{fig:bad24}b). Notice that by similar arguments we may assume that if the $\alpha$-arc starting at the intersection point right after $R_1$ lands on $\ell_2$, then it lands precisely on the intersection point right after $S_1$, and so on. 

\begin{figure}
\centering
\begin{tikzpicture}[scale=.8]
\draw [blue, thick] (1,0) -- (1,1.5);
\draw [blue, thick] (6,0) -- (6,1.5);
\draw [blue, thick, dashed] (1,1.5) -- (1,2);
\draw [blue, thick, dashed] (6,1.5) -- (6,2);
\draw [green, thick] (1,0) -- (6,0);
\draw [red, thick] (1,1) -- (2,1) arc (90:0:1 and 1);
\draw [red, thick] (2,0) arc (180:140:1 and 1);
\draw [red, thick] (4,0) arc (0:100:1 and 1);
\draw [orange, thick] (1,.1) to [out=0, in=-90] (1.1,.2) -- (1.1,1) to [out=90, in=180] (1.2,1.1) -- (2,1.1) to [out=0, in=135] (2.78,.78) to [out=-45, in=180] (3.7,.1) to [out=0, in=10] (2.9,.85);
\draw [orange, thick] (2.3,.5) to [out=-130, in=180] (2.3,.1) -- (2.7,.1) to [out=0, in=0] (2,.9) -- (1.4,.9) to [out=180, in=90] (1.2,.8) -- (1.2,.2) to [out=-90, in=180] (1.4,.1) -- (1.7,.1) to [out=0, in=-130] (2.15,.75);
\draw [orange, thick] (2.8,1.1) to [out=10, in=180] (5,.1) -- (6,.1);
\node [orange] at (.5,.5) {$\delta$};
\node at (4,-.32) {$Q_2$};
\node at (.68,1) {$P_1$};
\node at (2,-.32) {$P_2$};
\node at (3,-.32) {$Q_1$};
\node at (.7,-.3) {$C$};
\node at (3.5,-1) {a)};
\end{tikzpicture}
\begin{tikzpicture}
\draw [blue, thick] (0,0) -- (0,1.2);
\draw [blue, thick] (4,0) -- (4,1.2);
\draw [blue, thick, dashed] (0,1.2) -- (0,1.6);
\draw [blue, thick, dashed] (4,1.2) -- (4,1.6);
\draw [green, thick] (0,0) -- (4,0);
\draw [red, thick] (1,0) arc (180:0:.7 and .4);
\draw [red, thick] (3,0) arc (0:90:3 and 1);
\draw [cyan, thick] (.3,.9) to [out=180, in=90] (.1,.8) -- (.1,.3) to [out=-90, in=180] (.3,.1) -- (.6,.1) to [out=0, in=-135] (1.05,.35) to [out=45, in=180] (1.7,.5) to [out=0, in=180] (2.7,.2) to [out=0, in=-20] (1.8,.7) to [out=160, in=0] (.3,.9);
\node at (1,-.3) {$P_1$};
\node at (2.4,-.3) {$Q_1$};
\node at (3.1,-.3) {$Q_2$};
\node at (-.3,1) {$P_2$};
\node [cyan] at (.4,.5) {$\xi$};
\node at (-.25,-.25) {$C$};
\node at (2,-.8) {b)};
\end{tikzpicture}
\caption{Construction of the curve $\delta$ of \thref{lem:abcd} in Case B.}
\label{fig:624B}
\end{figure}
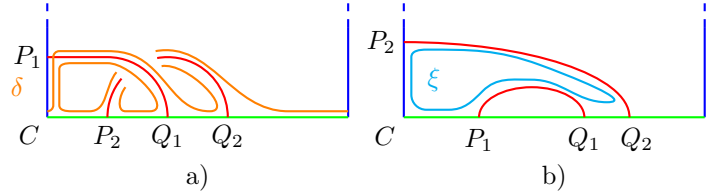

\paragraph{Case B: $Q_1,Q_2 \in \ell_1$.} If $Q_2$ lies between $P_1$ and $Q_1$ we take $\delta$ as in Figure \ref{fig:624B}a). If it lies further away from $C$, then we claim that it must be the next intersection point. Indeed, by a similar reasoning as in Case A, one of the curves $\widetilde{c_2}$ and $\xi$ of Figure \ref{fig:624B}b) is a $1$-curve, and we can take it as $\delta$ unless it is separating. If it is separating, then there are no points of $\ell_1 \cap \alpha$ between $Q_1$ and $Q_2$. Moreover, if $\widetilde{c_2}$ is separating, then we can take its obvious arc sum with $\ell_1$ as $\delta$, so we may assume that $\widetilde{c_2}$ is a $0$-curve. We get the bad configuration in Figure \ref{fig:bad24}c).

\paragraph{Case C: $Q_1 \in \ell_2$ and $Q_2 \in \ell_1$.} In this case, both $\widetilde{c_1}$ and $\widetilde{c_2}$ are nonseparating. If one of them is a $1$-curve, then we can take it as $\delta$. If both are $0$-curves, we get the bad configuration of Figure \ref{fig:bad24}d). 

Notice that in this situation we can assume that there are no arcs with both endpoints on $\ell_1$ (or on $\ell_2$). Indeed, let $d$ be such an arc, and call $R$ and $S$ its endpoints. Recall that $\widetilde{d}$ is a $0$-curve by assumption. If both $R$ and $S$ lie between $P_1$ and $Q_2$, then we may construct a curve $\xi$ as in Figure \ref{fig:624B}b); in this case, $\xi$ is nonseparating, and it is clearly spin. If $Q_2$ lies between $R$ and $S$, we can construct $\delta$ as in Figure \ref{fig:624B}a). Finally, if both $R$ and $S$ are further from $C$ than $Q_2$, then the arc sum of $\ell_1$ with $\widetilde{d}$ and $\widetilde{c_2}$ can be taken as $\delta$.

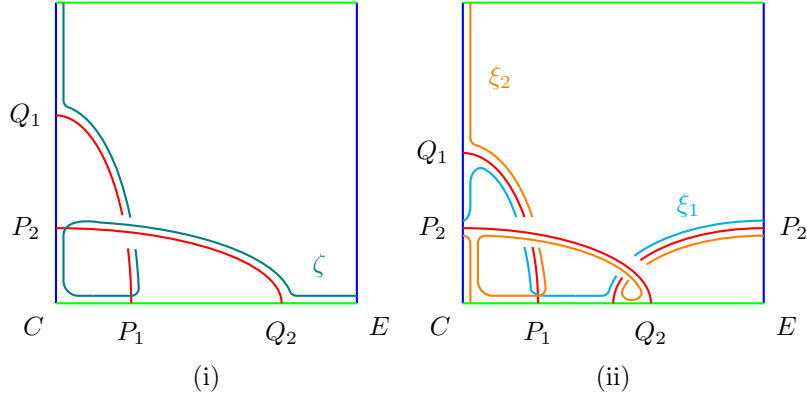
\begin{figure}
\centering
\begin{tikzpicture}
\draw [blue, thick] (0,0) -- (0,4);
\draw [blue, thick] (4,0) -- (4,4);
\draw [green, thick] (0,0) -- (4,0);
\draw [green, thick] (0,4) -- (4,4);
\draw [red, thick] (1,0) arc (0:17:1 and 2.5);
\draw [red, thick] (0,2.5) arc (90:28:1 and 2.5);
\draw [red, thick] (3,0) arc (0:90:3 and 1);
\draw [teal, thick] (.1,4) -- (.1,2.7) arc (180:260:.1 and .1) arc (80:25:1.1 and 2.6);
\draw [teal, thick] (.5,.1) -- (1,.1) arc (-90:5:.1 and .1) arc (5:17:1.1 and 2.5);
\draw [teal, thick] (4,.1) -- (3.2,.1) arc (270:190:.1 and .1) arc (10:80:3.1 and 1.1) to [out=170, in=90] (.1,.9) -- (.1,.3) arc (180:270:.2 and .2) -- (.5,.1);
\node at (1,-.4) {$P_1$};
\node at (-.4,2.5) {$Q_1$};
\node at (3,-.4) {$Q_2$};
\node at (-.4,1) {$P_2$};
\node at (-.3,-.3) {$C$};
\node at (4.3,-.3) {$E$};
\node [teal] at (3.5,.5) {$\zeta$};
\node at (2,-1) {(i)};
\end{tikzpicture}
\begin{tikzpicture}
\draw [blue, thick] (0,0) -- (0,4);
\draw [blue, thick] (4,0) -- (4,4);
\draw [green, thick] (0,0) -- (4,0);
\draw [green, thick] (0,4) -- (4,4);
\draw [red, thick] (1,0) arc (0:22:1 and 2);
\draw [red, thick] (0,2) arc (90:35:1 and 2);
\draw [red, thick] (2.5,0) arc (0:90:2.5 and 1);
\draw [red, thick] (4,1) arc (90:145:2 and 1);
\draw [red, thick] (2,0) arc (180:160:2 and 1);
\draw [cyan, thick] (0,1.1) arc(-90:0:.1 and .2) -- (.1,1.6) arc(180:80:.14 and .2) arc(80:40:.75 and 1.9);
\draw [cyan, thick] (1.5,.1) -- (1.1,.1) arc(270:170:.2 and .1) arc(10:27:.9 and 1.9);
\draw [cyan, thick] (1.5,.1) -- (1.85,.1) arc(-90:-10:.1 and .1) arc(170:160:2.1 and 1.1);
\draw [cyan, thick] (4,1.1) arc (90:145:2.1 and 1.1);
\draw [orange, thick] (.1,4) -- (.1,2.2) arc(180:260:.1 and .1) arc (80:32:1.1 and 2.1);
\draw [orange, thick] (.1,0) -- (.1,.8) arc (0:90:.1 and .1);
\draw [orange, thick] (4,.9) arc (90:145:1.9 and .9);
\draw [orange, thick] (.4,.1) -- (1,.1) arc (-90:10:.1 and .1) arc (10:25:1.1 and 2.1);
\draw [orange, thick] (.4,.1) arc (270:180:.2 and .2) -- (.2,.8) arc (180:85:.1 and .1) arc (85:10:2.3 and .9) arc (10:-190:.13 and .1) arc (170:160:1.9 and .9);
\node at (1,-.4) {$P_1$};
\node at (-.4,2) {$Q_1$};
\node at (2.5,-.4) {$Q_2$};
\node at (-.4,1) {$P_2$};
\node at (4.4,1) {$P_2$};
\node [cyan] at (3,1.3) {$\xi_1$};
\node [orange] at (.5,3) {$\xi_2$};
\node at (-.3,-.3) {$C$};
\node at (4.3,-.3) {$E$};
\node at (2,-1) {(ii)};
\end{tikzpicture}
\caption{Spin curves in the bad configuration of Figure \ref{fig:bad24}d).}
\label{fig:bad24d}
\end{figure}

Moreover, notice that the curve $\zeta$ of Figure \ref{fig:bad24d}(i) is a $1$-curve, and so are $t_{\widetilde{c_1}}^{\pm1}(\zeta)$ and $t_{\widetilde{c_2}}^{\pm1}(\zeta)$. Hence, we can assume that the sum of intersection points further from $C$ than $Q_2$ and $Q_1$ is at least $m-2$, with at least one intersection point on each edge. 

We can also assume that the arc coming out of $P_2$ on the edge opposite to $\ell_2$ does not land on $\ell_1$. Indeed, if it lands between $Q_2$ and $E$ it is straightforward to construct a $\delta$ that goes along $c_2$ and then $d$. If it lands between $P_1$ and $Q_2$, notice that the curves $\xi_1$ and $\xi_2$ of Figure \ref{fig:bad24d}(ii) are both spin, and at least one of them intersects $\alpha$ in less than $m$ points.

\begin{figure}
\centering
\begin{tikzpicture}[scale=.7]
\draw [blue, thick] (0,0) -- (0,4);
\draw [blue, thick] (4,0) -- (4,4);
\draw [green, thick] (0,0) -- (4,0);
\draw [green, thick] (0,4) -- (4,4);
\draw [red, thick] (1,0) arc (180:90:3 and 2);
\draw [orange, thick] (0,.2) to [out=0, in=-150] (1.7,1.5) to [out=30, in=180] (4,2.2);
\node at (1,-.4) {$P_1$};
\node at (4.4,2) {$Q_1$};
\node at (-.3,-.3) {$C$};
\node at (2,-1) {a)};
\end{tikzpicture}
\begin{tikzpicture}[scale=.7]
\draw [blue, thick] (0,0) -- (0,4);
\draw [blue, thick] (4,0) -- (4,4);
\draw [green, thick] (0,0) -- (4,0);
\draw [green, thick] (0,4) -- (4,4);
\draw [red, thick] (1,0) to [out=90, in=-135] (2.6,2.4) to [out=45, in=180] (4,3);
\draw [red, thick] (0,1) arc(-90:90:.7 and .7);
\draw [red, thick] (0,3) -- (.4,3);
\draw [orange, thick] (0,2.5) arc (90:-60:.8 and .8) to [out=-150, in=90] (.2,.7) -- (.2,.4) to [out=-90, in=180] (.4,.2) to [out=0, in=-100] (.8,.4) to [out=80, in=-135] (2.5,2.5) to [out=45, in=180] (4,3.1);
\node at (1,-.4) {$P_1$};
\node at (4.4,3) {$Q_1$};
\node at (-.4,2.4) {$Q_2$};
\node at (-.4,1) {$P_2$};
\node at (-.3,-.3) {$C$};
\node at (2,-1) {b)};
\end{tikzpicture}
\begin{tikzpicture}[scale=.7]
\draw [blue, thick] (0,0) -- (0,4);
\draw [blue, thick] (4,0) -- (4,4);
\draw [green, thick] (0,0) -- (4,0);
\draw [green, thick] (0,4) -- (4,4);
\draw [red, thick] (1,0) to [out=90, in=-135] (2.6,2.4) to [out=45, in=180] (4,3);
\draw [red, thick] (0,1) arc(-90:0:2 and 3);
\draw [red, thick] (3,4) -- (3,3.6);
\draw [orange, thick] (.2,0) to [out=90, in=-120] (1.5,1.7) to [out=60, in=-90] (2.2,4);
\node at (1,-.4) {$P_1$};
\node at (4.4,3) {$Q_1$};
\node at (3,4.4) {$Q_2$};
\node at (-.4,1) {$P_2$};
\node at (-.3,-.3) {$C$};
\node at (2,-1) {c)};
\end{tikzpicture}
\caption{Nice arcs in Case D of \thref{lem:abcd}.}
\label{fig:624C}
\end{figure}

\paragraph{Case D: $Q_1\notin \ell_1,\ell_2$.} In this case, $Q_1$ must lie in the edge opposite to $\ell_2$. We get a nice arc as in Figure \ref{fig:624C}a) unless $Q_1$ is the closest point to the corner opposite to $C$. Moreover, if that is the case, either the corner opposite to $C$ is in Case A or $Q_2$ does not lie on $\ell_1$ nor on $\ell_2$, otherwise we would again find nice arcs as in Figure \ref{fig:624C}b) or c). We get the bad configuration of Figure \ref{fig:bad24}e), where both $\widetilde{c_1}$ and $\widetilde{c_2}$ are $0$-curves. 

\paragraph{Case E: $Q_1=P_2$} We assume that $\widetilde{c_1}$ is either a $0$-curve or a separating curve, otherwise it can be taken as $\delta$. This is the bad configuration of Figure \ref{fig:bad24}f).

\vspace{12pt}

In order to deal with bad configurations at the corners, it is necessary to look at the global configuration. 

\begin{figure}
\centering
\begin{tikzpicture}[scale=.7]
\draw [blue, thick] (0,0) -- (0,4);
\draw [blue, thick] (4,0) -- (4,4);
\draw [green, thick] (0,0) -- (4,0);
\draw [green, thick] (0,4) -- (4,4);
\draw [red, thick] (1.1,0) arc (180:0:.9 and .4);
\draw [red, thick] (0,1.1) arc (-90:90:.4 and .9);
\draw [red, thick] (4,1.1) arc (270:180:2.9 and 2.9);
\draw [red, thick] (2.9,4) arc (180:270:1.1 and 1.1);
\draw [orange, thick] (.2,.1) -- (.5,.1) to [out=0, in=180] (2,.6) to [out=0, in=180] (3.5,.1) -- (3.8,.1) arc(-90:0:.1 and .1) -- (3.9,.9) arc(0:90:.1 and .1) to [out=180, in=-45] (1.8,1.8);
\draw [orange, thick] (.2,.1) arc(270:180:.1 and .1) -- (.1,.5) to [out=90, in=-90] (.6,2) to [out=90, in=-90] (.1,3.5) -- (.1,3.8) arc(180:90:.1 and .1) -- (.9,3.9) arc(90:0:.1 and .1) to [out=-90, in=135] (1.8,1.8);
\draw [cyan, thick] (3.9,3.8) -- (3.9,3.2) arc (0:-90:.1 and .1) to [out=180, in=-45] (3.3,3.3);
\draw [cyan, thick] (3.9,3.8) arc (0:90:.1 and .1) -- (3.2,3.9) arc (90:180:.1 and .1) to [out=-90, in=135] (3.3,3.3);
\node at (1.1,-.4) {$P_1$};
\node at (2.9,-.4) {$Q_1$};
\node at (-.4,1.1) {$P_2$};
\node at (-.4,2.9) {$Q_2$};
\node at (2.9,4.4) {$Q_1$};
\node at (4.4,2.9) {$Q_2$};
\node at (1.1,4.4) {$P_1$};
\node at (4.4,1.1) {$P_2$};
\node at (-.3,-.3) {$C$};
\node at (4.3,4.3) {$D$};
\node [cyan] at (3.55,3.55) {$\xi_2$};
\node [orange] at (.5,.5) {$\xi_1$};
\node at (2,-1) {(i)};
\end{tikzpicture}
\begin{tikzpicture}[scale=.7]
\draw [blue, thick] (0,0) -- (0,4);
\draw [blue, thick] (4,0) -- (4,4);
\draw [green, thick] (0,0) -- (4,0);
\draw [green, thick] (0,4) -- (4,4);
\draw [red, thick] (1.1,0) arc (180:0:.9 and .4);
\draw [red, thick] (0,1.1) arc (-90:90:.4 and .9);
\draw [red, thick] (4,1.1) arc (270:180:2.9 and 2.9);
\draw [red, thick] (2.9,4) -- (2.9,3.6);
\draw [red, thick] (4,2.9) arc (90:270:.4 and .5);
\draw [purple, thick] (.4,.2) to [out=0, in=180] (2,.6) to [out=0, in=90] (3.1,0);
\draw [purple, thick] (.4,.2) arc(270:180:.2 and .2) to [out=90, in=-90] (.6,2) to [out=90, in=0] (0,3.1);
\draw [purple, thick] (3.1,4) arc (180:270:.2 and .2) -- (3.6,3.8) arc (90:0:.2 and .2) -- (3.8,3.3) arc (0:-90:.1 and .1) to [out=180, in=90] (3.4,2.4);
\draw [purple, thick] (4,3.1) arc(90:180:.1 and .1) -- (3.9,1.9) arc(0:-90:.1 and .1) to [out=180, in=-90] (3.4,2.4);
\node at (1.1,-.4) {$P_1$};
\node at (2.9,-.4) {$Q_1$};
\node at (-.4,1.1) {$P_2$};
\node at (-.4,2.9) {$Q_2$};
\node at (2.9,4.4) {$Q_1$};
\node at (4.4,2.9) {$Q_2$};
\node at (4.4,1.9) {$R_2$};
\node at (1.1,4.4) {$P_1$};
\node at (4.4,1.1) {$P_2$};
\node at (-.3,-.3) {$C$};
\node at (4.3,4.3) {$D$};
\node [purple] at (1.7,.85) {$\eta_1$};
\node at (2,-1) {(ii)};
\end{tikzpicture}
\begin{tikzpicture}[scale=.7]
\draw [blue, thick] (0,0) -- (0,4);
\draw [blue, thick] (4,0) -- (4,4);
\draw [green, thick] (0,0) -- (4,0);
\draw [green, thick] (0,4) -- (4,4);
\draw [red, thick] (1.1,0) arc (180:0:.9 and .4);
\draw [red, thick] (0,1.1) arc (-90:90:.4 and .9);
\draw [red, thick] (4,1.1) arc (270:180:2.9 and 2.9);
\draw [red, thick] (2.9,4) arc (180:270:1.1 and 2);
\draw [red, thick] (4,2.9) -- (3.6,2.9);
\draw [magenta, thick] (.4,.2) to [out=0, in=180] (2,.6) to [out=0, in=90] (3.1,0);
\draw [magenta, thick] (.4,.2) arc(270:180:.2 and .2) to [out=90, in=-90] (.6,2) to [out=90, in=0] (0,3.1);
\draw [magenta, thick] (3.1,4) arc (180:265:.8 and 1.8) to [out=-5,in=-90] (3.9,2.3) -- (3.9,3) arc (180:90:.1 and .1);
\node at (1.1,-.4) {$P_1$};
\node at (2.9,-.4) {$Q_1$};
\node at (-.4,1.1) {$P_2$};
\node at (-.4,2.9) {$Q_2$};
\node at (2.9,4.4) {$Q_1$};
\node at (4.4,2) {$R_2$};
\node at (4.4,2.9) {$Q_2$};
\node at (1.1,4.4) {$P_1$};
\node at (4.4,1.1) {$P_2$};
\node at (-.3,-.3) {$C$};
\node at (4.3,4.3) {$D$};
\node [magenta] at (1.7,.85) {$\eta_2$};
\node at (2,-1) {(iii)};
\end{tikzpicture}
\caption{Dealing with the bad configuration of Figure \ref{fig:bad24}a).}
\label{fig:bad24a}
\end{figure}

Fix again a corner $C$, and assume that at each corner there is one of the bad configurations of Figure \ref{fig:bad24}. If the situation is that of Figure \ref{fig:bad24}a), the curve $\xi_1$ of Figure \ref{fig:bad24a} is a $1$-curve by assumption. We suppose that it is separating, otherwise it can be taken as $\delta$. Consider the corner $D$. If at $D$ we have the bad configuration of Figure \ref{fig:bad24}f), i.e. if the points corresponding to $Q_1$ and $Q_2$ are joined by an arc $\ell$, then $\widetilde{\ell}$ must be spin and nonseparating as $[\alpha]=[\xi_1]+[\widetilde{\ell}]$ in homology with $\sfrac{\Z}{2\Z}$ coefficients. 

If at $D$ we have a different bad configuration (i.e. that of Figure \ref{fig:bad24}b), c) or d)), then we take as $\delta$ the curve $\eta_1$ or $\eta_2$ of Figure \ref{fig:bad24a}(ii) and (iii). Here, we can assume that the $\alpha$-arcs from $Q_2$ to $R_2$ and from $Q_1$ to $R_1$ have spin value $0$ by the above discussion. This concludes the proof in the presence of the bad configuration of Figure \ref{fig:bad24}a).

\begin{figure}
\centering
\begin{tikzpicture}[scale=.75]
\draw [blue, thick] (0,0) -- (0,4);
\draw [blue, thick] (4,0) -- (4,4);
\draw [green, thick] (0,0) -- (4,0);
\draw [green, thick] (0,4) -- (4,4);
\draw [red, thick] (0,2.9) to [out=0, in=135] (1.8,1.8) to [out=-45, in=90] (2.9,0);
\draw [red, thick] (4,1.1) to [out=180, in=-45] (2.2,2.2) to [out=135, in=-90] (1.1,4);
\draw [red, thick] (2.9,4) -- (2.9,3.6);
\draw [red, thick] (4,2.9) arc (90:270:.4 and .5);
\draw [purple, thick] (0,3.8) -- (.8,3.8) arc(90:30:.2 and .2) to [out=-90, in=135] (2.1,2.1);
\draw [purple, thick] (3.8,0) -- (3.8,.8) arc(0:60:.2 and .2) to [out=180, in=-45] (2.1,2.1);
\draw [purple, thick] (3.8,4) -- (3.8,3.3) arc (0:-90:.1 and .1) to [out=180, in=90] (3.4,2.4);
\draw [purple, thick] (4,3.8) arc(90:180:.1 and .1) -- (3.9,1.9) arc(0:-90:.1 and .1) to [out=180, in=-90] (3.4,2.4);
\node at (2.9,-.4) {$Q_1$};
\node at (-.4,2.9) {$P_1$};
\node at (2.9,4.4) {$Q_1$};
\node at (4.4,2.9) {$P_1$};
\node at (4.4,1.9) {$R_1$};
\node at (1.1,4.4) {$P_2$};
\node at (4.4,1.1) {$Q_2$};
\node at (-.3,4.3) {$C$};
\node at (-.3,-.3) {$F$};
\node at (4.3,4.3) {$E$};
\node [purple] at (3.4,.5) {$\eta_1$};
\node at (2,-1) {(i)};
\end{tikzpicture}
\begin{tikzpicture}[scale=.75]
\draw [blue, thick] (0,0) -- (0,4);
\draw [blue, thick] (4,0) -- (4,4);
\draw [green, thick] (0,0) -- (4,0);
\draw [green, thick] (0,4) -- (4,4);
\draw [red, thick] (0,2.9) to [out=0, in=135] (1.8,1.8) to [out=-45, in=90] (2.9,0);
\draw [red, thick] (4,1.1) to [out=180, in=-45] (2.2,2.2) to [out=135, in=-90] (1.1,4);
\draw [red, thick] (2.9,4) arc (180:270:1.1 and 2);
\draw [red, thick] (4,2.9) -- (3.6,2.9);
\draw [magenta, thick] (0,3.8) -- (.8,3.8) arc(90:30:.2 and .2) to [out=-90, in=135] (2.1,2.1);
\draw [magenta, thick] (3.8,0) -- (3.8,.8) arc(0:60:.2 and .2) to [out=180, in=-45] (2.1,2.1);
\draw [magenta, thick] (3.8,4) arc(0:-90:.1 and .1) -- (3.2,3.9) arc(90:180:.1 and .1) arc (180:260:.8 and 1.6) to [out=-10,in=-90] (3.9,2.3) -- (3.9,3.7) arc (180:90:.1 and .1);
\node at (2.9,-.4) {$Q_1$};
\node at (-.4,2.9) {$P_1$};
\node at (2.9,4.4) {$Q_1$};
\node at (4.4,2.9) {$P_1$};
\node at (4.4,2) {$R_1$};
\node at (1.1,4.4) {$P_2$};
\node at (4.4,1.1) {$Q_2$};
\node at (-.3,4.3) {$C$};
\node at (-.3,-.3) {$F$};
\node at (4.3,4.3) {$E$};
\node [magenta] at (3.4,.5) {$\eta_2$};
\node at (2,-1) {(ii)};
\end{tikzpicture}
\caption{Dealing with the bad configuration of Figure \ref{fig:bad24}e).}
\label{fig:bad24e}
\end{figure}

Consider now the bad configuration of Figure \ref{fig:bad24}e). As already observed, we may assume that the curves $\widetilde{c_1}$ and $\widetilde{c_2}$ are $0$-curves. Notice that the configuration at the corners $E$ and $F$ cannot be that of Figure \ref{fig:bad24}f). Hence, there is an arc with spin value $0$ from $P_1$ to $R_1$ or from $Q_1$ to $R_1$ as in Figure \ref{fig:bad24e}(i) and (ii), and we can take as $\delta$ the corresponding curve $\eta_1$ or $\eta_2$. This settles the case of Figure \ref{fig:bad24}e). 

\begin{figure}
\centering
\begin{tikzpicture}
\draw [blue, thick] (0,0) -- (0,4);
\draw [blue, thick] (4,0) -- (4,4);
\draw [green, thick] (0,0) -- (4,0);
\draw [green, thick] (0,4) -- (4,4);
\draw [red, thick] (.5,0) arc (180:0:.5 and .5);
\draw [red, thick] (0,.5) arc (-90:90:.5 and .5);
\draw [red, thick] (2.5,0) arc (0:90:2.5 and 2.5);
\draw [red, thick] (1.9,0) arc (0:90:1.9 and 1.9);
\draw [red, thick] (3,0) arc (180:90:1 and 2.5);
\draw [red, thick] (0,3) arc (-90:0:2.5 and 1);
\draw [red, thick] (3,4) -- (3,3.6);
\draw [red, thick] (4,3) arc (90:195:1 and 1);
\draw [red, thick] (4,1) arc (270:232:1 and 1);
\draw [orange, thick] (0,2.6) arc (90:6:2.6 and 2.6) arc(-174:-6:.17 and .2) arc (174:100:1.1 and 2.6) arc (-80:100:.1 and .16) arc (100:192:.86 and .9);
\draw [orange, thick] (4,1.1) arc(270:230:.9 and .9);
\node at (2.5,-.4) {$R_n$};
\node at (-.4,2.5) {$S_n$};
\node at (2.5,4.4) {$R_n$};
\node at (4.4,2.5) {$S_n$};
\node at (3.3,-.4) {$R_{n+1}$};
\node at (-.5,3) {$S_{n+1}$};
\node at (3.3,4.4) {$R_{n+1}$};
\node at (4.5,3) {$S_{n+1}$};
\node [red] at (2.2,.1) {$\dots$};
\node [red] at (.1,2.25) {$\vdots$};
\node at (-.3,-.3) {$C$};
\node at (2,-1) {(i)};
\end{tikzpicture}
\begin{tikzpicture}
\draw [blue, thick] (0,0) -- (0,4);
\draw [blue, thick] (4,0) -- (4,4);
\draw [green, thick] (0,0) -- (4,0);
\draw [green, thick] (0,4) -- (4,4);
\draw [red, thick] (.5,0) arc (180:0:.5 and .5);
\draw [red, thick] (0,.5) arc (-90:90:.5 and .5);
\draw [red, thick] (2.5,0) arc (0:90:2.5 and 2.5);
\draw [red, thick] (1.9,0) arc (0:90:1.9 and 1.9);
\draw [red, thick] (3,0) arc (180:90:1 and 2.5);
\draw [red, thick] (0,3) arc (-90:0:2.5 and 1);
\draw [red, thick] (3,4) -- (3,3.6);
\draw [red, thick] (4,3) arc (270:225:2.5 and 1);
\draw [red, thick] (1.5,4) arc (180:205:2.5 and 1);
\draw [teal, thick] (2.9,0) arc(180:98:1.1 and 2.6) arc(-82:82:.1 and .16) arc(262:223:2.8 and 1.1);
\draw [teal, thick] (1.4,4) arc (180:207:2.6 and 1.1);
\node at (2.5,-.4) {$R_n$};
\node at (-.4,2.5) {$S_n$};
\node at (2.5,4.4) {$R_n$};
\node at (4.4,2.5) {$S_n$};
\node at (3.3,-.4) {$R_{n+1}$};
\node at (-.5,3) {$S_{n+1}$};
\node at (3.3,4.4) {$R_{n+1}$};
\node at (4.5,3) {$S_{n+1}$};
\node [red] at (2.2,.1) {$\dots$};
\node [red] at (.1,2.25) {$\vdots$};
\node at (-.3,-.3) {$C$};
\node at (2,-1) {(ii)};
\end{tikzpicture}
\begin{tikzpicture}
\draw [blue, thick] (0,0) -- (0,4);
\draw [blue, thick] (4,0) -- (4,4);
\draw [green, thick] (0,0) -- (4,0);
\draw [green, thick] (0,4) -- (4,4);
\draw [red, thick] (0,.5) arc (-90:90:.5 and .5);
\draw [red, thick] (.5,0) arc (180:100:.5 and .5);
\draw [red, thick] (1.5,0) arc (0:60:.5 and .5);
\draw [red, thick] (2.5,0) arc (0:40:2.5 and 2.5);
\draw [red, thick] (0,2.5) arc (90:55:2.5 and 2.5);
\draw [red, thick] (1.9,0) arc (0:35:1.9 and 1.9);
\draw [red, thick] (0,1.9) arc (90:55:1.9 and 1.9);
\draw [red, thick] (3,0) arc (180:90:1 and 2.5);
\draw [red, thick] (0,3) arc (-90:0:2.5 and 1);
\draw [red, thick] (3,4) -- (3,3.6);
\draw [red, thick] (4,3) arc (90:180:3 and 3);
\draw [cyan, thick] (2.6,4) arc (0:-85:2.6 and 1.1) arc (95:270:.1 and .17) arc (90:57:2.3 and 2.6);
\draw [cyan, thick] (1.1,0) arc (180:95:2.9 and 2.9) to [out=5, in=90] (3.9,2.75) arc (0:-86:.1 and .2) arc (94:175:1 and 2.6) to [out=-95, in=0] (2.75,.1) arc (270:196:.18 and .1) arc(4:40:2.6 and 2.6);
\node at (2.5,-.4) {$R_n$};
\node at (-.4,2.5) {$S_n$};
\node at (2.5,4.4) {$R_n$};
\node at (4.4,2.5) {$S_n$};
\node at (3.3,-.4) {$R_{n+1}$};
\node at (-.5,3) {$S_{n+1}$};
\node at (3.3,4.4) {$R_{n+1}$};
\node at (4.5,3) {$S_{n+1}$};
\node [red] at (2.2,.1) {$\dots$};
\node [red] at (.1,2.25) {$\vdots$};
\node at (-.3,-.3) {$C$};
\node at (2,-1) {(iii)};
\end{tikzpicture}
\begin{tikzpicture}[scale=.66]
\draw [blue, thick] (0,0) -- (0,6);
\draw [blue, thick] (6,0) -- (6,6);
\draw [green, thick] (0,0) -- (6,0);
\draw [green, thick] (0,6) -- (6,6);
\draw [red, thick] (2,0) arc(0:90:2 and 2);
\draw [red, thick] (4,0) arc(180:90:2 and 2);
\draw [red, thick] (2,6) arc(0:-90:2 and 2);
\draw [red, thick] (4,6) arc(-180:-90:2 and 2);
\draw [gray] (4.4,4.4) arc (-135:-160:2.1 and 2.1) to [out=110, in=0] (3,5.8);
\draw [gray] (4.4,4.4) arc (-135:-110:2.1 and 2.1) to [out=-20, in=90] (5.8,3);
\draw [gray] (1.6,4.4) arc (-45:-20:2.1 and 2.1) to [out=70, in=180] (3,5.8);
\draw [gray] (1.6,4.4) arc (-45:-70:2.1 and 2.1) to [out=-160, in=90] (.2,3);
\draw [gray] (4.4,1.6) arc (135:160:2.1 and 2.1) to [out=-110, in=0] (3,.2);
\draw [gray] (4.4,1.6) arc (135:110:2.1 and 2.1) to [out=20, in=-90] (5.8,3);
\draw [gray] (1.6,1.6) arc (45:20:2.1 and 2.1) to [out=-70, in=180] (3,.2);
\draw [gray] (1.6,1.6) arc (45:70:2.1 and 2.1) to [out=160, in=-90] (.2,3);
\node [gray] at (1.9,4.1) {$\zeta$};
\node at (-.45,-.45) {$C$};
\node at (3,-1.5) {(iv)};
\end{tikzpicture}
\caption{Dealing with the bad configuration of Figure \ref{fig:bad24}b).}
\label{fig:bad24b}
\end{figure}

Assume now that the bad configuration at $C$ is that of Figure \ref{fig:bad24}b). Recall that the intersection points with $\alpha$ are placed symmetrically on edges $\ell_1$ and $\ell_2$ with respect to $R_1$ and $S_1$. Call $R_2,\dots,R_k$ the intersection point on $\ell_1$ further from $C$ than $R_1$, and $S_2,\dots,S_k$ the symmetric points on $\ell_2$. As already observed, there may be arcs going from $R_2$ to $S_2$ and so on, but not all arcs starting at the $R_i$ are of this form as we have excluded the configurations of Figure \ref{fig:bad24}a) and e). Let $c_n$ be the last arc of this form, going from $R_n$ to $S_n$. We can assume that $\widetilde{c_n}$ is a $1$-curve as before. Then the $\alpha$-arc from $R_{n+1}$ does not land on $\ell_2$ by the arguments of Case A. We may assume that it lands on the point corresponding to $S_n$, as otherwise there would be a nice arc. Similarly, we assume that the arc from $S_{n+1}$ lands on $R_n$. 

We claim that $R_{n+1}$ and $S_{n+1}$ are joined by an $\alpha$-arc. Indeed, if the arc from $S_{n+1}$ lands on a different point, we can find nice arcs as in Figure \ref{fig:bad24b}(i), (ii) and (iii). Therefore, under the assumption that there are no nice arcs the configuration degenerates to that of Figure \ref{fig:bad24}f) on each corner. In this case, the homology class mod $2$ of $\alpha$ is the same as that of the curve $\zeta$ of Figure \ref{fig:bad24b}(iv), so $\zeta$ is a nonseparating $1$-curve. 

We are left to deal with the case where there are only the bad configurations of Figure \ref{fig:bad24}c), d) and f). It is easy to see that in this situation, on every edge there are a couple of adjacent intersection points that are the endpoints of $\alpha$-arcs landing on opposite sides. We can then repeat the analysis of Figure \ref{fig:bad24b} to show that it is always possible to find a nice arc and/or a curve $\delta$ as in the statement. \qedhere
\end{proof}

\begin{remark}
Note that \thref{lem:abcd} fails for $m=1$: just consider a bad triple on a genus $3$ surface.
\end{remark}

\begin{theorem}
The complex $X_g$ is simply connected.
\end{theorem}

\begin{proof}[Proof {\cite[Proposition 19]{waj:elem}}]
We just have to prove that paths of radius at least $2$ are null-homotopic. Let $\mathbf{p}$ be a path of radius $m \ge 2$ around some curve $\alpha$ contained in a vertex $v_0$ of $\mathbf{p}$. Let $v_1$ be the first vertex of $\mathbf{p}$ such that $d_{\alpha}(v_1)=m$; then $v_1$ contains a curve $\beta$ such that $\abs{\alpha \cap \beta}=m$. Consider the maximal $\beta$-segment starting from $v_1$ such that all its vertices have distance $m$ from $\alpha$, and call $v_2$ the last vertex of such segment. Moreover, call $u_1$ the last vertex before $v_1$, and $u_2$ the first vertex after $v_2$. Then, there are curves $\gamma_1 \in u_1$ and $\gamma_2 \in u_2$ such that $\abs{\gamma_i \cap \alpha}=d_{\alpha}(u_i)$; in particular, $\abs{\gamma_i \cap \alpha} \le m$. If $\gamma_i$ is disjoint from $\beta$, then set $\delta_i:=\gamma_i$; otherwise, call $\delta_i$ the curve given by \thref{lem:abcd}. We want to construct a shortcut as follows:
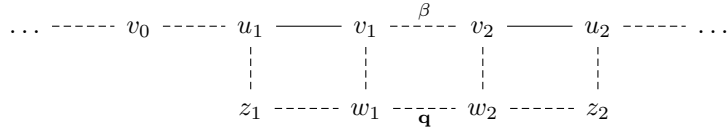
\[
\begin{tikzcd}
\dots \arrow[r, dashed, no head] & v_0 \arrow[r, no head, dashed] & u_1 \arrow[r, no head] \arrow[d, no head, dashed] & v_1 \arrow[r, no head, dashed, "\beta"] \arrow[d, no head, dashed] & v_2 \arrow[r, no head] \arrow[d, no head, dashed] & u_2 \arrow[d, no head, dashed] \arrow[r, dashed, no head] & \dots\\
& & z_1 \arrow[r, no head, dashed] & w_1 \arrow[r, no head, dashed, "\mathbf{q}"'] & w_2 \arrow[r, no head, dashed] & z_2 & 
\end{tikzcd}
\]

If $\delta_i=\gamma_i$, simply put $z_i=u_i$. If $\delta_i \ne \gamma_i$, assume for now that $\delta_i$ is neither homologous to $\beta$ nor to $\gamma_i$. If $\delta_i$ is disjoint and independent from $\gamma_i$, let $z_i$ be a vertex containing both $\gamma_i$ and $\delta_i$, and join it to $u_i$ via a $\gamma_i$-segment. Similarly, if $\delta_i$ is disjoint and independent from $\beta$, let $w_i$ be a vertex containing both $\beta$ and $\delta_i$, and join it to $v_i$ via a $\beta$-segment. If we have $\abs{\gamma_i \cap \delta_i}=1$, let $z_i$ be a vertex containing $\delta_i$, and join it to $u_i$ via a path of the following form:
\[
\begin{tikzcd}
u_i \arrow[r, no head, dashed, "\gamma_i"] & \Braket{\gamma_i} \arrow[r, no head] & \Braket{\delta_i} \arrow[r, no head, dashed, "\delta_i"] & z_i.
\end{tikzcd}
\]
Similarly, if $\abs{\beta \cap \delta_i}=1$, let $w_i$ be a vertex containing $\delta_i$, and join it to $v_i$ via a path of the following form:
\[
\begin{tikzcd}
v_i \arrow[r, no head, dashed, "\beta"] & \Braket{\beta} \arrow[r, no head] & \Braket{\delta_i} \arrow[r, no head, dashed, "\delta_i"] & w_i.
\end{tikzcd}
\]
Now join the vertex $z_i$ to $w_i$ via a $\delta_i$-path. We thus obtain a closed path $u_i-v_i\text{- -}w_i\text{- -}z_i\text{- -}u_i$ of radius $1$ around $\delta_i$; moreover, $u_1\text{- -}z_1\text{- -}w_1$ has radius strictly less than $m$ around $\alpha$.

Observe that $\delta_i$ cannot be homologous to both $\gamma_i$ and $\beta$. Assume that it is homologous to $\gamma_i$ (and disjoint from it). Then on each component $S_1,S_2$ of $\Sigma_g \setminus (\gamma_i\cup\delta_i)$ we can find an arc $a_i$ that connects the two boundary components and is disjoint from $\alpha$. Call $c_i,c_i'$ and $d_i,d_i'$ the two arcs in which the endpoints of $a_1,a_2$ divide $\gamma_i$ and $\delta_i$ respectively. Then we have $\abs{c_i\cap \alpha}+\abs{c_i'\cap\alpha} \le m$ and $\abs{d_i\cap \alpha}+\abs{d_i'\cap\alpha} < m$. Up to renaming, we can assume that the curves obtained by smoothing the unions ${a_1 \cup c_i \cup a_2 \cup d_i}$ and $a_1 \cup c_i' \cup a_2 \cup d_i'$ are $1$-curves, since their homology classes mod $2$ sum to $[\gamma_i]+[\delta_i]$ and their algebraic intersection is $0$. Moreover, one of them intersects $\alpha$ in less than $m$ points; call it $\eta_i$. Note that $\abs{\gamma_i \cap \eta_i}=\abs{\delta_i\cap \eta_i}=1$. Now let $z_i$ be a vertex containing $\delta_i$, and join it to $u_i$ via a path of the following form:
\[
\begin{tikzcd}
u_i \arrow[r, no head, dashed, "\gamma_i"] & \Braket{\gamma_i} \arrow[r, no head] & \Braket{\eta_i} \arrow[r, no head, dashed, "\eta_i"] & \Braket{\eta_i} \arrow[r, no head] & \Braket{\delta_i} \arrow[r, no head, dashed, "\delta_i"] & z_i.
\end{tikzcd}
\]
If $\delta_i$ is homologous to $\beta$ (and disjoint from it), we simply choose a curve $\xi_i$ which intersects both $\delta_i$ and $\beta$ once; up to Dehn twisting along $\beta$, we can assume that it is a $1$-curve. Now let $w_i$ be a vertex containing $\delta_i$, and join it to $v_i$ via a path of the following form:
\[
\begin{tikzcd}
u_i \arrow[r, no head, dashed, "\gamma_i"] & \Braket{\gamma_i} \arrow[r, no head] & \Braket{\xi_i} \arrow[r, no head, dashed, "\xi_i"] & \Braket{\xi_i} \arrow[r, no head] & \Braket{\delta_i} \arrow[r, no head, dashed, "\delta_i"] & w_i.
\end{tikzcd}
\]
Now we join $z_i$ to $w_i$ via a $\delta_i$-path and we get the same properties as before. 

Finally, applying \thref{lem:17}(b) we can join $w_1$ to $w_2$ via a path $\mathbf{q}$ such that all its vertices have distance and less than $m$ from $\alpha$ and from $\beta$, with the only possible exception of the last $\delta_2$-segment. This concludes the proof.  \qedhere
\end{proof}

\section{A finite presentation}
\label{fp}

Consider the even spin structure $\phi$ on a surface $\Sigma_g^1$ of genus $g$ with one boundary component $C$ defined by $\phi(C)=1$, $\phi(\alpha_i)=1$ and $\phi(\beta_i)=0$ for all $i=1,\dots,g$, in the notation of Figure \ref{fig:q}. In this section, we will find a finite presentation for $\Modd(\Sigma_g^1)[\phi]$ and $\Modd(H_g)[\phi]$, where $H_g$ is the handlebody in which the $\alpha_i$ bound disks. 

Given group elements $a,b$, we will denote by $a*b$ the conjugate $aba^{-1}$.

\subsection{The strategy} We start by recalling Hatcher and Thurston's strategy (see also Laudenbach's survey article \cite{lau:pres}). 

Fix a vertex $v_0 \in X_g$. By the spin change of coordinates, $\Modd(\Sigma_g^1)[\phi]$ acts transitively on the vertices of $X_g$, and we will see that there is a finite number of orbits of edges and faces with a vertex at $v_0$. For every orbit $O$ of edges with a vertex at $v_0$, let $r_O\in \Modd(\Sigma_g^1)[\phi]$ be such that $v_0-r_O(v_0)$ is a representative of $O$. Call $S$ the union of a generating set for $H[\phi]:=\operatorname{Stab}(v_0)$ and the elements $r_O$. 

There is a correspondence between paths in $X_g$ and words in $S$. Given $\varphi \in \Modd(\Sigma_g^1)[\phi]$, by \thref{prop:conn} there is an edge-path $v_0-v_1-\dots-v_k=\varphi(v_0)$. We can associate to such a path a word in $S$ as follows. Let $O_1$ be the edge orbit of $v_0-v_1$; then, there exists $h_1 \in H[\phi]$ such that $h_1^{-1}(v_1)=r_{O_1}(v_0)$, i.e. $h_1r_{O_1}(v_0)=v_1$. Now, let $O_2$ be the edge orbit of $v_0-(h_1r_{O_1})^{-1}(v_2)$, and find $h_2 \in H[\phi]$ such that $h_1r_{O_1}h_2r_{O_2}(v_0)=v_2$, and so on. Every $h_i$ can be expressed as a word in the generators of $H[\phi]$, so the resulting \emph{$h$-product} $h_1r_{O_1}\dots h_kr_{O_k}$ is indeed a word in $S$. Moreover, we have $(h_1r_1\dots h_kr_k)(v_0)=\varphi(v_0)$, so $\varphi^{-1}h_1r_1\dots h_kr_k$ is equal to some $h_{k+1}^{-1}\in H[\phi]$ and we can express $\varphi$ as a word in $S$.

In the other direction, given an $h$-product $h_1r_{1}\dots h_kr_{k}$ we can construct an edge path by setting $v_i:=h_1r_{1}\dots h_ir_{i}(v_0)$ for $i=0,1,\dots,k$. If an $h$-product $h_1r_1\dots h_kr_k$ corresponds to a closed edge-path, then $h_1r_1\dots h_kr_kh_{k+1}$ is a relation in $\Modd(\Sigma_g^1)[\phi]$, for some $h_{k+1}\in H[\phi]$. 

We use this correspondence to prove the following theorem, which is the main result of this section.

\begin{theorem}
\thlabel{thm:mcgs}
The group $\Modd(\Sigma_g^1)[\phi]$ admits a finite presentation with generating set $S$ and the following relations:
\begin{enumerate}
\item[\ref{aiaj}-\ref{a8}] relations in the presentation of the stabilizer $H[\phi]$ of $v_0$;
\item[\ref{backtracking}] an $h$-product representing each path $v_0-r(v_0)-v_0$, where $r \in S \setminus H[\phi]$;
\item[\ref{diffwr}] all relations of the form $r^{-1}*h=h_0$, where $r\in S \setminus H[\phi]$, $h$ is a generator of the stabilizer of the edge $v_0-r(v_0)$ and $h_0 \in H[\phi]$;
\item[\ref{triangles}-\ref{hyp}] an $h$-product representing each $\Modd(\Sigma_g^1)[\phi]$-orbit of faces in $X_g$ with a vertex at $v_0$.
\end{enumerate}
\end{theorem}

\begin{proof}
Call $G$ the group given by the presentation in the statement. Observe first that $H[\phi]$ is finitely presented, as it is a finite index subgroup of the stabilizer $H$ of $v_0$ under the action of the full mapping class group, which is finitely presented by \cite[Proposition 27]{waj:elem}.

The above discussion shows that $\Modd(\Sigma_g^1)[\phi]$ is a quotient of $G$. A relation in $\Modd(\Sigma_g^1)[\phi]$ can be written as an $h$-product $h_1r_1\dots h_kr_kh_{k+1}$ in $G$, which represents a closed edge path $\mathbf{p}$ in $X_g$. We want to show that $h_1r_1\dots h_kr_kh_{k+1}$ is equal to the identity in $G$. 

First, any other $h$-product $h_1'r_1'\dots h_k'r_k'h_{k+1}'$ representing $\mathbf{p}$ is equal to $h_1r_1\dots h_kr_kh_{k+1}$ in $G$. Indeed, we have $r_1(v_0)=h_1^{-1}h_1'r_1'(v_0)$, hence $r_1$ and $r_1'$ represent the same edge orbit and $r_1=r_1'$. Moreover, $h_1^{-1}h_1'$ fixes the edge $v_0-r_1(v_0)$, hence \ref{diffwr} gives $h_1^{-1}h_1'r_1=r_1h_1''$ for some $h_1'' \in H[\phi]$. As a consequence, 
\[
h_1'r_1'h_2'r_2'\dots h_k'r_k'h_{k+1}'=h_1r_1h_1''h_2'r_2'\dots h_k'r_k'h_{k+1}',
\]
so we get two shorter $h$-products representing the same edge-path and we conclude by induction on $k$. 

Moreover, we can assume that $\mathbf{p}$ does not contain \emph{backtrackings}, i.e. subpaths $v_i-v_{i+1}-v_{i+2}$ where $v_i=v_{i+2}$. Indeed, if there is such a subpath, we may assume that it is represented by a conjugate of an $h$-product representing the path $v_0-r(v_0)-v_0$ for some generator $r$, and these are trivial in $G$ by \ref{backtracking}. 

Finally, by \thref{thm:x} $\mathbf{p}$ is null-homotopic, hence it can be written as a composition of paths that go from $v_0$ to some vertex $v$, then go around a face of $X_g$ and finally go back from $v$ to $v_0$ along the same path as before. By the above discussion, we can assume that these paths are represented by conjugates of $h$-product representing faces with a vertex at $v_0$, which are trivial in $G$ by \ref{triangles}-\ref{hyp}. 

As the number of orbits of edges and faces of $X_g$ touching $v_0$ is finite, the resulting presentation is finite. \qedhere
\end{proof}

\begin{figure}
\centering
\begin{tikzpicture}[scale=0.75]
\draw (.5,0) to [out=0, in=-90] (1,2) to [out=90, in=180] (2,3.5) to [out=0, in=90] (3,2) to [out=-90, in=180] (3.5,1);
\draw (2,2.4) to [out=-60, in=60] (2,1.2);
\draw (2.12,2.7) to [out=-120, in=120] (2.12,.9);
\draw (3.5,0) to [out=0, in=-90] (4,2) to [out=90, in=180] (5,3.5) to [out=0, in=90] (6,2) to [out=-90, in=180] (6.5,1);
\draw (5,2.4) to [out=-60, in=60] (5,1.2);
\draw (5.12,2.7) to [out=-120, in=120] (5.12,.9);
\draw (7,0) to [out=0, in=-90] (7.5,2) to [out=90, in=180] (8.5,3.5) to [out=0, in=90] (9.5,2) to [out=-90, in=180] (10,1);
\draw (8.5,2.4) to [out=-60, in=60] (8.5,1.2);
\draw (8.62,2.7) to [out=-120, in=120] (8.62,.9);
\node at (6.75,1.7) {$\dots$};
\draw (.9,2) to [out=180, in=90] (-.5,.5) to [out=-90, in=180] (.9,-1) to [out=0, in=180] (9.6,-1) to [out=0, in=-90] (11,.5) to [out=90, in=0] (9.6,2);
\draw (1.95,2) -- (2.05,2);
\draw (3.1,2) -- (3.9,2);
\draw (4.95,2) -- (5.05,2);
\draw (6.1,2) -- (6.5,2);
\draw (7,2) -- (7.4,2);
\draw (8.45,2) -- (8.55,2);
\draw [red] (2,3.5) arc (90:-90:.2 and .55);
\draw [red, dashed] (2,3.5) arc (90:270:.2 and .55);
\draw [red] (5,3.5) arc (90:-90:.2 and .55);
\draw [red, dashed] (5,3.5) arc (90:270:.2 and .55);
\draw [red] (8.5,3.5) arc (90:-90:.2 and .55);
\draw [red, dashed] (8.5,3.5) arc (90:270:.2 and .55);
\draw [orange, dashed] (1.03,1.3) to [out=20, in=160] (1.93,1.3);
\draw [orange, dashed] (4.03,1.3) to [out=20, in=160] (4.93,1.3);
\draw [orange] (4.03,1.3) to [out=-60, in=0] (3.4,-.2) to [out=180, in=-100] (1.93,1.3);
\draw [orange] (1.03,1.3) to [out=-80, in=180] (3.4,-.6) to [out=0,  in=-100] (4.93,1.3);
\draw [blue] (2,2.9) to [out=0, in=90] (2.5,1.8) to [out=-90, in=0] (2,.7) to [out=180, in=-90] (1.5,1.8) to [out=90, in=180] (2,2.9);
\draw [green] (4.9,1.7) -- (5.1,1.7);
\draw [green] (3.9,1.7) to [out=180, in=0] (3.5,.8) to [out=180, in=-90] (2.7,1.8) to [out=90, in=0] (2,3.1) to [out=180, in=90] (1.3,1.8) to [out=-90, in=180] (3.5,-.4) to [out=0, in=-85] (4.3,1.8) to [out=95, in=180] (5,3.1) to [out=0, in=90] (5.7,1.8) to [out=-90, in=180] (6.5,.8) to [out=0, in=0] (6.1,1.7);
\draw [blue] (5,2.9) to [out=0, in=90] (5.5,1.8) to [out=-90, in=0] (5,.7) to [out=180, in=-90] (4.5,1.8) to [out=90, in=180] (5,2.9);
\draw [blue] (8.5,2.9) to [out=0, in=90] (9,1.8) to [out=-90, in=0] (8.5,.7) to [out=180, in=-90] (8,1.8) to [out=90, in=180] (8.5,2.9);
\node [red] at (2,3.8) {$\alpha_1$};
\node [red] at (5,3.8) {$\alpha_2$};
\node [red] at (8.5,3.8) {$\alpha_g$};
\node [blue] at (2.6,.6) {$\beta_1$};
\node [blue] at (5.4,.4) {$\beta_2$};
\node [blue] at (8.2,.4) {$\beta_g$};
\node [orange] at (1.6,-.4) {$\delta_{1,2}$};
\node [green] at (3.5,.5) {$\varepsilon_1$};
\draw (10,-.2) arc (0:360:.3 and .3);
\node at (9.1,-.2) {$C$};
\end{tikzpicture}
\caption{Some of the curves involved in the definition of the generators for $\Modd(\Sigma_g^1)[\phi]$.}
\label{fig:q}
\end{figure}
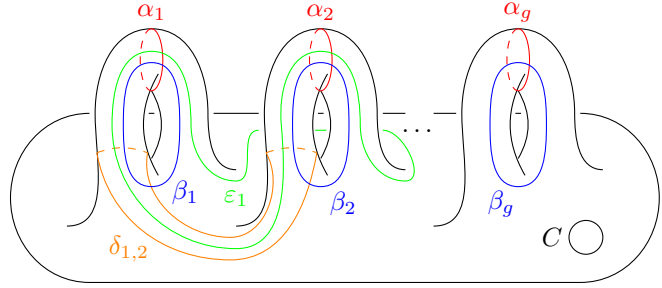

\subsection{Stabilizer of a vertex}
\label{stab}

Consider the spin cut-system $v_0=\Braket{\alpha_1,\dots,\alpha_g}$ of Figure \ref{fig:q}. We are going to give a presentation of its stabilizer $H[\phi]$ under the action of $\Modd(\Sigma_g^1)[\phi]$, which is a finite-index subgroup of the stabilizer $H$ of $v_0$ under the action of $\Modd(\Sigma_g^1)$. We will apply the Nielsen-Schreier method to Wajnryb's presentation of $H$ \cite[Proposition 27]{waj:elem}. 

First of all, we introduce Wajnryb's generators for $H$. In the notation of Figure \ref{fig:q}, set $a_i:=t_{\alpha_i}$ for $i=1,\dots,g$, $s:=t_{\beta_1}t_{\alpha_1}^2t_{\beta_1}$ and $t_i:=t_{\varepsilon_i}t_{\alpha_i}t_{\alpha_{i+1}}t_{\varepsilon_i}$ for $i=1,\dots,g-1$. Moreover, for all $i,j \in \set{\pm1,\dots,\pm g}$ with $i<j$, let $\delta_{i,j}$ be the curve in Figure \ref{fig:dij}, and set $d_{i,j}:=t_{\delta_{i,j}}$ and 
\begin{equation}
\label{eq:dij}
\overline{d}_{i,j}:=d_{i,j}a_{\abs{i}}^{-1}a_{\abs{j}}^{-1}.
\end{equation}

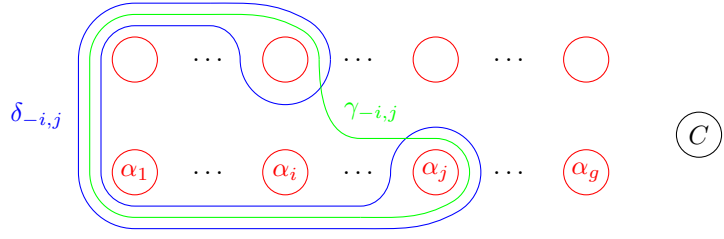
\begin{figure}
\centering
\begin{tikzpicture}
\draw [red] (1.5,.5) circle (3mm);
\draw [red] (1.5,2) circle (3mm);
\node at (2.5,.5) {$\dots$};
\node at (2.5,2) {$\dots$};
\draw [red] (3.5,.5) circle (3mm);
\draw [red] (3.5,2) circle (3mm);
\node at (4.5,.5) {$\dots$};
\node at (4.5,2) {$\dots$};
\draw [red] (5.5,.5) circle (3mm);
\draw [red] (5.5,2) circle (3mm);
\node at (6.5,.5) {$\dots$};
\node at (6.5,2) {$\dots$};
\draw [red] (7.5,.5) circle (3mm);
\draw [red] (7.5,2) circle (3mm);
\node [red] at (1.5,.5) {$\alpha_1$};
\node [red] at (3.5,.5) {$\alpha_i$};
\node [red] at (5.5,.5) {$\alpha_j$};
\node [red] at (7.5,.5) {$\alpha_g$};
\draw [blue] (4.1,2) arc(0:-180:.6 and .6) to [out=90, in=0] (2.5,2.45) -- (1.5,2.45) arc (90:180:.45 and .45) -- (1.05,.5) arc (180:270:.45 and .45) -- (4.5,.05) to [out=0, in=-90] (4.9,.5) arc (180:-60:.6 and .6) to [out=-150, in=0] (4.5,-.25) -- (1.5,-.25) arc (270:180:.75 and .75) -- (.75,2) arc (180:90:.75 and .75) -- (2.5,2.75);
\draw [blue] (4.1,2) arc(0:60:.6 and .6) to [out=150,in=0] (2.5,2.75);
\draw [green] (3.95,2) arc(0:60:.45 and .45) to [out=150,in=0] (2.5,2.6) -- (1.5,2.6) arc (90:180:.6 and .6) -- (.9,.5) arc (180:270:.6 and .6) -- (4.5,-.1);
\draw [green] (3.95,2) to [out=-90, in=180] (4.5,.95) -- (5.5,.95) arc (90:-60:.45 and .45) to [out=-150, in=0] (4.5,-.1);
\node [blue] at (.2,1.25) {$\delta_{-i,j}$};
\node [green] at (4.65,1.3) {$\gamma_{-i,j}$};
\draw (9,1) circle (3mm);
\node at (9,1) {$C$};
\end{tikzpicture}
\caption{Curves $\delta_{-i,j}$ and $\gamma_{-i,j}$. Here, we have cut $\Sigma_g^1$ along $\alpha_1,\dots,\alpha_g$, obtaining a planar surface.}
\label{fig:dij}
\end{figure}

\begin{proposition}
\thlabel{prop:stab}
The group $H[\phi]$ admits a presentation with generators $a_1^2,\dots,a_g^2$, $s$, $t_1,\dots$, $t_{g-1}$ and $\overline{d}_{i,j}$ for all $i,j \in \set{\pm1,\dots,\pm g}$ with $i<j$, and the following relations:
\begin{enumerate}[label=\textit{(A\arabic*)}, ref=(A\arabic*), series=arel]
\item $[a_i^2,a_j^2]=1$ and $[a_i^2,\overline{d}_{j,k}]=1$ for all $i,j,k$;  \label{aiaj}
\item pure braid relations: \label{pb}
\begin{enumerate}
\item $\overline{d}_{r,s}^{-1}*\overline{d}_{i,j}=\overline{d}_{i,j}$ if $r<s<i<j$ or $i<r<s<j$; \label{pba}
\item $\overline{d}_{r,s}^{-1}*\overline{d}_{s,j}=\overline{d}_{r,j}*\overline{d}_{s,j}$ if $r<s<j$; \label{pbb}
\item $\overline{d}_{r,j}^{-1}*\overline{d}_{r,s}=\overline{d}_{s,j}*\overline{d}_{r,s}$ if $r<s<j$;\label{pbc}
\item $[\overline{d}_{i,j},\overline{d}_{r,j}^{-1}*\overline{d}_{r,s}]=1$ if $r<i<s<j$; \label{pbd}
\end{enumerate}
\item $t_it_{i+1}t_i=t_{i+1}t_it_{i+1}$ for all $i$ and $[t_i,t_j]=1$ if $i<j-1$; \label{atiti}
\item $s^2=\overline{d}_{-1,1}a_1^{-2}$ and $t_i^2=\overline{d}_{i,i+1}\overline{d}_{-i-1,-i}$ for all $i$; \label{s2}
\item $[t_i,s]=1$ for all $i \ge 2$; \label{atisi}
\item $st_1st_1=t_1st_1s$; \label{asiti}
\item $[s,a_i^2]=1$ for all $i$, $t_i*a_i^2=a_{i+1}^2$ for all $i$ and $[a_i^2,t_j]=1$ if $j \ne i,i-1$; \label{sai}
\item other relations involving the generators $\overline{d}_{i,j}$: \label{a8}
\begin{enumerate}
\item $s*\overline{d}_{i,j}=\overline{d}_{i,j}$ if $\abs{i},\abs{j}\ge2$ or if $i=-1$ and $j=1$, $s*\overline{d}_{-1,j}=\overline{d}_{1,j}$ if $j \ge 2$, $s*\overline{d}_{i,-1}=\overline{d}_{i,1}$ if $i \le -2$; \label{a8a}
\item $t_k*\overline{d}_{i,j}=\overline{d}_{i,j}$ if $j-1=i=k$ or $j=i+1=-k$ or $\abs{i},\abs{j} \ne k,k+1$; \label{a8b}
\item $t_k*\overline{d}_{k,j}=\overline{d}_{k+1,j}$ if $j \ge k+2$ and $t_k*\overline{d}_{i,-k-1}=\overline{d}_{i,-k}$ if $i \le -k-2$; \label{a8c}
\item $t_k*\overline{d}_{-k-1,k}=\overline{d}_{-k,k+1}$;\label{a8d}
\item $t_k*\overline{d}_{-k-1,k+1}=\overline{d}_{k,k+1}*\overline{d}_{-k,k}$;\label{a8e}
\item $t_k*\overline{d}_{-k-1,j}=\overline{d}_{-k,j}$ if $j>-k$ and $j \ne k,k+1$ and $t_k*\overline{d}_{i,k}=\overline{d}_{i,k+1}$ if $i<k$ and $i \ne -k,-k-1$.\label{a8f}
\end{enumerate}
\end{enumerate}
\end{proposition}

Notice that relations (A4) and (A8) allow us to eliminate all the generators $\overline{d}_{i,j}$ apart from one, for example $\overline{d}_{1,2}$.

\begin{proof}
First of all, we apply the following Tietze moves to the presentation of $H$ given by \cite[Proposition 27]{waj:elem}. We add generators $\overline{d}_{i,j}$ for all $i,j$ and relations (\ref{eq:dij}). The $d_{i,j}$ only appear in relations (P1), (P2), (P4) and (P8), and can be replaced by the $\overline{d}_{i,j}$ using (\ref{eq:dij}).
\begin{enumerate}
\item[\textit{(P1)}] Since all the generators $a_i$ commute, we obtain $[\alpha_i,\overline{d}_{j,k}]=1$.
\item[\textit{(P2)}] The pure braid relations only involve the generators $d_{i,j}$; again, since all $a_i$ commute, we just replace $d_{i,j}$ with $\overline{d}_{i,j}$ for all $i,j$, obtaining (A2).
\item[\textit{(P4)}] The relations become $s^2=\overline{d}_{-1,1}a_1^{-2}$ and $t_i^2=\overline{d}_{i,i+1}\overline{d}_{-i-1,-i}$ for all $i$.
\item[\textit{(P8)}] By relations (P7), $s$ commutes with all the $a_i$, while $t_i*a_i=a_{i+1}$ and $[a_i,t_j]=1$ if $j \ne i,i-1$. Notice that by (P1) and (P4) we also have
\begin{equation}
\label{eq:tai+1}
t_i*a_{i+1}=t_i^2*a_i=(\overline{d}_{i,i+1}\overline{d}_{-i-1,-i})*a_i=a_i.
\end{equation}
Therefore, we just have to replace each $d_{i,j}$ with $\overline{d}_{i,j}$, obtaining (A8).
\end{enumerate}

Now, we can remove generators $d_{i,j}$ and relations \ref{eq:dij} from the presentation of $H$. Notice that all the new generators of $H$ preserve $\phi$, apart from $a_1,\dots,a_g$.

We claim that the subgroup $H[\phi]$ of $H$ is generated by the elements $a_i^2$, $s$, $t_i$ and $\overline{d}_{i,j}$. Indeed, let $w$ be a word in the generators of $H$. By relations (P1), (P7) and (\ref{eq:tai+1}), we can write it as $w=w'a_1^{\epsilon_1}\dots a_g^{\epsilon_g}$, where $w'$ is a word in the generators $s$, $t_i$ and $\overline{d}_{i,j}$. Therefore, $w'$ represents an element of $H[\phi]$, and by \thref{lem:elements}(2) $w$ represents an element of $H[\phi]$ if and only if each $\varepsilon_i$ is even. As a consequence, a Schreier transversal for $H^s$ in $H$ is 
\begin{equation}
\label{eq:u}
U:=\bigg\{u_J:=\prod_{j \in J} a_j \bigg| J \subseteq \set{1,\dots,g}\bigg\},
\end{equation}
ordered lexicographically. 

Now we determine the Schreier generators for $H^s$. Recall that they are of the form $ux\overline{ux}^{-1}$, where $u$ is an element of $U$, $x$ or $x^{-1}$ is a generator of $H$, and $g \mapsto \overline{g}$ is the function $H \rightarrow U$ that sends every element to the unique representative in $U$ of its $H[\phi]$-coset. 

Observe that if $x$ or $x^{-1}$ is equal to $s$ or to $\overline{d}_{i,j}$ for some $i,j$, then it commutes with all the elements of $U$; hence, in this case, we have $\overline{ux}=u$ for every $u \in U$. If $x=a_i^{\pm1}$, we have
\[
\overline{u_Ja_i^{\pm1}}=\begin{cases}
u_{J \cup \set{i}} & \text{if} \ i \notin J,\\
u_{J \setminus \set{i}} & \text{if} \ i \in J.
\end{cases}
\]
If $x=t_i^{\pm1}$ we have:
\[
\overline{u_Jt_i^{\pm1}}=\begin{cases}
u_J & \text{if} \ i,i+1 \in J  \ \ \text{or} \ i,i+1 \notin J;\\
u_{(J\setminus\set{i})\cup\set{i+1}} & \text{if} \ i \in J \ \ \text{and} \ i+1 \notin J;\\
u_{(J\setminus\set{i+1})\cup\set{i}} & \text{if} \ i \notin J \ \ \text{and} \ i+1 \in J.
\end{cases}
\]

In order to streamline the process, we can use directly the relations of $H$ to get rid of redundant generators. For example, by (P7) $s$ commutes with all the elements of $U$ in $H$, hence all the generators $us\overline{us}^{-1}$ coincide with $s$ in $H[\phi]$. More generally, using relations (P1) and (P7), we see that the Schreier generators boil down exactly to those in the statement. Indeed, each generator $g$ of $H$ that preserves $\phi$ gives a family of Schreier generators which are all equal to $g$ itself or to a product of $g$ and some $a_i^2$, and for all $i$ we have
\[
u_Ja_i\overline{u_Ja_i}^{-1}=\begin{cases}
1 & \text{if} \ i \in J, \\
a_i^2 & \text{if} \ i \notin J.
\end{cases}
\]

Finally, the relations for $H[\phi]$ are of the form $uru^{-1}$, where $u$ is an element of $U$ and $r$ is a relation for $H$. Clearly, the only relations that change are the ones involving some $a_i$, that is, (P1), (P4) and (P7).
\begin{enumerate}
\item[\textit{(P1)}] We obtain $[\alpha_i^2,\alpha_j^2]=1$ and $[\alpha_i^2,\overline{d}_{j,k}]=1$ for all $i,j,k$.
\item[\textit{(P4)}] The first relation becomes $s^2=\overline{d}_{-1,1}(a_1^2)^{-1}$.
\item[\textit{(P7)}] We get $[s,a_i^2]=1$ for all $i$, $t_i*a_i^2=a_{i+1}^2$ for all $i$ and $[a_i^2,t_j]=1$ if $j \ne i,i-1$. \qedhere
\end{enumerate}
\end{proof}

\subsection{Orbits of edges}
\label{act}

We can now derive a complete set of generators for $\Modd(\Sigma_g^1)[\phi]$. Consider the action of $\Modd(\Sigma_g^1)[\phi]$ on the edges of $X_g$ starting at $v_0$. Clearly, the orbits of edges of type (i) are disjoint from those of edges of type (ii). By the spin change of coordinates, that there is a unique orbit of edges of type (i). A representative for this orbit is the edge $v_0-b_1(v_0)$, where $b_1=\tau_{\beta_1}$ in the notations of Figure \ref{fig:q}.  

For the edges of type (ii), the situation is akin to the one considered by Wajnryb in his paper on the handlebody group \cite{wajh}. Wajnryb studies the action of $\Modd(H_g)$ on a cell complex $X_g^H$ whose vertices are cut-systems of meridians for $H_g$, and where two vertices $\Braket{\alpha_1,\dots,\alpha_g}$ and $\Braket{\alpha_1',\dots,\alpha_g'}$ are connected by an edge if $\abs{\alpha_1\cap\alpha_1'}=0$ and $\alpha_k=\alpha_k'$ for $k=2,\dots,g$.

Now, consider an edge of type (ii) $v_0-v_1$ in $X_g$. Up to renaming the curves, $v_1$ is of the form $\Braket{\gamma_1,\alpha_2,\dots,\alpha_g}$, where the curve $\gamma_1$ intersects $\alpha_1$ twice (algebraically and geometrically) and is disjoint from $\alpha_2,\dots,\alpha_g$. Cutting $\Sigma_g^1$ along $\alpha_2,\dots,\alpha_g$, we get a torus $T$ with a number of boundary components, that inherits an even spin structure. Capping each boundary component with a disk, we get a closed torus $\overline{T}$, and we can complete $\alpha_1$ to a geometric symplectic basis $\set{\alpha_1,\eta_1}$. Note that there are just two possible choices for $\eta_1$ up to squared Dehn twists along $\alpha_1$. Now, $\gamma_1$ corresponds to a curve $\overline{\gamma_1}$ on $\overline{T}$ whose homology class is $(2k+1)\alpha_1\pm2\eta_1$. Again, up to squared twists along $\alpha_1$, we may suppose that $k$ is either $0$ or $-1$. Hence, we get exactly two possible isotopy classes for $\overline{\gamma_1}$ up to the action of $\Modd(\overline{T})[\phi]$.

If $H_g$ is the handlebody with meridians $\alpha_1,\dots,\alpha_g$, the above argument shows that $\gamma_1$ is given by $\tau_{\eta_1}^{\pm2}(\xi_1)$, where $\eta_1$ is a curve that intersects $\alpha_1$ once and is disjoint from $\alpha_2,\dots,\alpha_g$, and there is an edge
\[
\Braket{\alpha_1,\dots,\alpha_g}-\Braket{\xi_1,\alpha_2,\dots,\alpha_g} 
\]
in Wajnryb's complex $X_g^H$. Up to the action of $\Modd(\Sigma_g^1)[\phi]$, we may suppose that $\eta_1=\beta_1$ (note that we did not require $\eta_1$ to be spin). Wajnryb classifies the possible choices for $\xi_1$ up to the action of $\Modd(H_g)$ (see \cite[page 220]{wajh}), but the classification up to the action of $\Modd(\Sigma_g^1)[\phi]$ is exactly the same: a mapping class that fixes $\alpha_2,\dots,\alpha_g$ and sends $\xi_1$ to another possible choice $\xi_1'$ necessarily extends to $H_g$, and can be made spin by composing it with suitable twists along $\alpha_1,\dots,\alpha_g$. 

\begin{remark}
Cutting $\Sigma_g^1$ along all the curves involved in an edge of type (ii) yields a planar surface with two connected components, one of which contains the hole coming from the boundary component $C$ of $\Sigma_g^1$. Hence, we actually have more edge orbits, according to the component on which $C$ sits. We will see that the faces of $X_g$ containing these extra orbits are superfluous for the simple connectivity, as a consequence of \thref{lem:hared} and \thref{lem:penta}. Hence, we are going to ignore them.
\end{remark}

We can now give a system of representatives for the orbits of edges of type (ii). The system of representatives found by Wajnryb is given by the edges $\Braket{\alpha_j}-\Braket{\gamma_{i,j}}$ for $i,j \in \set{\pm1,\dots,\pm g}$, $i \le 1$, $j+i \ge 1$, $j-i \le g$, where $\gamma_{i,j}$ is the curve in Figure \ref{fig:dij}. Define mapping classes $r_{i,j}=b_ja_jc_{i,j}b_j$, where $c_{i,j}:=t_{\gamma_{i,j}}$. Note that $r_{i,j}$ swaps $\alpha_j$ and $\gamma_{i,j}$, and fixes all other curves $\alpha_k$. Moreover, in homology mod $2$ we have
\[
\big[r_{i,j}(\beta_k)\big]\equiv_2\begin{cases}
[\beta_k] & \text{if} \ k \le \abs{i} \ \text{and} \ i\le -1, \ \ \text{or if} \ k \ge j+1,\\
[\beta_k+\gamma_{i,j}+\beta_j] & \text{if} \ k=i=1, \ \ \text{or if} \ \abs{i}<k<j,\\
[\alpha_j+\beta_j+\gamma_{i,j}] & \text{if} \ k=j.
\end{cases}
\]
In particular, $r_{i,j}$ only changes the spin value in the second case. Set 
\begin{equation}
\label{eq:rij}
\overline{r}_{i,j}:=\begin{cases} a_1^{-1}\dots a_{j-1}^{-1}r_{1,j} & if \text{i=1},\\ a_1^{-2}\dots a_{-i}^{-2}a_{-i+1}^{-1}\dots a_{j-1}^{-1}r_{i,j} & \text{if $i\le-1$}.
\end{cases}
\end{equation}
Clearly, $\overline{r}_{i,j}$ preserves the spin structure, swaps $\alpha_j$ and $\gamma_{i,j}$, and fixes all the other curves $\alpha_k$. Thus, a system of representatives for the orbits of edges of type (ii) under the action of $\Modd(\Sigma_g^1)[\phi]$ is given by $v_0-v^{\pm}$ and $v_0-v_{i,j}^{\pm}$, where 
\[
v^{\pm}:=b_1^{\pm2}(v_0), \quad v_{i,j}^{\pm}:=b_j^{\pm2}\overline{r}_{i,j}(v_0).
\]
Here, the indices $i,j$ are elements of $\set{\pm1,\dots,\pm g}$ such that $i \le 1$, $j+i \ge 1$ and $j-i \le g$.

Our generating set $S$ for $\Modd(\Sigma_g^1)[\phi]$ is then given by the generators of $H[\phi]$ from \thref{prop:stab}, $b$, $b_1^{\pm2}$ and $b_j^{\pm2}\overline{r}_{i,j}$ for $i,j$ as above. 

\begin{remark}
\thlabel{rmk:ivsii}
We only used the connectivity of the spin cut-system complex to determine a generating set for $\Modd(\Sigma_g^1)[\phi]$. By \thref{prop:conn}, edges of type (ii) are not necessary for the connectivity, hence generators $b_1^{\pm2}$ and $b_j^{\pm}\overline{r}_{i,j}$ are superfluous. We will keep them for now as they are needed to write the relations. Notice that for now we should also treat $b_1$, $b_1^2$ and $b_1^{-2}$ as independent generators. From \ref{backtracking} and \ref{triangles} we will obtain the obvious relations between them.
\end{remark}

\subsection{Backtracking}
\label{bt}

We now write explicitly relations \ref{backtracking}. Recall that these are $h$-products representing \emph{backtrackings}, i.e. loops of the form $v_0-r(v_0)-v_0$, for every generator $r\notin H[\phi]$. 

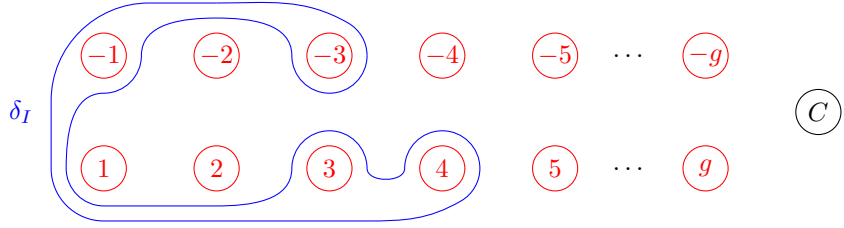
\begin{figure}
\centering
\begin{tikzpicture}
\draw [red] (1.5,.5) circle (3mm);
\draw [red] (1.5,2) circle (3mm);
\draw [red] (3,.5) circle (3mm);
\draw [red] (3,2) circle (3mm);
\draw [red] (4.5,.5) circle (3mm);
\draw [red] (4.5,2) circle (3mm);
\draw [red] (6,.5) circle (3mm);
\draw [red] (6,2) circle (3mm);
\draw [red] (7.5,.5) circle (3mm);
\draw [red] (7.5,2) circle (3mm);
\node at (8.5,.5) {$\dots$};
\node at (8.5,2) {$\dots$};
\draw [red] (9.5,.5) circle (3mm);
\draw [red] (9.5,2) circle (3mm);
\node [red] at (1.5,.5) {$1$};
\node [red] at (3,.5) {$2$};
\node [red] at (4.5,.5) {$3$};
\node [red] at (6,.5) {$4$};
\node [red] at (7.5,.5) {$5$};
\node [red] at (9.5,.5) {$g$};
\node [red] at (1.5,2) {$-1$};
\node [red] at (3,2) {$-2$};
\node [red] at (4.5,2) {$-3$};
\node [red] at (6,2) {$-4$};
\node [red] at (7.5,2) {$-5$};
\node [red] at (9.5,2) {$-g$};
\draw [blue] (5,2) arc(0:-180:.5 and .5) to [out=90, in=0] (3,2.5) to [out=180, in=90] (2,2) arc(0:-90:.5 and .5) to [out=180, in=90] (1,.5) arc (180:270:.5 and .5) -- (3,0) to [out=0, in=-90] (4,.5) arc (180:0:.5 and .5) to [out=-90, in=-90] (5.5,.5) arc (180:-60:.5 and .5) to [out=-150, in=0] (5,-.2) -- (1.5,-.2) arc (270:180:.7 and .7) -- (.8,1.4) to [out=90, in=180] (2.1,2.7) -- (3,2.7);
\draw [blue] (5,2) arc(0:60:.5 and .5) to [out=150,in=0] (3,2.7);
\node [blue] at (.4,1.25) {$\delta_{I}$};
\draw (11,1.25) circle (3mm);
\node at (11,1.25) {$C$};
\end{tikzpicture}
\caption{The curve $\delta_{I}$, for $I=\set{-3,-1,3,4}$.}
\label{fig:dijlant}
\end{figure}

We first introduce some additional notations. Cut $\Sigma_g^1$ along the curves $\alpha_0,\dots,\alpha_g$, obtaining a planar surface as in Figure \ref{fig:dij}. Given a subset $I$ of $\set{\pm1,\dots,\pm g}$, let $\delta_I$ be the curve that encircles the red holes corresponding to the elements of $I$, where the upper holes are indexed by negative integers (see Figure \ref{fig:diijj}). In particular, $\delta_{i,j}=\delta_{\set{i,j}}$, and $\gamma_{i,j}=\delta_{\set{i,i+1,\dots,j}}$. Finally, set $d_I:=t_{\delta_I}$. We define
\begin{equation}
\label{eq:dI}
\overline{d}_{\set{i_1,\dots,i_n}}:=d_{\set{i_1,\dots,i_n}}(a_{i_1}\dots a_{i_n})^{-1}.
\end{equation}
It can be shown that
\begin{equation}
\label{dset}
\overline{d}_{\set{i_1,\dots,i_n}}=(\overline{d}_{i_1,i_2}\overline{d}_{i_1,i_3}\dots \overline{d}_{i_1,i_n}\overline{d}_{i_2,i_3}\dots \overline{d}_{i_2,i_n}\dots \overline{d}_{i_{n-1},i_n}).
\end{equation}

\begin{remark}
\thlabel{rem:dI}
The expansion of (\ref{dset}) involves a number of lantern relations. The idea is the following. Consider the lantern specified by the curves $\delta_{i_{n-1},i_n}$ and $\delta_{I\setminus \set{i_n}}$, i.e.
\[
d_{i_{n-1},i_n}d_{I\setminus \set{i_n}}d_{I\setminus \set{i_{n-1}}}=a_{\abs{i_{n-1}}}a_{\abs{i_n}}d_{I}d_{I\setminus \set{i_{n-1},i_n}}.
\]
This allows us to write $d_I$ in terms of mapping classes $d_{I'}$, where $I'$ has one or two elements less than $I$ (see Figure \ref{fig:dijlant}). Notice that if we denote by $\ell(n)$ the number of lanterns needed to write $d_I$ as a product of elements $d_{i,j}$, we can write a recurrence relation and see that $\ell(n)=(n-1)(n-2)/2$. 
\end{remark}

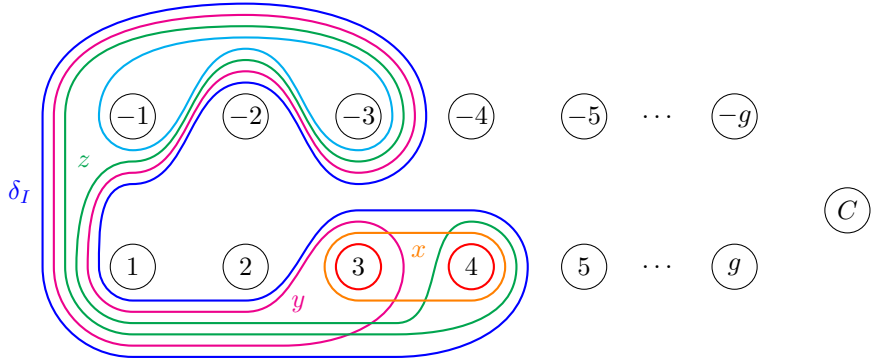
\begin{figure}
\centering
\begin{tikzpicture}
\draw (1.5,.5) circle (3mm);
\draw (1.5,2.5) circle (3mm);
\draw (3,.5) circle (3mm);
\draw (3,2.5) circle (3mm);
\draw [red, thick] (4.5,.5) circle (3mm);
\draw (4.5,2.5) circle (3mm);
\draw [red, thick] (6,.5) circle (3mm);
\draw (6,2.5) circle (3mm);
\draw (7.5,.5) circle (3mm);
\draw (7.5,2.5) circle (3mm);
\node at (8.5,.5) {$\dots$};
\node at (8.5,2.5) {$\dots$};
\draw (9.5,.5) circle (3mm);
\draw (9.5,2.5) circle (3mm);
\node at (1.5,.5) {$1$};
\node at (3,.5) {$2$};
\node at (4.5,.5) {$3$};
\node at (6,.5) {$4$};
\node at (7.5,.5) {$5$};
\node at (9.5,.5) {$g$};
\node at (1.5,2.5) {$-1$};
\node at (3,2.5) {$-2$};
\node at (4.5,2.5) {$-3$};
\node at (6,2.5) {$-4$};
\node at (7.5,2.5) {$-5$};
\node at (9.5,2.5) {$-g$};
\draw [blue, thick] (5.4,2.5) arc(0:-90:.9 and .9) to [out=180, in=0] (3,2.95) to [out=180, in=0] (1.5,1.6) to [out=180, in=90] (1.05,.5) arc (180:270:.45 and .45) -- (3,.05) to [out=0, in=180] (4.5,1.25) -- (6,1.25) arc(90:0:.75 and .75) to [out=-90, in=0] (5,-.7) -- (1.5,-.7) arc (270:180:1.2 and 1.2) -- (.3,2.5) to [out=90, in=180] (3,4) to [out=0, in=90] (5.4,2.5);
\draw [magenta, thick] (5.25,2.5) arc(0:-90:.75 and .75) to [out=180, in=0] (3,3.1) to [out=180, in=0] (1.5,1.75) to [out=180, in=90] (.9,.5) arc (180:270:.6 and .6) -- (3,-.1) to [out=0, in=180] (4.5,1.1) arc(90:0:.6 and .6) to [out=-90, in=0] (3.5,-.55) -- (1.5,-.55) arc (270:180:1.05 and 1.05) -- (.45,2.5) to [out=90, in=180] (3,3.85) to [out=0, in=90] (5.25,2.5);
\draw [Green, thick] (5.1,2.5) arc(0:-90:.6 and .6) to [out=180, in=0] (3,3.25) to [out=180, in=0] (1.5,1.9) to [out=180, in=90] (.75,.5) arc (180:270:.75 and .75) -- (5,-.25) to [out=0, in=180] (6,1.1) arc(90:0:.6 and .6) to [out=-90, in=0] (5,-.4) -- (1.5,-.4) arc (270:180:.9 and .9) -- (.6,2.5) to [out=90, in=180] (3,3.7) to [out=0, in=90] (5.1,2.5);
\draw [orange, thick] (4.5,.95) arc (90:270:.45 and .45) -- (6,.05) arc (-90:90:.45 and .45) -- (4.5,.95);
\draw [cyan, thick] (1.05,2.5) arc(180:270:.45 and .45) to [out=0, in=180] (3,3.4) to [out=0, in=180] (4.5,2.05) arc(-90:0:.45 and .45) to [out=90, in=0] (3,3.55) to [out=180, in=90] (1.05,2.5);
\node [orange] at (5.3,.7) {$x$};
\node [magenta] at (3.7,0) {$y$};
\node [Green] at (.85,1.9) {$z$};
\node [blue] at (0,1.5) {$\delta_I$};
\draw (11,1.25) circle (3mm);
\node at (11,1.25) {$C$};
\end{tikzpicture}
\caption{The lantern $xyz=a_3a_4d_{-3,-1}d_{I}$ can be used to define $d_{I}$ for $I:=\set{-3,-1,3,4}$. Here $x=d_{3,4}$, $y=d_{\set{-3,-1,3}}$ and $z=d_{\set{-3,-1,4}}$.}
\label{fig:diijj}
\end{figure}

Moreover, set $k_j:=t_j\overline{d}_{j,j+1}^{-1}$ for all $j$, and define $s_1:=s$ and
\[
s_j:=(k_{j-1}k_{j-2}\dots k_1)*s_1.
\]
It can be shown that $s_j=b_ja_j^2b_j$ (see (\ref{bi})). 

\begin{remark}
\thlabel{rmk:bt}
Backtracking on an edge in the orbit of $v_0-v_{i,j}^+$ results in an edge in the orbit of $v_0-v_{i,j}^-$. This can be seen by assigning orientations to the curves intersecting twice. Hence, it suffices to consider the $h$-products relative to backtrackings where the first edge is $v_0-v_{i,j}^+$.
\end{remark}

\begin{enumerate}[resume*=arel]
\item \textit{We have $b_1a_1^2b_1=s_1$, $b_1^{+2}b_1^{-2}=1$,}
\[
b_j^{+2}\overline{r}_{1,j}\overline{d}_{\set{1,\dots,j}}a_j^2 b_j^{-2}\overline{r}_{1,j}=a_1^{-2}\dots a_{j-1}^{-2}a_j^2 s_j\overline{d}_{\set{1,\dots,j}}s_j
\]
\textit{for every} $j \ge 2$ \textit{and}
\[
b_j^{+2}\overline{r}_{i,j}\overline{d}_{\set{i,\dots,j}}a_j^2 b_j^{-2}\overline{r}_{i,j}=a_1^{-4}\dots a_{-i}^{-4}a_{-i+1}^{-2}\dots a_{j-1}^{-2}a_j^2 s_j\overline{d}_{\set{i,\dots,j}}s_j
\]
\textit{for every} $i,j\in \set{\pm1,\dots,\pm g}$ \textit{with $i \le -1$, $j+i \ge 1$ and $j-i \le g$.}\label{backtracking}
\end{enumerate}

\begin{proof}[Proof of \ref{backtracking}.]
The first two relations are clear. The other relations follow from braid relations $T_1$ (in the whole $\Modd(\Sigma_g^1)$). We do the case $i=1$; the other is similar. We underline the places where a relation $T_1$ is applied:
\begin{align*}
&b_j^{\pm2}\overline{r}_{1,j}\overline{d}_{\set{1,\dots,j}}a_j^2 b_j^{\mp2}\overline{r}_{1,j}=(a_1\dots a_{j-1})^{-3}b_j^{\pm2}r_{1,j} c_{1,j}a_{j}b_j^{\mp2} r_{i,j}=\\
&\quad=(a_1\dots a_{j-1})^{-3}b_jb_j^{\pm2}a_j\underline{c_{1,j}b_j c_{1,j}}a_{j}b_j^{\mp2}b_ja_jc_{1,j}b_j=\\
&\quad=(a_1\dots a_{j-1})^{-3}b_j\underline{b_j^{\pm2}a_jb_j}c_{1,j}\underline{b_j a_{j}b_j^{\mp2}}b_ja_jc_{1,j}b_j=\\
&\quad=(a_1\dots a_{j-1})^{-3}\underline{b_ja_jb_j}a_j^{\pm2}c_{1,j}a_j^{\mp2} \underline{b_ja_jb_j}a_jc_{1,j}b_j=\\
&\quad=(a_1\dots a_{j-1})^{-3}a_jb_ja_ja_j\underline{c_{1,j}b_jc_{1,j}}a_ja_jb_j=a_1^{-2}\dots a_{j-1}^{-2}a_j^2s_j\overline{d}_{\set{1,\dots,j}}s_j.\qedhere
\end{align*}
\end{proof}

With the relations of \thref{prop:conn} and \ref{backtracking}, we can already obtain the following small set of generators, which may be of interest.

\begin{corollary}
The even spin mapping class group $\Modd(\Sigma_g)[\phi]$ is generated by $a_1^2,b_1,t_1,\overline{d}_{1,2}$ and $u:=t_1\dots t_{g-1}$.
\end{corollary}

\begin{proof}
Call $G$ the subgroup of $\Modd(\Sigma_g)[\phi]$ generated by the elements in the statement. By \ref{backtracking}, $s_1=b_1a_1^2b_1$ is contained in $G$. Thus, by \ref{atiti}, \ref{s2}, \ref{sai} and \ref{a8}, all the $a_i^2$, the $\overline{d}_{i,j}$ and the $t_i$ are also contained in $G$. As noted in \thref{rmk:ivsii}, the generators $b_1^{\pm2}$ and $b_j^{\pm2}\overline{r}_{i,j}$ are superfluous, so $G=\Modd(\Sigma_g)[\phi]$. \qedhere
\end{proof}

\subsection{Different writings of the same edge}
\label{dw}

Relations \ref{diffwr} come from different ways of associating an $h$-product to the same edge. In order to write down explicitly such relations, we must find a generating set for the stabilizer of each class of edges.

\begin{lemma}[{\cite[Lemma 29]{waj:elem}}]
\thlabel{lem:stab1}
The stabilizers of the edges $v_0-b_1(v_0)$ and $v_0-b_1^{\pm2}(v_0)$ are both generated by $a_1^2s$, $t_1st_1$, $a_2^2$, $\overline{d}_{2,3}$, $\overline{d}_{-2,2}$, $\overline{d}_{-1,1}\overline{d}_{-1,2}\overline{d}_{1,2}a_1^2$ and $t_2,\dots,t_{g-1}$.
\end{lemma}

\begin{lemma}
The stabilizer $H_{i,j}^{\pm}$ of the edge $v_0-v_{i,j}^{\pm}$ is generated by the following elements:
\begin{itemize}
\item $a_1^2,\dots,a_{j-1}^2,a_{j+1}^2,\dots,a_g^2$;
\item $t_k$ for $k >j$ or $1 \le k <j-1$ with $k \ne -i$;
\item $s_k$ for $k>j$ or $k \le -i$;
\item $\overline{d}_{k,m}$ for $k,m \in \set{i,i+1,\dots,j-1}$ or $k,m \notin \set{-j,i,i+1,\dots,j}$;
\item $a_j^2s_j \overline{d}_{\set{i,\dots,j}}$ for $(i,j)^+$, $a_j^2\overline{d}_{\set{i,\dots,j}}s_j$ for $(i,j)^-$.
\end{itemize}
\end{lemma}

\begin{proof}
The proof is exactly the same as that of \cite[Lemma 24]{wajh}. Note that some of Wajnryb's stabilizers have an additional generator $z_j$, which swaps the two connected components $S_1$ and $S_2$. In our case the presence of the boundary component $C$ prevents that from happening.  \qedhere
\end{proof}

We get the following set of relations. Here and elsewhere, $\overline{d}_I$ will be equal to $1$ if $I$ contains a single element. 
\begin{enumerate}[resume*=arel]
\item \label{diffwr} \begin{enumerate}[ref=\alph*)]
\item \textit{$b_1$ commutes with $a_1^2s$, $t_1st_1$, $a_2^2$, $\overline{d}_{2,3}$, $\overline{d}_{-1,1}\overline{d}_{-1,2}\overline{d}_{1,2}a_1^2$, $\overline{d}_{-2,2}$, $t_2,\dots,t_{g-1}$;} \label{dwb1}
\item \textit{$b_1^{\pm2}$ commutes with $a_1^2s$, $t_1st_1$, $a_2^2$, $\overline{d}_{2,3}$, $\overline{d}_{-1,1}\overline{d}_{-1,2}\overline{d}_{1,2}a_1^2$, $\overline{d}_{-2,2}$, $t_2,\dots,t_{g-1}$;}\label{dwb2}
\item \textit{$b_j^{\pm2}\overline{r}_{i,j}$ commutes with:} \label{dwbj} \begin{itemize}
\item \textit{$a_k^2$ for $k \ne j$;}
\item \textit{$t_k$ for $k >j$ or $1 \le k <j-1$ with $k \ne -i$;}
\item \textit{$s_k$ for $k>j$ or $k \le -i$;}
\item \textit{$\overline{d}_{k,m}$ for $k,m \in \set{i,\dots,j-1}$ or $k,m \notin \set{-j,i,i+1,\dots,j}$;}
\end{itemize}
\item \textit{$[b_j^{-2}\overline{r}_{i,j},a_j^2\overline{d}_{\set{i,\dots,j}}s_j]=1$, while} \label{dwb*}
\begin{align*}
\big(b_j^{2}\overline{r}_{1,j}\big)^{-1}*\big(a_j^2s_j\overline{d}_{\set{1,\dots,j}}\big)&=a_j^2 \overline{d}_{\set{1,\dots,j}}s_j^{-1} a_j^{-2}\overline{d}_{\set{1,\dots,j}}^{-2}\overline{d}_{\set{1,\dots,j-1}}\cdot\\
&\qquad\quad \cdot\big((t_{j-1}\dots t_{1})*\overline{d}_{\set{-1,\dots,j}}\big),
\end{align*}
\textit{and if $i<0$}
\begin{align*}
\big(b_j^{2}\overline{r}_{i,j}\big)^{-1}*\big(a_j^2s_j\overline{d}_{\set{i,\dots,j}}\big)&=a_j^2 \overline{d}_{\set{i,\dots,j}}s_j^{-1} a_j^{-2}\overline{d}_{\set{i,\dots,j}}^{-2}\overline{d}_{\set{i,\dots,j-1}}\cdot\\
&\quad \cdot\big((t_{j-1}\dots t_{-i+1})*\overline{d}_{\set{i-1,\dots,j}}\big).
\end{align*}
\end{enumerate}
\end{enumerate}

\begin{proof}[Proof of \ref{diffwr}.]
Most of these relations follow easily from the definitions. For the last point, observe that in the negative case we have
\[
s_j^{-1}a_j^{-1}c_{i,j}^{-1}r_{i,j}^{-1}b_j^{2} a_jc_{i,j}s_jb_j^{-2}r_{i,j}=1 
\]
by braid relations $T_1$, while in the positive case a 3-chain is involved:
\begin{align*}
&a_j^2c_{i,j}^2s_ja_j^{-1}c_{i,j}^{-1}r_{i,j}^{-1}b_j^{-2}s_ja_jc_{i,j}b_j^{2}r_{i,j}=(a_jb_jc_{i,j})^4=\\
&\qquad\qquad =c_{i,j-1}\big((t_{\abs{i}+1}\dots t_{j-1})*c_{i-1,j}\big).\qedhere
\end{align*}
\end{proof}

\subsection{Faces}
\label{f}

The last set of relations comes from the 2-cells in our complex. We are going to establish a list of closed edge paths $\mathbf{p}_i$, such that every closed edge path is a sum of paths conjugate to some $\mathbf{p}_i$, i.e. of the form $\mathbf{q}_1\mathbf{q}_2\mathbf{q}_1^{-1}$, where $\mathbf{q}_1$ starts at $v_0$ and $\mathbf{q}_2$ is the image of some $\mathbf{p}_i$ under the action of $\Modd(\Sigma_g^1)[\phi]$. The relations will be the $h$-products associated to the paths $\mathbf{p}_i$. For the proof that these relations hold in $\Modd(\Sigma_g^1)[\phi]$, see Subsection \ref{reldij}.

\emph{Triangles.} We apply Harer's reduction process \cite{har0}. We explain this method in detail for triangles involving an edge of type $(i,j)^+$; the negative case is symmetric. Let $\mathbf{p}$ be a triangle $v_0-v_1-v_2-v_0$. Cut $\Sigma_g^1$ along the $g-1$ curves in common, obtaining a $2g-1$-holed torus $T$ with three curves $\alpha,\beta,\gamma$, where $\alpha \in v_0$, $\gamma \in v_1$ and $\beta \in v_2$, and the edge of type (ii) is $\Braket{\alpha}-\Braket{\gamma}$. Call $\widehat{T}$ the closed torus obtained by capping all boundary components with disks. Then the universal cover of $\widehat{T}$ has a fundamental region which is a square with edges along $\alpha$ and $\beta$, cut into $4$ parts by $\gamma$. Orient the three curves in such a way that $(\alpha,\beta)=1$, $(\beta,\gamma)=-1$ and $(\alpha,\gamma)=2$, and name the $4$ regions as in Figure \ref{fig:fundreg}(a). Note that changing the orientations of all three curves switches the roles of $F_0$ and $F_3$ and of $F_1$ and $F_2$.

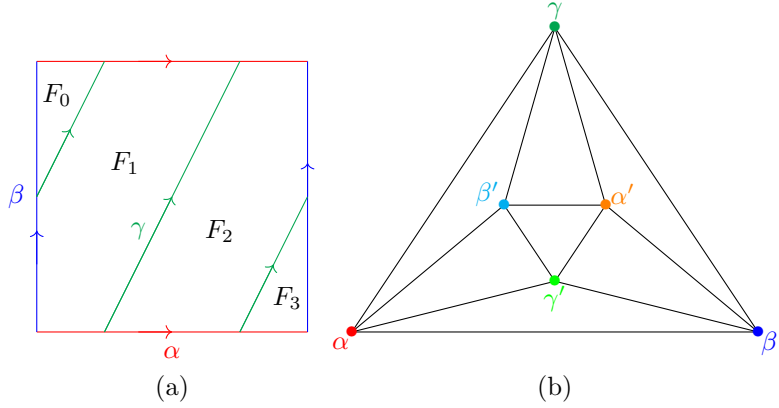
\begin{figure}
\centering
\begin{tikzpicture}[scale=0.75]
\draw [red] (0,0) -- (4.8,0);
\draw [red] (0,4.8) -- (4.8,4.8);
\draw [blue] (0,0) -- (0,4.8);
\draw [blue] (4.8,0) -- (4.8,4.8);
\draw [Green] (0,2.4) -- (1.2,4.8);
\draw [Green] (1.2,0) -- (3.6,4.8);
\draw [Green] (3.6,0) -- (4.8,2.4);
\node [red] at (2.4,-.36) {$\alpha$};
\node [blue] at (-.36,2.4) {$\beta$};
\node [Green] at (1.8,1.8) {$\gamma$};
\draw [->, red] (1.8,0) -- (2.4,0);
\draw [->, red] (1.8,4.8) -- (2.4,4.8);
\draw [->, blue] (0,1.2) -- (0,1.8);
\draw [->, blue] (4.8,2.4) -- (4.8,3);
\draw [->, Green] (0,2.4) -- (.6,3.6);
\draw [->, Green] (1.2,0) -- (2.4,2.4);
\draw [->, Green] (3.6,0) -- (4.2,1.2);
\node at (.36,4.2) {$F_0$};
\node at (1.56,3) {$F_1$};
\node at (3.24,1.8) {$F_2$};
\node at (4.44,.6) {$F_3$};
\node at (2.4,-1) {(a)};
\end{tikzpicture}
\begin{tikzpicture}[scale=0.75]
\draw (0,0) -- (7.2,0) -- (3.6,5.4) -- (0,0);
\draw (0,0) -- (3.6,.9) -- (2.7,2.25) -- (0,0);
\draw (7.2,0) -- (4.5,2.25) -- (3.6,.9) -- (7.2,0);
\draw (3.6,5.4) -- (4.5,2.25) -- (2.7,2.25) -- (3.6,5.4);
\node [red] at (-.2,-.2) {$\alpha$};
\node [orange] at (4.8,2.45) {$\alpha'$};
\node [blue] at (7.4,-.2) {$\beta$};
\node [cyan] at (2.4,2.45) {$\beta'$};
\node [Green] at (3.6,5.7) {$\gamma$};
\node [green] at (3.6,.6) {$\gamma'$};
\node [red] at (0,0) {$\bullet$};
\node [orange] at (4.5,2.25) {$\bullet$};
\node [blue] at (7.2,0) {$\bullet$};
\node [cyan] at (2.7,2.25) {$\bullet$};
\node [Green] at (3.6,5.4) {$\bullet$};
\node [green] at (3.6,.9) {$\bullet$};
\node at (3.6,-1) {(b)};
\end{tikzpicture}
\caption{(a) A fundamental region for the universal cover of $\widehat{T}$, with the curves $\alpha$, $\beta$ and $\gamma$. \\ (b) Octahedron associated to a push-off of curves $\alpha$, $\beta$ and $\gamma$.}
\label{fig:fundreg}
\end{figure}

Lifting the $2g-1$ boundary components of $T$ to the universal cover of $\widehat{T}$, we get a certain number $\ell_i$ of holes in each region $F_i$. We are going to push off slightly each curve $\alpha$, $\beta$ and $\gamma$, so that the triangles formed by the original curves and their push-offs have new values of $\ell_i$. Note that all push-offs are still spin, since homologically they differ from the original curves only by some spin boundary components. 

\begin{lemma}
\thlabel{lem:hared}
Every triangle is a sum of paths conjugated to triangles with $\ell_0=0$ and $\ell_3\le 1$, where the hole corresponding to the boundary component of $\Sigma_g^1$ lies in $F_1$. 
\end{lemma}

\begin{proof}
First of all, we prove that every triangle is a sum of paths conjugated to triangles with $\ell_3 \le 1$. Indeed, if a triangle $\Braket{\alpha}-\Braket{\gamma}-\Braket{\beta}-\Braket{\alpha}$ has $\ell_3 \ge 2$, we consider push-offs $\alpha'$, $\beta'$ and $\gamma'$ as in Figure \ref{fig:hared}a). Now, the curves fit into 8 triangles, which form an octahedron as in Figure \ref{fig:fundreg}(b). Here, all triangles have up to $\ell_3-1$ holes in region $F_3$, apart from the original one, and we can iterate the process until every triangle has up to 1 hole in region $F_3$. Note that $\ell_0$ stays the same throughout, so we can switch $F_0$ and $F_3$ and repeat the process until every triangle has up to 1 hole in both regions. 

We only have to deal with the case $\ell_0=\ell_3=1$. Consider push-offs $\alpha'$, $\beta'$ and $\gamma'$ as in Figure \ref{fig:hared}b). All the faces of the octahedron except the original triangle have either $F_0$ or $F_3$ without holes, and if one of them contains two holes we can do the same process as before. 

Up to changing the orientations of the curves, we are done. Notice that if the hole corresponding to the boundary components of $\Sigma_g^1$ lies in $F_3$ (or $F_2$), we can slide $\gamma$ on the whole $F_0$, so that in the new configuration the hole has moved to $F_2$ ($F_1$), and then repeat the above process, which does not remove any hole from $F_2$ ($F_1$). \qedhere
\end{proof}

\begin{figure}
\centering
\begin{tikzpicture}
\draw [red] (-.36,0) -- (4.8,0);
\draw [red] (-.36,4.8) -- (4.8,4.8);
\draw [orange] (-.36,.36) -- (4.8,.36);
\draw [orange] (-.36,5.16) -- (4.8,5.16);
\draw [blue] (0,0) -- (0,5.16);
\draw [blue] (4.8,0) -- (4.8,5.16);
\draw [cyan] (-.36,0) -- (-.36,5.16);
\draw [cyan] (4.44,0) -- (4.44,5.16);
\draw [Green] (-.36,1.68) -- (1.38,5.16);
\draw [Green] (1.2,0) -- (3.78,5.16);
\draw [Green] (3.6,0) -- (4.8,2.4);
\draw [green] (-.36,.96) -- (1.74,5.16);
\draw [green] (1.56,0) -- (4.14,5.16);
\draw [green] (3.96,0) -- (4.8,1.68);
\node [red] at (2.4,-.36) {$\alpha$};
\node [orange] at (2.4,5.52) {$\alpha'$};
\node [blue] at (5.1,2.4) {$\beta$};
\node [cyan] at (-.72,2.4) {$\beta'$};
\node [Green] at (1.8,1.8) {$\gamma$};
\node [green] at (2.52,1.08) {$\gamma'$};
\draw [->, red] (1.8,0) -- (2.4,0);
\draw [->, red] (1.8,4.8) -- (2.4,4.8);
\draw [->, orange] (1.8,.36) -- (2.4,.36);
\draw [->, orange] (1.8,5.16) -- (2.4,5.16);
\draw [->, blue] (0,2.4) -- (0,3);
\draw [->, blue] (4.8,2.4) -- (4.8,3);
\draw [->, cyan] (-.36,2.4) -- (-.36,3);
\draw [->, cyan] (4.44,2.4) -- (4.44,3);
\draw [->, Green] (0,2.4) -- (.66,3.72);
\draw [->, Green] (1.2,0) -- (2.46,2.52);
\draw [->, Green] (3.6,0) -- (4.02,.84);
\draw [->, green] (.36,2.4) -- (.96,3.6);
\draw [->, green] (1.56,0) -- (2.76,2.4);
\draw [->, green] (3.96,0) -- (4.32,.72);
\node at (.36,4.2) {$\bullet$};
\node at (1.56,3) {$\bullet$};
\node at (3.24,1.8) {$\bullet$};
\node at (4.32,.456) {$\bullet$};
\node at (4.62,.7) {$\circ$};
\node at (3.864,.156) {$\circ$};
\node at (-.18,.7) {$\circ$};
\node at (3.864,4.956) {$\circ$};
\node at (2.22,-1) {a)};
\end{tikzpicture} \qquad 
\begin{tikzpicture}
\draw [red] (-.36,0) -- (4.8,0);
\draw [red] (-.36,4.8) -- (4.8,4.8);
\draw [orange] (-.36,-.36) -- (4.8,-.36);
\draw [orange] (-.36,4.44) -- (4.8,4.44);
\draw [blue] (0,-.36) -- (0,4.8);
\draw [blue] (4.8,-.36) -- (4.8,4.8);
\draw [cyan] (-.36,-.36) -- (-.36,4.8);
\draw [cyan] (4.44,-.36) -- (4.44,4.8);
\draw [Green] (-.36,1.68) -- (1.2,4.8);
\draw [Green] (1.02,-.36) -- (3.6,4.8);
\draw [Green] (3.42,-.36) -- (4.8,2.4);
\draw [green] (-.36,.96) -- (1.56,4.8);
\draw [green] (1.38,-.36) -- (3.96,4.8);
\draw [green] (3.78,-.36) -- (4.8,1.68);
\node [red] at (2.4,5.1) {$\alpha$};
\node [orange] at (2.4,-.72) {$\alpha'$};
\node [blue] at (5.1,2.4) {$\beta$};
\node [cyan] at (-.72,2.4) {$\beta'$};
\node [Green] at (1.8,1.8) {$\gamma$};
\node [green] at (2.52,1.08) {$\gamma'$};
\draw [->, red] (1.8,0) -- (2.4,0);
\draw [->, red] (1.8,4.8) -- (2.4,4.8);
\draw [->, orange] (1.8,-.36) -- (2.4,-.36);
\draw [->, orange] (1.8,4.44) -- (2.4,4.44);
\draw [->, blue] (0,2.4) -- (0,3);
\draw [->, blue] (4.8,2.4) -- (4.8,3);
\draw [->, cyan] (-.36,2.4) -- (-.36,3);
\draw [->, cyan] (4.44,2.4) -- (4.44,3);
\draw [->, Green] (0,2.4) -- (.66,3.72);
\draw [->, Green] (1.2,0) -- (2.46,2.52);
\draw [->, Green] (3.6,0) -- (3.92,.64);
\draw [->, green] (.36,2.4) -- (.96,3.6);
\draw [->, green] (1.56,0) -- (2.76,2.4);
\draw [->, green] (3.96,0) -- (4.22,.52);
\node at (1.56,3) {$\bullet$};
\node at (3.24,1.8) {$\bullet$};
\node at (4.62,1.68) {$\circ$};
\node at (.5,-.18) {$\circ$};
\node at (-.18,1.68) {$\circ$};
\node at (.5,4.62) {$\circ$};
\node at (2.22,-1.36) {b)};
\end{tikzpicture}
\caption{Configurations of push-offs of curves $\alpha$, $\beta$ and $\gamma$ in the proof of \thref{lem:hared}. Circles indicate single holes, while dots indicate all the remaining holes in a certain region.}
\label{fig:hared}
\end{figure}
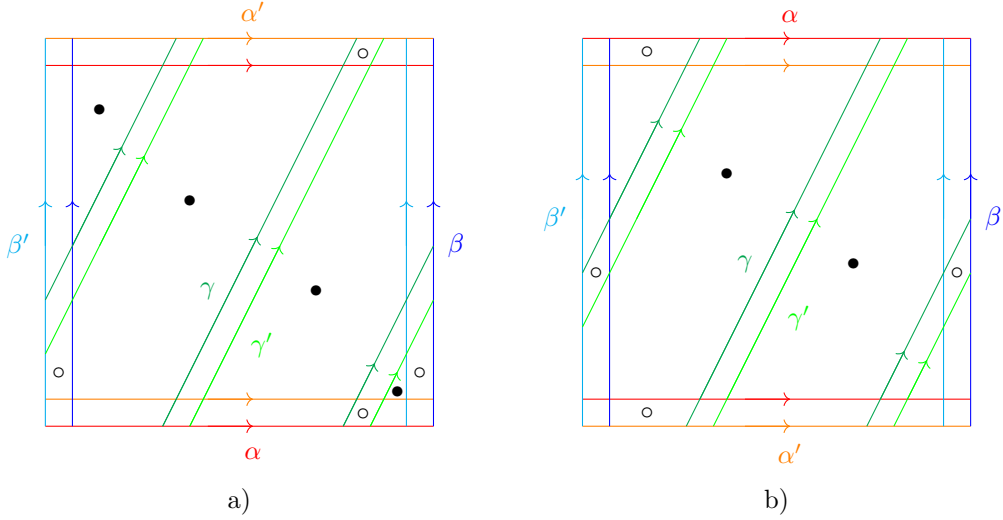

Finally, by \thref{rmk:bt}, every triangle with an edge of type $(i,j)^{\pm}$ is equivalent modulo backtracking and different writings to a triangle with an edge of type $(i,j)^{\mp}$. Hence, the relations corresponding to triangles with an edge of type $(i,j)^-$ are sufficient.

We now list all possible such triangles $\Braket{\alpha}-\Braket{\gamma}-\Braket{\beta}-\Braket{\alpha}$ starting from $v_0$ with $\ell_0=0$ and $\ell_3\le1$ up to the action of $\Modd(\Sigma_g^1)[\phi]$, and write the associated relations. Up to a suitable element of $H[\phi]$, we may assume that $\alpha=\alpha_j$ and $\gamma=b_j^{-2}\overline{r}_{i,j}(\alpha_j)$, where if $j=1$ we set $\overline{r}_{i,j}:=1$. Cutting along the curves of $v_0$, we get a disk with $2g$ holes, cut into two connected components by $\gamma$. These contain “single" holes and “paired" holes. Up to the action of the stabilizer of $\Braket{\alpha}-\Braket{\gamma}$, $\beta$ can be chosen as the curve which runs across the $j$-th handle, twisting along $\alpha_j$ once, and then may encircle one hole if $\ell_3=1$. According to whether this is a “single" or a “paired" hole, and to the connected component where it belongs, we get up to $5$ possibilities. These are listed in Figure \ref{fig:basictri-}. 

\begin{figure}
\centering
\begin{tikzpicture}[scale=0.95]
\node at (0,.75) {e)};
\draw (1.5,0) circle (2.5mm);
\draw (1.5,2) circle (2.5mm);
\node at (1.5,-1.5) {$1$};
\node at (2.5,0) {$\dots$};
\node at (2.5,2) {$\dots$};
\draw (3.5,0) circle (2.5mm);
\draw (3.5,2) circle (2.5mm);
\node at (3.5,-1.5) {$-i$};
\draw (5,0) circle (2.5mm);
\draw (5,2) circle (2.5mm);
\node at (5,-1.5) {$-i+1$};
\node at (6,0) {$\dots$};
\node at (6,2) {$\dots$};
\draw (7,0) circle (2.5mm);
\draw (7,2) circle (2.5mm);
\node at (7,-1.5) {$j-1$};
\draw (8.5,0) circle (2.5mm);
\draw (8.5,2) circle (2.5mm);
\node at (8.5,-1.5) {$j$};
\draw (10,0) circle (2.5mm);
\draw (10,2) circle (2.5mm);
\node at (10,-1.5) {$j+1$};
\node at (11,0) {$\dots$};
\node at (11,2) {$\dots$};
\draw (12,0) circle (2.5mm);
\draw (12,2) circle (2.5mm);
\node at (12,-1.5) {$g$};
\draw [thick] (13.5,1) circle (2mm);
\draw [blue, thick] (8.5,-.25) to [out=-90, in=0] (8,-.55) -- (4,-.55) to [out=180, in=-90] (3.1,0) arc(180:0:.4 and .4) to [out=-90, in=180] (5,-.4) -- (7.5,-.4) to [out=0, in=-90] (8,0) to [out=90, in=-90] (9.2,2) to [out=90, in=0] (8.75,2.5) to [out=180, in=90] (8.5,2.25);
\draw [orange, thick] (2,.85) to [out=180, in=90] (.95,0) to [out=-90, in=180] (1.5,-.7) -- (3.5,-.7) to [out=0, in=-90] (4.3,0) arc (180:0:.7 and .7) to [out=-90, in=180] (6.5,-.7) -- (8,-.7) to [out=0, in=-90] (8.9,0) to [out=90, in=0] (8,.85) -- (2,.85);
\draw [olive, thick] (2.2,1) to [out=180, in=-90] (.95,2) arc(180:0:.55 and .55) to [out=-90, in=180] (3,1.45) -- (5,1.45) to [out=0, in=-90] (5.5,2) to [out=90, in=0] (4.5,2.7) -- (1.5,2.7) arc(90:180:.7 and .7) -- (.8,0) to [out=-90, in=180] (1.5, -.85) -- (3.5,-.85) to [out=0, in=-90] (4.45,0) arc (180:0:.55 and .55) to [out=-90, in=180] (6.5,-.85) -- (8,-.85) to [out=0, in=-90] (9.05,0) to [out=90, in=0] (8,1) -- (2.2,1);
\draw [cyan, thick] (5.4,0) arc(0:180:.4 and .4) to [out=-90, in=180] (6.5,-1.15) -- (8,-1.15) to [out=0, in=-90] (9.35,0) to [out=90, in=0] (8,1.3) -- (3,1.3) to [out=180, in=-90] (1.9,2) arc(0:180:.4 and .4) to [out=-90, in=180] (2.2,1.15) -- (8,1.15) to [out=0, in=90] (9.2,0) to [out=-90, in=0] (8,-1) -- (6.5,-1) to [out=180, in=-90] (5.4,0);
\node at (0,4.25) {d)};
\draw (1.5,3.5) circle (2.5mm);
\draw (1.5,5) circle (2.5mm);
\node at (2.5,3.5) {$\dots$};
\node at (2.5,5) {$\dots$};
\draw (3.5,3.5) circle (2.5mm);
\draw (3.5,5) circle (2.5mm);
\draw (5,3.5) circle (2.5mm);
\draw (5,5) circle (2.5mm);
\node at (6,3.5) {$\dots$};
\node at (6,5) {$\dots$};
\draw (7,3.5) circle (2.5mm);
\draw (7,5) circle (2.5mm);
\draw (8.5,3.5) circle (2.5mm);
\draw (8.5,5) circle (2.5mm);
\draw (10,3.5) circle (2.5mm);
\draw (10,5) circle (2.5mm);
\node at (11,3.5) {$\dots$};
\node at (11,5) {$\dots$};
\draw (12,3.5) circle (2.5mm);
\draw (12,5) circle (2.5mm);
\draw [thick] (13.5,4.25) circle (2mm);
\draw [blue, thick] (8.5,3.25) to [out=-90, in=-90] (6.6,3.5) to [out=90, in=-90] (9.2,5) to [out=90, in=0] (8.75,5.5) to [out=180, in=90] (8.5,5.25);
\draw [orange, thick] (1.5,3.9) arc(90:270:.4 and .4) -- (7,3.1) arc(-90:90:.4 and .4) -- (1.5,3.9);
\draw [olive, thick] (2.2,4.05) to [out=180, in=-90] (.95,5) arc(180:0:.55 and .55) to [out=-90, in=180] (3,4.5) -- (5,4.5) to [out=0, in=-90] (5.5,5) to [out=90, in=0] (4.5,5.7) -- (1.5,5.7) arc(90:180:.7 and .7) -- (.8,3.8) to [out=-90, in=180] (1.5, 2.95) -- (7,2.95) to [out=0, in=-90] (7.55,3.5) arc(0:90:.55 and .55) -- (2.2,4.05);
\draw [cyan, thick] (8.9,3.5) arc(0:-180:.4 and .4) to [out=90, in=0] (7,4.2) -- (2.2,4.2) to [out=180, in=-90] (1.1,5) arc(180:0:.4 and .4) to [out=-90, in=180] (3,4.35) -- (8,4.35) to [out=0, in=90] (8.9,3.5);
\node at (0,7.75) {c)};
\draw (1.5,7) circle (2.5mm);
\draw (1.5,8.5) circle (2.5mm);
\node at (2.5,7) {$\dots$};
\node at (2.5,8.5) {$\dots$};
\draw (3.5,7) circle (2.5mm);
\draw (3.5,8.5) circle (2.5mm);
\draw (5,7) circle (2.5mm);
\draw (5,8.5) circle (2.5mm);
\node at (6,7) {$\dots$};
\node at (6,8.5) {$\dots$};
\draw (7,7) circle (2.5mm);
\draw (7,8.5) circle (2.5mm);
\draw (8.5,7) circle (2.5mm);
\draw (8.5,8.5) circle (2.5mm);
\draw (10,7) circle (2.5mm);
\draw (10,8.5) circle (2.5mm);
\node at (11,7) {$\dots$};
\node at (11,8.5) {$\dots$};
\draw (12,7) circle (2.5mm);
\draw (12,8.5) circle (2.5mm);
\draw [thick] (13.5,7.75) circle (2mm);
\draw [blue, thick] (8.5,6.75) to [out=-90, in=0] (8.35,6.6) to [out=180, in=-90] (8.05,7) arc(180:90:.45 and .45) to [out=0, in=180] (10,6.6) arc(270:360:.4 and .4) to [out=90,in=-90] (9.2,8.5) to [out=90, in=0] (8.75,9) to [out=180, in=90] (8.5,8.75);
\draw [orange, thick] (1.5,7.55) arc(90:180:.55 and .55) to [out=-90, in=180] (2,6.2) -- (10,6.2) to [out=0, in=-90] (10.55, 7) arc(0:90:.55 and .55) -- (1.5,7.55);
\draw [cyan, thick] (1.9,7) arc(0:180:.4 and .4) to [out=-90, in=180] (2,6.35) -- (10, 6.35) to [out=0, in=-90] (10.4,7) arc (0:90:.4 and .4) to [out=180, in=0] (8.5,6.5) -- (2.5,6.5) to [out=180, in=-90] (1.9,7);
\draw [olive, thick] (2.2,7.7) to [out=180, in=-90] (1,8.5) arc(180:0:.5 and .5) to [out=-90, in=180] (3,7.85) -- (5,7.85) to [out=0, in=-90] (5.7,8.5) arc(0:90:.7 and .7) -- (1.5,9.2) arc(90:180:.7 and .7) -- (.8,7) to [out=-90, in=180] (2, 6.05) -- (10,6.05) to [out=0, in=-90] (10.7,7) arc(0:90:.7 and .7) -- (2.2,7.7);
\node at (0,11.25) {b)};
\draw (1.5,10.5) circle (2.5mm);
\draw (1.5,12) circle (2.5mm);
\node at (2.5,10.5) {$\dots$};
\node at (2.5,12) {$\dots$};
\draw (3.5,10.5) circle (2.5mm);
\draw (3.5,12) circle (2.5mm);
\draw (5,10.5) circle (2.5mm);
\draw (5,12) circle (2.5mm);
\node at (6,10.5) {$\dots$};
\node at (6,12) {$\dots$};
\draw (7,10.5) circle (2.5mm);
\draw (7,12) circle (2.5mm);
\draw (8.5,10.5) circle (2.5mm);
\draw (8.5,12) circle (2.5mm);
\draw (10,10.5) circle (2.5mm);
\draw (10,12) circle (2.5mm);
\node at (11,10.5) {$\dots$};
\node at (11,12) {$\dots$};
\draw (12,10.5) circle (2.5mm);
\draw (12,12) circle (2.5mm);
\draw [thick] (13.5,11.25) circle (2mm);
\draw [blue, thick] (8.5,10.25) to [out=-90, in=0] (8.35,10.15) to [out=180, in=-90] (8,10.6) to [out=90, in=180] (8.5,11.15) arc (-90:90:.85 and .85) to [out=180, in=0] (7,12.5) arc(90:270:.5 and .5) to [out=0, in=180] (8.25,12.5) to [out=0, in=90] (8.5,12.25);
\draw [orange, thick] (1.5,11.05) arc(90:180:.55 and .55) to [out=-90, in=180] (2,9.8) -- (8.5,9.8) to [out=0, in=-90] (9.25,10.8) -- (9.25,11.7) to [out=90, in=0] (8.5, 12.55) arc(90:150:.55 and .55) to [out=-120, in=0] (6.8,11.05) -- (1.5,11.05);
\draw [cyan, thick] (1.9,10.5) arc(0:180:.4 and .4) to [out=-90, in=180] (2,9.95) -- (8.5, 9.95) to [out=0, in=-90] (9.1,10.8) -- (9.1,11.7) to [out=90, in=0] (8.5, 12.4) arc(90:180:.4 and .4) to [out=-90, in=90] (8.9,10.5) arc(0:-90:.4 and .4) -- (2.5,10.1) to [out=180, in=-90] (1.9,10.5);
\draw [olive, thick] (2.2,11.2) to [out=180, in=-90] (1,12) arc(180:0:.5 and .5) to [out=-90, in=180] (3,11.35) -- (5,11.35) to [out=0, in=-90] (5.7,12) arc(0:90:.7 and .7) -- (1.5,12.7) arc(90:180:.7 and .7) -- (.8,10.5) to [out=-90, in=180] (2, 9.65) -- (8.5,9.65) to [out=0, in=-90] (9.4,10.8) -- (9.4,11.7) to [out=90, in=0] (8.5, 12.7) arc(90:150:.7 and .7) to [out=-120, in=0] (6.8,11.2) -- (2.2,11.2);
\node [cyan] at (4.25,10.4) {$\varphi_2$};
\node at (0,14.75) {a)};
\draw (1.5,14) circle (2.5mm);
\draw (1.5,15.5) circle (2.5mm);
\node at (2.5,14) {$\dots$};
\node at (2.5,15.5) {$\dots$};
\draw (3.5,14) circle (2.5mm);
\draw (3.5,15.5) circle (2.5mm);
\draw (5,14) circle (2.5mm);
\draw (5,15.5) circle (2.5mm);
\node at (6,14) {$\dots$};
\node at (6,15.5) {$\dots$};
\draw (7,14) circle (2.5mm);
\draw (7,15.5) circle (2.5mm);
\draw (8.5,14) circle (2.5mm);
\draw (8.5,15.5) circle (2.5mm);
\draw (10,14) circle (2.5mm);
\draw (10,15.5) circle (2.5mm);
\node at (11,14) {$\dots$};
\node at (11,15.5) {$\dots$};
\draw (12,14) circle (2.5mm);
\draw (12,15.5) circle (2.5mm);
\draw [thick] (13.5,14.75) circle (2mm);
\node at (14,14.75) {$C$};
\draw [red, thick] (8.5,15.5) circle (5mm);
\draw [Green, thick] (8.25,15.5) to [out=180, in=60] (7.7,15.1) to [out=-120, in=0] (7,14.75) -- (5,14.75) to [out=180, in=-90] (4.25,15.5) to [out=90, in=0] (3.5,16.3) -- (1.5,16.3) to [out=180, in=90] (.7,15.5) -- (.7,14) to [out=-90, in=180] (1.5,13.2) -- (8.5,13.2) to [out=0, in=-90] (9.1,13.7) to [out=90, in=0] (8.75,14);
\draw [Green, thick] (8.75,15.5) to [out=0, in=90] (9.1,15.15) to [out=-90, in=0] (8.5,14.75) to [out=180, in=90] (7.9,14.35) to [out=-90, in=180] (8.25,14);
\draw [blue, thick] (8.5,13.75) to [out=-90, in=0] (8.3,13.6) to [out=180, in=-90] (7.9,14) arc(180:90:.6 and .6) to [out=0, in=-90] (9.4,15.3) to [out=90, in=0] (8.75,16) to [out=180, in=90] (8.5,15.75);
\draw [orange, thick] (1.5,14.5) arc(90:270:.5 and .5) -- (8.5,13.5) arc(-90:90:.5 and .5) -- (1.5,14.5);
\draw [olive, thick] (2.2,14.65) to [out=180, in=-90] (1,15.5) arc(180:0:.5 and .5) to [out=-90, in=180] (3,14.85) -- (5,14.85) arc(-90:90:.65 and .65) -- (1.5,16.15) arc(90:180:.65 and .65) -- (.85,14) arc(180:270:.65 and .65) -- (2.2, 13.35) -- (8.5,13.35) arc (-90:90:.65 and .65) -- (2.2,14.65);
\node [red] at (8,16.1) {$\alpha$};
\node [blue] at (9.3,16.1) {$\beta$};
\node [Green] at (6.4,15.05) {$\gamma$};
\node [orange] at (4.25,14.2) {$\zeta$};
\node [olive] at (4.6,15.85) {$\varphi_1$};
\end{tikzpicture}
\caption{Triangles starting from $v_0$ with an edge of type $(i,j)^{-}$.}
\label{fig:basictri-}
\end{figure}
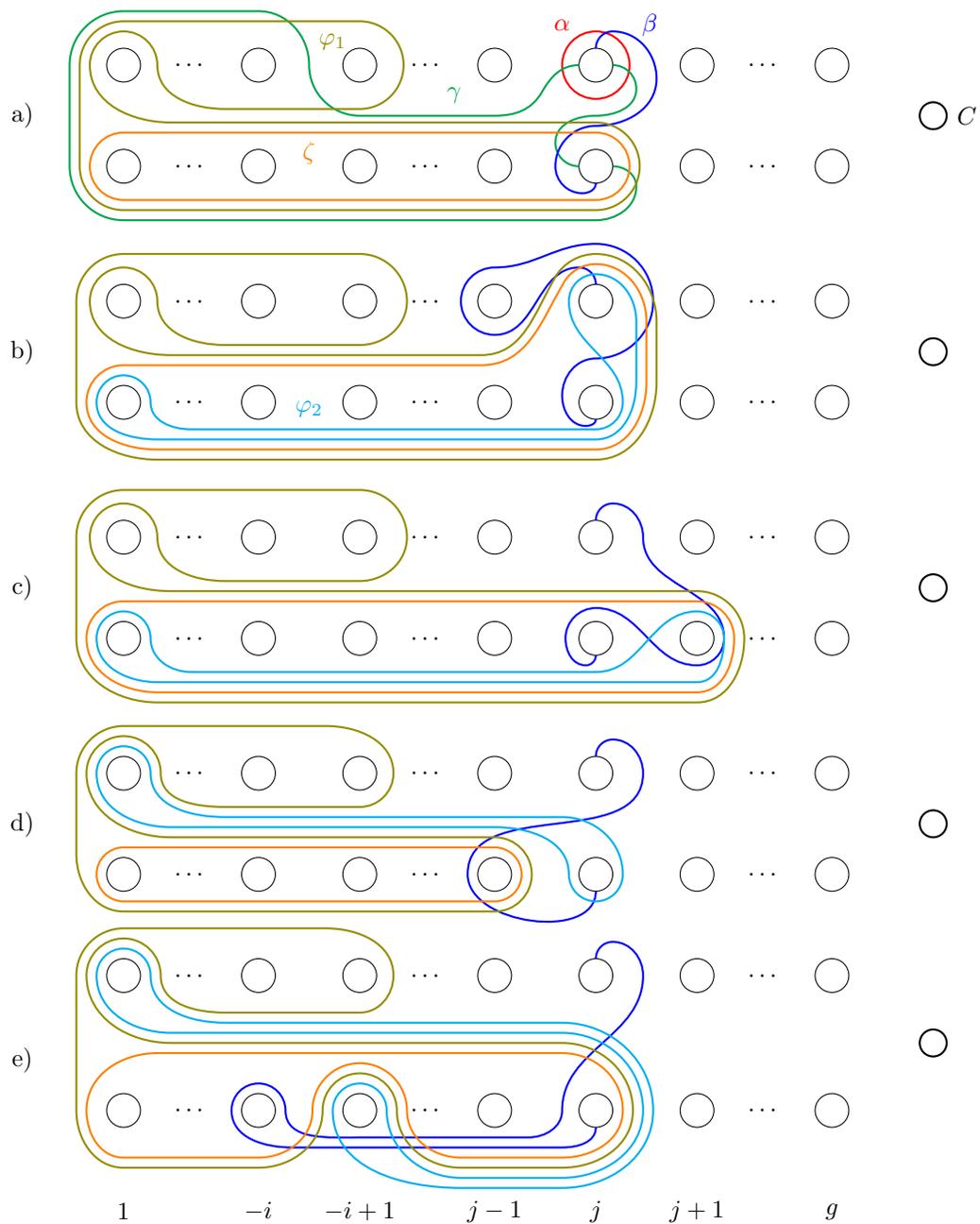

The corresponding relations are the following. Notice that we adopt the convention that an increasing sequence of consecutive indices is empty if the last is less than the first: hence, for example, $t_1,\dots,t_{-i}=1$ if $i=1$. 
\begin{enumerate}[resume*=arel]
\item \textit{The following relations hold:}
\begin{gather*}
(b_1^{-2})^{-1}=b_1b_1, \qquad b_1^{-2}s_1\overline{d}_{1,2}^{-1}a_1^2b_1s_1^{-1}\overline{d}_{1,2}^{-2}a_1^{-2}b_1=\overline{d}_{1,2}^{-1}a_2^2.
\end{gather*}

\textit{Now, assume that $j \ge 2$. If $i=1$, we have}
\[
b_j^{-2}\overline{r}_{1,j}\,\cdot \big((t_{j-1}\dots t_1\,\zeta^{-1})*(a_1^2b_1 \, s_1^{-1} \varphi_2^{-1}\zeta^{-1} s_1 \,b_1\, s_1^{-1}a_1^{-2})\big)=1,
\]
\textit{while if $i<1$}
\[
b_j^{-2}\overline{r}_{i,j}\,\cdot \big((t_{j-1}\dots t_1\,\zeta^{-1})*(a_1^2b_1 \, s_1^{-1} \varphi_2^{-1}\varphi_1^{-1} s_1 \,b_1\, s_1^{-1}a_1^{-2})\big)=1,
\]
\textit{where $\zeta$, $\varphi_1$ and $\varphi_2$ are shorthands for the following mapping classes, which correspond to the Dehn twists along the curves of Figure \ref{fig:basictri-} up to powers of the $a_i^2$:} \label{triangles}
\begin{enumerate}[label=\alph*)]
\item \label{tria} $\zeta:=\overline{d}_{\set{1,\dots,j}}$, \hspace{5pt} $\varphi_1:=\overline{d}_{\set{i,\dots,\widehat{-1},\dots,j}}$, \hspace{5pt} $\varphi_2:=1$; 
\item \label{trib} $\zeta:=(t_{j-1}\dots t_1)*\overline{d}_{\set{-1,\dots,j}}$, \newline \hspace{5pt} $\varphi_1:=a_j^2\cdot \big((t_{j-1}\dots t_1)*\overline{d}_{\set{i,\dots,\widehat{-2},\dots,j}}\big)$, \hspace{5pt} $\varphi_2:=s_j*\overline{d}_{1,j}$;
\item \label{tric} $\zeta:=\overline{d}_{\set{1,\dots,j+1}}$, \hspace{5pt} $\varphi_1:=\overline{d}_{\set{i,\dots,\widehat{-1},\dots,j+1}}\, a_{j+1}^2$, \hspace{5pt} $\varphi_2:=\overline{d}_{1,j+1}$;
\item \label{trid} $\zeta:=\overline{d}_{\set{1,\dots,j-1}}$, \hspace{5pt} $\varphi_1:=\overline{d}_{\set{i,\dots,\widehat{-1},\dots,j-1}}$,\newline 
\hspace{5pt} $\varphi_2:=(t_{j-1}^{-1}\dots t_2^{-1}s_1)*\overline{d}_{1,2}$;
\item $\zeta:=(t_{-i+1}^{-1}\dots t_{j-1}^{-1})*\overline{d}_{\set{1,\dots,j-1}}$, \newline \hspace{5pt} \label{trie} $\varphi_1:=a_{-i+1}^2\cdot\big((t_{-i+1}^{-1}\dots t_{j-1}^{-1})*\overline{d}_{\set{i,\dots,\widehat{-1},\dots,j-1}}\big)$, \newline \hspace{5pt} \hspace{5pt} $\varphi_2:=(t_{-i+1}^{-1}\dots t_{j-1}^{-1}\, t_{j-1}^{-1}\dots t_2^{-1}s_1)* \overline{d}_{1,2}$.
\end{enumerate}
\end{enumerate}

\vspace{10pt}

\emph{Squares.} By the spin change of coordinates principle, the spin mapping class group acts transitively on the set of squares with a vertex at $v_0$, so it is enough to take the relation corresponding to the square 
\[
\begin{tikzcd}
\Braket{\alpha_1,\alpha_2} \arrow[r, no head] \arrow[d, no head] & \Braket{\alpha_1,\beta_2} \arrow[d, no head] \\
\Braket{\beta_1,\alpha_2} \arrow[r, no head] & \Braket{\beta_1,\beta_2}, 
\end{tikzcd}
\]
in the notations of Figure \ref{fig:square}(a).

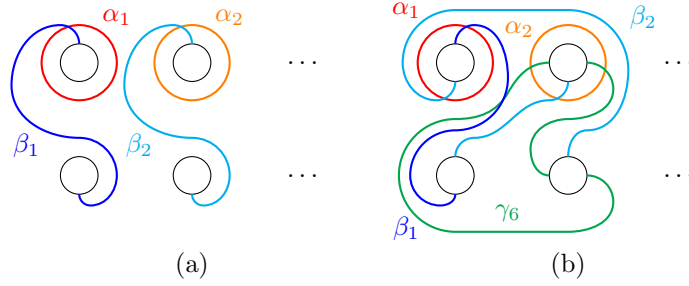
\begin{figure}
\centering
\begin{tikzpicture}
\draw (1.5,4) circle (2.5mm);
\draw (1.5,5.5) circle (2.5mm);
\draw (3,4) circle (2.5mm);
\draw (3,5.5) circle (2.5mm);
\node at (4.5,4) {$\dots$};
\node at (4.5,5.5) {$\dots$};
\draw [red, thick] (1.5,5.5) circle (5mm);
\draw [blue, thick] (1.5,3.75) to [out=-90, in=180] (1.65,3.6) to [out=0, in=-90] (2,4) arc(0:90:.5 and .5) to [out=180, in=-90] (.6,5.2) to [out=90, in=180] (1.25,6) to [out=0, in=90] (1.5,5.75);
\node [red, thick] at (2,6.1) {$\alpha_1$};
\node [blue, thick] at (.8,4.4) {$\beta_1$};
\draw [orange, thick] (3,5.5) circle (5mm);
\draw [cyan, thick] (3,3.75) to [out=-90, in=180] (3.15,3.6) to [out=0, in=-90] (3.5,4) arc(0:90:.5 and .5) to [out=180, in=-90] (2.1,5.2) to [out=90, in=180] (2.75,6) to [out=0, in=90] (3,5.75);
\node [orange, thick] at (3.5,6.1) {$\alpha_2$};
\node [cyan, thick] at (2.3,4.4) {$\beta_2$};
\node at (3,2.8) {(a)};
\draw (6.5,4) circle (2.5mm);
\draw (6.5,5.5) circle (2.5mm);
\draw (8,4) circle (2.5mm);
\draw (8,5.5) circle (2.5mm);
\node at (9.5,4) {$\dots$};
\node at (9.5,5.5) {$\dots$};
\draw [red, thick] (6.5,5.5) circle (5mm);
\draw [orange, thick] (8,5.5) circle (5mm);
\draw [Green, thick] (7.75,5.5) to [out=180, in=60] (7.2,5.1) to [out=-120, in=0] (6.5,4.75) arc(90:270:.75 and .75) -- (8,3.25) to [out=0, in=-90] (8.6,3.65) to [out=90, in=0] (8.25,4);
\draw [Green, thick] (8.25,5.5) to [out=0, in=90] (8.6,5.15) to [out=-90, in=0] (8,4.75) to [out=180, in=90] (7.5,4.35) to [out=-90, in=180] (7.75,4);
\draw [blue, thick] (6.5,3.75) to [out=-90, in=0] (6.3,3.6) to [out=180, in=-90] (5.9,4) arc(180:90:.6 and .6) to [out=0, in=-90] (7.2,5.3) to [out=90, in=0] (6.75,6) to [out=180, in=90] (6.5,5.75);
\draw [cyan, thick] (6.5,5.25) to [out=-90, in=0] (6.25,5) to [out=180, in=-90] (5.8,5.5) to [out=90, in=180] (6.5,6.2) -- (8,6.2) to [out=0, in=90] (8.8,5.4) to [out=-90, in=0] (8.25,4.6) to [out=180, in=90] (8,4.25);
\draw [cyan, thick] (8,5.25) to [out=-90, in=0] (7.75,5) to [out=180, in=0] (6.75,4.5) to [out=180, in=90] (6.5,4.25);
\node [red] at (5.85,6.2) {$\alpha_1$};
\node [blue] at (5.85,3.3) {$\beta_1$};
\node [orange] at (7.35,5.95) {$\alpha_2$};
\node [cyan] at (9,6.1) {$\beta_2$};
\node [Green] at (7.2,3.5) {$\gamma_6$};
\node at (8,2.8) {(b)};
\end{tikzpicture}
\caption{(a) Curves in the square corresponding to relation \ref{square}.\\ (b) Curves in the pentagon corresponding to relation \ref{pentagon}.}
\label{fig:square}
\end{figure}

\begin{samepage}
We get the following relation: 
\begin{enumerate}[resume*=arel]
    \item $(a_1^2b_1t_1\overline{d}_{1,2}^{-1}a_1^2b_1t_1\overline{d}_{-2,-1}^{-1})^2=a_1^2s_1a_2^2s_2$.  \label{square}
\end{enumerate}
\end{samepage}

\vspace{10pt}

\emph{Pentagons.} A single relation is sufficient also in this case. Our model pentagon will be that of Figure \ref{square}(b), i.e.
\[
\begin{tikzcd}[column sep=tiny]
\Braket{\alpha_1,\alpha_2}\arrow[r, no head, "(ii)"] & \Braket{\alpha_1,\gamma_6}\arrow[r, no head] & \Braket{\beta_2,\gamma_6}\arrow[r, no head] & \Braket{\beta_2,\beta_1}\arrow[r, no head] & \Braket{\alpha_2,\beta_1}\arrow[r, no head] & \Braket{\alpha_1,\alpha_2}.
\end{tikzcd}
\]

\begin{lemma}
\thlabel{lem:penta}
All pentagons are homotopic in $X_g$.
\end{lemma}

\begin{proof}
Consider another pentagon in $X_g$. As four of its five curves form a $4$-chain on $\Sigma_g^1$, by the spin change of coordinates, up to the action of $\Modd(\Sigma_g^1)[\phi]$ we may assume that it is of the form
\[
\begin{tikzcd}[column sep=tiny]
\Braket{\alpha_1,\alpha_2}\arrow[r, no head, "(ii)"] & \Braket{\alpha_1,\gamma}\arrow[r, no head] & \Braket{\beta_2,\gamma}\arrow[r, no head] & \Braket{\beta_2,\beta_1}\arrow[r, no head] & \Braket{\alpha_2,\beta_1}\arrow[r, no head] & \Braket{\alpha_1,\alpha_2},
\end{tikzcd}
\]
for some curve $\gamma$. Let $\gamma_5$ be the arc sum of $\beta_1$ and $\beta_2$ along an arc of $\alpha_2$. Then $\gamma_5$ is a nonseparating $1$-curve, and we have the homotopy of \cite[Figure 17]{waj:elem}. \qedhere
\end{proof}

We get the following relation:

\begin{enumerate}[resume*=arel]
\item $b_2^{-2}\overline{r}_{1,2} \overline{d}_{-2,-1}\overline{d}_{1,2}^{-1} b_1 t_1a_1^2 b_1 \overline{d}_{-2,-1}^{-3}t_1 b_1 t_1 b_1 s_2=\overline{d}_{\set{-2,-1,1,2}}$. \label{pentagon}
\end{enumerate}

\vspace{10pt}

\emph{Hyperelliptic faces.} Recall that a hyperelliptic face is uniquely determined by a $7$-chain of admissible curves by \thref{rem:h1}. Moreover, from the proof of \thref{prop:hyp} and \thref{rem:h2} we see that we need only the hyperelliptic faces corresponding to $7$-chains that split the surface into two components, one of which has genus $0$ and does not intersect $\partial\Sigma_g^1$. By the spin change of coordinates principle, we just need a single relation. 

Instead of writing a long $h$-product, we can directly state the relation as a product of admissible twist as follows. Fix a $7$-chain of admissible curves $\gamma_1,\dots,\gamma_7\subset \Sigma_g^1$ with the above properties, and let $\delta$ be the nontrivial boundary component of a tubular neighborhood of $\gamma_1\cup\dots\cup\gamma_7$. Then we have the following restatement of (\ref{hyp3}): 

\begin{enumerate}[resume*=arel]
    \item $(t_{\gamma_1}\dots t_{\gamma_6}t_{\gamma_7}^2t_{\gamma_6}\dots t_{\gamma_1})^2=t_{\delta}$. \label{hyp}
\end{enumerate}

\subsection{A finite presentation of the spin handlebody group}
\label{hmgqfp}

We conclude this section with a finite presentation for the spin handlebody mapping class group $\Modd(H_g)[\phi]$, where $H_g$ is the handlebody in which the curves $\alpha_1,\dots,\alpha_g$ of Figure \ref{fig:q} bound disks. This will be done using Nielsen-Schreier's method, and will not be needed in the following, but is relevant for \cite{fb}.

\begin{theorem}
\thlabel{thm:mhgq}
The spin handlebody mapping class group $\Modd(H_g)[\phi]$ has a presentation with generators $a_1^2,\dots,a_g^2$, $s$, $t_1,\dots,t_{g-1}$, $\overline{d}_{i,j}$ for all $i,j\in \set{\pm1,\dots,\pm g}$ with $i<j$ and $\overline{r}_{i,j}$ for all $i,j\in \set{\pm1,\dots,\pm g}$ with $i=1<j$ or $i\le -1$ and $-i+1 \le j \le g+i$, and the following relations:
\begin{enumerate}[label=\textit{(H\arabic*)}, ref=(H\arabic*)]
\item relations \ref{aiaj}-\ref{a8} of \thref{prop:stab};
\item $\overline{d}_{\set{\pm1,\dots,\pm g}}(a_1\dots a_g)^2=1$;
\item $\overline{d}_{\set{\pm1,\dots,\pm g}\setminus \set{k}}(a_1\dots \widehat{a_{\abs{k}}} \dots a_g)^2=1$ for all $k \in \set{\pm1,\dots,\pm g}$;\label{hatk}
\item $r_{1,j}^2=(a_1\dots a_{j-1})^{-2}s_j\overline{d}_{\set{1,\dots,j}}s_j\overline{d}_{\set{1,\dots,j}}^{-1}$, and if $i \le -1$ then \label{rij2}
\[
\overline{r}_{i,j}^2=\big(a_1^2\dots a_{-i}^2 a_{-i+1}\dots a_{j-1}\big)^{-2} s_j\overline{d}_{\set{i,\dots,j}}s_j\overline{d}_{\set{i,\dots,j}}^{-1};
\]
\item conjugates involving the generators $\overline{r}_{i,j}$:\label{conjr}
\begin{enumerate}
\item $[\overline{r}_{i,j},a_k^2]=1$ if $k \ne j$ and 
\[
\overline{r}_{i,j}*a_j^2=c_{i,j}^2=\begin{cases}
    \overline{d}_{\set{1,\dots,j}}^2(a_1^2\dots a_j^2) & \text{if $i=1$},\\
    \overline{d}_{\set{i,\dots,j}}^2(a_1^4\dots a_{-i}^4a_{-i+1}^2\dots a_j^2) & \text{if $i\le -1$};
\end{cases}
\]
\item $[\overline{r}_{i,j},t_k]=1$ if $k \ne j,j-1$ and $k \ne -i$;
\item $[\overline{r}_{i,j},s_k]=1$ if $k\le -i$ or $k>j$;
\item $[\overline{r}_{i,j},\overline{d}_{k,m}]=1$ if $k,m \in \set{i,\dots,j-1}$ or $k,m \notin\set{-j,i,i+1,\dots,j}$;
\item $[\overline{r}_{1,g},z_g]=1$, and $z_j*\overline{r}_{i,j}=a_1^2\dots a_{-i}^2\overline{r}_{i,j}a_{j+1}^{-2}\dots a_g^{-2}$ if $i\le -1$ and $j-i=g$;
\item $\overline{r}_{i,j}*\overline{d}_{i,j}=a_i^{-2}\overline{d}_{\set{i,\dots,j}}^{-1}\overline{d}_{\set{i,\dots,j}\setminus\set{i}}$;
\item $\overline{r}_{1,j}*\overline{d}_{-j,-j+1}=\overline{d}_{\set{1,\dots,j}}^{-1}(t_{j-2}t_{j-3}\dots t_1)*\overline{d}_{\set{-1,\dots,j}}$;
\item $\overline{r}_{i,j}*\overline{d}_{-j,-j+1}=\overline{d}_{\set{i,\dots,j}}^{-1}(t_{j-2}t_{j-3}\dots t_{-i+1})*\overline{d}_{\set{i-1,\dots,j}}$ if $i\le-1$ and $j+i>1$;
\item $\overline{r}_{i,j}^{-1}*\overline{d}_{-j-1,-j}=\overline{d}_{\set{i,\dots,j}}\,\big(s_{j+1}^{-1}*\overline{d}_{\set{i,\dots,j+1}}\big)$;
\end{enumerate}
\item $\overline{r}_{i,j}*t_{j-1}=\big(t_{j-1}^{-1}*\overline{r}_{i,j}\big)\,\overline{d}_{\set{i,\dots,j}}^{-1}$ if $i>-j+1$, and \label{riti}
\[
\overline{r}_{-j+1,j}*t_{j-1}=\big(t_{j-1}^{-1}*\overline{r}_{-j+1,j}\big)\,
(a_1^2\dots a_{j-1}^2)\overline{d}_{\set{-j+1,\dots,j}}^{-2};
\]\label{tr}
\item triangle relations:
\begin{align*}
\overline{r}_{1,j}&=s_j\overline{d}_{\set{1,\dots,j}}s_j\overline{d}_{\set{1,\dots,j}}^{-1}k_{j-1}\overline{d}_{\set{1,\dots,j-2}}t_{j-1}\overline{d}_{\set{1,\dots,j-1}}^{-1}t_{j-1}^{-1}\cdot\\
&\qquad\qquad \cdot\overline{r}_{1,j-1}^{-1}
s_{j-1}(a_1\dots a_{j-2})^{-2}Aa_1^{-2}\overline{r}_{1,2}^{-1}A^{-1}k_{j-1}^{-1} \quad \text{for $j \ge 3$},\\
\overline{r}_{-1,j}&=Ba_1^{-2}\overline{r}_{1,2}^{-1}B^{-1}s_j\overline{r}_{1,j}^{-1}(a_1a_2\dots a_{j-1})^{-2}\overline{d}_{\set{-1,\dots,j-1}}^{-1}\cdot\\
&\qquad\qquad \cdot\overline{d}_{\set{1,\dots,j-1}} s_j\overline{d}_{\set{-1,\dots,j}}s_j\overline{d}_{\set{-1,\dots,j}}^{-1},\\
\overline{r}_{i,j}&=Ca_{1}^{-1}\overline{r}_{1,2}^{-1}C^{-1}s_j(a_1^2\dots a_{-i-1}^2a_{-i}\dots a_{j-1})^{-2}\overline{r}_{i+1,j}^{-1}\overline{d}_{\set{i,\dots,j-1}}^{-1}\cdot\\
&\qquad\qquad \cdot\overline{d}_{\set{i+1,\dots,j-1}}s_j\overline{d}_{\set{i,\dots,j}}s_j\overline{d}_{\set{i,\dots,j}}^{-1} \quad \text{for $i \le -2$},
\end{align*}
where
\begin{gather*}
A=k_{j-1}^{-1}t_{j-2}^{-1}t_{j-3}^{-1}\dots t_1^{-1}k_{j-1}k_{j-2}\dots k_2,\\
B=s_1k_{j-1}k_{j-2}\dots k_2,\\
C=s_{-i}t_{-i-1}^{-1}t_{-i-2}^{-1}\dots t_1^{-1}k_{j-1}k_{j-2}\dots k_2.
\end{gather*}\label{trirel}
\end{enumerate}
\end{theorem}

\begin{proof}
We will apply the Nielsen-Schreier method to Wajnryb's presentation of $\Modd(H_g)$ \cite[Theorem 18]{wajh}. First of all, by inspecting his proof, it is easy to see that an equivalent presentation is given by the presentation \cite[Proposition 27]{waj:elem} of the stabilizer of $v_0=\Braket{\alpha_1,\dots,\alpha_g}$ in $\Modd(\Sigma_g^1)$, together with generators $r_{i,j}:=b_ja_jc_{i,j}b_j$ and relations (P3), (P4), (P9), (P10), (P11) and (P12) of \cite[Theorem 18]{wajh}. Indeed, since $v_0$ is a cut-system of meridians for $H_g$, its stabilizers under the action of $\Modd(H_g)$ and $\Modd(\Sigma_g)$ coincide, and relations (P3), (P4) of \cite[Theorem 18]{wajh} come from capping the boundary component of $\Sigma_g^1$ with a disk. Relation (P8) of \cite[Theorem 18]{wajh} is clearly implied by relation (P8) of \cite[Proposition 27]{waj:elem}, which on the other hand still holds in $\Modd(H_g)$. The last four relations of \cite[Theorem 18]{wajh} are derived by studying the action of $\Modd(H_g)$ on a complex of cut-system of meridians.

Now, we proceed exactly as in the proof of \thref{prop:stab}. We introduce new generators $\overline{d}_{i,j}$ and $\overline{r}_{i,j}$, and new relations (\ref{eq:dij}) and (\ref{eq:rij}). The relations coming from \cite[Proposition 27]{waj:elem} change as in the proof of \thref{prop:stab}, while the other relations change as follows.
\begin{enumerate}
\item[\textit{(P3)}] From (\ref{eq:dI}) we get $\overline{d}_{\set{\pm1,\dots,\pm g}}(a_1\dots a_g)^2=1$. 
\item[\textit{(P4)}] Similarly, we obtain $\overline{d}_{\set{\pm1,\dots,\pm g}\setminus \set{k}}(a_1\dots \widehat{a_{\abs{k}}} \dots a_g)^2=1$. 
\item[\textit{(P9)}] By (\ref{eq:rij}), we obtain
\[
\overline{r}_{i,j}^2=\begin{cases}
(a_1\dots a_{j-1})^{-2}s_j\overline{d}_{\set{1,\dots,j}}s_j\overline{d}_{\set{1,\dots,j}}^{-1} & \text{if $i=1$},\\
\big(a_1^2\dots a_{-i}^2 a_{-i+1}\dots a_{j-1}\big)^{-2} s_j\overline{d}_{\set{i,\dots,j}}s_j\overline{d}_{\set{i,\dots,j}}^{-1} & \text{if $i \le -1$}.
\end{cases}
\]
\item[\textit{(P10)}] We know that the $a_i$ commute with each other and with the $\overline{d}_{i,j}$ and $s$, and moreover $t_i*a_i=t_{i+1}*a_i=a_{i+1}$. As a consequence, we see that the $a_i$ commute with all the $s_j$ and the $c_{i,j}$, and that 
\[
z*a_i=a_{g+1-i}, \qquad z_j*a_i=\begin{cases}
a_{g+1-i} & \text{if $i>j$ or $j+i<g+1$},\\
a_{g-i} & \text{if $g+1-j\le i < j$},\\
a_i & \text{if $i=j$}.
\end{cases}
\]
We obtain the relations in the statement, apart from the following:
\begin{enumerate}
\item $\overline{r}_{i,j}*a_j=c_{i,j}$ and $[\overline{r}_{i,j},a_k]=1$ if $k \ne j$.
\end{enumerate}
\item[\textit{(P11)}] By (\ref{eq:rij}), we get the statement
\item[\textit{(P12)}] Again, it suffices to plug in (\ref{eq:rij}) to get the statement.
\end{enumerate}

A Schreier transversal is again given by (\ref{eq:u}). Indeed, notice that by (P9) and (P10)(a) we have
\begin{equation}
\label{eq:rija}
\overline{r}_{i,j}*c_{i,j}=\overline{r}_{i,j}^2*a_j=a_j.
\end{equation}
Hence, in every word in the generators we can move all the $a_i$ to the right, and apply the same reasoning as before. Moreover, the Schreier generators boil down to those in the statement. To see this for the $\overline{r}_{i,j}$, observe that
\[
\overline{u_J\overline{r}_{1,j}^{\pm1}}=\begin{cases}
u_{J\setminus\set{j}} & \text{if} \ j \in J,\\
u_{J\Delta\set{1,\dots,j}} & \text{if} \ j \notin J,
\end{cases}
\]
where by $\Delta$ we denote the symmetric difference, and
\[
\overline{u_J\overline{r}_{i,j}^{\pm1}}=\begin{cases}
u_{J\setminus\set{j}} & \text{if} \ j \in J,\\
u_{J\Delta\set{-i+1,\dots,j}} & \text{if} \ j \notin J
\end{cases}
\]
if $i \le -1$. 

Finally, the relations coming from the stabilizer of $v_0$ change as in \thref{prop:stab}, and the only other relation that changes is (P10)(a), which becomes $[\overline{r}_{i,j},a_k^2]=1$ if $k \ne j$ and 
\[
\overline{r}_{i,j}*a_j^2=c_{i,j}^2=\begin{cases}
    \overline{d}_{\set{1,\dots,j}}^2(a_1^2\dots a_j^2) & \text{if $i=1$},\\
    \overline{d}_{\set{i,\dots,j}}^2(a_1^4\dots a_{-i}^4a_{-i+1}^2\dots a_j^2) & \text{if $i\le -1$}. 
\end{cases}\qedhere
\]
\end{proof}

\begin{corollary}
\thlabel{cor:mhq}
The spin handlebody group $\Modd(H_g)[\phi]$ is isomorphic to $\Z\Braket{a_1^2}\oplus\sfrac{\Z}{2\Z}\Braket{sa_1^2}$ if $g=1$, and is generated by elements $a_1^2,s,\overline{r}_{1,2}$, $t_1$ and $u:=t_1\cdots t_{g-1}$ if $g \ge 2$.
\end{corollary}

\begin{proof}
The expression for $g=1$ is clear. If $g>2$, the subgroup of $\Modd(H_g)[\phi]$ generated by $a_1^2,s,\overline{r}_{1,2},t_1$ and $u$ contains all the $t_i$ since $u*t_i=t_{i+1}$ by \ref{atiti}, hence all the $a_i^2$ and the $\overline{d}_{i,j}$ by \ref{s2}, \ref{sai}, \ref{a8} and \ref{tr}. Since all the $k_j$ and $s_j$ are equal to products of generators $\overline{d}_{i,j}$, $t_k$ and $s$, \ref{trirel} implies that all the $\overline{r}_{i,j}$ are products of the elements in the statement. \qedhere
\end{proof}

\begin{corollary}
The abelianization of the spin handlebody group is the following:
\[
H_1(\Modd(H_g)[\phi];\Z) \cong \begin{cases}
    \Z \oplus \sfrac{\Z}{2\Z} & \text{if $g=1$},\\
    \Z \oplus \sfrac{\Z}{2\Z} \oplus \sfrac{\Z}{2\Z} & \text{if $g=2$},\\
    \sfrac{\Z}{2\Z} & \text{if $g\ge3$}.
\end{cases}
\]
\end{corollary}

\begin{proof}
In the abelianization, all the $a_i^2$ become equal to an element $y$ by \ref{sai}, all the $\overline{d}_{i,j}$ with $i+j \ne 0$ become equal to an element $x$ by \ref{a8} and all the $t_i$ become equal to an element $t$ by \ref{atiti}. Moreover, all the $\overline{d}_{-i,i}$ are equal by \ref{a8}, and by \ref{s2} we get $\overline{d}_{-1,1}=s^2y$. 

By \ref{conjr}a) and \ref{s2}, we have $y=x^{-2}$ and $t^2=x^2$. Moreover, by \ref{riti} all the $\overline{r}_{i,j}$ become equal to products of $t$ and $x$. Now, by \ref{rij2} and \ref{riti} we get $(tx)^2=\overline{r}_{1,2}^2=x^2s^2$, hence $s^2=x^2$. This shows that the abelianization is generated by $t$, $x$ and $s$. 

If $g=2$, the other relations become superfluous.

If $g \ge 3$, we obtain $t=x^3s$ as a consequence of \ref{trirel} for $\overline{r}_{1,3}$ and \ref{riti}. Moreover, from \ref{conjr}f) for $i=1$ and $j=3$ we obtain $x=1$, and this implies that $s^2=1$.  \qedhere 
\end{proof}

\section{Passing to Dehn twist generators}
\label{dt}

In this section, we apply Tietze moves to the presentation of \thref{thm:mcgs} to find a presentation where all the generators are admissible twists.

\subsection{Fake 3-chains}

By a theorem of Gervais \cite{ger}, every relation in the mapping class group can be written in terms of braids, $3$-chains and lanterns. It is easy to see that a single $3$-chain cannot involve only admissible twists. However, this can be fixed via some lantern substitutions. We call the result a \emph{fake $3$-chain}. 

\begin{proposition}
\thlabel{prop:fake}
Let $\gamma_1,\gamma_2,\gamma_3$ be a $3$-chain on a spin surface $\Sigma_g$, $g \ge 3$. If $\phi(\gamma_i)=1$ for some $i$, it is possible to construct an admissible relation from the $3$-chain relation $C(\gamma_1,\gamma_2,\gamma_3)$ by exactly $6$ lantern substitutions (and various braid substitutions). 
\end{proposition}

\begin{proof}
First of all, we reduce to two basic cases. We will often apply tacitly braid substitutions. If $\phi(\gamma_1)=1$ or $\phi(\gamma_3)=1$, we can assume by symmetry that $\phi(\gamma_1)=1$. If $\phi(\gamma_2)=1$, we have
\[
t_{\gamma_1}t_{\gamma_2}t_{\gamma_3}=t_{\gamma_2}t_{\gamma_2}^{-1}t_{\gamma_1}t_{\gamma_2}t_{\gamma_3}=t_{\gamma_2}t_{t_{\gamma_2}^{-1}(\gamma_1)} t_{\gamma_3},
\]
so $C(\gamma_1,\gamma_2,\gamma_3)$ is equivalent modulo braids to $C(t_{\gamma_2}^{-1}(\gamma_1),\gamma_2,\gamma_3)$. Hence, also in this case we can assume that $\phi(\gamma_1)=1$. This allows us to choose freely the value of $\phi(\gamma_2)$ by twisting along $\gamma_1$. 

If $\phi(\gamma_3)=0$, the two boundary components $\delta_1,\delta_2$ of a neighborhood of the $3$-chain are both admissible. Assume that $\phi(\gamma_2)=0$. Notice that we can rewrite the $3$-chain relation as
\[
t_{\gamma_1}^2\,t_{\gamma_2}t_{\gamma_1}^2t_{\gamma_1}\,t_{\gamma_3}t_{\gamma_2}t_{\gamma_1}^2t_{\gamma_2}t_{\gamma_3}=t_{\delta_1}t_{\delta_2}.
\]
By (\ref{eq:all}) we can apply two lantern substitutions for each occurrence of $t_{\gamma_1}^2$ to get an admissible relation.

If $\phi(\gamma_3)=1$, assume that $\phi(\gamma_2)=1$. We can rewrite the $3$-chain relation as
\[
(t_{\gamma_2}t_{\gamma_3}t_{\gamma_1}t_{\gamma_2})^2=t_{\delta_1}t_{\gamma_1}^{-1}t_{\gamma_3}^{-1}\,t_{\delta_2}t_{\gamma_1}^{-1}t_{\gamma_3}^{-1}.
\]
Note that on the right hand side there are two fundamental multitwists, which can be made admissible via a lantern substitution each by (\ref{eq:fml}). For the left hand side, notice that 
\begin{equation}
\label{eq:t}
t_{\gamma_2}t_{\gamma_3}t_{\gamma_1}t_{\gamma_2}=t_{t_{\gamma_1}(\gamma_2)}t_{\gamma_1}^2t_{t_{\gamma_3}(\gamma_2)}.
\end{equation}
Using again (\ref{eq:all}) we conclude.

Observe that in both cases all lanterns have the same sign. \qedhere
\end{proof}

We work out an example in detail.

\begin{example}
\thlabel{ex:fake}
Consider the $3$-chain $\gamma_1,\gamma_2,\gamma_3$ in Figure \ref{fig:fake}. We have the following lantern relators:
\begin{gather*}
L_1:=t_{y_1}t_{z_1}t_{\delta_1}t_{\gamma_1}^{-1}t_{\gamma_3}^{-1}t_{\varepsilon_1}^{-1}t_{\varepsilon_2}^{-1},\\
L_2:=t_{y_2}t_{z_2}t_{\delta_2}t_{\gamma_1}^{-1}t_{\gamma_3}^{-1}t_{\varepsilon_1}^{-1}t_{\varepsilon_3}^{-1}, \\
L_3:=t_{y_3}t_{z_3}t_{\gamma_3}t_{\gamma_1}^{-1}t_{\delta_2}^{-1}t_{y_1}^{-1}t_{\varepsilon_1}^{-1}.
\end{gather*}
Here, $z_1,z_2$ and $z_3$ can be determined from the other curves, and are easily seen to be admissible. Taking the product of the inverses of $L_2$ and $L_3$, we get the relator
\[
A:=t_{\varepsilon_3}t_{\varepsilon_1}t_{z_2}^{-1}t_{y_2}^{-1}t_{\gamma_1}^2t_{z_3}^{-1}t_{y_3}^{-1}t_{\varepsilon_1}t_{y_2}.
\]
Since $L_1$, $L_2$ and $A$ contain as subwords $t_{\delta_1}t_{\gamma_1}^{-1}t_{\gamma_3}^{-1}$, $t_{\delta_2}t_{\gamma_1}^{-1}t_{\gamma_3}^{-1}$ and $t_{\gamma_1}^2$ respectively, and all the other twists that appear are admissible, we can plug some conjugates of $L_1$, $L_2$ and $A^{-1}$ in $C:=C(\gamma_1,\gamma_2,\gamma_3)$ so that all non-admissible twist get canceled out, and we get the fake $3$-chain relator
\[
F:=t_{\varepsilon_1}^{-1}t_{\varepsilon_2}^{-1}t_{y_1}t_{z_1}
\,\big(t_{t_{\gamma_1}(\gamma_2)} 
t_{y_2}t_{z_2}t_{\varepsilon_1}^{-1}t_{\varepsilon_3}^{-1}t_{y_2}^{-1}t_{\varepsilon_1}^{-1}t_{y_3}t_{z_3}
t_{t_{\gamma_3}(\gamma_2)}\big)^2\,
t_{\varepsilon_1}^{-1}t_{\varepsilon_3}^{-1}t_{y_2}t_{z_2}.
\]
\end{example}

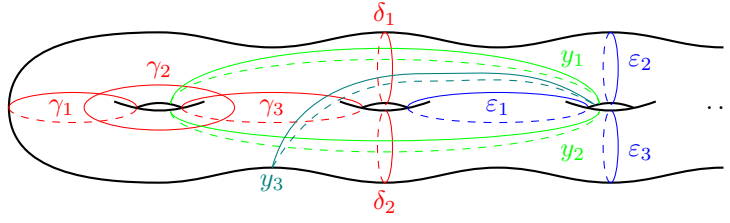
\begin{figure}
\centering
\begin{tikzpicture}
\draw [thick] (0,0) to [out=90, in=-180] (2,1) to [out=0, in=180] (3.5,.8) to [out=0, in=180] (5,1) to [out=0, in=180] (6.5,.8) to [out=0, in=180] (8,1) to [out=0, in=180] (9.5,.8);
\draw [thick] (0,0) to [out=-90, in=180] (2,-1) to [out=0, in=-180] (3.5,-.8) to [out=0, in=180] (5,-1) to [out=0, in=180] (6.5,-.8) to [out=0, in=180] (8,-1) to [out=0, in=180] (9.5,-.8);
\draw [thick] (1.4,.08) to [out=-20, in=-160] (2.6,.08);
\draw [thick] (1.7,0) to [out=20, in=160] (2.3,0);
\draw [red] (1,0) arc (-180:180:1 and .3);
\node [red] at (2,.5) {$\gamma_2$};
\draw [red] (0,0) arc (180:0:.85 and .2);
\draw [red, dashed] (0,0) arc (-180:0:.85 and .2);
\node [red] at (.7,0) {$\gamma_1$};
\draw [red] (2.3,0) arc (180:0:1.2 and .2);
\draw [red, dashed] (2.3,0) arc (-180:0:1.2 and .2);
\node [red] at (3.5,0) {$\gamma_3$};
\draw [thick] (4.4,.08) to [out=-20, in=-160] (5.6,.08);
\draw [thick] (4.7,0) to [out=20, in=160] (5.3,0);
\draw [red] (5,1) arc (90:-90:.1 and .48);
\draw [red, dashed] (5,1) arc (90:270:.1 and .48);
\node [red] at (5,1.25) {$\delta_1$};
\draw [red] (5,-1) arc (-90:90:.1 and .48);
\draw [red, dashed] (5,-1) arc (270:90:.1 and .48);
\node [red] at (5,-1.25) {$\delta_2$};
\draw [blue] (5.3,0) arc (180:0:1.2 and .2);
\draw [blue, dashed] (5.3,0) arc (-180:0:1.2 and .2);
\node [blue] at (6.5,0) {$\varepsilon_1$};
\draw [thick] (7.4,.08) to [out=-20, in=-160] (8.6,.08);
\draw [thick] (7.7,0) to [out=20, in=160] (8.3,0);
\draw [blue] (8,1) arc (90:-90:.1 and .48);
\draw [blue, dashed] (8,1) arc (90:270:.1 and .48);
\node [blue] at (8.4,.6) {$\varepsilon_2$};
\draw [blue] (8,-1) arc (-90:90:.1 and .48);
\draw [blue, dashed] (8,-1) arc (270:90:.1 and .48);
\node [blue] at (8.4,-.6) {$\varepsilon_3$};
\draw [green] (2.15,-.05) arc(181:359:2.85 and .4);
\draw [green, dashed] (2.15,-.05) arc(181:359:2.85 and .55);
\node [green] at (7.5,-.6) {$y_2$};
\draw [green] (2.15,.05) arc(179:1:2.85 and .75);
\draw [green, dashed] (2.15,.05) arc(179:1:2.85 and .6);
\node [green] at (7.5,.65) {$y_1$};
\draw [teal] (3.5,-.8) to [out=75, in=180] (5.5,.45) to [out=0, in=145] (7.8,.03);
\draw [teal, dashed] (3.5,-.8) to [out=60, in=180] (5.5,.35) to [out=0, in=150] (7.8,.03);
\node [teal] at (3.5,-1) {$y_3$};
\node at (9.5,0) {$\dots$};
\end{tikzpicture}
\caption{Curves involved in a fake $3$-chain relation. Here blue and green curves are admissible and red curves are not admissible.}
\label{fig:fake}
\end{figure}

\subsection{First relations in the new generators}

Denote by $b_1,\dots,b_g$, $\xi_1,\dots,\xi_{g-1}$, $\eta_2,\dots,\eta_g$ the Dehn twists along the corresponding curves in Figure \ref{fig:gens}. This will be our new generating set. Let $\phi$ be the unique spin structure on $\Sigma_g^1$ such that all the curves of Figure \ref{fig:gens} are admissible (not only the blue ones).

\begin{remark}
Hamenstädt's system of generating twist for $\Modd(\Sigma_g)[\phi]$ \cite{ham} is different than ours for $g \ge 5$, and has a smaller cardinality. Indeed, our generating set is not an “admissible curve-system" in the sense of \cite[Definition 1.2]{ham}, as its intersection graph is not a tree.
\end{remark} 

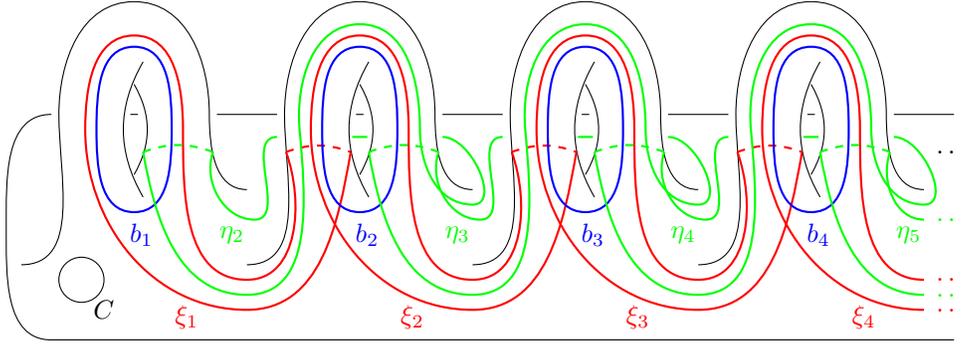
\begin{figure}
\centering
\begin{tikzpicture}
\draw (.5,0) to [out=0, in=-90] (1,2) to [out=90, in=180] (2,3.5) to [out=0, in=90] (3,2) to [out=-90, in=180] (3.5,1);
\draw (2,2.4) to [out=-60, in=60] (2,1.2);
\draw (2.12,2.7) to [out=-120, in=120] (2.12,.9);
\draw (3.5,0) to [out=0, in=-90] (4,2) to [out=90, in=180] (5,3.5) to [out=0, in=90] (6,2) to [out=-90, in=180] (6.5,1);
\draw (5,2.4) to [out=-60, in=60] (5,1.2);
\draw (5.12,2.7) to [out=-120, in=120] (5.12,.9);
\draw (6.5,0) to [out=0, in=-90] (7,2) to [out=90, in=180] (8,3.5) to [out=0, in=90] (9,2) to [out=-90, in=180] (9.5,1);
\draw (8,2.4) to [out=-60, in=60] (8,1.2);
\draw (8.12,2.7) to [out=-120, in=120] (8.12,.9);
\draw (9.5,0) to [out=0, in=-90] (10,2) to [out=90, in=180] (11,3.5) to [out=0, in=90] (12,2) to [out=-90, in=180] (12.5,1);
\draw (11,2.4) to [out=-60, in=60] (11,1.2);
\draw (11.12,2.7) to [out=-120, in=120] (11.12,.9);
\node at (12.9,1.5) {$\dots$};
\draw (.9,2) to [out=180, in=90] (.3,.5) to [out=-90, in=180] (.9,-1) to [out=0, in=180] (12.9,-1);
\draw (1.95,2) -- (2.05,2);
\draw (3.1,2) -- (3.9,2);
\draw (4.95,2) -- (5.05,2);
\draw (6.1,2) -- (6.9,2);
\draw (7.95,2) -- (8.05,2);
\draw (9.1,2) -- (9.9,2);
\draw (10.95,2) -- (11.05,2);
\draw (12.1,2) -- (12.9,2);
\draw [blue, thick] (2,2.9) to [out=0, in=90] (2.5,1.8) to [out=-90, in=0] (2,.7) to [out=180, in=-90] (1.5,1.8) to [out=90, in=180] (2,2.9);
\draw [red, thick, dashed] (4.02,1.5) to [out=20, in=160] (4.88,1.5);
\draw [red, thick] (4.02,1.5) to [out=-70, in=0] (3.5,-.2) to [out=180, in=-90] (2.65,1.8) to [out=90, in=0] (2,3.05) to [out=180, in=90] (1.35,1.8) to [out=-90, in=180] (3.5,-.6) to [out=0,  in=-100] (4.88,1.5);
\draw [green, thick, dashed] (2.12,1.5) to [out=20, in=160] (3.05,1.5);
\draw [green, thick] (2.12,1.5) to [out=-80, in=180] (3.5,-.4) to [out=0, in=-85] (4.2,1.8) to [out=95, in=180] (5,3.2) to [out=0, in=90] (5.8,1.8) to [out=-90, in=180] (6.5,.8) to [out=0, in=0] (6.1,1.7);
\draw [green, thick] (4.9,1.7) -- (5.1,1.7);
\draw [green, thick] (3.9,1.7) to [out=180, in=0] (3.6,.6) to [out=180,in=-110] (3.05,1.5);
\draw [blue, thick] (5,2.9) to [out=0, in=90] (5.5,1.8) to [out=-90, in=0] (5,.7) to [out=180, in=-90] (4.5,1.8) to [out=90, in=180] (5,2.9);
\draw [red, thick, dashed] (7.02,1.5) to [out=20, in=160] (7.88,1.5);
\draw [red, thick] (7.02,1.5) to [out=-70, in=0] (6.5,-.2) to [out=180, in=-90] (5.65,1.8) to [out=90, in=0] (5,3.05) to [out=180, in=90] (4.35,1.8) to [out=-90, in=180] (6.5,-.6) to [out=0,  in=-100] (7.88,1.5);
\draw [green, thick, dashed] (5.12,1.5) to [out=20, in=160] (6.05,1.5);
\draw [green, thick] (5.12,1.5) to [out=-80, in=180] (6.5,-.4) to [out=0, in=-85] (7.2,1.8) to [out=95, in=180] (8,3.2) to [out=0, in=90] (8.8,1.8) to [out=-90, in=180] (9.5,.8) to [out=0, in=0] (9.1,1.7);
\draw [green, thick] (7.9,1.7) -- (8.1,1.7);
\draw [green, thick] (6.9,1.7) to [out=180, in=0] (6.6,.6) to [out=180,in=-110] (6.05,1.5);
\draw [blue, thick] (8,2.9) to [out=0, in=90] (8.5,1.8) to [out=-90, in=0] (8,.7) to [out=180, in=-90] (7.5,1.8) to [out=90, in=180] (8,2.9);
\draw [red, thick, dashed] (10.02,1.5) to [out=20, in=160] (10.88,1.5);
\draw [red, thick] (10.02,1.5) to [out=-70, in=0] (9.5,-.2) to [out=180, in=-90] (8.65,1.8) to [out=90, in=0] (8,3.05) to [out=180, in=90] (7.35,1.8) to [out=-90, in=180] (9.5,-.6) to [out=0,  in=-100] (10.88,1.5);
\draw [green, thick, dashed] (8.12,1.5) to [out=20, in=160] (9.05,1.5);
\draw [green, thick] (8.12,1.5) to [out=-80, in=180] (9.5,-.4) to [out=0, in=-85] (10.2,1.8) to [out=95, in=180] (11,3.2) to [out=0, in=90] (11.8,1.8) to [out=-90, in=180] (12.5,.8) to [out=0, in=0] (12.1,1.7);
\draw [green, thick] (10.9,1.7) -- (11.1,1.7);
\draw [green, thick] (9.9,1.7) to [out=180, in=0] (9.6,.6) to [out=180,in=-110] (9.05,1.5);
\draw [blue, thick] (11,2.9) to [out=0, in=90] (11.5,1.8) to [out=-90, in=0] (11,.7) to [out=180, in=-90] (10.5,1.8) to [out=90, in=180] (11,2.9);
\draw [red, thick] (12.5,-.2) to [out=180, in=-90] (11.65,1.8) to [out=90, in=0] (11,3.05) to [out=180, in=90] (10.35,1.8) to [out=-90, in=180] (12.5,-.6);
\draw [green, thick, dashed] (11.12,1.5) to [out=20, in=160] (12.05,1.5);
\draw [green, thick] (11.12,1.5) to [out=-80, in=180] (12.5,-.4);
\draw [green, thick] (12.5,.6) to [out=180,in=-110] (12.05,1.5);
\node [red] at (12.9,-.2) {$\dots$};
\node [red] at (12.9,-.6) {$\dots$};
\node [green] at (12.9,.6) {$\dots$};
\node [green] at (12.9,-.4) {$\dots$};
\node [blue] at (2.1,.4) {$b_1$};
\node [red] at (2.7,-.7) {$\xi_1$};
\node [green] at (3.3,.4) {$\eta_2$};
\node [blue] at (5.1,.4) {$b_2$};
\node [red] at (5.7,-.7) {$\xi_2$};
\node [green] at (6.3,.4) {$\eta_3$};
\node [blue] at (8.1,.4) {$b_3$};
\node [red] at (8.7,-.7) {$\xi_3$};
\node [green] at (9.3,.4) {$\eta_4$};
\node [blue] at (11.1,.4) {$b_4$};
\node [red] at (11.7,-.7) {$\xi_4$};
\node [green] at (12.3,.4) {$\eta_5$};
\draw (1.6,-.2) arc (0:360:.3 and .3);
\node at (1.6,-.6) {$C$};
\end{tikzpicture}
\caption{Generators $b_1,\dots,b_g$, $\xi_1,\dots,\xi_{g-1}$ and $\eta_2,\dots,\eta_g$. All the colored curves in the picture are admissible.}
\label{fig:gens}
\end{figure}

We first express the generators $\overline{d}_{1,2}$ and $a_1^2$ as products of admissible twists, using explicit embeddings of the subsurface of Figure \ref{fig:lants} in $\Sigma_g^1$. 

\begin{lemma}
\thlabel{lem:da}
Consider the following elements of $\Modd(\Sigma_g^1)[\phi]$ (see Figure \ref{fig:lanterns}):
\begin{gather*}
H_3:=\xi_2\xi_1b_2\eta_3\eta_2b_1^2\eta_2\eta_3b_2\xi_1\xi_2, \quad m_1:=H_3*b_3, \\
m_2:=(\eta_2\eta_3)*b_2, \quad m_3:=(b_1\eta_2\eta_3b_2)*\xi_1, \quad m_4:=(\xi_1b_2\eta_3\eta_2)*b_1.
\end{gather*}
The following relations hold in $\Modd(\Sigma_g^1)[\phi]$:
\begin{enumerate}[label=\textit{(\roman*)}]
\item $\overline{d}_{1,2}=m_1m_2\big((\eta_4m_1m_2\eta_4)*m_3^{-1}\big)m_3^{-1}$;\label{d12}
\item $\overline{d}_{1,2}^{-1}a_1^{-2}=\eta_3b_3\big((\eta_4\eta_3b_3\eta_4)*m_4^{-1}\big)m_4^{-1}$;
\item $a_1^2=m_4\big((\eta_4\eta_3b_3\eta_4)*m_4\big)\eta_3^{-1}b_3^{-1}m_3\big((\eta_4m_1m_2\eta_4)*m_3\big)m_1^{-1}m_2^{-1}$.\label{a2da}
\end{enumerate}
\end{lemma}

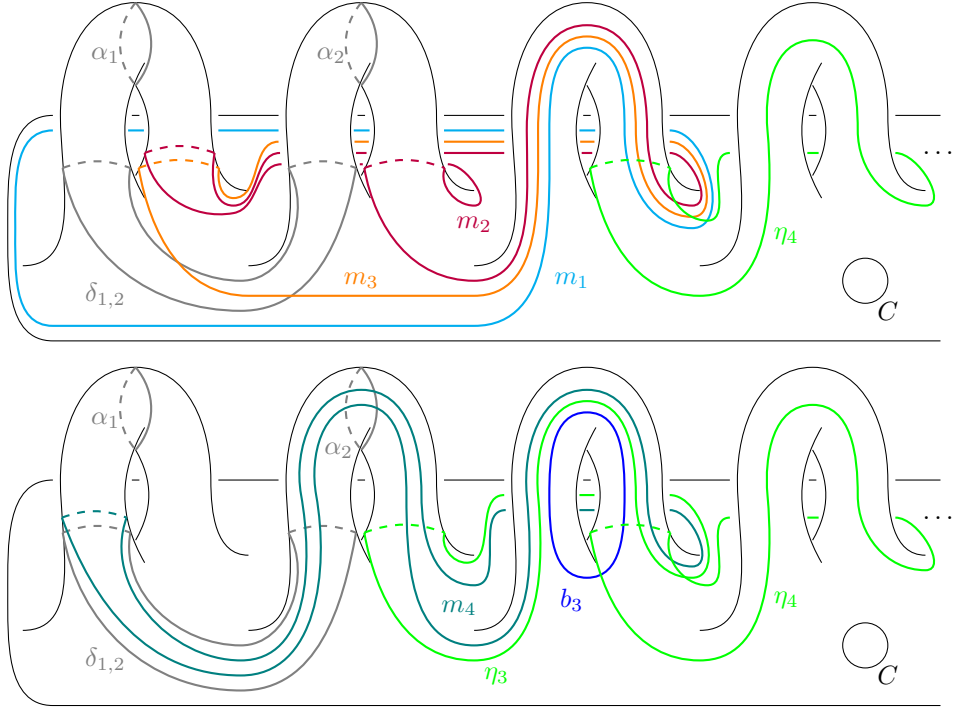
\begin{figure}
\centering
\begin{tikzpicture}
\draw (.5,0) to [out=0, in=-90] (1,2) to [out=90, in=180] (2,3.5) to [out=0, in=90] (3,2) to [out=-90, in=180] (3.5,1);
\draw (2,2.4) to [out=-60, in=60] (2,1.2);
\draw (2.12,2.7) to [out=-120, in=120] (2.12,.9);
\draw (3.5,0) to [out=0, in=-90] (4,2) to [out=90, in=180] (5,3.5) to [out=0, in=90] (6,2) to [out=-90, in=180] (6.5,1);
\draw (5,2.4) to [out=-60, in=60] (5,1.2);
\draw (5.12,2.7) to [out=-120, in=120] (5.12,.9);
\draw (6.5,0) to [out=0, in=-90] (7,2) to [out=90, in=180] (8,3.5) to [out=0, in=90] (9,2) to [out=-90, in=180] (9.5,1);
\draw (8,2.4) to [out=-60, in=60] (8,1.2);
\draw (8.12,2.7) to [out=-120, in=120] (8.12,.9);
\draw (9.5,0) to [out=0, in=-90] (10,2) to [out=90, in=180] (11,3.5) to [out=0, in=90] (12,2) to [out=-90, in=180] (12.5,1);
\draw (11,2.4) to [out=-60, in=60] (11,1.2);
\draw (11.12,2.7) to [out=-120, in=120] (11.12,.9);
\node at (12.7,1.5) {$\dots$};
\draw (.9,2) to [out=180, in=90] (.3,.5) to [out=-90, in=180] (.9,-1) to [out=0, in=180] (12.7,-1);
\draw (1.95,2) -- (2.05,2);
\draw (3.1,2) -- (3.9,2);
\draw (4.95,2) -- (5.05,2);
\draw (6.1,2) -- (6.9,2);
\draw (7.95,2) -- (8.05,2);
\draw (9.1,2) -- (9.9,2);
\draw (10.95,2) -- (11.05,2);
\draw (12.1,2) -- (12.7,2);
\draw [gray, thick] (2,3.5) to [out=-50, in=50] (2,2.4);
\draw [gray, thick, dashed] (2,3.5) to [out=-130, in=130] (2,2.4);
\draw [gray, thick] (5,3.5) to [out=-50, in=50] (5,2.4);
\draw [gray, thick, dashed] (5,3.5) to [out=-130, in=130] (5,2.4);
\draw [gray, thick, dashed] (1.03,1.3) to [out=20, in=160] (1.93,1.3);
\draw [gray, thick, dashed] (4.03,1.3) to [out=20, in=160] (4.93,1.3);
\draw [orange, thick, dashed] (2.04,1.3) to [out=20, in=160] (3.11,1.3);
\draw [gray, thick] (4.03,1.3) to [out=-60, in=0] (3.4,-.2) to [out=180, in=-100] (1.93,1.3);
\draw [gray, thick] (1.03,1.3) to [out=-80, in=180] (3.4,-.6) to [out=0,  in=-100] (4.93,1.3);
\draw [orange, thick] (3.9,1.65) to [out=180, in=0] (3.3,.9) to [out=180,in=-100] (3.11,1.3);
\draw [purple, thick] (3.92,1.5) to [out=180, in=0] (3.3,.8) to [out=180,in=-100] (3.05,1.5);
\draw [purple, thick] (3.94,1.35) to [out=180, in=10] (3.3,.7) to [out=190,in=-80] (2.12,1.5);
\draw [cyan, thick] (1.9,1.8) -- (2.11,1.8);
\draw [cyan, thick] (3.1,1.8) -- (3.9,1.8);
\draw [cyan, thick] (4.9,1.8) -- (5.11,1.8);
\draw [orange, thick] (4.91,1.65) -- (5.1,1.65);
\draw [purple, thick] (4.93,1.5) -- (5.07,1.5);
\draw [purple, thick] (4.98,1.35) -- (5.02,1.35);
\draw [cyan, thick] (6.1,1.8) -- (6.9,1.8);
\draw [orange, thick] (6.1,1.65) -- (6.9,1.65);
\draw [purple, thick] (6.1,1.5) -- (6.9,1.5);
\draw [cyan, thick] (7.9,1.8) -- (8.11,1.8);
\draw [orange, thick] (7.91,1.65) -- (8.1,1.65);
\draw [purple, thick] (7.93,1.5) -- (8.07,1.5);
\draw [cyan, thick]  (9.1,1.8) to [out=0, in=0] (9.4,.5) to [out=180, in=-90] (8.5,1.8) to [out=90, in=0] (8,2.9) to [out=180, in=90] (7.5,1.8) to [out=-90, in=0] (6.5,-.8) -- (.9,-.8) to [out=180, in=-90] (.4,.5) to [out=90, in=180] (.9,1.8);
\draw [orange, thick] (9.1,1.65) to [out=0, in=0] (9.4,.65) to [out=180, in=-90] (8.65,1.8) to [out=90, in=0] (8,3.05) to [out=180, in=90] (7.35,1.8) to [out=-90, in=0] (6.5,-.4) -- (3.5,-.4) to [out=180, in=-80] (2.04,1.3);
\draw [purple, thick, dashed] (2.12,1.5) to [out=20, in=160] (3.05,1.5);
\draw [purple, thick] (5.04,1.3) to [out=-80, in=180] (6.5,-.2) to [out=0, in=-90] (7.2,1.8) to [out=90, in=180] (8,3.2) to [out=0, in=90] (8.8,1.8) to [out=-90, in=180] (9.4,.8) to [out=0, in=0] (9.1,1.5);
\draw [purple, thick, dashed] (5.04,1.3) to [out=20, in=160] (6.11,1.3);
\draw [purple, thick] (6.15,1.35) to [out=0, in=0] (6.5,.8) to [out=180,in=-110] (6.11,1.3);
\draw [green, thick, dashed] (8.04,1.3) to [out=20, in=160] (9.11,1.3);
\draw [green, thick] (8.04,1.3) to [out=-80, in=180] (9.5,-.4) to [out=0, in=-85] (10.4,1.8) to [out=95, in=180] (11,3) to [out=0, in=90] (11.6,1.8) to [out=-90, in=180] (12.5,.8) to [out=0, in=0] (12.1,1.5);
\draw [green, thick] (10.93,1.5) -- (11.08,1.5);
\draw [green, thick] (9.9,1.5) to [out=180, in=0] (9.6,.6) to [out=180,in=-110] (9.11,1.3);
\node [green] at (10.65,.4) {$\eta_4$};
\node [cyan] at (7.8,-.2) {$m_1$};
\node [orange] at (5,-.2) {$m_3$};
\node [purple] at (6.5,.55) {$m_2$};
\node [gray] at (1.6,-.4) {$\delta_{1,2}$};
\node [gray] at (1.6,2.8) {$\alpha_1$};
\node [gray] at (4.6,2.8) {$\alpha_2$};
\draw (12,-.2) arc (0:360:.3 and .3);
\node at (12,-.6) {$C$};
\end{tikzpicture}

\bigskip
\begin{tikzpicture}
\draw (.5,0) to [out=0, in=-90] (1,2) to [out=90, in=180] (2,3.5) to [out=0, in=90] (3,2) to [out=-90, in=180] (3.5,1);
\draw (2,2.4) to [out=-60, in=60] (2,1.2);
\draw (2.12,2.7) to [out=-120, in=120] (2.12,.9);
\draw (3.5,0) to [out=0, in=-90] (4,2) to [out=90, in=180] (5,3.5) to [out=0, in=90] (6,2) to [out=-90, in=180] (6.5,1);
\draw (5,2.4) to [out=-60, in=60] (5,1.2);
\draw (5.12,2.7) to [out=-120, in=120] (5.12,.9);
\draw (6.5,0) to [out=0, in=-90] (7,2) to [out=90, in=180] (8,3.5) to [out=0, in=90] (9,2) to [out=-90, in=180] (9.5,1);
\draw (8,2.4) to [out=-60, in=60] (8,1.2);
\draw (8.12,2.7) to [out=-120, in=120] (8.12,.9);
\draw (9.5,0) to [out=0, in=-90] (10,2) to [out=90, in=180] (11,3.5) to [out=0, in=90] (12,2) to [out=-90, in=180] (12.5,1);
\draw (11,2.4) to [out=-60, in=60] (11,1.2);
\draw (11.12,2.7) to [out=-120, in=120] (11.12,.9);
\node at (12.7,1.5) {$\dots$};
\draw (.9,2) to [out=180, in=90] (.3,.5) to [out=-90, in=180] (.9,-1) to [out=0, in=180] (12.7,-1);
\draw (1.95,2) -- (2.05,2);
\draw (3.1,2) -- (3.9,2);
\draw (4.95,2) -- (5.05,2);
\draw (6.1,2) -- (6.9,2);
\draw (7.95,2) -- (8.05,2);
\draw (9.1,2) -- (9.9,2);
\draw (10.95,2) -- (11.05,2);
\draw (12.1,2) -- (12.7,2);
\draw [gray, thick] (2,3.5) to [out=-50, in=50] (2,2.4);
\draw [gray, thick, dashed] (2,3.5) to [out=-130, in=130] (2,2.4);
\draw [gray, thick] (5,3.5) to [out=-50, in=50] (5,2.4);
\draw [gray, thick, dashed] (5,3.5) to [out=-130, in=130] (5,2.4);
\draw [gray, thick, dashed] (1.03,1.3) to [out=20, in=160] (1.93,1.3);
\draw [gray, thick, dashed] (4.03,1.3) to [out=20, in=160] (4.93,1.3);
\draw [gray, thick] (4.03,1.3) to [out=-60, in=0] (3.4,-.2) to [out=180, in=-100] (1.93,1.3);
\draw [gray, thick] (1.03,1.3) to [out=-80, in=180] (3.4,-.8) to [out=0,  in=-100] (4.93,1.3);
\draw [blue, thick] (8,2.9) to [out=0, in=90] (8.5,1.8) to [out=-90, in=0] (8,.7) to [out=180, in=-90] (7.5,1.8) to [out=90, in=180] (8,2.9);
\draw [teal, thick, dashed] (1.02,1.5) to [out=20, in=160] (1.88,1.5);
\draw [teal, thick] (1.02,1.5) to [out=-70, in=180] (3.4,-.6) to [out=0, in=-80] (4.4,1.8) to [out=100, in=180] (5,3) to [out=0, in=90] (5.6,1.8) to  [out=-90, in=180] (6.5,-.2) to [out=0, in=-85] (7.2,1.8) to [out=95, in=180] (8,3.2) to [out=0, in=90] (8.8,1.8) to [out=-90, in=180] (9.4,.85) to [out=0, in=0] (9.1,1.6);
\draw [teal, thick] (7.9,1.6) -- (8.1,1.6);
\draw [teal, thick] (6.92,1.6) to [out=180, in=0] (6.5,.6) to [out=180,in=-90] (5.8,1.8) to [out=90, in=0] (5,3.2) to [out=180, in=100] (4.2,1.8) to [out=-80, in=0] (3.4,-.4) to [out=180, in=-110] (1.88,1.5);
\draw [green, thick, dashed] (5.04,1.3) to [out=20, in=160] (6.11,1.3);
\draw [green, thick] (5.04,1.3) to [out=-80, in=180] (6.5,-.4) to [out=0, in=-85] (7.35,1.8) to [out=95, in=180] (8,3.05) to [out=0, in=90] (8.65,1.8) to [out=-90, in=180] (9.4,.7) to [out=0, in=0] (9.1,1.8);
\draw [green, thick] (7.9,1.8) -- (8.1,1.8);
\draw [green, thick] (6.9,1.8) to [out=180, in=0] (6.4,.9) to [out=180,in=-110] (6.11,1.3);
\draw [green, thick, dashed] (8.04,1.3) to [out=20, in=160] (9.11,1.3);
\draw [green, thick] (8.04,1.3) to [out=-80, in=180] (9.5,-.4) to [out=0, in=-85] (10.4,1.8) to [out=95, in=180] (11,3) to [out=0, in=90] (11.6,1.8) to [out=-90, in=180] (12.5,.8) to [out=0, in=0] (12.1,1.5);
\draw [green, thick] (10.93,1.5) -- (11.08,1.5);
\draw [green, thick] (9.9,1.5) to [out=180, in=0] (9.6,.6) to [out=180,in=-110] (9.11,1.3);
\node [green] at (10.65,.4) {$\eta_4$};
\node [gray] at (1.6,-.4) {$\delta_{1,2}$};
\node [gray] at (1.6,2.8) {$\alpha_1$};
\node [gray] at (4.7,2.4) {$\alpha_2$};
\node [teal] at (6.3,.3) {$m_4$};
\node [green] at (6.8,-.6) {$\eta_3$};
\node [blue] at (7.8,.4) {$b_3$};
\draw (12,-.2) arc (0:360:.3 and .3);
\node at (12,-.6) {$C$};
\end{tikzpicture}
\caption{Lanterns of \thref{lem:da}. Here, the gray curves have spin value $1$, and the other curves are admissible.}
\label{fig:lanterns}
\end{figure}

\begin{proof}
Relation (iii) is an immediate consequence of (i) and (ii), which are the lantern relations depicted in Figure \ref{fig:lanterns}, and are true in $\Modd(\Sigma_g^1)[\phi]$ by \thref{thm:mcgs}.
\qedhere
\end{proof}

We are now ready to state the presentation. We will use the symbol $R_i(x,y)$ to indicate that elements $x$ and $y$ satisfy an Artin relation of length $i$, i.e. that the words $xyxy\dots$ and $yxyx\dots$ of length $i$ are equal. We will only encounter relations of length $2$, $3$ or $4$. Moreover, to simplify the exposition, we will use various shorthands that have already appeared in the above.

\begin{remark}
\thlabel{rem:sh}
From now on, the symbols $\overline{d}_{1,2}$ and $a_1^2$ will always be used as shorthands for the products of \thref{lem:da}\ref{d12} and \ref{a2da}. We will also use the following shorthands:
\begin{itemize}
\item $s_1:=b_1a_1^2b_1$ and $t_1:=(b_1^{-1}\eta_2b_1)a_1^2(b_2^{-1}\xi_1b_2)$;
\item $a_{i+1}^2:=t_ia_i^2t_i^{-1}$, $s_{i+1}=b_ia_{i+1}^2b_i$ and 
\begin{equation}
\label{tiNE}
t_{i+1}:=(b_{i+1}^{-1}\eta_{i+2}b_{i+1}) a_{i+1}^2(b_{i+2}^{-1}\xi_{i+1}b_{i+2})
\end{equation}
for all $i \ge 1$;
\item $\overline{d}_{-1,1}:=s_1^2a_1^2$, and
\begin{equation}
\label{eq:dijdij}
\overline{d}_{i,j}:=\begin{cases}
\big(t_{i-1}t_{i-2}\dots t_1t_{j-1}t_{j-2}\dots t_2\big)*\overline{d}_{1,2} & \text{if $i>0$},\\
\big(t_{-i-1}^{-1}\dots t_1^{-1}s_1^{-1}t_{j-1}\dots t_2\big)*\overline{d}_{1,2} & \text{if $0<-i<j$},\\
\big(t_{-i-1}^{-1}\dots t_1^{-1}s_1^{-1}t_j\dots t_2\big)*\overline{d}_{1,2} & \text{if $0<j<-i$},\\
\big(t_{-j-1}^{-1}\dots t_1^{-1}t_{-i-1}^{-1}\dots t_2^{-1}s_1^{-1}t_1^{-1}s_1^{-1}\big)*\overline{d}_{1,2} & \text{if $j<0$},\\
\big(t_{j-1}^{-1}\overline{d}_{j-1,j}t_{j-2}^{-1}\overline{d}_{j-2,j-1}\dots t_1^{-1}\overline{d}_{1,2}\big)*\overline{d}_{-1,1} & \text{if $i+j=0$};
\end{cases}
\end{equation}
\item $\overline{d}_{\set{i_1,\dots,i_n}}:=(\overline{d}_{i_1,i_2}\overline{d}_{i_1,i_3}\dots \overline{d}_{i_1,i_n}\overline{d}_{i_2,i_3}\dots \overline{d}_{i_2,i_n}\dots \overline{d}_{i_{n-1},i_n})$;
\item $\overline{r}_{i,j}:=b_ja_j^2\overline{d}_{\set{i,\dots,j}}b_j$.
\end{itemize}
\end{remark}

\begin{theorem}
\thlabel{thm:presD}
If $g\ge4$, the spin mapping class group $\Modd(\Sigma_g^1)[\phi]$ admits a presentation with generators $b_1,\dots,b_g$, $\xi_1,\dots,\xi_{g-1}$, $\eta_2,\dots,\eta_g$ and the following relations:
\begin{enumerate}[label=\textit{(S\arabic*)}, ref=(S\arabic*), series=drel]
\item two generators satisfy $R_2$ or $R_3$ if the corresponding curves are disjoint or intersect once; \label{t01}
\item the $5$-chain $(b_1\eta_2\eta_3b_2\xi_1)^6=b_3m_1$;\label{5c}
\item the hyperelliptic relation 
\[
(b_3\xi_2\xi_1b_2\eta_3\eta_2b_1^2\eta_2\eta_3b_2\xi_1\xi_2b_3)^2=(b_1\eta_2\eta_3b_2\xi_1\xi_2)^{14};\label{hypg}
\]
\item the fake $3$-chain relations $t_1^2=\overline{d}_{1,2}\overline{d}_{-2,-1}$, 
\[
a_1^2(b_1a_1^2b_1)(\eta_2b_1a_1^2b_1\eta_2)=b_2H_2b_2H_2^{-1}, 
\]
where $H_2:=\eta_3\eta_2b_1a_1^2b_1\eta_2\eta_3$, and \label{3c}
\begin{align}
\begin{split}
\label{eq:3c}
\overline{r}_{i,j}^2&=a_1^{-4}\dots a_{-i}^{-4}a_{-i+1}^{-2}\dots a_{j}^{-2} \overline{d}_{\set{i,\dots,j}}^{-2}\cdot\\
&\qquad\qquad \cdot\overline{d}_{\set{i,\dots,j-1}}\big((t_{j-1}\dots t_{-i+1})*\overline{d}_{\set{i-1,\dots,j}}\big)
\end{split}
\end{align}
for all $i,j$ such that $i=1$ or $i<0$, $j+i>0$ and $j-i<g$;
\item $R_3(m_1,\eta_4)$, $[\eta_4,(b_3\xi_2)*m_4]=1$ and $[(\eta_3^{-1}\eta_4^{-1})*m_4,b_1^{-1}*\eta_2]=1$; \label{boh} 
\item $(\eta_4m_1m_2\eta_4)*m_3$ commutes with $b_3$, $\eta_3$ and $\eta_5$, while $(\eta_4\eta_3b_3\eta_4)*m_4$ commutes with $m_1$, $m_2$ and $\eta_5$; \label{lt}
\item $[m_3 \cdot (\eta_4m_1m_2\eta_4)*m_3,m_4\cdot ((\eta_4\eta_3b_3\eta_4)*m_4)]=1$; \label{lad}
\item $\overline{d}_{i,i+1}*b_i=(a_i^{-2}s_i^{-1})*\xi_i$ and $(\overline{d}_{1,2}b_1^{-1})*\eta_2=(a_2^{-2}b_2^{-1})*\xi_1$; \label{dbar}
\item $[a_i^2, (\xi_ib_i^{-1}b_{i+1}\eta_{i+2})*\eta_{i+1}]=1$ and $[\overline{d}_{i-1,i},(\eta_{i}\eta_{i+1})*b_{i}]=1$; \label{fore1}
\item $[\overline{d}_{1,2},a_3^2]=1$, $[\overline{d}_{1,2},\overline{d}_{3,4}]=1$, $[\overline{d}_{1,2},\overline{d}_{-3,-1}]=1$ and $[\overline{d}_{1,2},t_2\overline{d}_{1,2}t_2]=1$; \label{d1234}
\item $R_4(a_1^2,b_1)$, $[b_1*a_1^2,\xi_1*a_1^2]=1$ and $a_1^2$ commutes with $b_2$, $\xi_2$ and $\eta_2$;\label{a2}
\item $b_2 \big((b_1^{-1}\eta_2^{-1})*\overline{d}_{1,2} \overline{d}_{-2,-1}^{-1}\big)=\overline{d}_{-1,1,2}a_1^2 b_2$, $\overline{d}_{-2,-1}*\eta_2=b_2$ and \label{penta}
\begin{equation}
\label{lant}
\overline{d}_{-2,-1}\overline{d}_{-1,1,2} (t_1*\overline{d}_{-1,1,2}) =\overline{d}_{1,2}\overline{d}_{\set{-2,-1,1,2}};
\end{equation}
\item $[a_j^2\overline{d}_{\set{i,\dots,j}},\overline{r}_{i,j}]=1$ and 
\[
\overline{r}_{i,j}\overline{d}_{\set{i,\dots,j}}=a_1^{-4}\dots a_{-i}^{-4}a_{-i+1}^{-2}\dots a_{j-1}^{-2}\overline{d}_{\set{i,\dots,j}}^{-1} \overline{r}_{i,j}
\]
for all $i,j$ such that $i=1$ or $i<0$, $j+i>0$ and $j-i<g$;\label{rijrel}
\item the triangle relations \label{dttri} $a_1^2b_1\overline{d}_{1,2}^{-1}b_1^{-1}\overline{d}_{1,2}^{-2}a_1^{-2}b_1=\overline{d}_{1,2}^{-1}a_2^2$,
\[
b_j^{-1}a_j^2\overline{d}_{\set{1,\dots,j}}b_j=(t_{j-1}\dots t_1\,\zeta^{-1})*(b_1^{-1} \varphi_2\zeta a_1^{2}b_1)
\]
and for $i<0$
\[
b_j^{-1}a_j^2\overline{d}_{\set{i,\dots,j}}b_j=(t_{j-1}\dots t_1\,\zeta^{-1})*(b_1^{-1} \varphi_2\varphi_1a_1^2 b_1),
\]
where $\zeta$, $\varphi_1$ and $\varphi_2$ are the mapping classes defined in \ref{triangles}\ref{trib}-\ref{trie}.
\end{enumerate}
\end{theorem}

\begin{remark}
Relation \ref{dttri} is almost a restatement of \ref{triangles}. However, the conjugates of the elements on the right hand side are often familiar mapping classes. For example, the relation \ref{triangles}\ref{tric} becomes 
\[
b_j^{-1}a_j^2\overline{d}_{\set{i,\dots,j}}b_j=\xi_j^{-1} \overline{d}_{j,j+1}\overline{d}_{i,\dots,j+1}a_j^2a_{j+1}^2 \xi_j.
\]
\end{remark}

We are going to prove \thref{thm:presD} by applying Tietze moves to the presentation of \thref{thm:mcgs}. First of all, we add to the presentation of \thref{thm:mcgs} the following generators:
\begin{gather}
b_{k+1}:=(t_k\overline{d}_{k,k+1}^{-1})*b_k, \label{bi} \\
\xi_k:=\overline{d}_{k,k+1}^{-1}*b_k, \quad \eta_{k+1}:=t_k^{-1}*b_k. \label{xietai}
\end{gather}
Notice that in this enlarged presentation, the relations of \thref{rem:sh} hold true by \thref{thm:mcgs}. This implies that the $b_i$, $\xi_i$ and $\eta_i$ generate $\Modd(\Sigma_g^1)[\phi]$, and the “old" generators can be removed from the presentation. Nonetheless, we will still keep them as shorthands, as explained in \thref{rem:sh}. 

Then, we add relations \ref{t01}-\ref{dttri}. We are going to prove that relations (\ref{bi}), (\ref{xietai}) and \ref{aiaj}-\ref{hyp} are consequence of \ref{t01}-\ref{dttri}.

We start deriving additional relations from \ref{t01}-\ref{a2}, in order to prove that these are sufficient. As a first step, we prove some basic relations involving $a_1^2$ and $\overline{d}_{1,2}$. 

\begin{lemma}
\thlabel{lem:darel}
The following relations are consequences of \ref{t01}-\ref{a2}:
\begin{enumerate}[label=\textit{(\roman*)}, ref=(\roman*)]
    \item $R_3(m_1,\xi_2)$, $R_3(m_1,\eta_4)$ and $m_1$ commutes with all the other generators; \label{m1}
    \item $R_3(m_2,b_1)$, $R_3(m_2,\xi_1)$, $R_3(m_2,b_2)$, $R_3(m_2,\eta_2)$, $R_3(m_2,\eta_4)$ and $m_2$ commutes with all the other generators and $m_1$; \label{m2}
    \item for $i=3,4$ we have $R_3(m_i,b_1)$, $R_3(m_i,\xi_1)$, $R_3(m_i,\xi_2)$, $R_3(m_i,\eta_4)$, and $m_i$ commutes with all the other generators, $m_1$ and $m_2$; \label{m34} 
    \item $\overline{d}_{1,2}$ commutes with all $b_k$, $\xi_k$ and $\eta_k$ for $k \ge 3$ and with $m_1$, $m_2$, and $\overline{d}_{1,2}*\xi_1=b_1$; \label{dbarvs}
    \item $a_1^2$ commutes with $\overline{d}_{1,2}$, $m_1$, $m_2$, $m_3$, $m_4$ and all $b_k$, $\xi_k$ and $\eta_k$ for $k \ge 2$, and $R_4(a_1^2,b_1)$, $R_4(a_1^2,\xi_1)$ hold. \label{a2vsall}
\end{enumerate}
\end{lemma}

\begin{proof}
\begin{enumerate}[label=\textit{(\roman*)}]
\item We prove this point in some detail, as in the following we will often perform similar computations more tacitly. If $k$ is at least $4$, generators $b_k$, $\xi_k$ and $\eta_{k+1}$ commute with all generators appearing in the definition of $m_1$ by \ref{t01}. It is also easy to see that all the generators appearing in the definition of $m_1$ except for $\xi_2$ and $b_3$ again commute with $m_1$ by \ref{t01}. As an example, we have
\begin{align*}
\xi_1*m_1&=(\xi_2\xi_1\xi_2b_2\eta_3\eta_2b_1^2\eta_2\eta_3b_2\xi_1\xi_2)*b_3=\\
&=(\xi_2\xi_1b_2\eta_3\eta_2b_1^2\eta_2\eta_3b_2\xi_1\xi_2\xi_1)*b_3=m_1.
\end{align*}
For $\xi_2$, we have
\begin{align*}
m_1*\xi_2&=(H_3\,b_3\, \xi_2^{-1}\xi_1^{-1}b_2^{-1}\eta_3^{-1}\eta_2^{-1}b_1^{-2}\eta_2^{-1}\eta_3^{-1}b_2^{-1}\xi_1^{-1})*\xi_2=\\
&=(\xi_1\,H_3\,b_3\, \xi_2^{-1}\xi_1^{-1}b_2^{-1}\eta_3^{-1}\eta_2^{-1}b_1^{-2}\eta_2^{-1}\eta_3^{-1}b_2^{-1})*\xi_1=\dots=\\
&=(\xi_1b_2\eta_3\eta_2b_1^2\eta_2\eta_3b_2\,H_3\,b_3)*\xi_2=\\
&=(\xi_1b_2\eta_3\eta_2b_1^2\eta_2\eta_3b_2\,\xi_2\xi_1\xi_2b_2\eta_3\eta_2 b_1^2\eta_2\eta_3b_2 \xi_1)*b_3=\\
&=(\xi_1b_2\eta_3\eta_2b_1^2\eta_2\eta_3b_2\xi_1\xi_2\xi_1)*b_3= \xi_2^{-1}*m_1.
\end{align*}
From \ref{5c} we get $m_1=b_3^{-1}(b_1\eta_2\eta_3b_2\xi_1)^6$, and all the generators appearing here commute with both $b_3$ and $\xi_3$ by \ref{t01}. Finally, the desired relation for $\eta_4$ is just \ref{boh}.
\item Clearly, $m_2$ commutes with $\xi_k$ for $k \ge 2$, $b_n$ for $n\ge 3$ and $\eta_m$ for $m \ge 5$. Moreover, $\eta_3*m_2=m_2$ by \ref{t01}. In order to prove the $R_3$ relations, we rearrange the curves in the definition of $m_2$ by \ref{t01}. For example, to show $R_3(m_2,b_1)$, observe that $m_2=(b_2^{-1}\eta_3^{-1})*\eta_2$, so the desired relation is just $(b_2^{-1}\eta_3^{-1})*R_3(b_1,\eta_2)$. Finally, the fact that $m_2*m_1=m_1$ follows from \ref{m1}.
\item The commutators follow easily from \ref{t01} as above. The relation $R_3$ follow from \ref{t01} using the same trick as in \ref{m2}.
\item Clearly, $\overline{d}_{1,2}$ commutes with $b_k$ for $k \ge 4$, $\xi_n$ for $n\ge 3$, $\eta_m$ for $m \ge 6$ and $m_1$, $m_2$ by \ref{m1}, \ref{m2}, \ref{m34} and \ref{t01}. Moreover, by \ref{lt} it also commutes with $\eta_3$ and $\eta_5$. For $\eta_4$, we have
\begin{align*}
&\overline{d}_{1,2}*\eta_4=(m_1m_2\eta_4m_1m_2\eta_4m_3^{-1}\eta_4^{-1} m_2^{-1}m_1^{-1}\eta_4^{-1}m_3^{-1})*\eta_4=\\
&\qquad=(m_1m_2\eta_4m_1m_2\eta_4m_3^{-1}\eta_4^{-1})*m_3=(m_1\eta_4m_2\eta_4m_1)*\eta_4=\eta_4.
\end{align*}
For $\xi_1$, observe that $m_3^{-1}*\xi_1=(b_1\eta_2\eta_3b_2\xi_1^{-1}b_2^{-1})*\xi_1=b_1*m_2=m_2^{-1}*b_1$. As a consequence,
\begin{align*}
&\overline{d}_{1,2}*\xi_1=(m_1m_2\eta_4m_1m_2\eta_4m_3^{-1}\eta_4^{-1} m_2^{-1}m_1^{-1}\eta_4^{-1}m_2^{-1})*b_1=\\
&\qquad=(\eta_4m_1m_2\eta_4m_3^{-1}\eta_4^{-1}m_2^{-1})*b_1=(\eta_4m_1m_2m_3^{-1}\eta_4^{-1})*\xi_1=b_1.
\end{align*}
\item That $a_1^2$ commutes with $\overline{d}_{1,2}$ is a consequence of \ref{lad}. Moreover, it follows from \ref{lt} and the preceding points that $a_1^2$ commutes also with $b_k$, $\xi_k$ and $\eta_k$ for $k \ge 3$, and with all the $m_i$. The other relations are \ref{a2} except for $a_1^2\xi_1a_1^2\xi_1=\xi_1a_1^2\xi_1a_1^2$, which is obtained conjugating by $\overline{d}_{1,2}^{-1}$ the corresponding relation for $b_1$. \qedhere
\end{enumerate}
\end{proof}

From \thref{lem:darel}, we now start deriving more complex relations. We will often use the shorthands 
\[
E_i:=b_{i+1}^{-1}*\xi_i, \qquad N_{i+1}:=b_i^{-1}*\eta_{i+1}. 
\] 
It is an immediate consequence of \ref{t01} that the following relations hold:
\begin{itemize}
\item $R_3(E_i,E_{i+1})$, $R_3(E_i,\xi_i)$, $R_3(E_i,\eta_{i+2})$ and $E_i$ satisfies the same relations $R_2$ and $R_3$ as $\xi_i$ with all the other generators;
\item $R_3(N_i,N_{i+1})$, $R_3(N_i,\eta_i)$, $R_3(N_i,\xi_{i-2})$ and $N_i$ satisfies the same relations $R_2$ and $R_3$ as $\eta_i$ with all the other generators;
\item $R_2(E_i,N_j)$ for all $i,j$.
\end{itemize}
Note that $t_i=N_{i+1}a_i^2E_i$ by (\ref{tiNE}). 

\begin{enumerate}[label=\textbf{(E\arabic*)}, ref=(E\arabic*), series=extra]
    \item $R_2(a_i^2,X_j)$ for $i \ne j$ and $R_4(a_i^2,X_i)$ for $i \ge 1$, where $X \in \set{b,\xi,\eta}$; \label{a2vs}
    \item $t_i*\eta_{i+1}=b_i$, $t_i*\xi_i=b_{i+1}$; \label{tiact}
    \item $a_{i+1}^2=(\eta_{i+1}b_ib_{i+1}^{-1}\xi_i^{-1})*a_i^2$ and $t_i*a_{i+1}^2=a_i^2$; \label{altai}
\end{enumerate}

\begin{proof}[Proof of \ref{a2vs}-\ref{altai}]
We do induction on $i$. For $i=1$, \ref{a2vs} is \thref{lem:darel}\ref{a2vsall}. We prove that for fixed $i\ge1$, \ref{tiact} and \ref{altai} follow from $\ref{a2vs}$ for indices up to $i$. Relations $R_2(a_i^2,\eta_{i+1})$ and $R_2(a_i^2,b_{i+1})$ imply that
\begin{gather*}
t_i*\eta_{i+1}=(b_i^{-1}\eta_{i+1}b_i)*\eta_{i+1}=b_i,\\
t_i*\xi_i=\big((b_i^{-1}\eta_{i+1}b_i)a_i^2\big)*b_{i+1}=b_{i+1}.
\end{gather*}
Using also $R_4(a_i^2,\xi_i)$, we get
\begin{align*}
a_{i+1}^2&=t_i*a_i^2=(b_i^{-1}\eta_{i+1}b_ib_{i+1}^{-1}a_i^2\xi_i)*a_i^2=\\
&=(\eta_{i+1}b_i\eta_{i+1}^{-1}b_{i+1}^{-1}\xi_i^{-1})*a_i^2=(\eta_{i+1}b_ib_{i+1}^{-1}\xi_i^{-1})*a_i^2. 
\end{align*}
Moreover, relations $R_4(a_i^2,E_i)$ and $R_4(a_i^2,N_{i+1})$ follow easily from \ref{a2vs} and \ref{t01}, hence
\begin{align*}
t_i*a_{i+1}^2=t_i^2*a_i^2&=(N_{i+1}a_i^2N_{i+1}E_ia_i^2E_i)*a_i^2=\\
&=(N_{i+1}a_i^2N_{i+1})*a_i^2=a_i^2. 
\end{align*}

To start the induction, we also need to prove \ref{a2vs} for $i=2$. Clearly, it holds for $b_k$, $\xi_k$ and $\eta_{k+1}$ when $k \ge 3$. From \ref{tiact} and \ref{altai} for $i=1$ we get
\begin{gather*}
R_2(a_2^2,\xi_1)=t_1^{-1}*R_2(a_1^2,b_2), \quad R_2(a_2^2,b_1)=t_1*R_2(a_1^2,\eta_2), \\
R_4(a_2^2,\eta_2)=t_1^{-1}*R_4(a_1^2,b_1),\quad R_4(a_2^2,b_2)=t_1*R_4(a_1^2,\xi_1).
\end{gather*}
Finally, by \ref{altai} we have
\begin{gather*}
R_2(a_2^2,\eta_3)=(\eta_2b_1b_2^{-1}\xi_1^{-1})*R_2(a_1^2,m_4), \\
R_4(a_2^2,\xi_2)=(\eta_2b_1b_2^{-1}\xi_1^{-1})*R_4(a_1^2,\xi_2^{-1}*\xi_1),
\end{gather*}
so by \thref{lem:darel}\ref{a2vsall} and \ref{a2vs} for $i=1$ we conclude. Then, relations \ref{tiact} and \ref{altai} for $i=2$ follow by the above arguments. 

Assume now that all relations have been proved for some $i\ge 2$. By the above, it suffices to prove that \ref{a2vs} holds for $i+1$. It is clearly true for $b_k$ when $k \ge i+2$ or $k \le i-1$ and for $\xi_{\ell}$, $\eta_{\ell+1}$ when $\ell \ge i+2$ or $\ell \le i-2$. Relations $R_2(a_{i+1}^2,\xi_i)$, $R_2(a_{i+1}^2,b_i)$, $R_4(a_{i+1}^2,\eta_{i+1})$ and $R_4(a_{i+1}^2,b_{i+1})$ are conjugates by $t_i^{\pm1}$ of relations for $a_i^2$. For the remaining relations, observe first that 
\begin{gather*}
R_4(a_{i+1}^2,\xi_{i+1})=(\eta_{i+1}b_ib_{i+1}^{-1}\xi_i^{-1})*R_4(a_i^2,\xi_{i+1}^{-1}*\xi_i),\\
R_2(a_{i+1}^2,\xi_{i-1})=\big((\eta_{i+1}b_ib_{i+1}^{-1}\xi_i^{-1})(\eta_{i}b_{i-1}b_{i}^{-1}\xi_{i-1}^{-1})\big)*R_2(a_{i-1}^2,\xi_i),\\
R_2(a_{i+1}^2,\eta_{i})=\big((\eta_{i+1}b_ib_{i+1}^{-1}\xi_i^{-1})(\eta_{i}b_{i-1}b_{i}^{-1}\xi_{i-1}^{-1})\big)*R_2(a_{i-1}^2,\eta_{i+1}).
\end{gather*}
For $\eta_{i+2}$, we use \ref{fore1}:
\[
(\eta_{i+1}b_ib_{i+1}^{-1}\xi_i^{-1})*R_2\big(a_i^2,(\xi_ib_i^{-1}b_{i+1}\eta_{i+2})*\eta_{i+1}\big)=R_2(a_{i+1}^2,\eta_{i+2}).\qedhere
\]
\end{proof}

Notice that (\ref{bi}) and (\ref{xietai}) follow immediately from \ref{tiact} and \thref{lem:darel}\ref{dbarvs}.

\begin{enumerate}[resume*=extra]
    \item $[a_i^2,a_j^2]=1$, $s_j*a_i^2=a_i^2$ for all $i,j$ and $t_k*a_i^2=a_i^2$ for $i \ne k,k+1$. \label{a2a}
\end{enumerate}

\begin{proof} 
For $i=1$, observe that $s_1*a_1^2=a_1^2$ is exactly $R_4(a_1^2,b_1)$, and
\[
[a_1^2,a_2^2]=[a_1^2,(\eta_2b_1b_2^{-1}\xi_1^{-1})*a_1^2]=(\eta_2b_2^{-1}\xi_1^{-1})*[\xi_1*a_1^2,b_1*a_1^2]=1
\]
follows from \ref{altai} and \ref{a2}. Using \ref{a2vs} and induction we get $s_k*a_1^2=a_1^2$, $t_k*a_1^2=a_1^2$ and $[a_1^2,a_{k+1}^2]=1$ for $k \ge 2$. 

For $i \ge2$, the proof is similar: $s_j*a_i^2=a_i^2$ for $j\le i$, $t_k*a_i^2=a_i^2$ for $k<i-1$ and $t_{i-1}*a_{i+1}^2=a_{i+1}^2$ follow from \ref{a2vs} and induction, so we get
\[
R_2(a_i^2,a_{i+1}^2)=t_{i-1}*R_2(a_{i-1}^2,a_{i+1}^2)
\]
and the other relations follow as in the case $i=1$. \qedhere
\end{proof}

From \ref{a2vs} we also obtain relations $R_4(a_i^2,E_{i-1})$, $R_4(a_i^2,E_i)$, $R_4(a_i^2,N_i)$, $R_4(a_i^2,N_{i+1})$ and $R_2(a_i^2,E_j)$, $R_2(a_i^2,N_{j+1})$ for every $j\ne i,i-1$. It is also useful to observe that 
\begin{equation}
\label{eq:altt}
t_i=E_ia_{i+1}^2N_{i+1}.
\end{equation}
Indeed, by \ref{altai} and $R_4(a_i^2,N_{i+1})$ we have
\[
E_ia_{i+1}^2N_{i+1}=E_it_i^{-1}a_i^2t_iN_{i+1}=a_i^{-2}N_{i+1}^{-1}a_i^2N_{i+1}a_i^2N_{i+1}E_i=N_{i+1}a_i^2E_i.
\]

\begin{enumerate}[resume*=extra]
\item $[\overline{d}_{1,2},t_j]=1$ for all $j \ne 2$ and $[\overline{d}_{i,j},a_k^2]=1$ for all $i,j,k$. \label{diti}
\end{enumerate}

\begin{proof}
We first show that $[\overline{d}_{1,2},t_1]=1$. By \ref{dbar}, we have $\overline{d}_{1,2}*N_2=a_2^{-2}*E_1$. Moreover, by (\ref{bi}), we have $\overline{d}_{1,2}*b_2=t_1*b_1$. Now, using \thref{lem:darel}\ref{dbarvs}, we get
\begin{align}
\begin{split}
\label{d12E}
\overline{d}_{1,2}*E_1&=(t_1b_1^{-1}t_1^{-1})*b_1=(N_2a_1^2E_1b_1^{-1}E_1^{-1}a_1^{-2}N_2^{-1})*b_1=\\
&=(N_2a_1^2b_1^{-1}a_1^{-2})*\eta_2=(N_2a_1^2)*N_2=a_1^{-2}*N_2.
\end{split}
\end{align}
As a consequence,
\begin{align*}
\overline{d}_{1,2}*t_1&=a_2^{-2}E_1a_2^2N_2a_1^2=a_2^{-2}E_1t_1a_1^2t_1^{-1}N_2a_1^2=\\
&=a_2^{-2}N_2E_1a_1^2E_1a_1^2E_1^{-1}=a_2^{-2}N_2a_1^2E_1a_1^2=a_2^{-2}t_1a_1^2=t_1.
\end{align*}

Since from \ref{d1234} we have $[\overline{d}_{1,2},a_3^2]=1$, the rest then follows from \ref{a2a} and \thref{lem:darel}\ref{dbarvs}. \qedhere
\end{proof}

As a consequence of \ref{diti} and (\ref{bi}), we see that 
\begin{equation}
\label{d12b2}
\overline{d}_{1,2}*b_2=t_1*b_1=(b_1^{-1}\eta_2a_1^{-2})*b_1=s_1^{-1}*\eta_2.
\end{equation}

\begin{enumerate}[resume*=extra]
    \item $t_i*t_{i+1}=t_{i+1}^{-1}*t_i$ for all $i$, $[t_i,t_k]=1$ if $\abs{i-k}>1$ and $[t_i,s_j]=1$ for $i \ne j$. \label{titi}
\end{enumerate}

\begin{proof}
All the relations follow easily from \ref{a2vs} and \ref{a2a} except for the first one, which can be rewritten using (\ref{eq:altt}) as 
\[
(t_i*N_{i+2})a_i^2(t_i*E_{i+1})=(t_{i+1}^{-1}*N_{i+1})a_i^2(t_{i+1}^{-1}*E_i). 
\]
We are going to prove that $t_i*N_{i+2}=t_{i+1}^{-1}*N_{i+1}$ and $t_i*E_{i+1}=t_{i+1}^{-1}*E_i$. We have $t_i*N_{i+2}=(N_{i+1}a_i^2E_i)*N_{i+2}=N_{i+1}*N_{i+2}$ and 
\begin{align*}
t_{i+1}^{-1}*N_{i+1}&=(E_{i+1}^{-1}a_{i+1}^{-2}N_{i+2}^{-1})*N_{i+1}=\\
&=(E_{i+1}^{-1}N_{i+1}a_i^2E_ia_i^2E_i^{-1}a_i^{-2}N_{i+1}^{-1}N_{i+1})*N_{i+2}=N_{i+1}*N_{i+2}.
\end{align*}
For the other equality, we have
\begin{align*}
t_{i+1}^{-1}*E_{i}&=(E_{i+1}^{-1}a_{i+1}^{-2})*E_i=(E_{i+1}^{-1}E_i^{-1}a_i^{-2}N_{i+1}^{-1}a_i^{-2}N_{i+1}a_i^2E_i)*E_i=\\
&=(E_{i+1}^{-1}E_i^{-1}N_{i+1}a_i^{-2}N_{i+1}^{-1})*E_i=(N_{i+1}E_{i+1}^{-1}E_i^{-1}a_i^{-2})*E_i=\\
&=(N_{i+1}E_{i+1}^{-1}a_i^2)*E_i=t_i*E_{i+1}. \qedhere
\end{align*}
\end{proof}

\begin{enumerate}[resume*=extra]
\item For $i>0$, we have
\begin{gather*}
\overline{d}_{i,i+1}=(t_{i-1}t_it_{i-2}t_{i-1}\dots t_1t_2)*\overline{d}_{1,2}, \\ \overline{d}_{-i-1,-i}=(t_{i-1}^{-1}t_i^{-1}t_{i-2}^{-1}t_{i-1}^{-1}\dots t_1^{-1}t_2^{-1})*\overline{d}_{-2,-1}.
\end{gather*}
As a consequence, we have 
\[
\overline{d}_{i,i+1}=(t_{i-1}t_i)*\overline{d}_{i-1,i}, \quad \overline{d}_{-i-1,-i}=(t_{i-1}^{-1}t_i^{-1})*\overline{d}_{-i,-i+1}
\]
for $i>1$, and $t_k*\overline{d}_{i,i+1}=\overline{d}_{i,i+1}$ if $\abs{k-i} \ne 1$. \label{ditiii}
\end{enumerate}

\begin{proof}
This follows easily from (\ref{eq:dijdij}), \ref{diti} and \ref{titi}. \qedhere
\end{proof}

Recall that $H_2:=\eta_3\eta_2b_1a_1^2b_1\eta_2\eta_3$.

\begin{enumerate}[resume*=extra]
    \item $H_2*b_2=H_2^{-1}*b_2$ and $[H_2*b_2,b_2]=1$.
\end{enumerate}

\begin{proof}
Applying repeatedly relations \ref{t01} and \ref{a2vs}, we see that $H_2$ and $b_2$ both commute with $a_1^2(b_1a_1^2b_1)(\eta_2b_1a_1^2b_1\eta_2)$. Then, using \ref{3c} we can do the same proof as \cite[(9)]{waj:elem}. \qedhere
\end{proof}

\begin{enumerate}[resume*=extra]
\item $s_1t_1s_1t_1=t_1s_1t_1s_1$. \label{siti}
\end{enumerate}

\begin{proof}
We have to prove that $(s_1t_1)*s_1=t_1^{-1}*s_1$. Since 
\[
(s_1t_1)*a_1^2=s_1*a_2^2=a_2^2=t_1^{-1}*a_1^2
\]
by \ref{a2a}, it suffices to prove that $(st_1)*b_1$ is equal to $t_1^{-1}*b_1=\eta_2$. We have
\[
(s_1t_1)*b_1=(b_1a_1^2b_1b_1^{-1}\eta_2b_1a_1^2E_1)*b_1=(b_1\eta_2a_1^2b_1a_1^2)*b_1=\eta_2.\qedhere
\]
\end{proof}

\begin{enumerate}[resume*=extra]
\item $s_1t_1s_1=t_1w_1=w_1t_1$, where $w_1:=(\eta_2a_2^2N_2a_1^2N_2\eta_2)$. \label{w1}
\end{enumerate}

\begin{proof} The idea is similar to that of \cite[(10)]{waj:elem}. By \ref{tiact} and \ref{a2vs}, we have
\[
t_1w_1=t_1\eta_2t_1^{-1}a_1^2t_1N_2a_1^2\eta_2b_1=b_1a_1^2t_1\eta_2b_1a_1^2b_1=b_1a_1^2b_1t_1b_1a_1^2b_1=s_1t_1s_1.
\]
By \ref{siti}, $st_1s=t_1st_1st_1^{-1}$, so the other equality follows. \qedhere
\end{proof}

\begin{enumerate}[resume*=extra]
\item \label{t2} $t_1^2=\overline{d}_{1,2}\overline{d}_{-2,-1}$, $[\overline{d}_{1,2},\overline{d}_{-2,-1}]=1$ and 
\[
\overline{d}_{-2,-1}=w_i^{-1}*\overline{d}_{1,2}=w_i*\overline{d}_{1,2}
\]
for $i=1,2,3$, where
\[
w_2:=N_3a_2^2N_2a_1^2N_2N_3, \quad w_3:=E_2E_1a_1^2E_1a_2^2E_2. 
\]
\end{enumerate}

\begin{proof}
The first relation is \ref{3c}, and the second relation follows easily. For the third one, we have $\overline{d}_{-2,-1}=(st_1s)^{-1}*\overline{d}_{1,2}=(t_1w_1)^{-1}*\overline{d}_{1,2}=w_1^{-1}*\overline{d}_{1,2}$ by \ref{diti} and \ref{w1}. On the other hand,
\[
w_1*\overline{d}_{1,2}=w_1*(t_1^2\overline{d}_{-2,-1}^{-1})=t_1^2 (w_1*\overline{d}_{-2,-1}^{-1})=t_1^2\overline{d}_{1,2}^{-1}=\overline{d}_{-2,-1}. 
\]
Conjugating by $m_2=(\eta_2\eta_3)*b_2$ we obtain the last relation for $i=2$. Indeed, notice first that
\[
w_1=\eta_2N_2a_1^2E_1a_1^2E_1^{-1}N_2\eta_2, \quad w_2=N_3N_2a_1^2E_1a_1^2E_1^{-1}N_2N_3;
\]
then, it suffices to apply \thref{lem:darel}\ref{dbarvs}. Since $m_2$ commutes with $\overline{d}_{1,2}$ and $t_1$, it also commutes with $\overline{d}_{-2,-1}=t_1^2\overline{d}_{1,2}^{-1}$. 

For $i=3$, we claim that $w_3$ is the result of conjugation of $w_2$ by 
\[
\psi:=\big((E_2\eta_4\overline{d}_{1,2})*N_3\big)^{-1}\cdot \eta_4\overline{d}_{1,2}. 
\]
First of all, by (\ref{d12b2}) we get $\overline{d}_{1,2}*N_3=(\eta_3s_1^{-1})*\eta_2$. This clearly implies that relations $R_3(\eta_4,\overline{d}_{1,2}*N_3)$ and $R_2(E_2,\overline{d}_{1,2}*N_3)$ hold. As a consequence,
\[
\psi*N_3=\big(E_2\eta_4\cdot (\overline{d}_{1,2}*N_3)^{-1}(E_2*\eta_4)^{-1})\big)*(\overline{d}_{1,2}*N_3)=(E_2\eta_4E_2)*\eta_4=E_2.
\]

Now, applying \thref{lem:darel}\ref{m34} we have
\begin{align*}
\overline{d}_{1,2}^{-1}*E_2&=(\eta_4\eta_3b_3\eta_4m_4^{-1} \eta_4^{-1}b_3^{-1}\eta_3^{-1}\eta_4^{-1}m_4^{-1})*\xi_2=(\eta_4m_4^{-1}\eta_3b_3 \eta_4^{-1}b_3^{-1}\eta_3^{-1}\eta_4^{-1}m_4^{-1})*\xi_2=\\
&=(\eta_4m_4^{-1}\eta_3 \eta_4^{-1}\eta_3^{-1}b_3^{-1}\eta_4^{-1}m_4^{-1})*\xi_2=(m_4^{-1}\eta_4^{-1}m_4\eta_3^{-1}b_3^{-1}\eta_4^{-1}b_3\xi_2)*m_4.
\end{align*}
Since by \ref{boh} $\eta_4$ commutes with $(b_3\xi_2)*m_4$, we get
\[
\overline{d}_{1,2}^{-1}*E_2=(m_4^{-1}\eta_4^{-1}m_4\xi_2)*m_4=m_4^{-1}*\xi_2
\]
Hence, by \ref{boh} and by \ref{dbar} we have
\begin{align*}
\psi*N_2&=(E_2\eta_4\overline{d}_{1,2}\,N_3^{-1}\,\overline{d}_{1,2}^{-1} \eta_4^{-1}E_2^{-1}\,a_2^{-2})*E_1=(E_2\overline{d}_{1,2}\,N_3^{-1}\eta_4^{-1}N_3\,\overline{d}_{1,2}^{-1} E_2^{-1}\,a_2^{-2})*E_1=\\
&=(\overline{d}_{1,2}\xi_2m_4\xi_2^{-1}\,N_3^{-1}\eta_4^{-1}N_3\, \xi_2m_4^{-1}\xi_2^{-1}\overline{d}_{1,2}^{-1}\,a_2^{-2})*E_1=\\
&=(\overline{d}_{1,2}b_2^{-1}\xi_2m_4\,\eta_3^{-1}\eta_4^{-1}\eta_3 \,m_4^{-1})*N_2=(\overline{d}_{1,2}b_2^{-1}\xi_2)*N_2=a_2^{-2}*E_1.
\end{align*}
This proves that $\psi*w_2=w_3$. Now, $\psi$ is easily seen to commute with $\overline{d}_{1,2}$ and $t_1$, hence also with $\overline{d}_{-2,-1}$. \qedhere
\end{proof}

\begin{enumerate}[resume*=extra]
    \item $\overline{d}_{i+1,i+2}=(t_i^{-1}t_{i+1}^{-1})*\overline{d}_{i,i+1}$ for all $i=1,\dots,g-2$. \label{diti2}
\end{enumerate}

\begin{proof}
By \ref{ditiii}, it suffices to prove that $(t_{i+1}t_i^2t_{i+1})*\overline{d}_{i,i+1}=\overline{d}_{i,i+1}$. We do induction on $i$. For $i=1$, notice that
\[
t_2t_1^2t_2=N_3a_2^2E_2N_2a_1^2E_1N_2a_1^2E_1N_3a_2^2E_2=N_3a_2^2N_2a_1^2N_2N_3E_2E_1a_1^2E_1a_2^2E_2=w_2w_3.
\]
Thus, by \ref{t2} we have $(t_2t_1^2t_2)*\overline{d}_{1,2}=(w_2w_3)*\overline{d}_{1,2}=w_2*\overline{d}_{-2,-1}=\overline{d}_{1,2}$. For the inductive step see \cite[(13)]{waj:elem}. \qedhere
\end{proof}

\begin{enumerate}[resume*=extra]
    \item $t_i^2=\overline{d}_{i,i+1}\overline{d}_{-i-1,-i}$ for all $i=1,\dots,g-1$. \label{ti2}
\end{enumerate}

\begin{proof} 
See \cite[(14)]{waj:elem}. \qedhere
\end{proof}

Notice that so far we have obtained relations \ref{aiaj}, \ref{atiti}, \ref{s2}, \ref{atisi}, \ref{asiti} and \ref{sai} as consequences of \ref{t01}-\ref{dttri}.

\subsection{Further relations}
\label{reldij}

Now we derive relations \ref{pb} and \ref{a8} from \ref{t01}-\ref{dttri}. Many steps of the proof are similar to those in \cite{waj:elem}. Observe that the notion of symmetry considered by Wajnryb does not really apply to our context, essentially as a consequence of the asymmetry in (\ref{tiNE}). However, it is easy to adapt Wajnryb's arguments to the extra cases.

\begin{enumerate}[resume*=extra]
    \item $[b_1,\overline{d}_{-2,2}]=1$. \label{b22}
\end{enumerate}

\begin{proof}
As in \cite[(16)]{waj:elem}, we find that $\overline{d}_{-2,2}=a_2^2((\overline{d}_{1,2}t_1^{-1})*s_1^2)$. We have to prove that $b_1$ commutes with $(\overline{d}_{1,2}t_1^{-1})*s_1^2$. We have
\[
(b_1\overline{d}_{1,2}t_1^{-1})*s_1^2=(\overline{d}_{1,2}\xi_1\,E_1a_1^{-2}N_2^{-1})*s^2=(\overline{d}_{1,2}E_1\,b_2a_1^{-2}N_2^{-1})*s^2=(\overline{d}_{1,2}t_1^{-1})*s^2. \qedhere
\]
\end{proof}

\begin{enumerate}[resume*=extra]
    \item $\xi_k=(b_{k+1}t_{k-1}t_kb_k^{-1})*\xi_{k-1}$ and $\eta_{k+1}=(b_kt_{k-1}t_kb_{k-1}^{-1})*\eta_k$ for all $k\ge2$. \label{xiketak}
\end{enumerate}

\begin{proof}
The statement is equivalent to $E_k=(t_{k-1}t_k)*E_{k-1}$ and $N_{k+1}=(t_{k-1}t_k)*N_k$. Using (\ref{eq:altt}), we get
\[
(t_{k-1}t_k)*E_{k-1}=(N_ka_{k-1}^2E_{k-1}E_ka_{k+1}^2N_{k+1})*E_{k-1}=(N_ka_{k-1}^2E_{k-1}E_k)*E_{k-1}=E_k.
\]
Since $N_{i+1}=t_iE_i^{-1}a_i^{-2}$ for all $i$, the second relation follows from \ref{titi} and the first one. \qedhere
\end{proof}

\begin{enumerate}[resume*=extra]
\item $\overline{d}_{i,i+1}*E_i=a_i^{-2}*N_{i+1}$ and $\overline{d}_{i,i+1}*N_{i+1}=a_{i+1}^{-2}*E_{i}$ for all $i$. \label{dEN}
\end{enumerate}

\begin{proof}
We do induction on $i$. The base case is \ref{dbar} and (\ref{d12E}). In general, by \ref{ditiii} and \ref{xiketak}, we obtain
\[
\overline{d}_{i,i+1}*E_i=(t_{i-1}t_i\overline{d}_{i-1,i}t_i^{-1}t_{i-1}^{-1})*E_i=(t_{i-1}t_i\overline{d}_{i-1,i})*(E_{i-1})=(t_{i-1}t_ia_{i-1}^{-2})*N_i=a_i^{-2}*N_{i+1}
\]
and
\[
\overline{d}_{i,i+1}*N_{i+1}=(t_{i-1}t_i\overline{d}_{i-1,i}t_i^{-1}t_{i-1}^{-1})*N_{i+1}=(t_{i-1}t_i\overline{d}_{i-1,i})*N_{i}=(t_{i-1}t_ia_i^{-2})*E_{i-1}=a_{i+1}^{-2}*E_{i}. \qedhere
\]
\end{proof}

\begin{enumerate}[resume*=extra]
\item $[s_k,\overline{d}_{i,i+1}]=1$ if $k \ne i,i+1$ and $[s_i,\overline{d}_{i,i+1}s_i\overline{d}_{i,i+1}]=1$. \label{den}
\end{enumerate}

\begin{proof}
For the first relation, it is enough to show that $\overline{d}_{i,i+1}*b_k=b_k$ if $k \ne i,i+1$. Clearly $\overline{d}_{i,i+1}$ commutes with $b_k$ for $k \ge i+2$. For the rest, we do induction on $i$. The base case is \thref{lem:darel}\ref{dbarvs}. Suppose that $\overline{d}_{i-1,i}*b_k=b_k$ for $k<i-1$. Then also $\overline{d}_{i,i+1}*b_k=b_k$ for $k<i-1$ by \ref{ditiii}. For $k=i-1$, it is enough to prove that $\overline{d}_{i-1,i}$ commutes with $(t_{i}^{-1}t_{i-1}^{-1})*b_{i-1}$. Observe that by (\ref{eq:altt}) we have
\[
(t_{i}^{-1}t_{i-1}^{-1})*b_{i-1}=(N_{i+1}^{-1}a_{i+1}^{-2}E_{i}^{-1})*\eta_{i}=N_{i+1}^{-1}*\eta_{i}=(\eta_{i}\eta_{i+1})*b_{i},
\]
so the desired relation is \ref{fore1}.

The second relation follows from \ref{dbar}. Indeed, we have
\[
(s_i\overline{d}_{i,i+1})*b_i=(b_ia_i^2b_ia_i^{-2}b_i^{-1}a_i^{-2})*\xi_i= a_i^{-2}*\xi_i=(a_i^{-2}\overline{d}_{i,i+1}^{-1})*b_i,
\]
and then it is straightforward to conclude. \qedhere
\end{proof}

\begin{enumerate}[resume*=extra]
\item $[t_j,\overline{d}_{i,i+1}t_j\overline{d}_{i,i+1}]=1$ for $j=i\pm1$. \label{td}
\end{enumerate}

\begin{proof}
We do induction on $i$. The base case is \ref{d1234}. Then, as a consequence of \ref{titi} and \ref{ditiii}, we have
\[
[t_i,\overline{d}_{i+1,i+2}t_i\overline{d}_{i+1,i+2}]=(t_it_{i+1}t_i)*[t_{i+1},\overline{d}_{i,i+1}t_{i+1}\overline{d}_{i,i+1}]
\]
and
\[
[t_{i+2},\overline{d}_{i+1,i+2}t_{i+2}\overline{d}_{i+1,i+2}]=(t_{i}t_{i+1}t_{i+2})*[t_{i+1},\overline{d}_{i,i+1}t_{i+1}\overline{d}_{i,i+1}]. \qedhere
\]
\end{proof}

\begin{lemma}
\thlabel{lem:a8}
Relations \ref{a8} follow from \ref{t01}-\ref{dttri}.
\end{lemma}

\begin{proof}
The proof is the same as that \cite[Lemma 33]{waj:elem}; notice that all the relations needed are either \ref{t01}-\ref{dttri} or have been proved above. \qedhere
\end{proof}

Now we move on to relations \ref{pb}. 

\begin{enumerate}[resume*=extra]
    \item $[\overline{d}_{i,j},\overline{d}_{-1,1}]=1$ if $i,j \ne \pm1$, and $[\overline{d}_{i,j},\overline{d}_{k,k+1}]=1$ if $i,j \ne k,k+1$. \label{prepb}
\end{enumerate}

\begin{proof}
See \cite[(20)]{waj:elem}. Recall that $\overline{d}_{1,2}$ commutes with $\overline{d}_{3,4}$ and $\overline{d}_{-3,-1}$ by \ref{d1234}. \qedhere
\end{proof}

\begin{lemma}[{\cite[Lemma 34]{waj:elem}}]
\thlabel{lem:a8bis}
The following relations hold:
\begin{enumerate}[label=\textit{(\alph*)}]
    \item $t_k^{-1}\overline{d}_{k,k+1}$ commutes with $\overline{d}_{i,j}$ if $i,j \ne \pm k,\pm(k+1)$; \label{k8a}
    \item $t_k^{-1}\overline{d}_{k,k+1}$ commutes with $\overline{d}_{k,k+1}$ and $\overline{d}_{-k-1,-k}$;\label{k8b}
    \item $(t_k^{-1}\overline{d}_{k,k+1})*\overline{d}_{i,\pm k}=\overline{d}_{i,\pm k\pm1}$ if $i \ne -k-1$ and $i+k\ne0$;\label{k8c}
    \item $(t_k^{-1}\overline{d}_{k,k+1})*\overline{d}_{\pm k,j}=\overline{d}_{\pm k\pm 1,j}$ if $j \ne k+1$ and $-k+j \ne 0$;\label{k8d}
    \item $(t_k^{-1}\overline{d}_{k,k+1})*\overline{d}_{-k,k}=\overline{d}_{-k-1,k+1}$;\label{k8e}
\end{enumerate}
\end{lemma}

\begin{lemma}
Relations \ref{pb} follow from \ref{t01}-\ref{dttri}.
\end{lemma}

\begin{proof}
See \cite[(21), (22), (23), (24)]{waj:elem}. \qedhere
\end{proof}

To conclude the proof of \thref{thm:mcgs}, we only have to deal with the relations coming from the action of $\Modd(\Sigma_{g,1})[\phi]$ on the edges of $X_g$. 

We start from relations \ref{diffwr}. Notice that we have
\[
(t_1s_1t_1)*b_1=(t_1b_1\eta_2a_1^2b_1a_1^2)*b_1=(t_1)*\eta_2=b_1
\]
and 
\[
\overline{d}_{-1,1}\overline{d}_{-1,2}\overline{d}_{1,2}a_1^2=a_1^4s_1\overline{d}_{1,2}s_1\overline{d}_{1,2},
\]
so by \ref{dbar} and \ref{den} we get 
\[
(a_1^4s_1\overline{d}_{1,2}s_1\overline{d}_{1,2})*b_1=(a_1^2s_1\overline{d}_{1,2})*\xi_1=(a_1^2s_1)*b_1=b_1.
\]
This proves \ref{dwb1}, and \ref{dwb2} follows immediately. It is easy to see that relations \ref{dwbj} follow from the above. For \ref{dwb*}, notice that the first relation is equivalent to
\begin{equation}
\label{adr}
[a_j^2\overline{d}_{i,\dots,j},\overline{r}_{i,j}]=1,
\end{equation}
which is the first half of \ref{rijrel}. We will deal with the last part of \ref{dwb*} later on.

Consider now relation \ref{backtracking}. We deal only with the case $i<0$, as the other is similar. By (\ref{adr}), we get
\[
b_j^{+2}\overline{r}_{i,j}\overline{d}_{\set{i,\dots,j}}a_j^2 b_j^{-2}\overline{r}_{i,j}=b_j^{+2}\overline{d}_{\set{i,\dots,j}}a_j^2 \overline{r}_{i,j}b_j^{-2}\overline{r}_{i,j}=b_j^{+2}\overline{d}_{\set{i,\dots,j}}a_j^2b_j\overline{d}_{\set{i,\dots,j}}a_j^2\overline{d}_{\set{i,\dots,j}}a_j^2b_j.
\]
Hence, writing out the definition of $s_j$, the relation becomes
\[
b_j\overline{d}_{\set{i,\dots,j}}a_j^2 b_ja_j^2\overline{d}_{\set{i,\dots,j}}^2=a_1^{-4}\dots a_{-i}^{-4}a_{-i+1}^{-2}\dots a_{j-1}^{-2} a_j^2b_ja_j^2\overline{d}_{\set{i,\dots,j}}b_j.
\]
Applying once more (\ref{adr}), we obtain the second half of \ref{rijrel}.

Consider now relation \ref{diffwr}\ref{dwb*} for $i<0$. Applying \ref{backtracking}, we have
\begin{align*}
s_ja_j^{-2}\overline{d}_{\set{1,\dots,j}}^{-1}\overline{r}_{i,j}^{-1}b_j^{-2}a_j^2s_j\overline{d}_{\set{i,\dots,j}}b_j^{2}\overline{r}_{i,j}&=a_1^4\dots a_{-i}^4a_{-i+1}^2\dots a_{j-1}^2 s_jb_j^{-2}\overline{r}_{i,j}s_j^{-1}b_j^{2}\overline{r}_{i,j}=\\
&=a_1^4\dots a_{-i}^4a_{-i+1}^2\dots a_{j-1}^2 \overline{r}_{i,j}^2.
\end{align*}
The relation becomes
\[
\overline{r}_{i,j}^2=a_1^{-4}\dots a_{-i}^{-4}a_{-i+1}^{-2}\dots a_{j}^{-2} \overline{d}_{\set{i,\dots,j}}^{-2}\cdot\overline{d}_{\set{i,\dots,j-1}}\big((t_{j-1}\dots t_{-i+1})*\overline{d}_{\set{i-1,\dots,j}}\big),
\]
which is \ref{3c}.

For the triangles, we have already observed that relations \ref{dttri} easily imply \ref{triangles}.

Relation \ref{square} follows from \ref{t2}:
\[
(a_1^2b_1t_1\overline{d}_{1,2}^{-1}a_1^2b_1t_1\overline{d}_{-2,-1}^{-1})^2=(a_1^2b_1a_2^2b_2t_1\overline{d}_{1,2}^{-1}t_1\overline{d}_{-2,-1}^{-1})^2=(a_1^2b_1a_2^2b_2)^2=a_1^2s_1a_2^2s_2.
\]

Observe that applying \thref{lem:darel}\ref{dbarvs}, \ref{tiact} and \ref{t2}, relation \ref{pentagon} simplifies as follows:
\begin{align*}
&s_2b_2^{-2}\overline{r}_{1,2} \overline{d}_{-2,-1}\overline{d}_{1,2}^{-1} b_1 t_1a_1^2 b_1 \overline{d}_{-2,-1}^{-3}t_1 b_1 t_1 b_1 \overline{d}_{\set{-2,-1,1,2}}^{-1}=\\
&\quad=b_2a_2^2\overline{d}_{1,2}a_2^2b_2 \overline{d}_{-2,-1}\overline{d}_{1,2}^{-1} b_1 t_1a_1^2 b_1 \overline{d}_{1,2}t_1^{-1}\overline{d}_{-2,-1}^{-2}t_1 \eta_2  b_1 \overline{d}_{\set{-2,-1,1,2}}^{-1}=\\
&\quad=b_2a_2^2\overline{d}_{1,2}a_2^2b_2 \overline{d}_{-2,-1}\overline{d}_{1,2}^{-1} b_1 \overline{d}_{1,2}a_2^2 b_2 \overline{d}_{-2,-1}^{-2}t_1 \eta_2  b_1 \overline{d}_{\set{-2,-1,1,2}}^{-1}=\\
&\quad=b_2a_2^2\overline{d}_{1,2}a_2^2b_2 \overline{d}_{-2,-1} \xi_1 a_2^2 b_2 t_1^{-1}\overline{d}_{1,2} \overline{d}_{-2,-1}^{-1} \eta_2  b_1 \overline{d}_{\set{-2,-1,1,2}}^{-1}=\\
&\quad=b_2a_2^2\overline{d}_{1,2}a_2^2b_2 \overline{d}_{-2,-1} a_2^2a_1^{-2} \eta_2 b_2 b_1^{-1}\eta_2^{-1}
t_1^{-1}\overline{d}_{1,2} \overline{d}_{-2,-1}^{-1} \eta_2  b_1 \overline{d}_{\set{-2,-1,1,2}}^{-1}.
\end{align*}
Now, applying \ref{penta}, this last term is equal to
\begin{align*}
&b_2a_2^2\overline{d}_{1,2}a_2^2b_2 \overline{d}_{-2,-1} a_2^2a_1^{-2} \eta_2 b_2 b_1^{-1}\eta_2^{-1}
t_1^{-1}\overline{d}_{1,2} \overline{d}_{-2,-1}^{-1} \eta_2  b_1 \overline{d}_{\set{-2,-1,1,2}}^{-1}=\\
&\quad=b_2a_2^2\overline{d}_{1,2}a_2^2b_2 \overline{d}_{-2,-1} a_2^2a_1^{-2} \eta_2 \overline{d}_{-1,1,2}a_1^2 b_2  \overline{d}_{\set{-2,-1,1,2}}^{-1}=\\
&\quad=b_2a_2^2\overline{d}_{1,2}a_2^2b_2 a_2^2 b_2 \overline{d}_{-2,-1}\overline{d}_{-1,1,2}   \overline{d}_{\set{-2,-1,1,2}}^{-1} b_2=\\
&\quad=b_2a_2^2\overline{d}_{1,2}b_2 a_2^2 b_2a_2^2 \overline{d}_{1,2} (t_1*\overline{d}_{-1,1,2})^{-1} b_2=\overline{r}_{1,2} a_2^2 \overline{r}_{1,2} (t_1*\overline{d}_{-1,1,2})^{-1}.
\end{align*}
Observe that by \ref{rijrel} we have
\[
\overline{r}_{1,2} a_2^2 \overline{r}_{1,2}=\overline{r}_{1,2} a_2^2\overline{d}_{1,2}\overline{d}_{1,2}^{-1} \overline{r}_{1,2}=\overline{r}_{1,2}^2 a_1^2a_2^2\overline{d}_{1,2}^2,
\]
and by \ref{3c} we know that 
\[
\overline{r}_{1,2}^2=a_1^{-2}a_2^{-2}\overline{d}_{1,2}^{-2}(t_1*\overline{d}_{-1,1,2}).
\]
This proves \ref{pentagon}.

Finally, \ref{hyp} follows easily from \ref{hypg}.

\section{The case of closed surfaces}

In this section, we finally determine a finite presentation of the spin mapping class group of a closed surface. There is a standard procedure to relate the mapping class group of $\Sigma_g^1$ and that of $\Sigma_g$: first we cap off the boundary component $\partial$ with a once-marked disk, obtaining a surface $\Sigma_{g,1}$, and then we forget the marked point $p$. These two steps correspond to two well-known exact sequences of groups (see \cite[Proposition 3.19 and Theorem 4.6]{primer}):
\begin{gather*}
1 \longrightarrow \Braket{t_{\partial}} \longrightarrow \Modd(\Sigma_g^1) \overset{Cap}{\longrightarrow} \Modd(\Sigma_{g,1}) \longrightarrow 1, \\
1 \longrightarrow \pi_1(\Sigma_g,p) \overset{Push}{\longrightarrow} \Modd(\Sigma_{g,1}) \overset{Forget}{\longrightarrow} \Modd(\Sigma_g) \longrightarrow 1.
\end{gather*}
The second sequence is known as the Birman exact sequence. By \cite[Fact 4.7]{primer}, the kernel of $Forget$ is generated by mapping classes of the form $t_{\gamma}t_{\gamma'}^{-1}$, where $\gamma$ and $\gamma'$ bound an annulus containing $p$, i.e. Birman's “spin maps" \cite{birmanes}. 

Consider now an even spin structure $\phi$ on $\Sigma_g^1$. By our assumptions, $\phi(\partial)=1$, so $\Sigma_{g,1}$ and $\Sigma_g$ inherit well-defined spin structures, which we still denote by $\phi$. Moreover, it is clear that the maps $Cap$ and $Forget$ restrict to surjections between the stabilizer subgroups. Since all the mapping classes in both kernels are easily seen to preserve $\phi$, we get analogous exact sequences
\begin{gather}
1 \longrightarrow \Braket{t_{\partial}} \longrightarrow \Modd(\Sigma_g^1)[\phi] \overset{Cap}{\longrightarrow} \Modd(\Sigma_{g,1})[\phi] \longrightarrow 1, \label{cap} \\
1 \longrightarrow \pi_1(\Sigma_g,p) \overset{Push}{\longrightarrow} \Modd(\Sigma_{g,1})[\phi] \overset{Forget}{\longrightarrow} \Modd(\Sigma_g)[\phi] \longrightarrow 1.\label{bes}
\end{gather}

We can now obtain a presentation for $\Modd(\Sigma_g)[\phi]$, where $\phi$ is the even spin structure of the preceding sections.

\begin{theorem}
\thlabel{thm:presDC}
The spin mapping class group $\Modd(\Sigma_g)[\phi]$ admits a presentation with generators $b_1,\dots,b_g$, $\xi_1,\dots,\xi_{g-1}$, $\eta_2,\dots,\eta_g$ and relations \ref{t01}-\ref{dttri} and the following:
\begin{enumerate}[resume*=drel]
    \item $\overline{d}_{\set{-g,\dots,g}}a_1^2\dots a_g^2=1$; \label{Delta}
    \item $\overline{d}_{\set{-g,\dots,-1}}=\overline{d}_{\set{1,\dots,g}}$.\label{push}
\end{enumerate}
\end{theorem}

\begin{proof}
Notice that $\partial=\delta_{\set{-g,\dots,g}}$ in the notation of Figure \ref{fig:diijj}. Hence, 
\[
t_{\partial}=d_{\set{-g,\dots,g}}=\overline{d}_{\set{-g,\dots,g}}a_1^2\dots a_g^2
\]
by (\ref{dset}), and modding out $\Modd(\Sigma_g^1)[\phi]$ by \ref{Delta} we get $\Modd(\Sigma_{g,1})[\phi]$ by (\ref{cap}).

In order to obtain $\Modd(\Sigma_g)[\phi]$, we have to mod out by the subgroup generated by mapping classes $t_{\gamma}t_{\gamma'}^{-1}$ as above. Clearly, the action of $\Modd(\Sigma_g^1)[\phi]$ on these elements has two orbits, according to the spin value of $\gamma$ and $\gamma'$. A couple of $\gamma,\gamma'$ with spin value $1$ is given by $\delta_{\set{-g,\dots,-1}}$ and $\delta_{\set{1,\dots,g}}$, and relation \ref{push} implies that
\[
d_{\set{1,\dots,g}}d_{\set{-g,\dots,-1}}^{-1}=\overline{d}_{\set{1,\dots,g}}\overline{d}_{\set{-g,\dots,-1}}^{-1}=1.
\]
A couple of admissible $\gamma,\gamma'$ is given by the curves $\beta_g$ and $\beta_g'$ of Figure \ref{fig:spin}. Notice that
\[
\beta_g'=\big(\overline{d}_{\set{-g,\dots,-1}}^{-1}\overline{d}_{\set{1,\dots,g}}\big)(\beta_g)=\beta_g
\]
by \ref{push}. This concludes the proof. \qedhere
\end{proof}

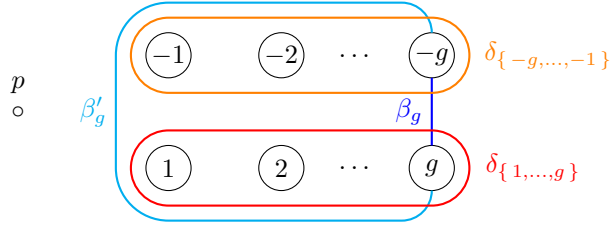
\begin{figure}
\centering
\begin{tikzpicture}
\draw (1.5,.5) circle (3mm);
\draw (1.5,2) circle (3mm);
\draw (3,.5) circle (3mm);
\draw (3,2) circle (3mm);
\node at (4,.5) {$\dots$};
\node at (4,2) {$\dots$};
\draw (5,.5) circle (3mm);
\draw (5,2) circle (3mm);
\node at (1.5,.5) {$1$};
\node at (3,.5) {$2$};
\node at (5,.5) {$g$};
\node at (1.5,2) {$-1$};
\node at (3,2) {$-2$};
\node at (5,2) {$-g$};
\draw [blue, thick] (5,.8) -- (5,1.7);
\draw [cyan, thick] (5,.2) arc (0:-90:.4 and .4) -- (1.5,-.2) arc (-90:-180:.7 and .7) -- (.8,2) arc (180:90:.7 and .7) -- (4.6,2.7) arc(90:0:.4 and .4);
\draw [red, thick] (5,0) -- (1.5,0) arc (270:90:.5 and .5) -- (5,1) arc (90:-90:.5 and .5);
\draw [orange, thick] (5,1.5) -- (1.5,1.5) arc (270:90:.5 and .5) -- (5,2.5) arc (90:-90:.5 and .5);
\node [blue] at (4.7,1.25) {$\beta_g$};
\node [cyan] at (.5,1.25) {$\beta_g'$};
\node [red] at (6.35,.5) {$\delta_{\set{1,\dots,g}}$};
\node [orange] at (6.55,2) {$\delta_{\set{-g,\dots,-1}}$};
\node at (-.5,1.6) {$p$};
\node at (-.5,1.25) {$\circ$};
\end{tikzpicture}
\caption{Pairs of curves that bound an annulus containing the marked point $p$.}
\label{fig:spin}
\end{figure}

\begin{remark}
\thlabel{rem:0lant}
By \cite{ger}, all the relations appearing in \thref{thm:presD} and \thref{thm:presDC} can be expressed in terms of braids, chains and lanterns. This is obvious for \ref{t01} and \ref{5c}. It follows from the remarks after (\ref{hyp3}) that \ref{hypg} is the combination of two positive $7$-chains and two negative $6$-chains. Moreover, \ref{boh}-\ref{a2} are Artin relations, so they are clearly products of lanterns with total exponent zero modulo braids. 

For the remaining relations, recall from \thref{rem:dI} that if $\abs{I}=n$, then $\overline{d}_I$ can be written as a product of $\overline{d}_{i,j}$ using exactly $(n-1)(n-2)/2$ lanterns. Since each $\overline{d}_{i,j}$ can be written as a product of admissible twists by taking the product with a negative lantern, we see that $\overline{d}_{I}$ can be factored using $1-\abs{I}$ lanterns, counted with multiplicity. For example, according to \thref{prop:fake} a fake $3$-chain involves a $3$-chain and $6$ lanterns. This is clear for the first two relations of \ref{3c}. For (\ref{eq:3c}), we have $2(3-(j-i))$ lanterns in the left hand side, and $-2(j-i)$ lanterns on the right hand side, so this remains true.

Counting the lanterns with signs in this way, we see that every relation appearing in \ref{penta}-\ref{triangles} and \ref{push} involves a total amount of zero lanterns. Notice that (\ref{lant}) comes from a lantern in $\Modd(\Sigma_g)$, so it involves an extra (positive) lantern. 

For \ref{push}, we obtain a total of $1-2g+2g=1$ lantern. Notice that since the curve $\delta_{\set{-g,\dots,g}}$ bounds a disk on $\Sigma_g$, the remaining lantern has a trivial boundary component.
\end{remark}

\begin{corollary}[Randal-Williams \cite{rw2}, Sierra \cite{sierra}]
\label{cor:abh}
\thlabel{cor:highab}
The abelianization of the even spin mapping class group is $\sfrac{\Z}{4\Z}$ for $g \ge 4$.
\end{corollary}

\begin{proof}
Take the presentation of \thref{thm:presD} and add all commutators. Relation \ref{t01} implies that all the generators become equal to some $x$. Relations \ref{boh}-\ref{a2} become trivial. From relation \ref{5c} we find $28x=0$, while relation \ref{hypg} gives $56x=0$. 

For the remaining relations, we have to write the classes of $s_i$, $t_i$, $a_i^2$, $\overline{d}_{I}$ and $\overline{r}_{i,j}$ in the abelianization. By \thref{lem:da}, $\overline{d}_{1,2}$ and $a_1^2$ become zero, and so do their conjugates $a_i^2$ and $\overline{d}_{i,j}$ with $i \ne j$. As a consequence, we get $s_j=t_j=\overline{r}_{1,j}=2x$ for all $j$. Finally, $\overline{d}_{-1,1}$ and its conjugates $\overline{d}_{-j,j}$ become equal to $4x$. Hence, $\overline{d}_{\set{i,\dots,j}}=4\,\abs{i}\,x$ and $\overline{r}_{i,j}=4\,\abs{i}\,x+2$ for $i<0$ and $j>-i$. 

Hence, relations \ref{3c} give $4x=0$, while \ref{rijrel}-\ref{push} do not give other restriction to the order of $x$. This concludes the proof. \qedhere
\end{proof}

\subsection{Surfaces of low genus}

In this section, we give presentations of the even spin mapping class group for surfaces of genus $1$, $2$ and $3$. We start from the case of the torus, where a presentation can be derived by hand.

\begin{proposition}
\thlabel{prop:g1}
The even spin mapping class group of a torus is given by
\[
\Modd(\Sigma_1)[\phi]=\Braket{a^2,b|(a^2b)^2=(ba^2)^2,\,(a^2b)^4=1},
\]
where $a^2=t_{\alpha_1}^2$ and $b=t_{\beta_1}$ in the notation of Figure \ref{fig:q}.
\end{proposition}

\begin{proof}
It suffices to show that $\Modd(\Sigma_1)[\phi]$ is generated by $a^2$ and $b$; then the statement follows easily using the Nielsen-Schreier method. 

An element $\varphi\in\Modd(\Sigma_1)$ can be represented by a word $a^{k_1}b^{\ell_1} \dots a^{k_n}b^{\ell_n}$, for some integers $k_1,\ell_1,\dots,k_n,\ell_n$. If $\varphi$ preserves $\phi$, we prove that $\varphi$ can be written as a product of $a^2$ and $b$ by induction on $n$. The base case is \thref{lem:elements}(b). For the inductive step, observe that if $\varphi$ preserves $\phi$ then also $\varphi b^{-\ell_n}=a^{k_1}b^{\ell_1} \dots b^{\ell_{n-1}}a^{k_n}$ does. If $k_n$ is even, we conclude by induction. Otherwise,
\[
\varphi b^{-\ell_n}a^{-k_n+1}=a^{k_1}b^{\ell_1} \dots a^{k_{n-1}}b^{\ell_{n-1}}abb^{-1}=a^{k_1}b^{\ell_1} \dots a^{k_{n-1}+1}ba^{\ell_{n-1}}b^{-1}
\]
preserves $\phi$. If $\ell_{n-1}$ is even, we conclude by induction. Otherwise, 
\[
\varphi b^{-\ell_n}a^{-k_n+1}ba^{-\ell_{n-1}+1}=a^{k_1}b^{\ell_1} \dots b^{\ell_{n-2}}a^{k_{n-1}+1}ba=a^{k_1}b^{\ell_1} \dots b^{\ell_{n-2}+1}ab^{k_{n-1}+1}
\]
preserves $\phi$, and we conclude by induction. \qedhere
\end{proof}

A similar reasoning can be done for genus $2$, but in this case we apply \thref{thm:mcgs} and \thref{thm:presDC}.

\begin{corollary}
\thlabel{cor:g2}
The even spin mapping class group $\Modd(\Sigma_2)[\phi]$ admits a presentation with generators $a_1^2$, $b_1$, $t_1$, $\overline{d}_{1,2}$, and the following relations:
\begin{enumerate}[label=(\roman*)]
\item $[a_1^2,t_1*a_1^2]=1$, $[a_1^2,\overline{d}_{1,2}]=1$, $[t_1*a_1^2,b_1]=1$ and $[b_1,(t_1\overline{d}_{1,2}^{-1})*b_1]=1$;\label{aaaa}
\item $R_4(a_1^2,b_1)$ and $R_4(b_1a_1^2b_1,t_1)$;\label{r4}
\item $t_1^2=\overline{d}_{1,2}^2$; \label{t2d2}
\item $R_2(t_1,\overline{d}_{1,2})$ and $R_4(b_1a_1^2b_1,\overline{d}_{1,2})$; \label{rd12}
\item $b_1=(t_1b_1a_1^2b_1t_1)*b_1$, $[a_1^2\overline{d}_{1,2}, b_1\overline{d}_{1,2}a_1^2b_1]=1$ and 
\[
t_1b_1\overline{d}_{1,2}a_1^2b_1\overline{d}_{1,2}
=a_1^{-2}t_1\overline{d}_{1,2}^{-1}b_1\overline{d}_{1,2}a_1^2b_1;\label{r}
\]
\item $(t_1\overline{d}_{1,2}^{-1}b_1\overline{d}_{1,2}a_1^2b_1\overline{d}_{1,2}t_1^{-1})^2=\overline{d}_{1,2}^{-2}a_1^{-2}(t_1*a_1^2)^{-2}$ \label{r3c};
\item $b_1^{-1}a_1^2b_1\overline{d}_{1,2}^{-1}b_1^{-1}\overline{d}_{1,2}^{-2}a_1^{-2}b_1=\overline{d}_{1,2}^{-1}(t_1*a_1^2)$ and \label{tri2}
\[
\overline{d}_{1,2}^{-1}b_1^{-1}\overline{d}_{1,2}a_1^2b_1\overline{d}_{1,2}=b_1^{-1} \big((t_1^{-1}b_1^{-1}a_1^{-2}b_1^{-1})*\overline{d}_{1,2}\big) a_1^2b_1;
\]
\item $t_1\overline{d}_{1,2}^{-1}b_1\overline{d}_{1,2}a_1^4b_1\overline{d}_{1,2}t_1^{-1} b_1 t_1a_1^2 b_1 \overline{d}_{1,2}^{-1}t_1^{-1} b_1 t_1 b_1 a_1^2(t_1*a_1^2)=1$;\label{penta2}
\item $\overline{d}_{1,2}=(b_1a_1^2b_1t_1b_1a_1^2b_1)*\overline{d}_{1,2}$;\label{push2}
\item $\overline{d}_{-2,-1,1,2}a_1^2(t_1*a_1^2)=1$.
\end{enumerate}
\end{corollary}

\begin{proof}
Write the presentation of \thref{thm:mcgs} for $g=2$, add relations \ref{Delta} and \ref{push} and remove $a_2^2=t_1*a_1^2$ and $s_1=b_1a_1^2b_1$. We can eliminate all $\overline{d}_{i,j}$ except for $\overline{d}_{1,2}$ using \ref{s2} and \ref{a8}. Then, from relations \ref{aiaj}, \ref{atiti}, \ref{asiti} and \ref{sai} only \ref{aaaa} and \ref{r4} survive. Relation \ref{s2} reduces to \ref{t2d2} by \ref{push2}. Relations \ref{a8} and \ref{pb} reduce to \ref{rd12} as in Section \ref{reldij}. Relations \ref{backtracking} and \ref{diffwr} become \ref{r} and \ref{r3c}, while \ref{square} follows from \ref{t2d2} as in Subsection \ref{reldij}. Finally, \ref{triangles} and \ref{pentagon} reduce to \ref{tri2} and \ref{penta2}. \qedhere
\end{proof}

Finally, we turn to genus $3$. In this case, we are able to prove that the even spin mapping class group is generated by admissible Dehn twists, which was not previously known. Note that by \cite{ham}, the intersection graph of the curves of the generating set cannot be a tree. 

Our generating set will be given by the Dehn twists of \thref{thm:presD} along with $z_1,z_2$, that are the twists along the corresponding curves of Figure \ref{fig:zi}. We get the following restatement of \thref{lem:darel}.

\begin{lemma}
\thlabel{lem:daz}
The following relations hold in $\Modd(\Sigma_3)[\phi]$, in the notation of \thref{lem:darel} and Figure \ref{fig:zi}:
\begin{enumerate}[label=\textit{(\roman*)}]
\item $\overline{d}_{1,2}=b_3m_2z_1^{-1}m_3^{-1}$;
\item $\overline{d}_{1,2}^{-1}a_1^{-2}=\eta_3b_3z_2^{-1}m_4^{-1}$;
\item $a_1^2=m_4z_2\eta_3^{-1}b_3^{-1}m_3z_1b_3^{-1}m_2^{-1}$.
\end{enumerate}
\end{lemma}

\begin{figure}
\centering
\begin{tikzpicture}
\draw (.5,0) to [out=0, in=-90] (1,2) to [out=90, in=180] (2,3.5) to [out=0, in=90] (3,2) to [out=-90, in=180] (3.5,1);
\draw (2,2.4) to [out=-60, in=60] (2,1.2);
\draw (2.12,2.7) to [out=-120, in=120] (2.12,.9);
\draw (3.5,0) to [out=0, in=-90] (4,2) to [out=90, in=180] (5,3.5) to [out=0, in=90] (6,2) to [out=-90, in=180] (6.5,1);
\draw (5,2.4) to [out=-60, in=60] (5,1.2);
\draw (5.12,2.7) to [out=-120, in=120] (5.12,.9);
\draw (6.5,0) to [out=0, in=-90] (7,2) to [out=90, in=180] (8,3.5) to [out=0, in=90] (9,2) to [out=-90, in=180] (9.5,1);
\draw (8,2.4) to [out=-60, in=60] (8,1.2);
\draw (8.12,2.7) to [out=-120, in=120] (8.12,.9);
\draw (.9,2) to [out=180, in=90] (-.5,.5) to [out=-90, in=180] (.9,-1) to [out=0, in=180] (9.1,-1) to [out=0, in=-90] (10.5,.5) to [out=90, in=0] (9.1,2);
\draw (1.95,2) -- (2.05,2);
\draw (3.1,2) -- (3.9,2);
\draw (4.95,2) -- (5.05,2);
\draw (6.1,2) -- (6.9,2);
\draw (7.95,2) -- (8.05,2);
\draw [gray, thick] (2,3.5) to [out=-50, in=50] (2,2.4);
\draw [gray, thick, dashed] (2,3.5) to [out=-130, in=130] (2,2.4);
\draw [gray, thick] (5,3.5) to [out=-50, in=50] (5,2.4);
\draw [gray, thick, dashed] (5,3.5) to [out=-130, in=130] (5,2.4);
\draw [gray, thick, dashed] (1.03,1.3) to [out=20, in=160] (1.93,1.3);
\draw [gray, thick, dashed] (4.03,1.3) to [out=20, in=160] (4.93,1.3);
\draw [orange, thick, dashed] (2.04,1.3) to [out=20, in=160] (3.11,1.3);
\draw [gray, thick] (4.03,1.3) to [out=-60, in=0] (3.4,-.2) to [out=180, in=-100] (1.93,1.3);
\draw [gray, thick] (1.03,1.3) to [out=-80, in=180] (3,-.6) to [out=0, in=-110] (4.93,1.3);
\draw [orange, thick] (3.9,1.65) to [out=180, in=0] (3.3,.9) to [out=180,in=-100] (3.11,1.3);
\draw [purple, thick] (3.92,1.5) to [out=180, in=0] (3.3,.8) to [out=180,in=-100] (3.05,1.5);
\draw [purple, thick] (3.94,1.35) to [out=180, in=10] (3.3,.7) to [out=190,in=-80] (2.12,1.5);
\draw [cyan, thick] (1.9,1.8) -- (2.11,1.8);
\draw [cyan, thick] (3.1,1.8) -- (3.9,1.8);
\draw [cyan, thick] (4.9,1.8) -- (5.11,1.8);
\draw [orange, thick] (4.91,1.65) -- (5.1,1.65);
\draw [purple, thick] (4.93,1.5) -- (5.07,1.5);
\draw [purple, thick] (4.98,1.35) -- (5.02,1.35);
\draw [cyan, thick] (6.1,1.8) -- (6.9,1.8);
\draw [orange, thick] (6.1,1.65) -- (6.9,1.65);
\draw [purple, thick] (6.1,1.5) -- (6.9,1.5);
\draw [cyan, thick] (7.9,1.8) -- (8.11,1.8);
\draw [orange, thick] (7.91,1.65) -- (8.1,1.65);
\draw [purple, thick] (7.93,1.5) -- (8.07,1.5);
\draw [cyan, thick]  (9.1,1.8) to [out=0, in=0] (9.4,.5) to [out=180, in=-90] (8.5,1.8) to [out=90, in=0] (8,2.9) to [out=180, in=90] (7.5,1.8) to [out=-90, in=0] (6.5,-.55) -- (5.5,-.55) to [out=180, in=-100] (4.88,1.5);
\draw [cyan, thick, dashed] (4.02,1.5) to [out=20, in=160] (4.88,1.5);
\draw [cyan, thick] (4.02,1.5) to [out=-60, in=0] (3.5,-.8) -- (.9,-.8) to [out=180, in=-90] (0,.5) to [out=90, in=180] (.9,1.8);
\draw [orange, thick] (9.1,1.65) to [out=0, in=0] (9.4,.65) to [out=180, in=-90] (8.65,1.8) to [out=90, in=0] (8,3.05) to [out=180, in=90] (7.35,1.8) to [out=-90, in=0] (6.5,-.4) -- (3.5,-.4) to [out=180, in=-80] (2.04,1.3);
\draw [purple, thick, dashed] (2.12,1.5) to [out=20, in=160] (3.05,1.5);
\draw [purple, thick] (5.04,1.3) to [out=-80, in=180] (6.5,-.2) to [out=0, in=-90] (7.2,1.8) to [out=90, in=180] (8,3.2) to [out=0, in=90] (8.8,1.8) to [out=-90, in=180] (9.4,.8) to [out=0, in=0] (9.1,1.5);
\draw [purple, thick, dashed] (5.04,1.3) to [out=20, in=160] (6.11,1.3);
\draw [purple, thick] (6.15,1.35) to [out=0, in=0] (6.5,.8) to [out=180,in=-110] (6.11,1.3);
\draw [blue, thick] (8,2.75) to [out=0, in=90] (8.35,1.8) to [out=-90, in=0] (8,.85) to [out=180, in=-90] (7.65,1.8) to [out=90, in=180] (8,2.75);
\node [cyan] at (7.6,-.2) {$z_1$};
\node [orange] at (4.7,-.6) {$m_3$};
\node [purple] at (6.5,.55) {$m_2$};
\node [gray] at (1.6,-.4) {$\delta_{1,2}$};
\node [gray] at (1.55,2.8) {$\alpha_1$};
\node [gray] at (4.55,2.8) {$\alpha_2$};
\node [blue] at (7.8,.6) {$b_3$};
\end{tikzpicture}

\bigskip
\begin{tikzpicture}
\draw (.5,0) to [out=0, in=-90] (1,2) to [out=90, in=180] (2,3.5) to [out=0, in=90] (3,2) to [out=-90, in=180] (3.5,1);
\draw (2,2.4) to [out=-60, in=60] (2,1.2);
\draw (2.12,2.7) to [out=-120, in=120] (2.12,.9);
\draw (3.5,0) to [out=0, in=-90] (4,2) to [out=90, in=180] (5,3.5) to [out=0, in=90] (6,2) to [out=-90, in=180] (6.5,1);
\draw (5,2.4) to [out=-60, in=60] (5,1.2);
\draw (5.12,2.7) to [out=-120, in=120] (5.12,.9);
\draw (6.5,0) to [out=0, in=-90] (7,2) to [out=90, in=180] (8,3.5) to [out=0, in=90] (9,2) to [out=-90, in=180] (9.5,1);
\draw (8,2.4) to [out=-60, in=60] (8,1.2);
\draw (8.12,2.7) to [out=-120, in=120] (8.12,.9);
\draw (.9,2) to [out=180, in=90] (-.5,.5) to [out=-90, in=180] (.9,-1) to [out=0, in=180] (9.1,-1) to [out=0, in=-90] (10.5,.5) to [out=90, in=0] (9.1,2);
\draw (1.95,2) -- (2.05,2);
\draw (3.1,2) -- (3.9,2);
\draw (4.95,2) -- (5.05,2);
\draw (6.1,2) -- (6.9,2);
\draw (7.95,2) -- (8.05,2);
\draw [gray, thick] (2,3.5) to [out=-50, in=50] (2,2.4);
\draw [gray, thick, dashed] (2,3.5) to [out=-130, in=130] (2,2.4);
\draw [gray, thick] (5,3.5) to [out=-50, in=50] (5,2.4);
\draw [gray, thick, dashed] (5,3.5) to [out=-130, in=130] (5,2.4);
\draw [gray, thick, dashed] (1.03,1.3) to [out=20, in=160] (1.93,1.3);
\draw [gray, thick, dashed] (4.03,1.3) to [out=20, in=160] (4.93,1.3);
\draw [gray, thick] (4.03,1.3) to [out=-60, in=0] (3.4,-.2) to [out=180, in=-100] (1.93,1.3);
\draw [gray, thick] (1.03,1.3) to [out=-80, in=180] (3.4,-.8) to [out=0,  in=-100] (4.93,1.3);
\draw [blue, thick] (8,2.75) to [out=0, in=90] (8.35,1.8) to [out=-90, in=0] (8,.85) to [out=180, in=-90] (7.65,1.8) to [out=90, in=180] (8,2.75);
\draw [teal, thick, dashed] (1.02,1.5) to [out=20, in=160] (1.88,1.5);
\draw [teal, thick] (1.02,1.5) to [out=-70, in=180] (3.4,-.7) to [out=0, in=-80] (4.65,1.8) to [out=100, in=180] (5,2.8) to [out=0, in=90] (5.4,1.8) to  [out=-90, in=180] (6.5,-.2) to [out=0, in=-85] (7.2,1.8) to [out=95, in=180] (8,3.2) to [out=0, in=90] (8.8,1.8) to [out=-90, in=180] (9.4,.85) to [out=0, in=0] (9.1,1.6);
\draw [teal, thick] (7.9,1.6) -- (8.1,1.6);
\draw [teal, thick] (6.92,1.55) to [out=210, in=0] (6.5,.6) to [out=180,in=-90] (5.85,1.8) to [out=90, in=0] (5,3.25) to [out=180, in=100] (4.2,1.8) to [out=-80, in=0] (3.4,-.3) to [out=180, in=-110] (1.88,1.5);
\draw [magenta, thick] (7.9,1.75) -- (8.1,1.75);
\draw [magenta, thick] (6.91,1.7) to [out=200, in=0] (6.4,.75) to [out=180, in=-100] (6.05,1.5);
\draw [magenta, thick, dashed] (5.12,1.5) to [out=20, in=160] (6.05,1.5);
\draw [magenta, thick] (5.12,1.5) to [out=70, in=0] (5,2.95) to [out=180, in=100] (4.5,1.8) to [out=-80, in=0] (3.4,-.6) to [out=180, in=-70] (1.01,1.7);
\draw [magenta, thick, dashed] (1.85,1.7) to [out=160, in=20] (1.01,1.7);
\draw [magenta, thick] (1.85,1.7) to [out=-120, in=180] (3.4,-.4) to [out=0, in=-80] (4.35,1.8) to [out=100, in=180] (5,3.1) to [out=0, in=90] (5.7,1.8) to  [out=-90, in=180] (6.5,-.35) to [out=0, in=-85] (7.35,1.8) to [out=95, in=180] (8,3.05) to [out=0, in=90] (8.65,1.8) to [out=-90, in=180] (9.4,.7) to [out=0, in=0] (9.1,1.75);
\draw [green, thick, dashed] (5.04,1.3) to [out=20, in=160] (6.11,1.3);
\draw [green, thick] (5.04,1.3) to [out=-80, in=180] (6.5,-.5) to [out=0, in=-85] (7.5,1.8) to [out=95, in=180] (8,2.9) to [out=0, in=90] (8.5,1.8) to [out=-90, in=180] (9.4,.55) to [out=0, in=0] (9.1,1.9);
\draw [green, thick] (7.9,1.9) -- (8.1,1.9);
\draw [green, thick] (6.9,1.9) to [out=180, in=0] (6.4,.9) to [out=180,in=-110] (6.11,1.3);
\node [gray] at (1.6,-.4) {$\delta_{1,2}$};
\node [gray] at (2,3.7) {$\alpha_1$};
\node [gray] at (5,3.7) {$\alpha_2$};
\node [magenta] at (1.5,2) {$z_2$};
\node [teal] at (6.3,.3) {$m_4$};
\node [green] at (7.2,-.6) {$\eta_3$};
\node [blue] at (7.9,.6) {$b_3$};
\end{tikzpicture}
\caption{Additional generators for $\Modd(\Sigma_3)$ and corresponding lanterns. Here, the gray curves have spin value 1, and the other curves are admissible.}
\label{fig:zi}
\end{figure}
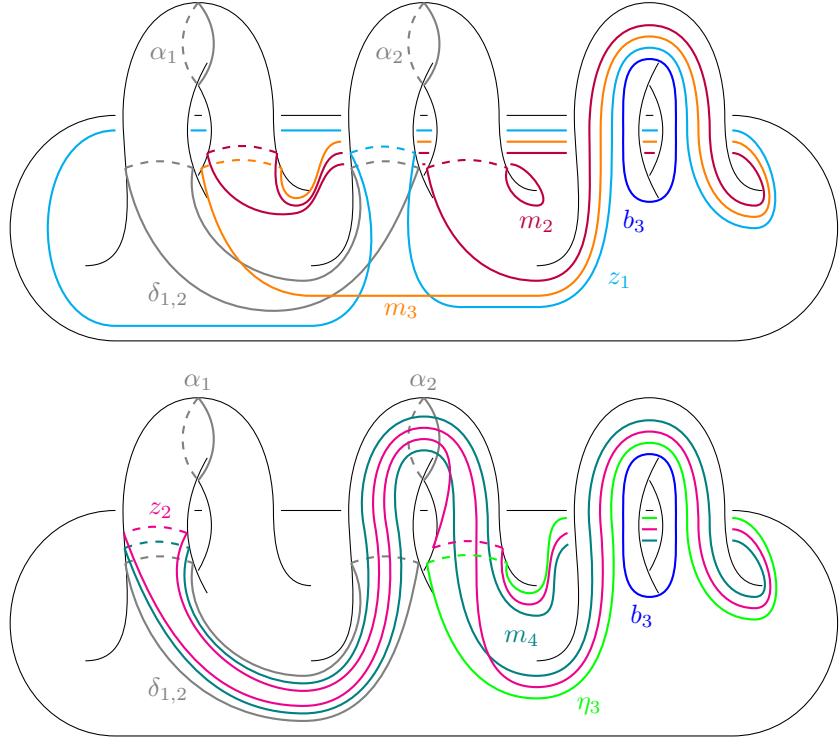

We obtain the following immediate consequence of \thref{thm:presDC} (recall also \thref{rem:sh}).

\begin{corollary}
\thlabel{cor:pres3}
The even spin mapping class group $\Modd(\Sigma_3)[\phi]$ admits a presentation with generators $b_1,b_2,b_3$, $\xi_1,\xi_{2}$, $\eta_2,\eta_3$, $z_1,z_2$ and the following relations:
\begin{enumerate}[label=(\roman*)]
\item obvious commutativity and braid relations between the generators;
\item relations \ref{5c}, \ref{3c} and \ref{dbar}-\ref{dttri};
\item the hyperelliptic relation $(b_3\xi_2\xi_1b_2\eta_3\eta_2b_1^2\eta_2\eta_3b_2\xi_1\xi_2b_3)^2=1$;
\item $[z_2,m_2]=1$ and $[m_3z_1,m_4z_2]=1$.
\end{enumerate}
\end{corollary}

We can then compute the abelianization of the even spin mapping class group in every genus. Sierra \cite{sierra} has obtained similar results for $\Modd(\Sigma_g^1)[\phi]$ using GAP.

\begin{corollary}
\label{cor:ab}
\thlabel{cor:lowab}
The abelianization of $\Modd(\Sigma_g)[\phi]$ for $g \le 3$ is the following:
\[
H_1\big(\Modd(\Sigma_g)[\phi];\Z\big)=\begin{cases}
\Z\oplus\sfrac{\Z}{4\Z} & \text{if $g=1$},\\
\Z\oplus\sfrac{\Z}{2\Z} & \text{if $g=2$},\\
\sfrac{\Z}{4\Z} & \text{if $g=3$}.
\end{cases}
\]
\end{corollary}

\begin{proof}
We start from $g=1$. Consider the presentation of \thref{prop:g1} and add the commutator $[a^2,b]=1$. Relation $R_4(a^2,b)$ becomes redundant, and $a^2b$ has order $4$. This implies the statement. 

For $g=2$, we start from the presentation of \thref{cor:g2} and add all commutators. Relation \ref{penta2} gives $t_1=(b_1^6a_1^{10})^{-1}$, and from \ref{r} we get $a_1^2=\overline{d}_{1,2}^{-2}$. Hence, the abelianization is generated by $d:=\overline{d}_{1,2}$ and $b:=b_1$. Relation \ref{r3c} yields $(b^2d^{-3})^2=1$, and all the other relations become superfluous.

For $g=3$ the proof is exactly the same as that of \thref{cor:highab}. \qedhere
\end{proof}

\end{document}